\documentclass[11pt]{article}
\usepackage{amsmath,amsfonts,amssymb,amsthm,enumerate,graphicx}
\usepackage{tikz,authblk}
\usepackage{subfig}
\usepackage{url}
\usepackage{float}
\usepackage{easyReview}

\newtheorem{theorem}{{\bf Theorem}}[section]

\newtheorem{proposition}[theorem]{{\bf Proposition}}
\newtheorem{remark}[theorem]{{\bf Remark}}

\newcommand{\ZZ}{ \ensuremath{\mathbb{Z}}}

\begin{document}

\title{Quotient maps of 2,3-uniform tilings of the plane on the torus}

	\author[1] {Marbarisha M. Kharkongor}
 	\author[2] {Debashis Bhowmik}
 	\author[1] {Dipendu Maity}
 	\affil[1]{Department of Sciences and Mathematics,
 		Indian Institute of Information Technology Guwahati, Bongora, Assam-781\,015, India.\linebreak
 		\{marbarisha.kharkongor, dipendu\}@iiitg.ac.in/\{marbarisha.kharkongor, dipendumaity\}@gmail.com.}
 	\affil[2]{Department of Mathematics, Indian Institute of Technology Patna, Patna 801\,106, India.
 		debashisiitg@gmail.com.}

\date{\today}

\maketitle

\begin{abstract}

A 2-uniform tiling is an edge-to-edge tiling by regular polygons having $2$ distinct transitivity classes of vertices. 
There are 20 distinct 2-uniform tilings (these are of $14$ different types) on the plane, and since the plane is the universal cover of the torus, it is natural to explore maps on the torus that correspond to the 2-uniform tilings. In this article, we discuss that if a map is the quotient of a plane's $2$-uniform lattice then what would be the bounds of the number of vertex orbits.

A 3-uniform tiling is an edge-to-edge tiling by regular polygons having $3$ distinct transitivity classes of vertices. There are $61$ distinct $3$-uniform tilings on the plane.   
In this article, we discuss that if a map is the quotient of a plane's $3$-uniform lattice then what would be the bounds of the number of vertex orbits.
\end{abstract}

\noindent {\small {\em MSC 2010\,:} 52C20, 52B70, 51M20, 57M60.

}

\section{Introduction}

A map is a connected $2$-dimensional cell complex on a surface. Equivalently, it is a cellular embedding of a connected graph on a surface. 
For a map $\mathcal{K}$, let $V(\mathcal{K})$ be the vertex set of $\mathcal{K}$ and $u\in V(\mathcal{K})$. The faces containing  $u $ form a cycle (called the {\em face-cycle} at  $u $)  $C_u $ in the dual graph  of  $\mathcal{K} $. That is,  $C_u $ is of the form  $(F_{1,1}\mbox{-}\cdots\mbox{-}F_{1,n_1})\mbox{-}\cdots\mbox{-}(F_{k,1}\mbox{-}\cdots \mbox{-}F_{k,n_k})\mbox{-}$ $F_{1,1} $, where  $F_{i,\ell} $ is a  $p_i $-gon for  $1\leq \ell \leq n_i $,  $1\leq i \leq k $,  $p_r\neq p_{r+1} $ for  $1\leq r\leq k-1 $ and  {$p_k\neq p_1 $}. In this case, the vertex $u$ is said to be of type $ [p_1^{n_1}, \dots, p_k^{n_k}]$ (addition in the suffix is modulo  $k $).
A map  $\mathcal{K} $ is said to be {\em k-semiequivelar} (or in short, $semiequivelar$) if  $V(\mathcal{K}) = V_1 \sqcup V_2 \sqcup \dots \sqcup V_k$ such that the type of the vertices of $V_i$ ($1 \le i \le k$) is same and  the types of the vertices in $V_i$ and $V_j$ ($i \ne j$, $1 \le i, j \le k$) are  different. In this article, we use semiequivelar to mean $k$-semiequivelar for some $k$. 
A semiequivelar map is said to be an $equivelar$ map if it consists of single type of faces with all the vertices are of same type.   

A {\em $2$-uniform tiling} is a tiling of the plane having $2$ distinct transitivity classes of vertices. A vertex-transitive map (or tiling) is a map (or tiling) on a closed surface (or on the plane) on which the automorphism group acts transitively on the set of vertices. A $2$-uniform tiling or map will have vertices that we could label $X$, and others that we could label $Y$. Each $X$ vertex can be mapped onto every other $X$ vertex, but cannot be mapped to any $Y$ vertex. 


A {\em semiregular} tiling of $\mathbb{R}^2$ is also known as {\em Archimedean}, or {\em homogeneous}, or {\em uniform} tiling. In \cite{GS1977}, Gr\"{u}nbaum and Shephard showed that there are exactly eleven types of Archimedean tilings on the plane. These types are $[3^6]$, $[3^4,6^1]$, $[3^3,4^2]$,  $[3^2,4^1,3^1,4^1]$, $[3^1,6^1,3^1,6^1]$, $[3^1,4^1,6^1,4^1]$, $[3^1,12^2]$, $[4^4]$, $[4^1,6^1,12^1]$, $[4^1,8^2]$, $[6^3]$.
Clearly, a {\em semiregular} tiling on $\mathbb{R}^2$ gives a semiequivelar map on $\mathbb{R}^2$. But, there are semiequivelar maps on the plane which are not (not isomorphic to) an Archimedean tiling. In fact, there exists $[p^q]$ equivelar maps on $\mathbb{R}^2$ whenever $1/p+1/q<1/2$ (see in \cite{CM1957}, \cite{FT1965}). Thus, we have

\begin{proposition} \label{prop:plane}
There are infinitely many types of equivelar maps on the plane $\mathbb{R}^2$.
\end{proposition}

We know that the plane is the universal cover of the torus. Since there are infinitely many equivelar maps on the plane, it is natural to ask what are the other types of semiequivelar maps that exist on the torus.  Here we have the following result. 

\begin{proposition} \cite{DM2017, DM2018} \label{theo:GrSh}
Let $X$ be a semiequivelar map on a surface $M$. If $M$ is the torus then the type of $X$ is $[3^6]$, $[6^3]$, $[4^4]$, $[3^4,6^1]$, $[3^3,4^2]$, $[3^2,4^1,3^1,4^1]$,  $[3^1,6^1,3^1,6^1]$,
$[3^1,4^1,6^1,4^1]$,  $[3^1,12^2]$, $[4^1,8^2]$  or $[4^1,6^1,12^1]$.
\end{proposition}



We know that all the Archimedean {tilings} are vertex-transitive. But, it is not true on the torus. Here, we know the following.   

\begin{proposition} \cite{DM2017, DM2018} \label{prop:36&44}
Let $X$ be an equivelar map on the torus. If the type of $X$ is $[3^6]$, $[4^4]$, $[6^3]$ or $[3^3,4^2]$ then $X$ is vertex-transitive. If the type is $[3^2,4^1,3^1,4^1]$,  $[3^1,6^1,3^1,6^1]$, $[3^1,4^1,6^1,4^1]$,  $[3^1,12^2]$, $[4^1,8^2]$, $[3^4,6^1]$  or $[4^1,6^1,12^1]$, then there exists a semiequivelar toroidal map which is not vertex-transitive.
\end{proposition}

The {\em $2$-uniform} tilings of the plane $\mathbb{R}^2$ are the generalization of vertex-transitive tilings on the plane. We know from \cite{GS1977, GS1981, Otto1977} that there are {twenty} 2-uniform tiling of types
\begin{align*}
   & [3^6;3^3,4^{2}], [3^{6};3^2,4^1,3^1,4^1],  [3^4,6^1;3^2,6^{2}], [3^{3}, 4^2;3^1,4^1,6^1,4^1], [3^3, 4^2;3^2,4^1,3^1,4^1],\\
   &  [3^{6}; 3^2, 4^1,12^1], [3^1, 4^1, 6^1, 4^1; 4^1, 6^1, 12^1], [3^2,4^1,3^1,4^1;3^1, 4^1,6^1,4^1], [3^2,6^2; 3^1, 6^1, 3^1, 6^1],\\
   & [3^1, 4^1, 3^1, 12^1; 3^1, 12^2], [3^1,4^2,6^1; 3^1, 4^1, 6^1, 4^1], [3^1,4^2, 6^1; 3^1, 6^1, 3^1, 6^1], [3^3,4^2;4^4], [3^{6};3^4,6^1].
\end{align*}
 on the plane (see in Section \ref{2uniform}). 
We know that the plane is the universal cover of the torus. We also know that the maps of the above fourteen types also exist on the torus. Thus, it is natural to ask that if $X$ is a map on the torus and $\eta \colon Y \to X$ is {a} covering map where $Y$ is an $2$-uniform tiling on the plane, then what would be the number of orbits of vertices of $X$. Here, we know the following. 


\begin{proposition} \cite{KBM2021} Let $Z$ be a semiequivelar map on the torus and $\eta \colon K_5 \to Z$ be a covering map ($K_5$ given in Section \ref{2uniform}). Let the vertices of $Z$ form $m$ ${ Aut}(Z)$-orbits. Then, $m \le 4$.
\end{proposition}

 In this article, we prove the following.  

\begin{theorem} \label{theo1} Let $X$ be a semiequivelar map on the torus and $\eta \colon Y \to X $ be a covering map. Let the vertices of $X$ form $m$ ${\rm Aut}(X)$-orbits. Let $K_i$ for $1 \le i \le 20$ (in {Section} \ref{2uniform}) denote the $2$-uniform tilings on the plane. 

{\rm (1)} If $Y = K_{1}, K_9, K_{11}, K_{14}$, or $K_{17}$ then $m \le 6$.

\smallskip

{\rm (2)} If $Y = K_{2}, K_7, K_{10},$ or $K_{16}$ then $m \le 4$.

\smallskip

{\rm (3)}  If $Y = K_3, K_4, K_8, K_{12}, K_{13},$ or $K_{15}$ then $m = 2$.  

\smallskip

{\rm (4)} If $Y = K_{18}$ or $K_{19}$ then $m \le 3$.

\smallskip

{\rm (5)} If $Y = K_{6}$ then $m \le 7$.

\smallskip

{\rm (6)} If $Y = K_{20}$ then $m \le 9$.
\end{theorem}

A {\em $3$-uniform tiling} is an edge-to-edge tiling by regular polygons having $3$ distinct transitivity classes of vertices. A vertex-transitive map is a map on a closed surface on which the automorphism group acts transitively on the set of vertices. A $3$-uniform tiling or map will have vertices that we could label $X_1$, $X_2,$ $X_3$. Each $X_i$ vertex can be mapped onto every other $X_i$ vertex, but cannot be mapped to any $X_j$ for $j \neq i$ vertices. 

The {\em $k$-uniform} tilings of the plane $\mathbb{R}^2$ are the generalization of vertex-transitive tilings on the plane. We know that the plane is the universal cover of the torus. 

If $k=3$, we know from \cite{GS1977, GS1981, Otto1977} that there are $61$ $3$-uniform tiling  on the plane.
In this article, we prove the following.  If $X$ is a map on the torus and $\eta \colon Y \to X$ is a covering map where $Y$ is an $3$-uniform tiling on the plane, then what would be the number of orbits of vertices of $X$. Precisely, we have have the following result.  

\begin{theorem} \label{theo2} Let $X$ be a semiequivelar map on the torus and $\eta \colon Y \to X $ be a covering map. Let the vertices of $X$ form $m$ ${\rm Aut}(X)$-orbits. Let $K_i$ for $1 \le i \le 61$ (in {Section} \ref{3uniform}) denote the $3$-uniform tilings on the plane. If $X = K_1, K_2, \dots, or, K_{61}$ then $m \le 15$.
\end{theorem}

\section{Proof of Theorems \ref{theo1}}\label{sec:proofs-1}
 Let $K_i$ (given in Section \ref{2uniform}) be of type $A$, where 
 \begin{align*}
   A   \in &\{ [3^{6};3^4,6^1], [3^6;3^3,4^{2}], [3^{6};3^2,4^1,3^1,4^1], [3^{6}; 3^2,4^1,12^1], [3^6;3^2,6^2],\\
   &  [3^4,6^1;3^2,6^{2}], [3^3, 4^2;3^2,4^1,3^1,4^1], [3^{3}, 4^2;3^1,4^1,6^1,4^1], [3^3,4^2;4^4], \\
   &  [3^2,4^1,3^1,4^1;3^1,4^1,6^1,4^1], [3^2,6^2; 3^1, 6^1, 3^1, 6^1], [3^1, 4^1, 3^1, 12^1; 3^1, 12^2],\\
   &  [3^1,4^2,6^1; 3^1, 4^1, 6^1, 4^1], [3^1,4^2, 6^1; 3^1, 6^1, 3^1, 6^1], [3^1, 4^1, 6^1, 4^1; 4^1, 6^1, 12^1] \}.  
 \end{align*}
 
Gr\"{u}nbaum and G. C. Shephard \cite{GS1977, GS1981} and Kr\"{o}tenheerdt \cite{Otto1977} have discussed the existence and uniqueness of the $2$-uniform tilings $K_i$, $i =1, 2, \dots, 20$ of the plane. 
Thus, we have the following. 

\begin{proposition}\label{prop1}
The $2$-uniform maps $K_{i}$ ($1 \le i \le 20$)  are unique up to isomorphism. 
\end{proposition} 

\begin{remark}
Let $S = \{0, 1, 2, \dots, m-1\}$. Let $S + n := \{0+n, 1+n, 2+n, \dots, m+n-1\}$. Define $S_m = \{S+mk ~:~k \in \mathbb{Z} ~\& ~2 \mid (k-1)\}$. Let $K_m$ be a map on the plane defined as follows. 
Let the vertex set of $K_m$ be $\{(i, j) \in \mathbb{Z}\times\mathbb{Z}\}$ on $\mathbb{R}^2$. The points $(i, j)$ and $(i_1, j_1)$ are connected by an edge if  $(i, j)\sim (i_1, j_1)$. 
Thus, the edges of $K_m$ : 
\begin{align*}
& (1)~ (i, j)\sim (i_1, j_1) \mbox{ if } |i-i_1| = 1 \mbox{ and } j = j_1, \mbox{ or } i = i_1 \mbox{ and } |j-j_1| = 1,\\ 
& (2)~ (i, j)\sim (i+1, j+1), (t, j)\sim (t+1, j-1) \mbox{ if } i \in S_m, j \mbox{ is even and } t \in \mathbb{Z} \setminus S_m. 
\end{align*}
Thus, we get an object $K_m$ for each $m \ge 2$. Clearly, $K_x \not\cong K_y$ for $x \neq y$ since the number of continuous quadrilaterals are different. Therefore, there are infinitely many $2$-semiequivelar tilings of the plane.
\end{remark}

\begin{proof}[Proof of Theorem \ref{theo1}]
Let $X_1$ be a semiequivelar map on the torus that is the quotient of the plane's $2$-uniform lattice $K_1$. Let the vertices of $X_1$ form $m_1$ ${\rm Aut}(X_1)$-orbits. Let $K_1$ be as in Section \ref{2uniform}.
Let $V_{1} = V(K_1)$ be the vertex set of $K_1$. Let $H_{1}$ be the group of all the translations of $K_1$. So, $H_1 \leq {\rm Aut}(K_1)$.

Since $X_1$ is a semiequivelar map on the torus that is the quotient of the plane's $2$-uniform lattice $K_1$ (as by Proposition \ref{prop1} $K_1$ is unique), so, we can assume, there is a polyhedral covering map $\eta_{1} : K_1 \to X_1$ where $X_1 = K_1/\Gamma_{1}$  for some fixed element (vertex, edge or face) free subgroup $\Gamma_{1} \le {\rm Aut}(K_1)$. Hence $\Gamma_{1}$
consists of translations and glide reflections. Since $X_1 =
K_1/\Gamma_{1}$ is orientable, $\Gamma_{1}$ does not contain any glide reflection. Thus $\Gamma_{1} \leq H_{1}$.

 We take the middle point of the line segment joining vertices $a_{0}$ and $a_3$ as the origin $(0,0)$ of $K_1$. Let $\alpha_1 := a_{9} - a_{0}$, $\beta_1 := a_{14}- a_{0}$ and  $\gamma_1 := a_{20} - a_{0}$ in $K_1$. Then $$H_1 := \langle x\mapsto x+\alpha_1, x\mapsto x+\beta_1, x\mapsto x+\gamma_1\rangle.$$ Under the action of $H_1$, vertices of $K_1$ form twelve orbits.
Consider the subgroup $G_1$ of ${\rm Aut}(K_1)$ generated by $H_1$ and the map (the half rotation) $x\mapsto -x$. So,
\begin{align*}
  G_1 & =\{ \alpha : x\mapsto \varepsilon x + m\alpha_1 + n\beta_1 +r\gamma_1 \, : \, \varepsilon=\pm 1, m, n, r \in \ZZ\} \cong H_1\rtimes \mathbb{Z}_2.
\end{align*}
Clearly, under the action of $G_1$, vertices of $K_1$ form six orbits. The orbits are 
\begin{align*}
& O_1 :=\langle a_{0} \rangle, O_2 :=\langle a_{1} \rangle, O_3 :=\langle a_{2} \rangle, O_4 :=\langle b_{18} \rangle, O_5 :=\langle b_{20} \rangle, O_6 :=\langle b_{28} \rangle.
\end{align*}

\noindent {\bf Claim 1.} If $K \leq H_1$ then $K \unlhd G_1$. {The alphabet $K$ has already been used in the statement of \ref{theo1}}

\smallskip

Let $g \in G_1$ and $k\in K$. Then $g(x) = \varepsilon x+ma+nb+rc$ and $k(x) = x + pa+ qb+\ell c$ for some $m, n, r, p, q, \ell \in \mathbb{Z}$ and $\varepsilon\in\{1, -1\}$.
Therefore, 
\begin{align*}
(g\circ k\circ g^{-1})(x) & = (g\circ k)(\varepsilon(x-ma-nb-rc))\\
                          & = g(\varepsilon(x-ma-nb-rc)+pa+qb+\ell c)\\                          
                          & =x-ma-nb-rc+\varepsilon(pa+qb+\ell c)+ma+nb+rc\\
                          & = x+\varepsilon(pa+qb+\ell c)\\
                          & = k^{\varepsilon}(x).
\end{align*}
 Thus, $g\circ k\circ g^{-1} = k^{\varepsilon}\in K$. This completes the claim.

\smallskip

By Claim 1, $\Gamma_1$ is a normal subgroup of $G_1$. Therefore, $G_1/\Gamma_1$ acts on $X_1= K_1/\Gamma_1$.
Since 
\begin{align*}
& O_1 :=\langle a_{0} \rangle, O_2 :=\langle a_{1} \rangle, O_3 :=\langle a_{2} \rangle, O_4 :=\langle b_{18} \rangle, O_5 :=\langle b_{20} \rangle, O_6 :=\langle b_{28} \rangle.
\end{align*}
 are the $G_1$-orbits, it follows that $\eta_1(O_j)$ for $j=1, 2, \dots, 6$ are the $(G_1/\Gamma_1)$-orbits. Since the vertex set of $K_1$ is $\sqcup_{j=1}^{6}\eta_j(O_j)$ and $G_1/\Gamma_1 \leq {\rm Aut}(X_1)$, it follows that the number of ${\rm Aut}(X_1)$-orbits of vertices is $\leq 6$. 
 
 \medskip

Let $X_{9} = K_{9}/\Gamma_{9}$ be a semiequivelar map on the torus for some fixed element (vertex, edge or face) free subgroup $\Gamma_{9} \le {\rm Aut}(K_{9})$. Let the vertices of $X_{9}$ form $m_{9}$ ${\rm Aut}(X_{9})$-orbits.
We take the middle point of the line segment joining vertices $a_0$ and $a_5$ as the origin $(0,0)$ of $K_9$  (see Section \ref{2uniform}). Let  $\alpha_9 := a_{1} - a_0$, $\beta_9 := a_{2} - a_{0}$ and $\gamma_9 := a_{3} - a_{0} \in \mathbb{R}^2$. Similarly as above, define $H_9 := \langle x\mapsto x+\alpha_9, x\mapsto x+\beta_9, x\mapsto x+\gamma_9\rangle$ and 
\begin{align*}
  G_9 & =\{ \alpha : x\mapsto \varepsilon x + m_9\alpha_9 + n_9\beta_9 +r_9\gamma_9   \, : \, \varepsilon=\pm 1, m_9, n_9, r_9 \in \ZZ\} \cong H_9\rtimes \mathbb{Z}_2.
\end{align*}
By the same arguments as above and in Claim 1, $\Gamma_9 \unlhd G_9$ and the number of $G_9/\Gamma_9$-orbits of vertices of $X_9$ is six. Therefore, $G_9/\Gamma_9$ acts on $X_9 = K_9/\Gamma_9$.
Since  $O_1 = \langle a_0 \rangle,$ $O_2 = \langle b_0 \rangle,$  $O_3 = \langle b_1 \rangle,$ $O_4 = \langle b_{33} \rangle$, $O_5 = \langle a_{8} \rangle$, $O_6 = \langle b_{34} \rangle$
 are the $G_9$-orbits, it follows that $O_1 = \langle a_0 \rangle,$ $O_2 = \langle b_0 \rangle,$  $O_3 = \langle b_1 \rangle,$ $O_4 = \langle b_{33} \rangle$, $O_5 = \langle a_{8} \rangle$, $O_6 = \langle b_{34} \rangle$ are the $(G_9/\Gamma_9)$-orbits.  
Since the vertex set of $X_9$ is $\sqcup_{j=1}^6 \eta_9(O_j)$ and $G_9/\Gamma_9 \leq {\rm Aut}(X_9)$, it follows that the number of ${\rm Aut}(X_9)$-orbits of vertices is $\leq 6$. 

 \medskip

Let $X_{11} = K_{11}/\Gamma_{11}$ be a semiequivelar map on the torus for some fixed element (vertex, edge or face) free subgroup $\Gamma_{11} \le {\rm Aut}(K_{11})$. Let the vertices of $X_{11}$ form $m_{11}$ ${\rm Aut}(X_{11})$-orbits. We take the middle point of the line segment joining vertices $a_{0}$ and $a_{3}$ as the origin $(0,0)$ of $K_{11}$. Let  $\alpha_{11} := a_{15} - a_0, \beta_{11} := a_{6} - a_{0}$ and $\gamma_{11} := a_{37} - a_{0} \in \mathbb{R}^2$. Similarly as above, define $H_{11} := \langle x\mapsto x+\alpha_{11}, x\mapsto x+\beta_{11}, x\mapsto x+\gamma_{11}\rangle$ and 
\begin{align*}
  G_{11} & =\{ \alpha : x\mapsto \varepsilon x + m_{11}\alpha_{11} + n_{11}\beta_{11} +r_{11}\gamma_{11}   \, : \, \varepsilon=\pm 1, m_{11}, n_{11}, r_{11} \in \ZZ\} \cong H_{11}\rtimes \mathbb{Z}_2.
\end{align*}
By the same arguments as above and in Claim 1, $\Gamma_{11} \unlhd G_{11}$ and the number of $G_{11}/\Gamma_{11}$-orbits of vertices of $X_{11}$ is six. Therefore, $G_{11}/\Gamma_{11}$ acts on $X_{11} = K_{11}/\Gamma_{11}$.
Since  $O_1 = \langle a_0 \rangle,$ $O_2 = \langle a_1 \rangle,$  $O_3 = \langle a_2 \rangle,$ $O_4 = \langle b_0 \rangle,$ $O_5 = \langle b_2 \rangle,$  $O_6 = \langle b_{46} \rangle$
 are the $G_{11}$-orbits, it follows that $O_1 = \langle a_0 \rangle,$ $O_2 = \langle a_1 \rangle,$  $O_3 = \langle a_2 \rangle,$ $O_4 = \langle b_0 \rangle,$ $O_5 = \langle b_2 \rangle,$  $O_6 = \langle b_{46} \rangle$ are the $(G_{11}/\Gamma_{11})$-orbits.  
Since the vertex set of $X_{11}$ is $\sqcup_{j=1}^6\eta_{11}(O_j)$ and $G_{11}/\Gamma_{11} \leq {\rm Aut}(X_{11})$, it follows that the number of ${\rm Aut}(X_{11})$-orbits of vertices is $\leq 6$. 

\medskip

Let $X_{14} = K_{14}/\Gamma_{14}$ be a semiequivelar map on the torus for some fixed element (vertex, edge or face) free subgroup $\Gamma_{14} \le {\rm Aut}(K_{14})$. Let the vertices of $X_{14}$ form $m_{14}$ ${\rm Aut}(X_{14})$-orbits. We take the middle point of the line segment joining vertices $a_{0}$ and $a_{3}$ as the origin $(0,0)$ of $K_{14}$. Let  $\alpha_{14} := a_9 - a_3$, $\beta_{14} := a_{30} - a_3$ and $\gamma_{14} := a_{34} - a_3 $ $\in \mathbb{R}^2$. Similarly as above, define $H_{14} := \langle x\mapsto x+\alpha_{14}, x\mapsto x+\beta_{14}, x\mapsto x+\gamma_{14}\rangle$ and 
\begin{align*}
  G_{14} & =\{ \alpha : x\mapsto \varepsilon x + m_{14}\alpha_{14} + n_{14}\beta_{14} +r_{14}\gamma_{14}   \, : \, \varepsilon=\pm 1, m_{14}, n_{14}, r_{14} \in \ZZ\} \cong H_{14}\rtimes \mathbb{Z}_2.
\end{align*}
By the same arguments as above and in Claim 1, $\Gamma_{14} \unlhd G_{14}$ and the number of $G_{14}/\Gamma_{14}$-orbits of vertices of $X_{14}$ is six. Therefore, $G_{14}/\Gamma_{14}$ acts on $X_{14} = K_{14}/\Gamma_{14}$.
Since  $O_1 = \langle a_0 \rangle,$ $O_2 = \langle a_1 \rangle,$  $O_3 = \langle a_2 \rangle,$ $O_4 = \langle b_0 \rangle$, $O_5 = \langle b_2 \rangle$, $O_6 = \langle b_4 \rangle$
 are the $G_{14}$-orbits, it follows that $O_1 = \langle a_0 \rangle,$ $O_2 = \langle a_1 \rangle,$  $O_3 = \langle a_2 \rangle,$ $O_4 = \langle b_0 \rangle$, $O_5 = \langle b_2 \rangle$, $O_6 = \langle b_4 \rangle$ are the $(G_{14}/\Gamma_{14})$-orbits.  
Since the vertex set of $X_{14}$ is $\sqcup_{j=1}^6 \eta_{14}(O_j)$ and $G_{14}/\Gamma_{14} \leq {\rm Aut}(X_{14})$, it follows that the number of ${\rm Aut}(X_{14})$-orbits of vertices is $\leq 6$. 

\medskip

Let $X_{17} = K_{17}/\Gamma_{17}$ be a semiequivelar map on the torus for some fixed element (vertex, edge or face) free subgroup $\Gamma_{17} \le {\rm Aut}(K_{17})$. Let the vertices of $X_{17}$ form $m_{17}$ ${\rm Aut}(X_{17})$-orbits. We take the middle point of the line segment joining vertices $a_{0}$ and $a_3$ as the origin $(0,0)$ of $K_{17}$. Let  $\alpha_{17} := a_{20} - a_4$, $\beta_{17} := a_7 - a_4$ and $\gamma_{17} := a_{16} - a_4$ $\in \mathbb{R}^2$. Similarly as above, define $H_{17} := \langle x\mapsto x+\alpha_{17}, x\mapsto x+\beta_{17}, x\mapsto x+\gamma_{17}\rangle$ and 
\begin{align*}
  G_{17} & =\{ \alpha : x\mapsto \varepsilon x + m_{17}\alpha_{17} + n_{17}\beta_{17} +r_{17}\gamma_{17}   \, : \, \varepsilon=\pm 1, m_{17}, n_{17}, r_{17} \in \ZZ\} \cong H_{17}\rtimes \mathbb{Z}_2.
\end{align*}
By the same arguments as above and in Claim 1, $\Gamma_{17} \unlhd G_{17}$ and the number of $G_{17}/\Gamma_{17}$-orbits of vertices of $X_{17}$ is nine. Therefore, $G_{17}/\Gamma_{17}$ acts on $X_{17} = K_{17}/\Gamma_{17}$.
Since  $O_1 = \langle a_0 \rangle$, $O_2 = \langle a_1 \rangle$, $O_3 = \langle a_2 \rangle$, $O_4 = \langle b_0 \rangle$, $O_5 = \langle b_1 \rangle$, $O_6 = \langle b_2 \rangle$
 are the $G_{17}$-orbits, it follows that $O_1 = \langle a_0 \rangle$, $O_2 = \langle a_1 \rangle$, $O_3 = \langle a_2 \rangle$, $O_4 = \langle b_0 \rangle$, $O_5 = \langle b_1 \rangle$, $O_6 = \langle b_2 \rangle$ are the $(G_{17}/\Gamma_{17})$-orbits.  
Since the vertex set of $X_{17}$ is $\sqcup_{j=1}^6 \eta_{17}(O_j)$ and $G_{17}/\Gamma_{17} \leq {\rm Aut}(X_{17})$, it follows that the number of ${\rm Aut}(X_{17})$-orbits of vertices is $\leq 6$. Thus, it completes the part {\rm (1)}.

\medskip

Let $X_{2} = K_{2}/\Gamma_{2}$ be a semiequivelar map on the torus for some fixed element (vertex, edge or face) free subgroup $\Gamma_{2} \le {\rm Aut}(K_{2})$. Let the vertices of $X_{2}$ form $m_{2}$ ${\rm Aut}(X_{2})$-orbits.
We take the middle point of the line segment joining vertices $a_{0}$ and $a_3$ as the origin $(0,0)$ of $K_2$.  Let $\alpha_2 := a_{6} - a_0$, $\beta_2 := a_{18} - a_0$ and $\gamma_2 := a_{24} - a_{0}$ $\in \mathbb{R}^2$. Similarly as above, define $H_2 := \langle x\mapsto x+\alpha_2, x\mapsto x+\beta_2, x\mapsto x+\gamma_2\rangle$ and 
\begin{align*}
  G_2 & =\{ \alpha : x\mapsto \varepsilon x + m_1\alpha_2 + n_1\beta_2 +r_1\gamma_2 \, : \, \varepsilon=\pm 1, m_1, n_1, r_1 \in \ZZ\} \cong H_2\rtimes \mathbb{Z}_2.
\end{align*}
By the same arguments as above and in Claim 1, $\Gamma_2 \unlhd G_2$ and the number of $G_2/\Gamma_2$-orbits of vertices of $X_2$ is four. Therefore, $G_2/\Gamma_2$ acts on $X_2= K_2/\Gamma_2$.
Since 
\begin{align*}
& O_1 = \langle a_0 \rangle, O_2 = \langle a_1 \rangle, O_3 = \langle a_2 \rangle, O_4 = \langle b_0 \rangle
\end{align*}
 are the $G_2$-orbits, it follows that $\eta_{2}(O_j)$ for $j=1, 2, 3, 4$ are the $(G_2/\Gamma_2)$-orbits.  
Since the vertex set of $Y$ is $\sqcup_{j=1}^{4}\eta_j(O_j)$ and $G_2/\Gamma_2 \leq {\rm Aut}(X_2)$, it follows that the number of ${\rm Aut}(X_2)$-orbits of vertices is $\leq 4$. 

\medskip

Let $X_{7} = K_{7}/\Gamma_{7}$ be a semiequivelar map on the torus for some fixed element (vertex, edge or face) free subgroup $\Gamma_{7} \le {\rm Aut}(K_{7})$. Let the vertices of $X_{7}$ form $m_{7}$ ${\rm Aut}(X_{7})$-orbits.
We take the middle point of the line segment joining vertices $b_{0}$ and $b_{7}$ as the origin $(0,0)$ of $K_7$. Let  $\alpha_7 := b_{22} - b_0$, $\beta_7 := b_{28} - b_{0}$ and $\gamma_7 := b_{41} - b_{0}$ $\in \mathbb{R}^2$. Similarly as above, define $H_7 := \langle x\mapsto x+\alpha_7, x\mapsto x+\beta_7, x\mapsto x+\gamma_7\rangle$ and 
\begin{align*}
  G_7 & =\{ \alpha : x\mapsto \varepsilon x + m_7\alpha_7 + n_7\beta_7 +r_7\gamma_7   \, : \, \varepsilon=\pm 1, m_7, n_7, r_7 \in \ZZ\} \cong H_7\rtimes \mathbb{Z}_2.
\end{align*}
By the same arguments as above and in Claim 1, $\Gamma_7 \unlhd G_7$ and the number of $G_7/\Gamma_7$-orbits of vertices of $X_7$ is four. Therefore, $G_7/\Gamma_7$ acts on $X_7 = K_7/\Gamma_7$.
Since  $O_1 = \langle a_0 \rangle,$ $O_2 = \langle b_0 \rangle,$  $O_3 = \langle b_5 \rangle,$ $O_4 = \langle b_6 \rangle$
 are the $G_7$-orbits, it follows that $O_1 = \langle a_0 \rangle,$ $O_2 = \langle b_0 \rangle,$  $O_3 = \langle b_5 \rangle,$ $O_4 = \langle b_6 \rangle$ are the $(G_7/\Gamma_7)$-orbits.  
Since the vertex set of $X_7$ is $\sqcup_{j=1}^4 \eta_7(O_j)$ and $G_7/\Gamma_7 \leq {\rm Aut}(X_7)$, it follows that the number of ${\rm Aut}(X_7)$-orbits of vertices is $\leq 4$. 

\medskip

Let $X_{10} = K_{10}/\Gamma_{10}$ be a semiequivelar map on the torus for some fixed element (vertex, edge or face) free subgroup $\Gamma_{10} \le {\rm Aut}(K_{10})$. Let the vertices of $X_{10}$ form $m_{10}$ ${\rm Aut}(X_{10})$-orbits. We take the middle point of the line segment joining vertices $b_0$ and $b_{1}$ as the origin $(0,0)$ of $K_{10}$. Let  $\alpha_{10} := b_6-b_0$ and $\beta_{10} := b_{23}-b_0\in \mathbb{R}^2$. Similarly as above, define $H_{10} := \langle x\mapsto x+\alpha_{10}, x\mapsto x+\beta_{10}\rangle$ and 
\begin{align*}
  G_{10} & =\{ \alpha : x\mapsto \varepsilon x + m_{10}\alpha_{10} + n_{10}\beta_{10}  \, : \, \varepsilon=\pm 1, m_{10}, n_{10} \in \ZZ\} \cong H_{10}\rtimes \mathbb{Z}_2.
\end{align*}
By the same arguments as above and in Claim 1, $\Gamma_{10} \unlhd G_{10}$ and the number of $G_{10}/\Gamma_{10}$-orbits of vertices of $X_{10}$ is three. Therefore, $G_{10}/\Gamma_{10}$ acts on $X_{10} = K_{10}/\Gamma_{10}$.
Since  $O_1 = \langle a_0 \rangle, O_2 = \langle a_1 \rangle, O_3 = \langle b_0 \rangle, O_4 = \langle b_2 \rangle$
 are the $G_{10}$-orbits, it follows that $\eta_{10}(O_1), \eta_{10}(O_2), \eta_{10}(O_3),\eta_{10}(O_4)$ are the $(G_{10}/\Gamma_{10})$-orbits.  
Since the vertex set of $X_{10}$ is $\eta_{10}(O_1) \sqcup \eta_{10}(O_2) \sqcup \eta_{10}(O_3)\sqcup \eta_{10}(O_4)$ and $G_{10}/\Gamma_{10} \leq {\rm Aut}(X_{10})$, it follows that the number of ${\rm Aut}(X_{10})$-orbits of vertices is $\leq 4$.

\medskip

Let $X_{16} = K_{16}/\Gamma_{16}$ be a semiequivelar map on the torus for some fixed element (vertex, edge or face) free subgroup $\Gamma_{16} \le {\rm Aut}(K_{16})$. Let the vertices of $X_{16}$ form $m_{16}$ ${\rm Aut}(X_{16})$-orbits. We take the middle point of the line segment joining vertices $a_{0}$ and $a_4$ as the origin $(0,0)$ of $K_{16}$. Let  $\alpha_{16} := a_4 - a_1$ and $\beta_{16} := a_8 - a_1$ $\in \mathbb{R}^2$. Similarly as above, define $H_{16} := \langle x\mapsto x+\alpha_{16}, x\mapsto x+\beta_{16}\rangle$ and 
\begin{align*}
  G_{16} & =\{ \alpha : x\mapsto \varepsilon x + m_{16}\alpha_{16} + n_{16}\beta_{16}  \, : \, \varepsilon=\pm 1, m_{16}, n_{16} \in \ZZ\} \cong H_{16}\rtimes \mathbb{Z}_2.
\end{align*}
By the same arguments as above and in Claim 1, $\Gamma_{16} \unlhd G_{16}$ and the number of $G_{16}/\Gamma_{16}$-orbits of vertices of $X_{16}$ is three. Therefore, $G_{16}/\Gamma_{16}$ acts on $X_{16} = K_{16}/\Gamma_{16}$.
Since  $O_1 = \langle a_1 \rangle, O_2 = \langle b_1 \rangle, O_3 = \langle a_3 \rangle, O_4 = \langle b_2 \rangle$
 are the $G_{16}$-orbits, it follows that $\eta_{16}(O_1), \eta_{16}(O_2), \eta_{16}(O_3), \eta_{16}(O_4)$ are the $(G_{16}/\Gamma_{16})$-orbits.  
Since the vertex set of $X_{16}$ is $\eta_{16}(O_1) \sqcup \eta_{16}(O_2) \sqcup \eta_{16}(O_3) \sqcup \eta_{16}(O_4) $ and $G_{16}/\Gamma_{16} \leq {\rm Aut}(X_{16})$, it follows that the number of ${\rm Aut}(X_{16})$-orbits of vertices is $\leq 2$ and hence, $m_{16} \le 4$. Thus, it completes the part {\rm (2)}.

\medskip

Let $X_{3} = K_{3}/\Gamma_{3}$ be a semiequivelar map on the torus for some fixed element (vertex, edge or face) free subgroup $\Gamma_{3} \le {\rm Aut}(K_{3})$. Let the vertices of $X_{3}$ form $m_{3}$ ${\rm Aut}(X_{3})$-orbits.
We take the middle point of the line segment joining vertices $a_{0}$ and $a_{11}$ as the origin $(0,0)$ of $K_3$. Let  $\alpha_3 := a_1 - a_0$ and $\beta_3 := a_{35} - a_0$ $\in \mathbb{R}^2$. Similarly as above, define $H_3 := \langle x\mapsto x+\alpha_3, x\mapsto x+\beta_3\rangle$ and 
\begin{align*}
  G_3 & =\{ \alpha : x\mapsto \varepsilon x + m_3\alpha_3 + n_3\beta_3  \, : \, \varepsilon=\pm 1, m_3, n_3 \in \ZZ\} \cong H_3\rtimes \mathbb{Z}_2.
\end{align*}
By the same arguments as above and in Claim 1, $\Gamma_3 \unlhd G_3$ and the number of $G_3/\Gamma_3$-orbits of vertices of $X_3$ is two. Therefore, $G_3/\Gamma_3$ acts on $X_3= K_3/\Gamma_3$.
Since  $O_1 = \langle a_1 \rangle, O_2 = \langle b_1 \rangle$ are the $G_3$-orbits, it follows that $\eta_3(O_1), \eta_3(O_2)$ are the $(G_3/\Gamma_3)$-orbits.  
Since the vertex set of $X_3$ is $\eta_3(O_1) \sqcup \eta_3(O_2)$ and $G_3/\Gamma_3 \leq {\rm Aut}(X_3)$, it follows that the number of ${\rm Aut}(X_3)$-orbits of vertices is $\leq 2$ and hence, $m_3 =2$.

\medskip

Let $X_{4} = K_{4}/\Gamma_{4}$ be a semiequivelar map on the torus for some fixed element (vertex, edge or face) free subgroup $\Gamma_{4} \le {\rm Aut}(K_{4})$. Let the vertices of $X_{4}$ form $m_{4}$ ${\rm Aut}(X_{4})$-orbits. 
We take the middle point of the line segment joining vertices $a_{0}$ and $a_{12}$ as the origin $(0,0)$ of $K_4$. Let  $\alpha_4 := a_1 - a_0$ and $\beta_4 := a_{43} - a_0$ $\in \mathbb{R}^2$. Similarly as above, define $H_4 := \langle x\mapsto x+\alpha_4, x\mapsto x+\beta_4\rangle$ and 
\begin{align*}
  G_4 & =\{ \alpha : x\mapsto \varepsilon x + m_4\alpha_4 + n_4\beta_4  \, : \, \varepsilon=\pm 1, m_4, n_4 \in \ZZ\} \cong H_4\rtimes \mathbb{Z}_2.
\end{align*}
By the same arguments as above and in Claim 1, $\Gamma_4 \unlhd G_4$ and the number of $G_4/\Gamma_4$-orbits of vertices of $X_4$ is two. Therefore, $G_4/\Gamma_4$ acts on $X_4 = K_4/\Gamma_4$.
Since  $O_1 = \langle a_1 \rangle, O_2 = \langle b_1 \rangle$
 are the $G_4$-orbits, it follows that $\eta_4(O_1), \eta_4(O_2)$ are the $(G_4/\Gamma_4)$-orbits.  
Since the vertex set of $X_4$ is $\eta_4(O_1) \sqcup \eta_4(O_2)$ and $G_4/\Gamma_4 \leq {\rm Aut}(X_4)$, it follows that the number of ${\rm Aut}(X_4)$-orbits of vertices is $\leq 2$ and hence, $m_4 =2$.

\medskip

Let $X_{8} = K_{8}/\Gamma_{8}$ be a semiequivelar map on the torus for some fixed element (vertex, edge or face) free subgroup $\Gamma_{8} \le {\rm Aut}(K_{8})$. Let the vertices of $X_{8}$ form $m_{8}$ ${\rm Aut}(X_{8})$-orbits. We take the middle point of the line segment joining vertices $a_0$ and $a_1$ as the origin $(0,0)$ of $K_8$. Let  $\alpha_8 := a_3 - a_0$ and $\beta_8 := a_4 - a_0$ $\in \mathbb{R}^2$. Similarly as above, define $H_8 := \langle x\mapsto x+\alpha_8, x\mapsto x+\beta_8\rangle$ and 
\begin{align*}
  G_8 & =\{ \alpha : x\mapsto \varepsilon x + m_8\alpha_8 + n_8\beta_8  \, : \, \varepsilon=\pm 1, m_8, n_8 \in \ZZ\} \cong H_8\rtimes \mathbb{Z}_2.
\end{align*}
By the same arguments as above and in Claim 1, $\Gamma_8 \unlhd G_8$ and the number of $G_8/\Gamma_8$-orbits of vertices of $X_8$ is two. Therefore, $G_8/\Gamma_8$ acts on $X_8 = K_8/\Gamma_8$.
Since  $O_1 = \langle a_0 \rangle, O_2 = \langle b_0 \rangle$
 are the $G_8$-orbits, it follows that $\eta_8(O_1), \eta_8(O_2)$ are the $(G_8/\Gamma_8)$-orbits.  
Since the vertex set of $X_8$ is $\eta_8(O_1) \sqcup \eta_8(O_2)$ and $G_8/\Gamma_8 \leq {\rm Aut}(X_8)$, it follows that the number of ${\rm Aut}(X_8)$-orbits of vertices is $\leq 2$ and hence, $m_8 =2$.

\medskip

Let $X_{12} = K_{12}/\Gamma_{12}$ be a semiequivelar map on the torus for some fixed element (vertex, edge or face) free subgroup $\Gamma_{12} \le {\rm Aut}(K_{12})$. Let the vertices of $X_{12}$ form $m_{12}$ ${\rm Aut}(X_{12})$-orbits. We take the middle point of the line segment joining vertices $a_{0}$ and $a_{1}$ as the origin $(0,0)$ of $K_{12}$. Let  $\alpha_{12} := a_1 - a_0 $ and $\beta_{12} := a_{12} - a_0$ $\in \mathbb{R}^2$. Similarly as above, define $H_{12} := \langle x\mapsto x+\alpha_{12}, x\mapsto x+\beta_{12}\rangle$ and 
\begin{align*}
  G_{12} & =\{ \alpha : x\mapsto \varepsilon x + m_{12}\alpha_{12} + n_{12}\beta_{12}  \, : \, \varepsilon=\pm 1, m_{12}, n_{12} \in \ZZ\} \cong H_{12}\rtimes \mathbb{Z}_2.
\end{align*}
By the same arguments as above and in Claim 1, $\Gamma_{12} \unlhd G_{12}$ and the number of $G_{12}/\Gamma_{12}$-orbits of vertices of $X_{12}$ is two. Therefore, $G_{12}/\Gamma_{12}$ acts on $X_{12} = K_{12}/\Gamma_{12}$.
Since  $O_1 = \langle a_0 \rangle, O_2 = \langle b_0 \rangle$
 are the $G_{12}$-orbits, it follows that $\eta_{12}(O_1), \eta_{12}(O_2)$ are the $(G_{12}/\Gamma_{12})$-orbits.  
Since the vertex set of $X_{12}$ is $\eta_{12}(O_1) \sqcup \eta_{12}(O_2)$ and $G_{12}/\Gamma_{12} \leq {\rm Aut}(X_{12})$, it follows that the number of ${\rm Aut}(X_{12})$-orbits of vertices is $\leq 2$ and hence, $m_{12} =2$.

\medskip

Let $X_{13} = K_{13}/\Gamma_{13}$ be a semiequivelar map on the torus for some fixed element (vertex, edge or face) free subgroup $\Gamma_{13} \le {\rm Aut}(K_{13})$. Let the vertices of $X_{13}$ form $m_{13}$ ${\rm Aut}(X_{13})$-orbits. We take the middle point of the line segment joining vertices $a_{0}$ and $a_{12}$ as the origin $(0,0)$ of $K_{13}$. Let  $\alpha_{13} := a_1 - a_0 $ and $\beta_{13} := a_{19} - a_0 $ $\in \mathbb{R}^2$. Similarly as above, define $H_{13} := \langle x\mapsto x+\alpha_{13}, x\mapsto x+\beta_{13}\rangle$ and 
\begin{align*}
  G_{13} & =\{ \alpha : x\mapsto \varepsilon x + m_{13}\alpha_{13} + n_{13}\beta_{13}  \, : \, \varepsilon=\pm 1, m_{13}, n_{13} \in \ZZ\} \cong H_{13}\rtimes \mathbb{Z}_2.
\end{align*}
By the same arguments as above and in Claim 1, $\Gamma_{13} \unlhd G_{13}$ and the number of $G_{13}/\Gamma_{13}$-orbits of vertices of $X_{13}$ is two. Therefore, $G_{13}/\Gamma_{13}$ acts on $X_{13} = K_{13}/\Gamma_{13}$.
Since  $O_1 = \langle a_0 \rangle, O_2 = \langle b_0 \rangle$
 are the $G_{13}$-orbits, it follows that $\eta_{13}(O_1), \eta_{13}(O_2)$ are the $(G_{13}/\Gamma_{13})$-orbits.  
Since the vertex set of $X_{13}$ is $\eta_{13}(O_1) \sqcup \eta_{13}(O_2)$ and $G_{13}/\Gamma_{13} \leq {\rm Aut}(X_{13})$, it follows that the number of ${\rm Aut}(X_{13})$-orbits of vertices is $\leq 2$ and hence, $m_{13} =2$.

\medskip

Let $X_{15} = K_{15}/\Gamma_{15}$ be a semiequivelar map on the torus for some fixed element (vertex, edge or face) free subgroup $\Gamma_{15} \le {\rm Aut}(K_{15})$. Let the vertices of $X_{15}$ form $m_{15}$ ${\rm Aut}(X_{15})$-orbits. We take the middle point of the line segment joining vertices $a_{0}$ and $a_{19}$ as the origin $(0,0)$ of $K_{15}$. Let  $\alpha_{15} := a_2 - a_0 $ and $\beta_{15} := a_{10} - a_0$ $\in \mathbb{R}^2$. Similarly as above, define $H_{15} := \langle x\mapsto x+\alpha_{15}, x\mapsto x+\beta_{15}\rangle$ and 
\begin{align*}
  G_{15} & =\{ \alpha : x\mapsto \varepsilon x + m_{15}\alpha_{15} + n_{15}\beta_{15}  \, : \, \varepsilon=\pm 1, m_{15}, n_{15} \in \ZZ\} \cong H_{15}\rtimes \mathbb{Z}_2.
\end{align*}
By the same arguments as above and in Claim 1, $\Gamma_{15} \unlhd G_{15}$ and the number of $G_{15}/\Gamma_{15}$-orbits of vertices of $X_{15}$ is two. Therefore, $G_{15}/\Gamma_{15}$ acts on $X_{15} = K_{15}/\Gamma_{15}$.
Since  $O_1 = \langle a_0 \rangle, O_2 = \langle b_0 \rangle$
 are the $G_{15}$-orbits, it follows that $\eta_{15}(O_1), \eta_{15}(O_2)$ are the $(G_{15}/\Gamma_{15})$-orbits.  
Since the vertex set of $X_{15}$ is $\eta_{15}(O_1) \sqcup \eta_{15}(O_2)$ and $G_{15}/\Gamma_{15} \leq {\rm Aut}(X_{15})$, it follows that the number of ${\rm Aut}(X_{15})$-orbits of vertices is $\leq 2$ and hence, $m_{15} =2$. Thus, it completes the part {\rm (3)}.

\medskip

Let $X_{18} = K_{18}/\Gamma_{18}$ be a semiequivelar map on the torus for some fixed element (vertex, edge or face) free subgroup $\Gamma_{18} \le {\rm Aut}(K_{18})$. Let the vertices of $X_{18}$ form $m_{18}$ ${\rm Aut}(X_{18})$-orbits. We take the middle point of the line segment joining vertices $a_{0}$ and $a_{28}$ as the origin $(0,0)$ of $K_{18}$. Let  $\alpha_{18} := a_2 - a_0$ and $\beta_{18} := a_{20} - a_0$ $\in \mathbb{R}^2$. Similarly as above, define $H_{18} := \langle x\mapsto x+\alpha_{18}, x\mapsto x+\beta_{18}\rangle$ and 
\begin{align*}
  G_{18} & =\{ \alpha : x\mapsto \varepsilon x + m_{18}\alpha_{18} + n_{18}\beta_{18}  \, : \, \varepsilon=\pm 1, m_{18}, n_{18} \in \ZZ\} \cong H_{18}\rtimes \mathbb{Z}_2.
\end{align*}
By the same arguments as above and in Claim 1, $\Gamma_{18} \unlhd G_{18}$ and the number of $G_{18}/\Gamma_{18}$-orbits of vertices of $X_{18}$ is three. Therefore, $G_{18}/\Gamma_{18}$ acts on $X_{18} = K_{18}/\Gamma_{18}$.
Since  $O_1 = \langle a_0 \rangle,$ $O_2 = \langle a_1 \rangle$, $O_3 = \langle b_1 \rangle$
 are the $G_{18}$-orbits, it follows that $\eta_{18}(O_1), \eta_{18}(O_2) , \eta_{18}(O_3)$ are the $(G_{18}/\Gamma_{18})$-orbits.  
Since the vertex set of $X_{18}$ is $\eta_{18}(O_1) \sqcup \eta_{18}(O_2) \sqcup \eta_{18}(O_3)$ and $G_{18}/\Gamma_{18} \leq {\rm Aut}(X_{18})$, it follows that the number of ${\rm Aut}(X_{18})$-orbits of vertices is $\leq 3$.

\medskip

Let $X_{19} = K_{19}/\Gamma_{19}$ be a semiequivelar map on the torus for some fixed element (vertex, edge or face) free subgroup $\Gamma_{19} \le {\rm Aut}(K_{19})$. Let the vertices of $X_{19}$ form $m_{19}$ ${\rm Aut}(X_{19})$-orbits. We take the middle point of the line segment joining vertices $b_{0}$ and $b_2$ as the origin $(0,0)$ of $K_{19}$. Let  $\alpha_{19} := b_0 - b_2$ and $\beta_{19} := b_7 - b_2$ $\in \mathbb{R}^2$. Similarly as above, define $H_{19} := \langle x\mapsto x+\alpha_{19}, x\mapsto x+\beta_{19}\rangle$ and 
\begin{align*}
  G_{19} & =\{ \alpha : x\mapsto \varepsilon x + m_{19}\alpha_{19} + n_{19}\beta_{19}  \, : \, \varepsilon=\pm 1, m_{19}, n_{19} \in \ZZ\} \cong H_{19}\rtimes \mathbb{Z}_2.
\end{align*}
By the same arguments as above and in Claim 1, $\Gamma_{19} \unlhd G_{19}$ and the number of $G_{19}/\Gamma_{19}$-orbits of vertices of $X_{19}$ is three. Therefore, $G_{19}/\Gamma_{19}$ acts on $X_{19} = K_{19}/\Gamma_{19}$.
Since  $O_1 = \langle b_0 \rangle,$ $O_2 = \langle a_2 \rangle$, 
$O_3 = \langle a_3 \rangle$  are the $G_{19}$-orbits, it follows that $\eta_{19}(O_1), \eta_{19}(O_2) , \eta_{19}(O_3)$ are the $(G_{19}/\Gamma_{19})$-orbits.  
Since the vertex set of $X_{19}$ is $\eta_{19}(O_1) \sqcup \eta_{19}(O_2) \sqcup \eta_{19}(O_3)$ and $G_{19}/\Gamma_{19} \leq {\rm Aut}(X_{19})$, it follows that the number of ${\rm Aut}(X_{19})$-orbits of vertices is $\leq 3$. Thus, it completes the part {\rm (4)}.

\medskip

Let $X_{6} = K_{6}/\Gamma_{6}$ be a semiequivelar map on the torus for some fixed element (vertex, edge or face) free subgroup $\Gamma_{6} \le {\rm Aut}(K_{6})$. Let the vertices of $X_{6}$ form $m_{6}$ ${\rm Aut}(X_{6})$-orbits. We take the middle point of the line segment joining vertices $a_{0}$ and $a_{6}$ as the origin $(0,0)$ of $K_6$. Let  $\alpha_6 := a_{23} - a_3$, $\beta_6 := a_{16} - a_{6}$ and $\gamma_6 := a_{39} - a_{7}$ $\in \mathbb{R}^2$. Similarly as above, define $H_6 := \langle x\mapsto x+\alpha_6, x\mapsto x+\beta_6, x\mapsto x+\gamma_6\rangle$ and 
\begin{align*}
  G_6 & =\{ \alpha : x\mapsto \varepsilon x + m_6\alpha_6 + n_6\beta_6 +r_6\gamma_6   \, : \, \varepsilon=\pm 1, m_6, n_6, r_6 \in \ZZ\} \cong H_6\rtimes \mathbb{Z}_2.
\end{align*}
By the same arguments as above and in Claim 1, $\Gamma_6 \unlhd G_6$ and the number of $G_6/\Gamma_6$-orbits of vertices of $X_6$ is seven. Therefore, $G_6/\Gamma_6$ acts on $X_6 = K_6/\Gamma_6$.
Since  $O_j = \langle a_j \rangle,$ $j=0,1, \dots, 5$  $O_6 = \langle b_1 \rangle$
 are the $G_6$-orbits, it follows that $\eta_6(O_j),$ $j=0,1,\dots, 6$ are the $(G_6/\Gamma_6)$-orbits.  
Since the vertex set of $X_6$ is $\sqcup_{j=0}^6 \eta_6(O_j)$ and $G_6/\Gamma_6 \leq {\rm Aut}(X_6)$, it follows that the number of ${\rm Aut}(X_6)$-orbits of vertices is $\leq 7$. Thus, it completes the part {\rm (5)}.

\medskip

Let $X_{20} = K_{20}/\Gamma_{20}$ be a semiequivelar map on the torus for some fixed element (vertex, edge or face) free subgroup $\Gamma_{20} \le {\rm Aut}(K_{20})$. Let the vertices of $X_{20}$ form $m_{20}$ ${\rm Aut}(X_{20})$-orbits. We take the middle point of the line segment joining vertices $a_{0}$ and $a_6$ as the origin $(0,0)$ of $K_{20}$. Let  $\alpha_{20} :=a_{24} - a_7$, $\beta_{20} := a_{14} - a_5$ and $\gamma_{20} := a_{20} - a_4$ $\in \mathbb{R}^2$. Similarly as above, define $H_{20} := \langle x\mapsto x+\alpha_{20}, x\mapsto x+\beta_{20}, x\mapsto x+\gamma_{20}\rangle$ and 
\begin{align*}
  G_{20} & =\{ \alpha : x\mapsto \varepsilon x + m_{20}\alpha_{20} + n_{20}\beta_{20} +r_{20}\gamma_{20}   \, : \, \varepsilon=\pm 1, m_{20}, n_{20}, r_{20} \in \ZZ\} \cong H_{20}\rtimes \mathbb{Z}_2.
\end{align*}
By the same arguments as above and in Claim 1, $\Gamma_{20} \unlhd G_{20}$ and the number of $G_{20}/\Gamma_{20}$-orbits of vertices of $X_{20}$ is nine. Therefore, $G_{20}/\Gamma_{20}$ acts on $X_{20} = K_{20}/\Gamma_{20}$.
Since  $O_1 = \langle a_1 \rangle,$ $O_2 = \langle a_2 \rangle,$ $O_3 = \langle a_3 \rangle,$ $O_4 = \langle a_4 \rangle,$ $O_5 = \langle a_5 \rangle,$ $O_6 = \langle a_6 \rangle,$ $O_7 = \langle \ b_3 \rangle,$ $O_8 = \langle b_0 \rangle,$ $O_9 = \langle b_{11} \rangle$
 are the $G_{20}$-orbits, it follows that $O_1 = \langle a_1 \rangle,$ $O_2 = \langle a_2 \rangle,$ $O_3 = \langle a_3 \rangle,$ $O_4 = \langle a_4 \rangle,$ $O_5 = \langle a_5 \rangle,$ $O_6 = \langle a_6 \rangle,$ $O_7 = \langle \ b_3 \rangle,$ $O_8 = \langle b_0 \rangle,$ $O_9 = \langle b_{11} \rangle$ are the $(G_{20}/\Gamma_{20})$-orbits.  
Since the vertex set of $X_{20}$ is $\sqcup_{j=1}^9 \eta_{20}(O_j)$ and $G_{20}/\Gamma_{20} \leq {\rm Aut}(X_{20})$, it follows that the number of ${\rm Aut}(X_{20})$-orbits of vertices is $\leq 9$. Thus, it completes the part {\rm (6)}.
\end{proof}

\section{$2$-uniform tilings of the plane}\label{2uniform}

\begin{figure}[H]
    \centering
    \includegraphics[height=6cm, width= 6cm]{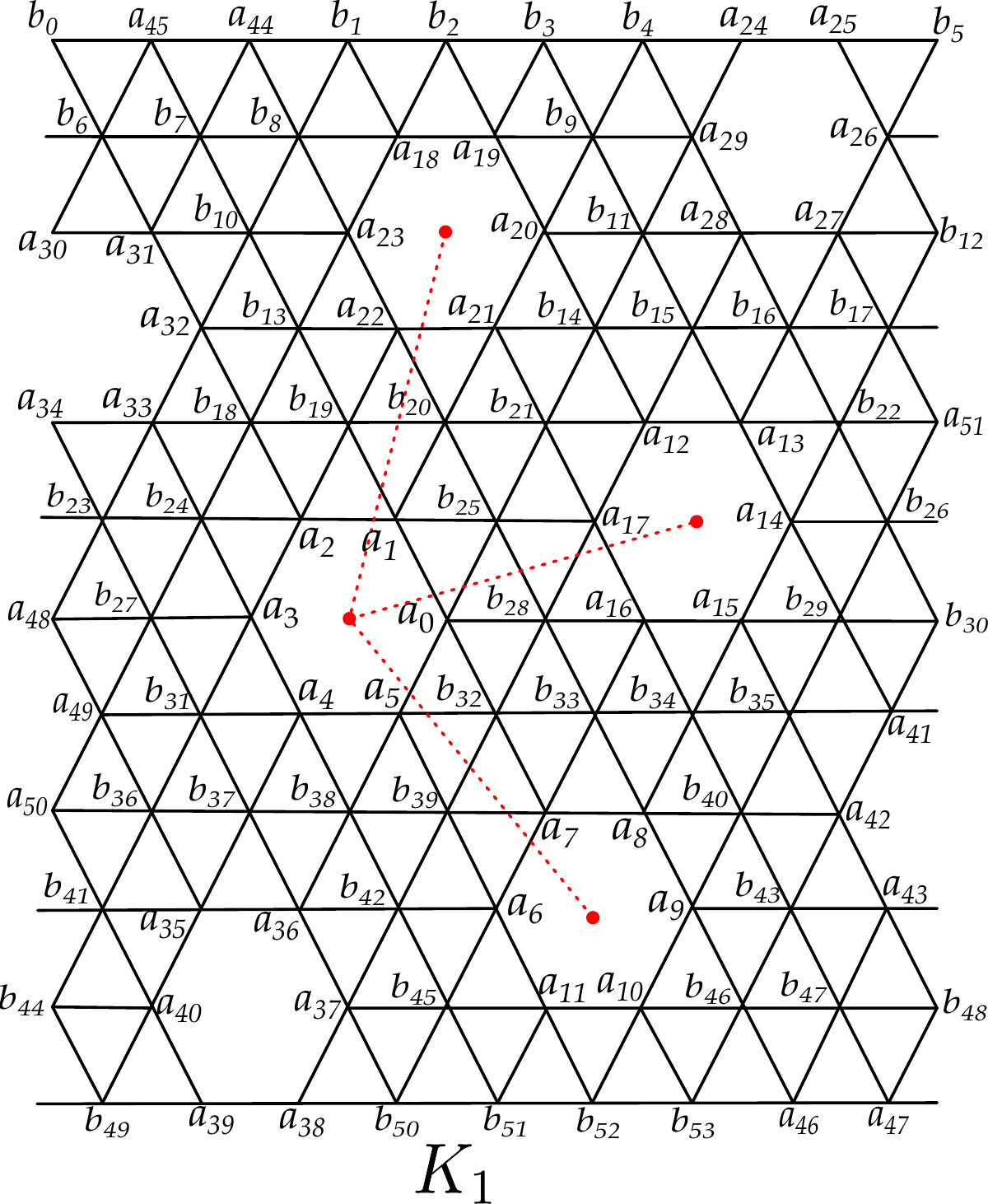}\hspace{5mm}
    \includegraphics[height=6cm, width=6cm]{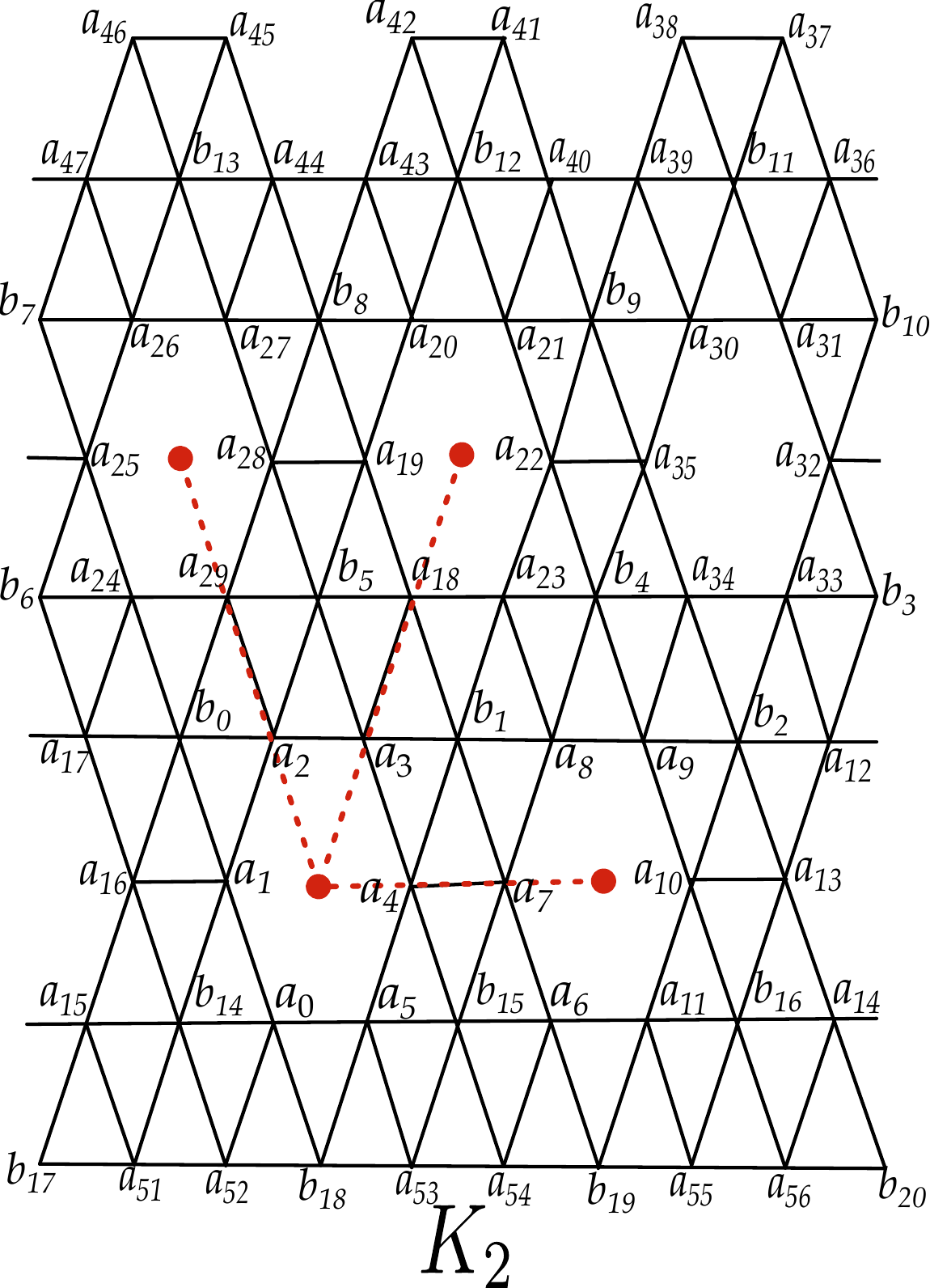}
     \vspace{10mm}
     
    \includegraphics[height=6cm, width= 6cm]{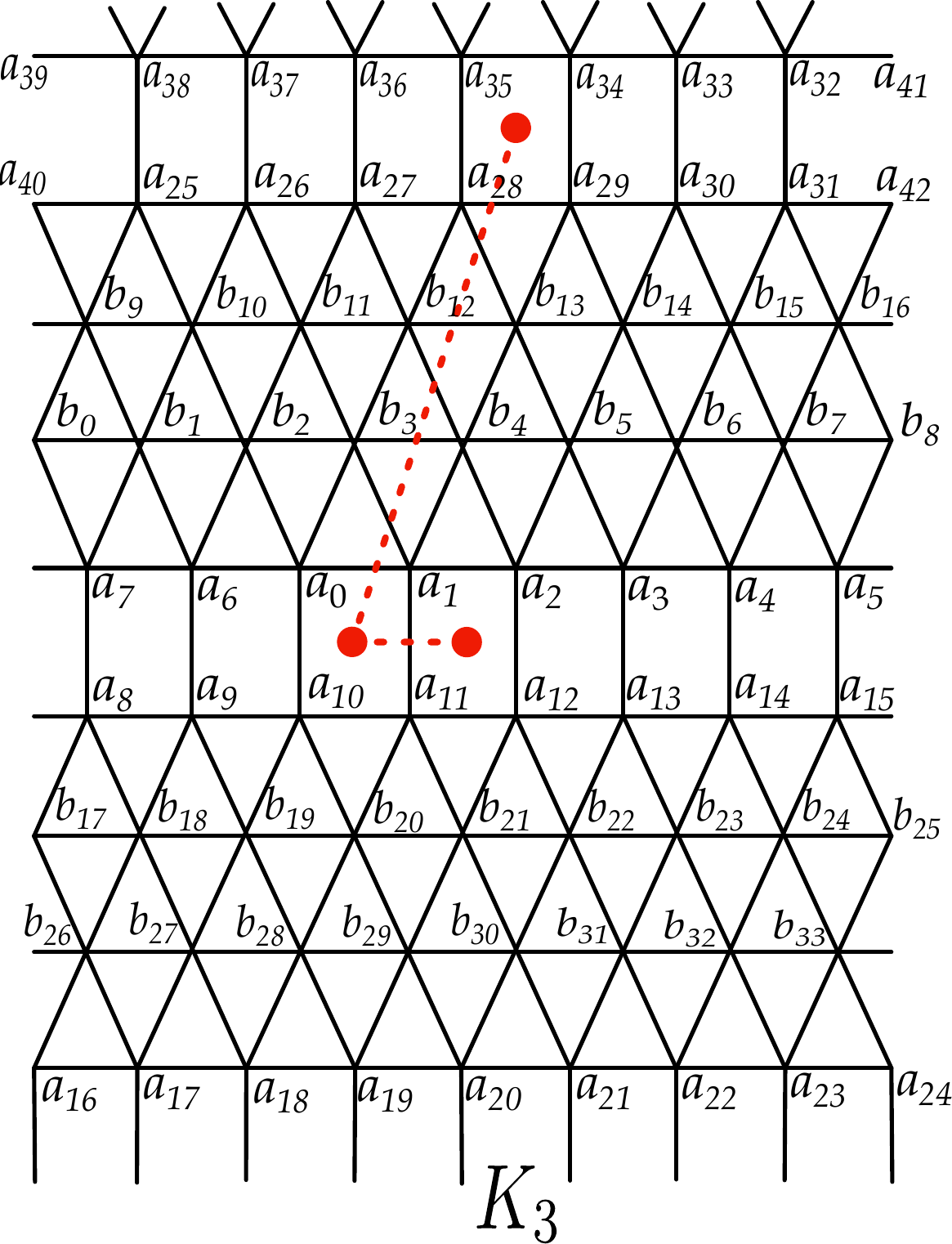}\hspace{5mm}
    \includegraphics[height=6cm, width= 6cm]{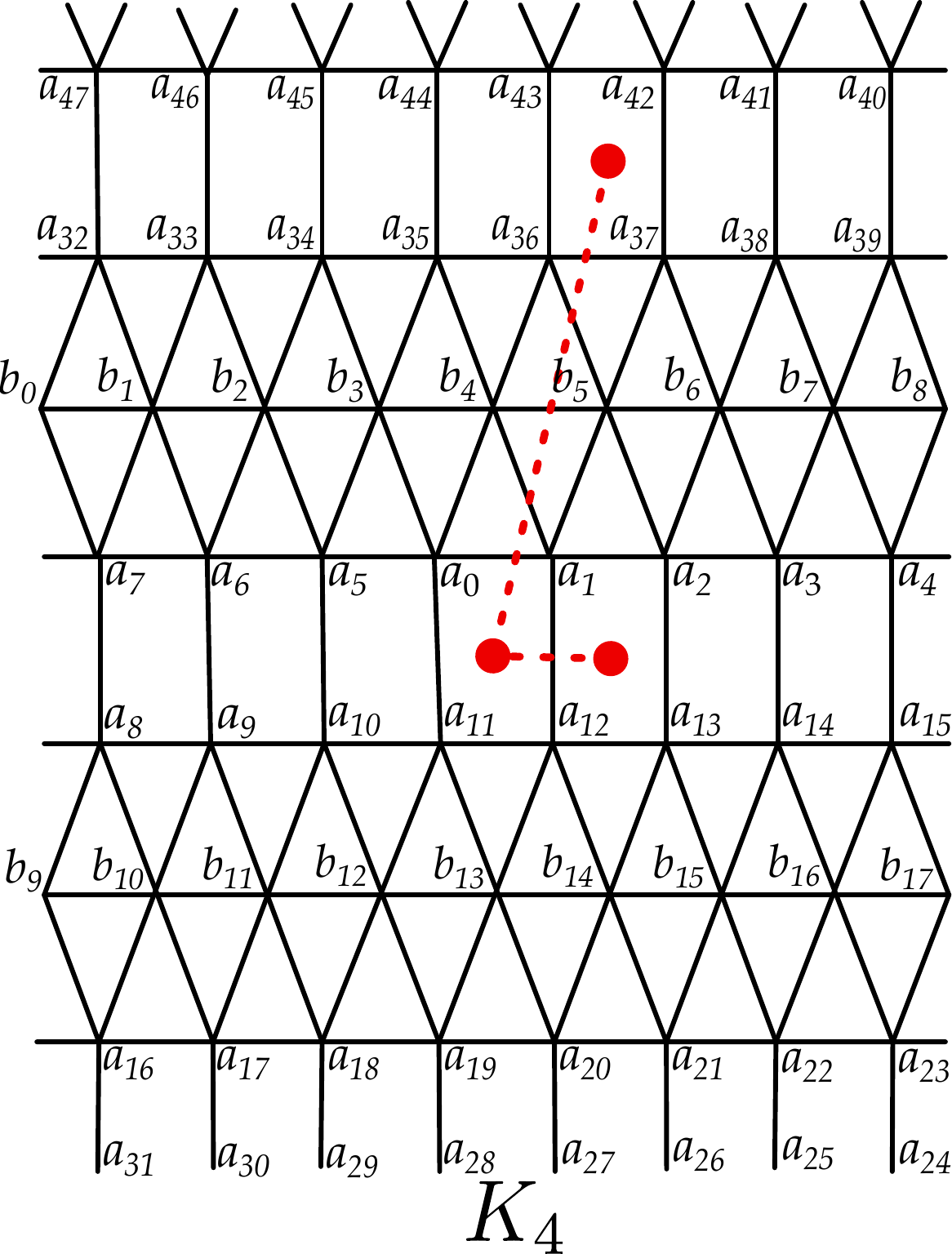}
    \vspace{10mm}
    
    \includegraphics[height=6cm, width= 6cm]{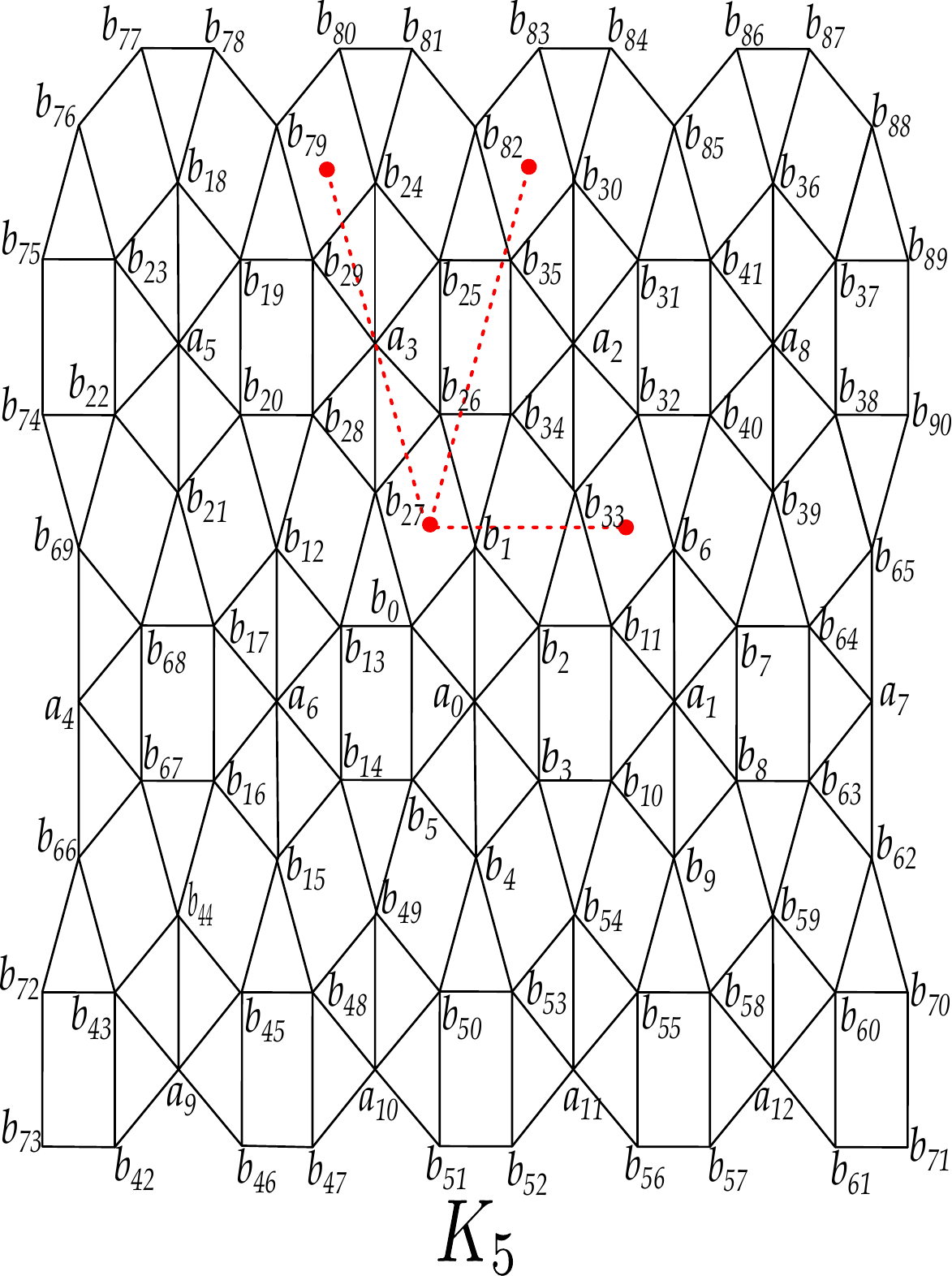}\hspace{5mm}
    \includegraphics[height=6cm, width= 6cm]{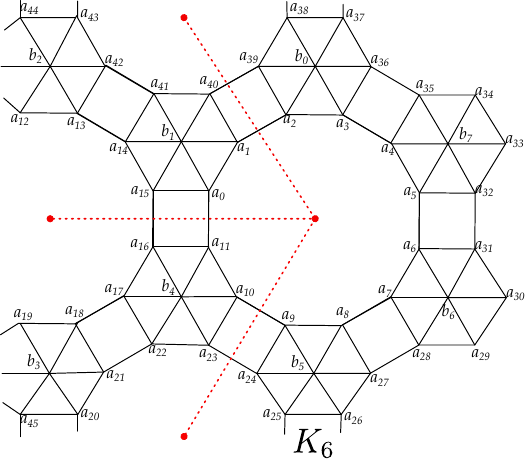}

    \end{figure}
    \begin{figure}[H]
    \centering
    \includegraphics[height=6cm, width= 6cm]{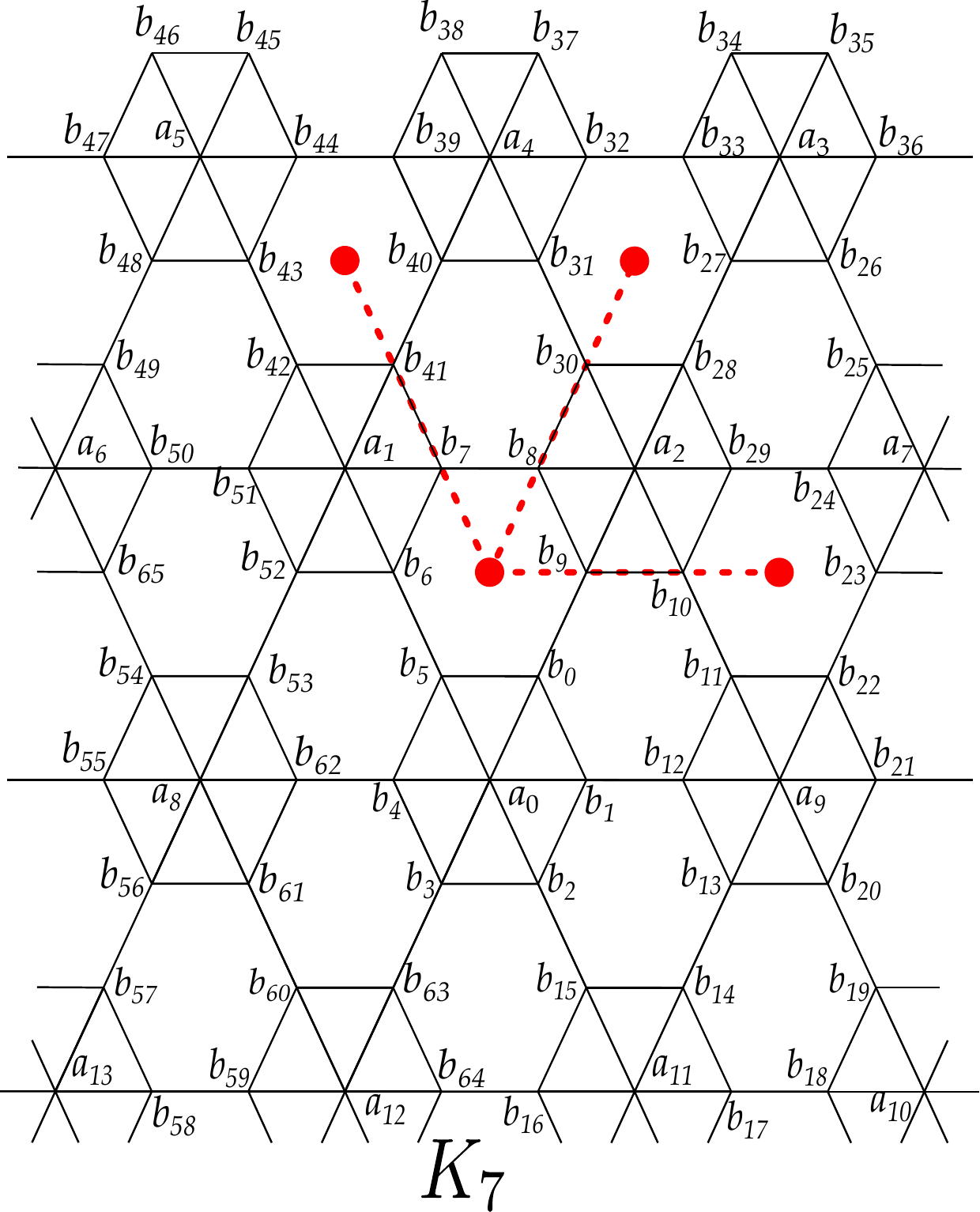}
    \includegraphics[height=6cm, width= 6cm]{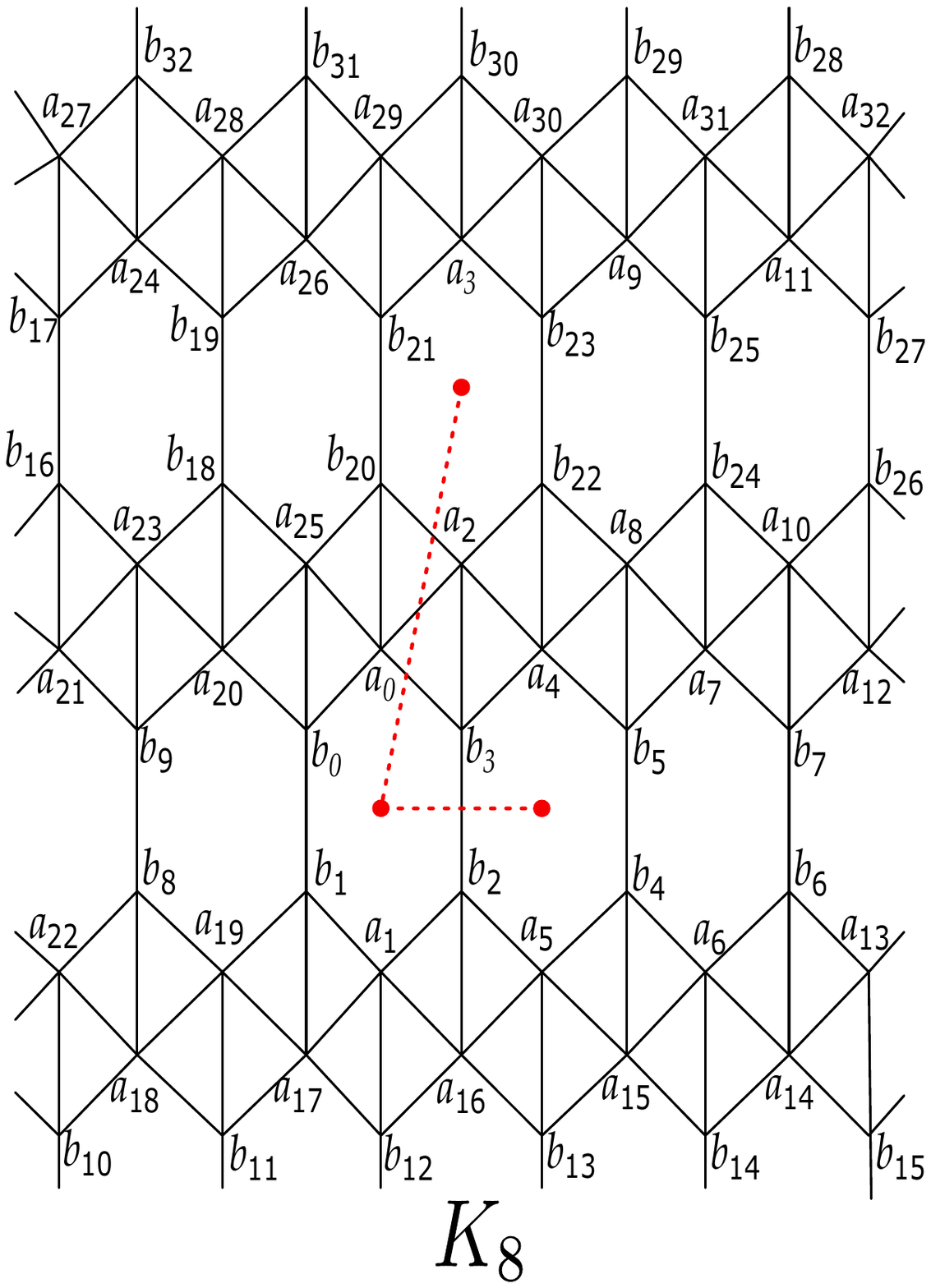}
     \vspace{10mm}
     
    \includegraphics[height=6cm, width= 6cm]{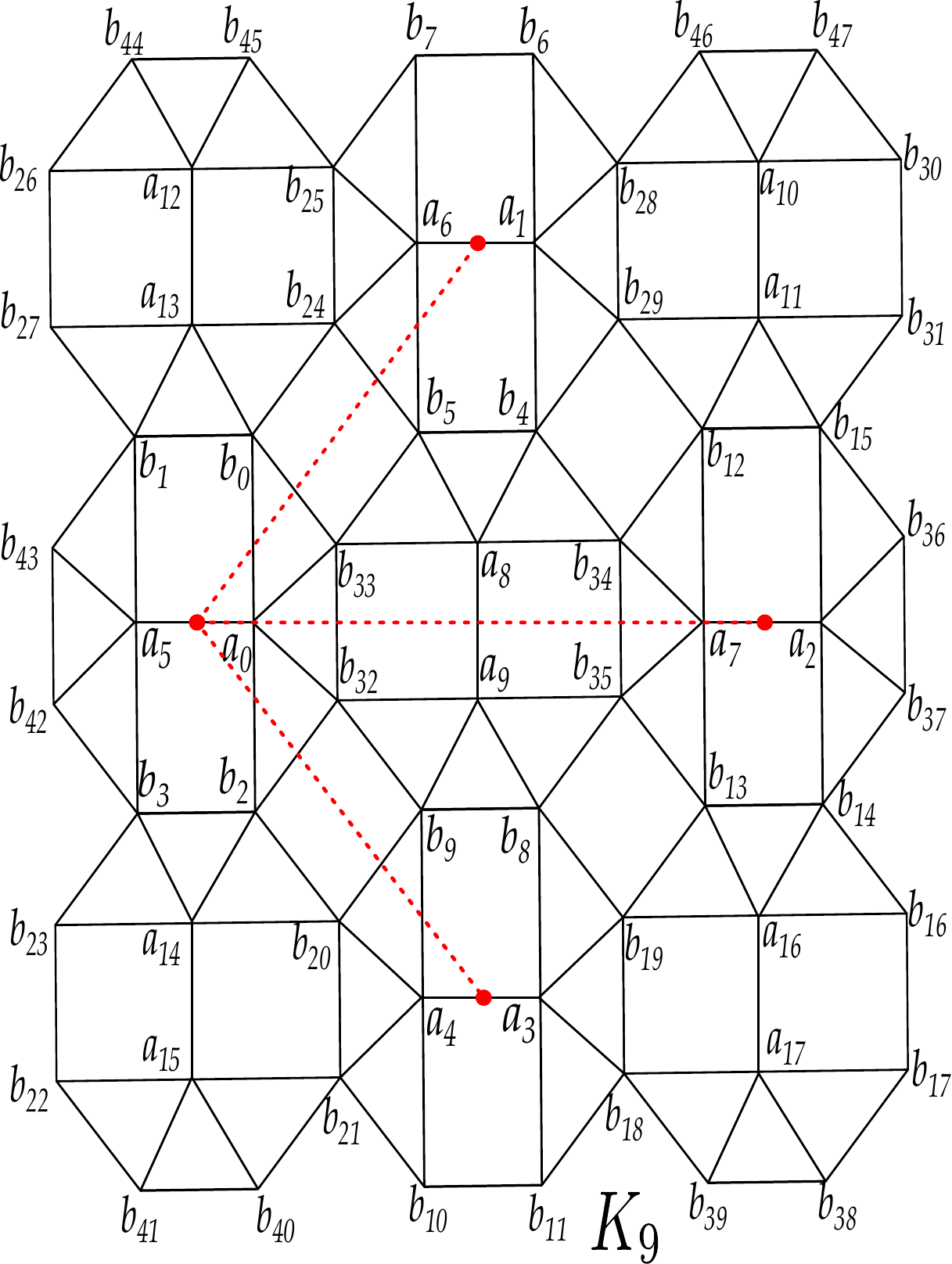}\hspace{5mm}
    \includegraphics[height=6cm, width= 6cm]{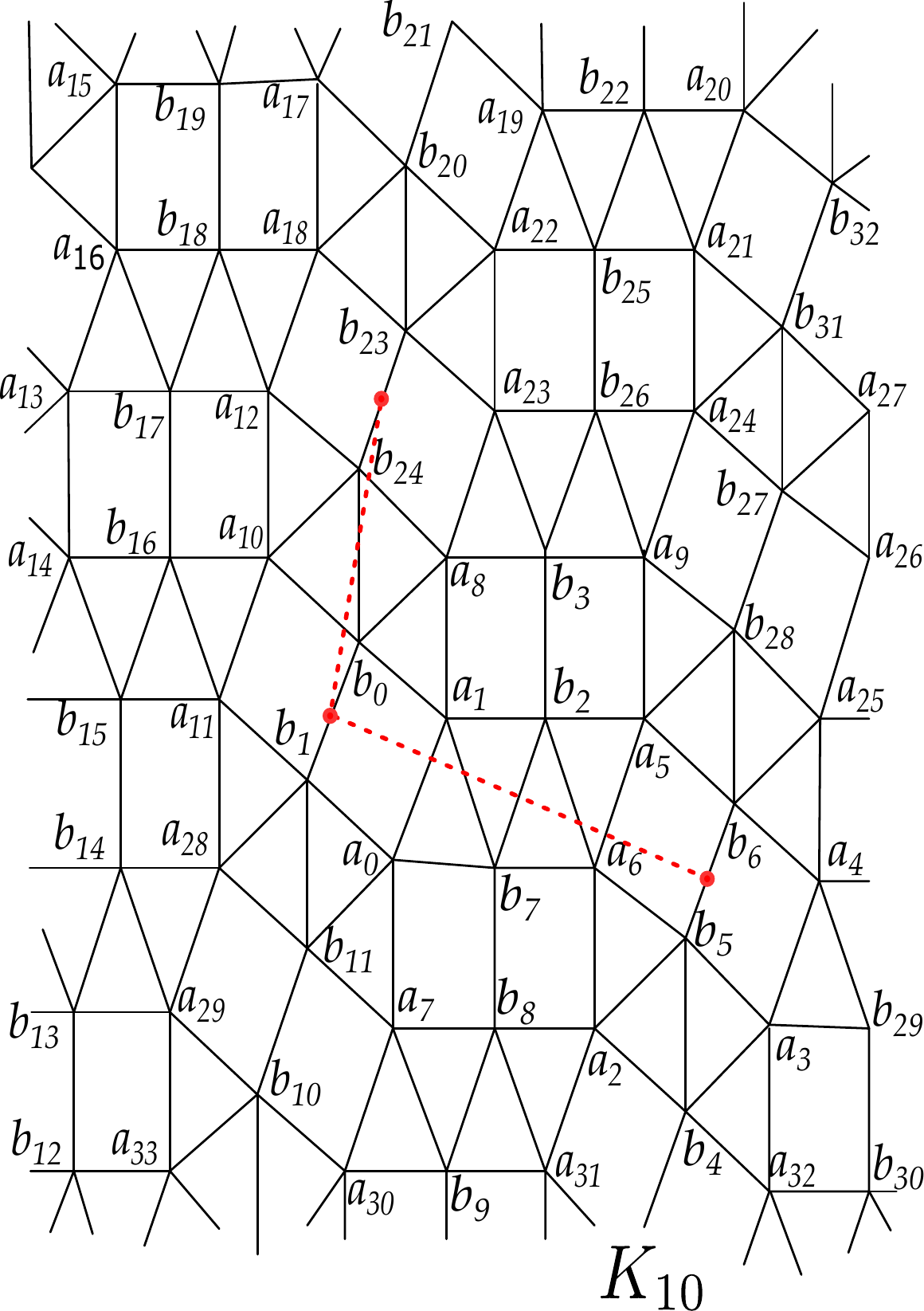}
    
     \vspace{10mm}
     
    \includegraphics[height=6cm, width= 6cm]{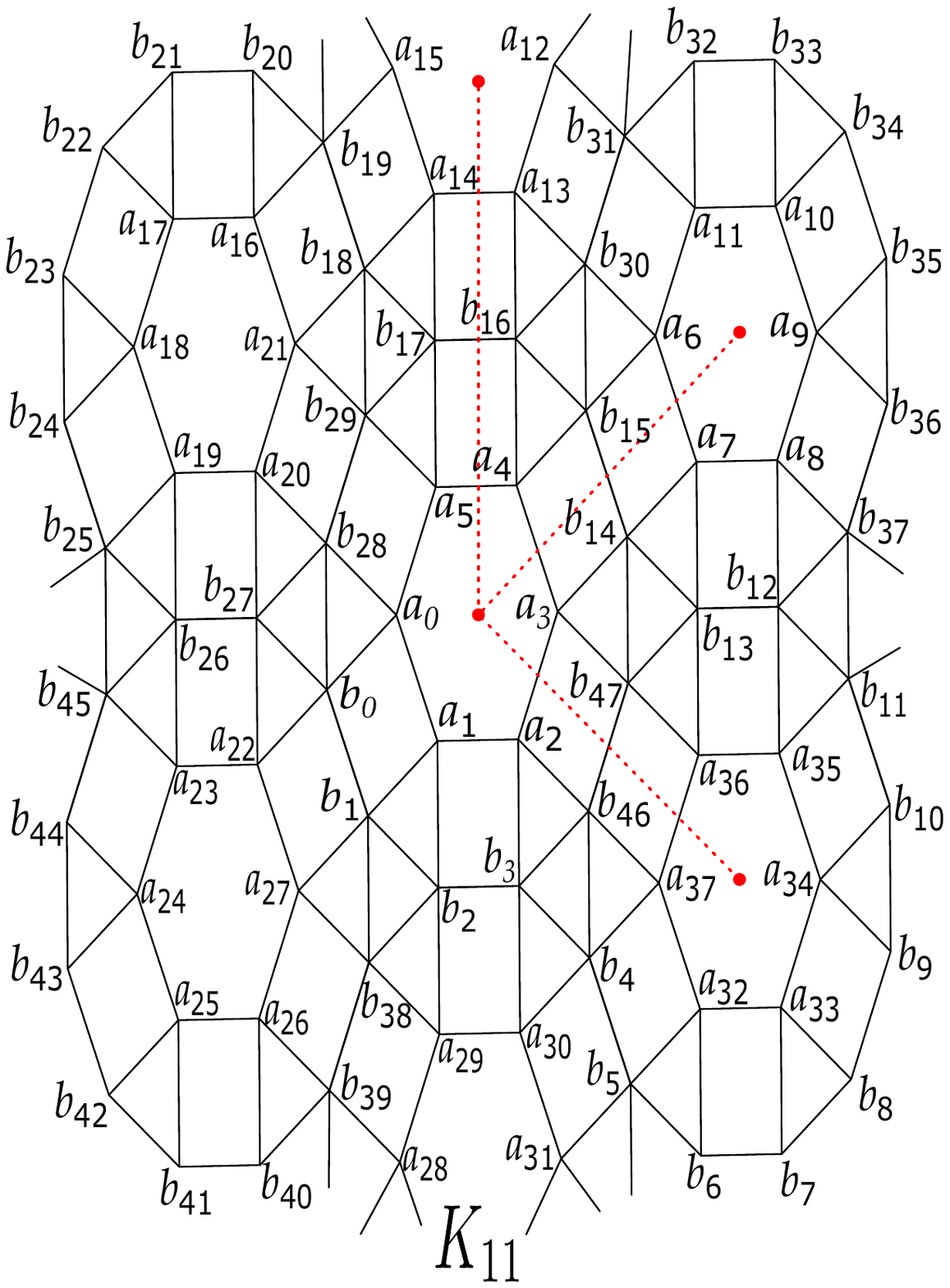}\hspace{5mm}
    \includegraphics[height=6cm, width= 6cm]{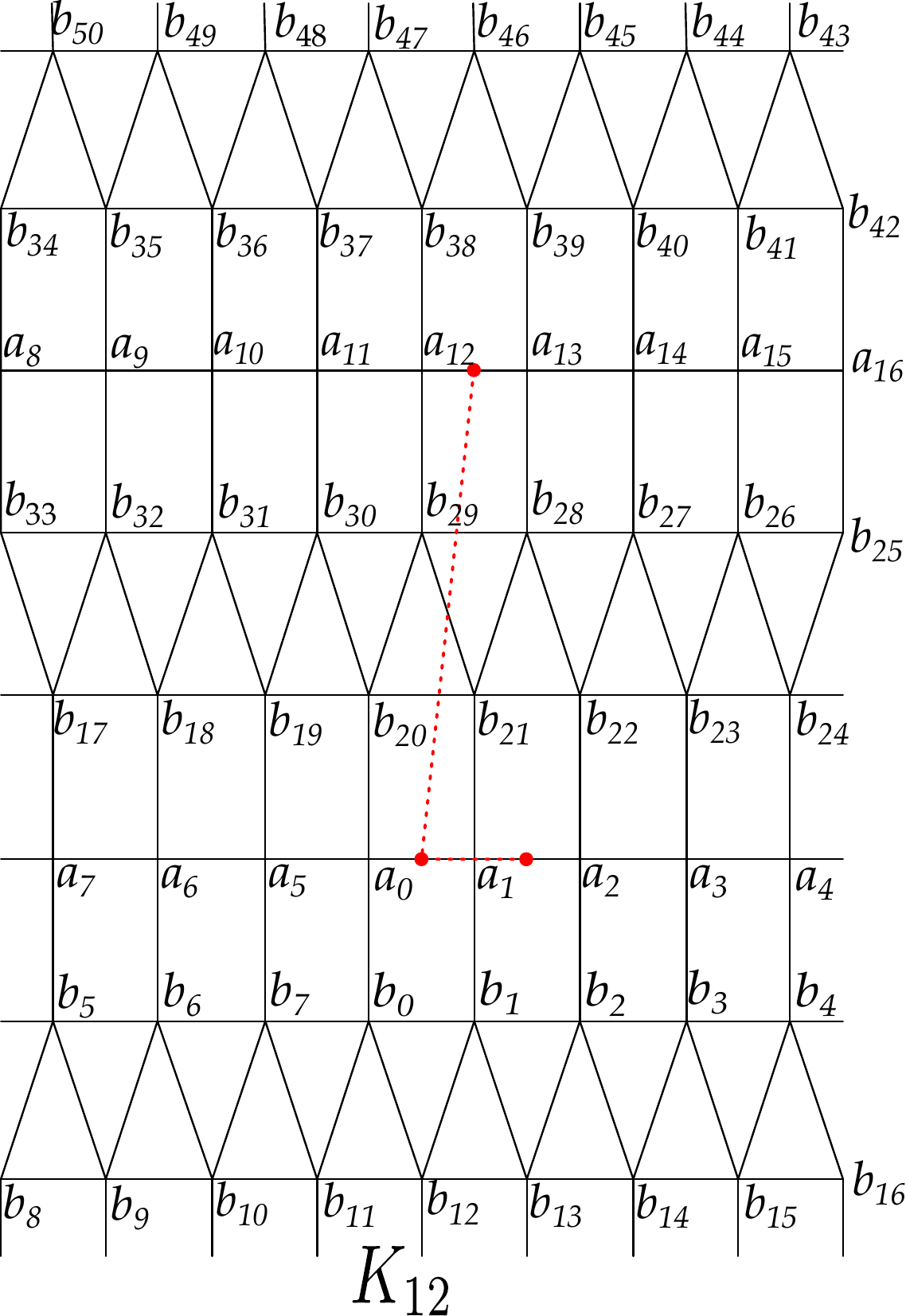}
    
     \end{figure}
     
     \begin{figure}[H]
    \centering
    \includegraphics[height=6cm, width= 6cm]{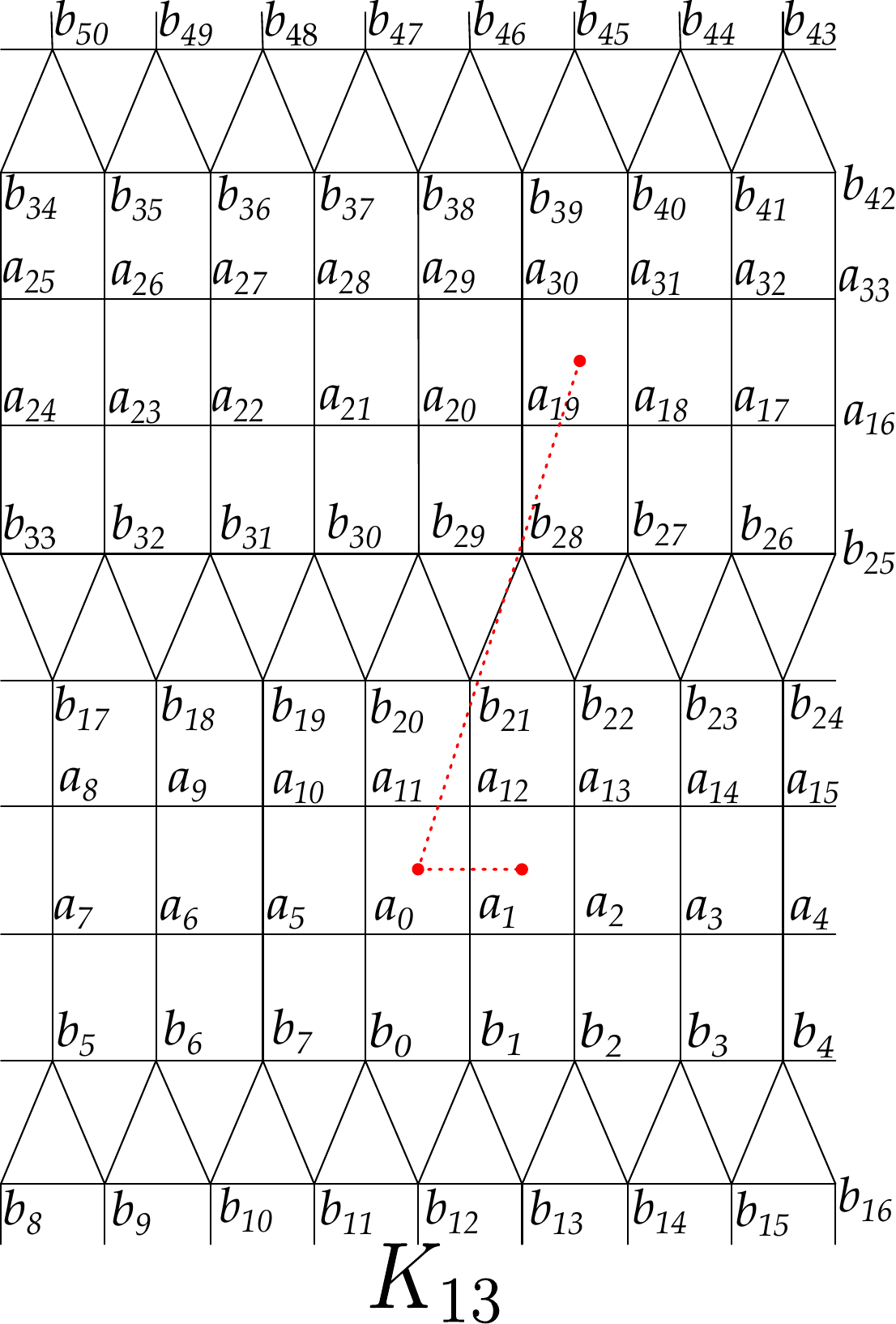}\hspace{5mm}
    \includegraphics[height=6cm, width= 6cm]{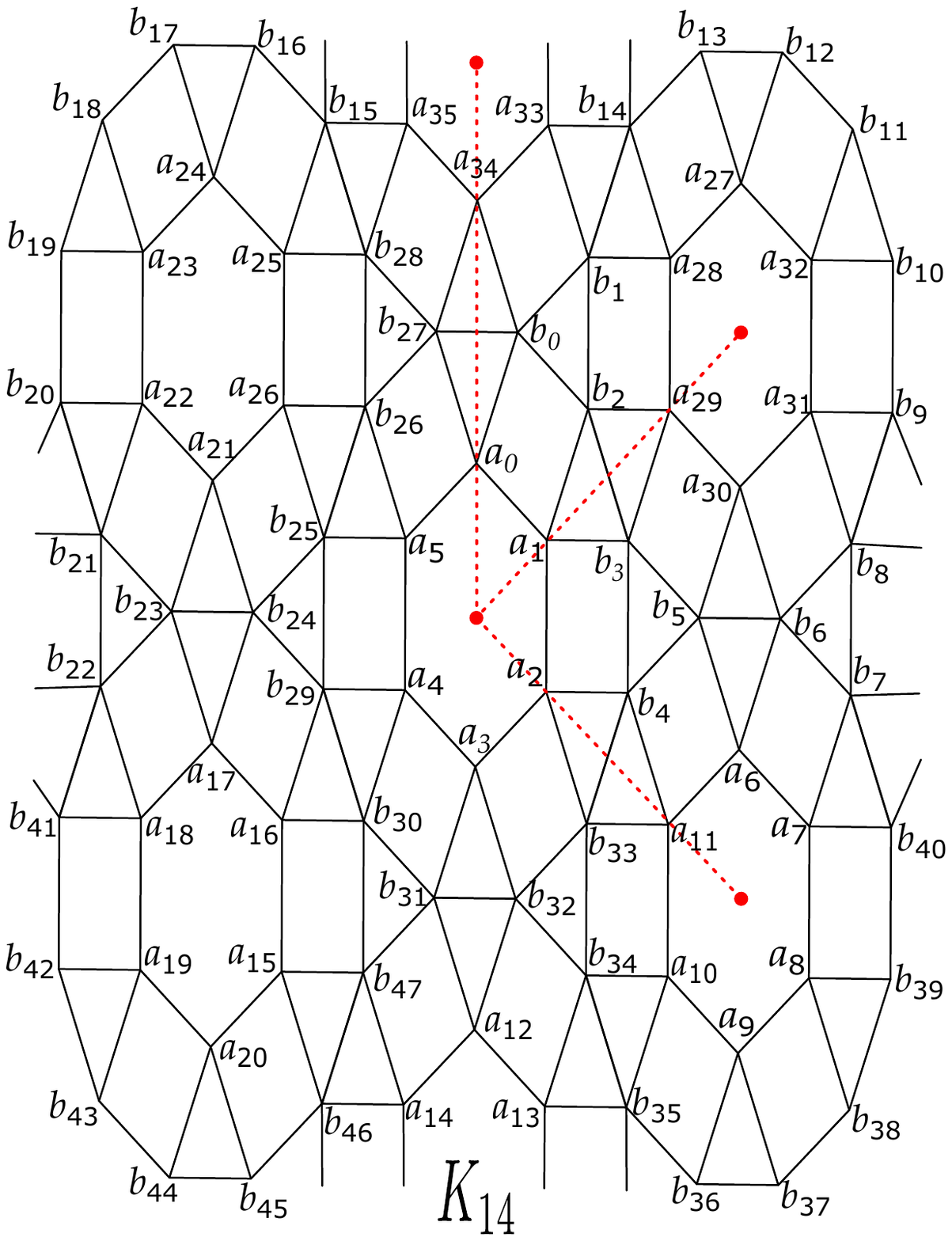}
    
     \vspace{10mm}
     
    \includegraphics[height=6cm, width= 6cm]{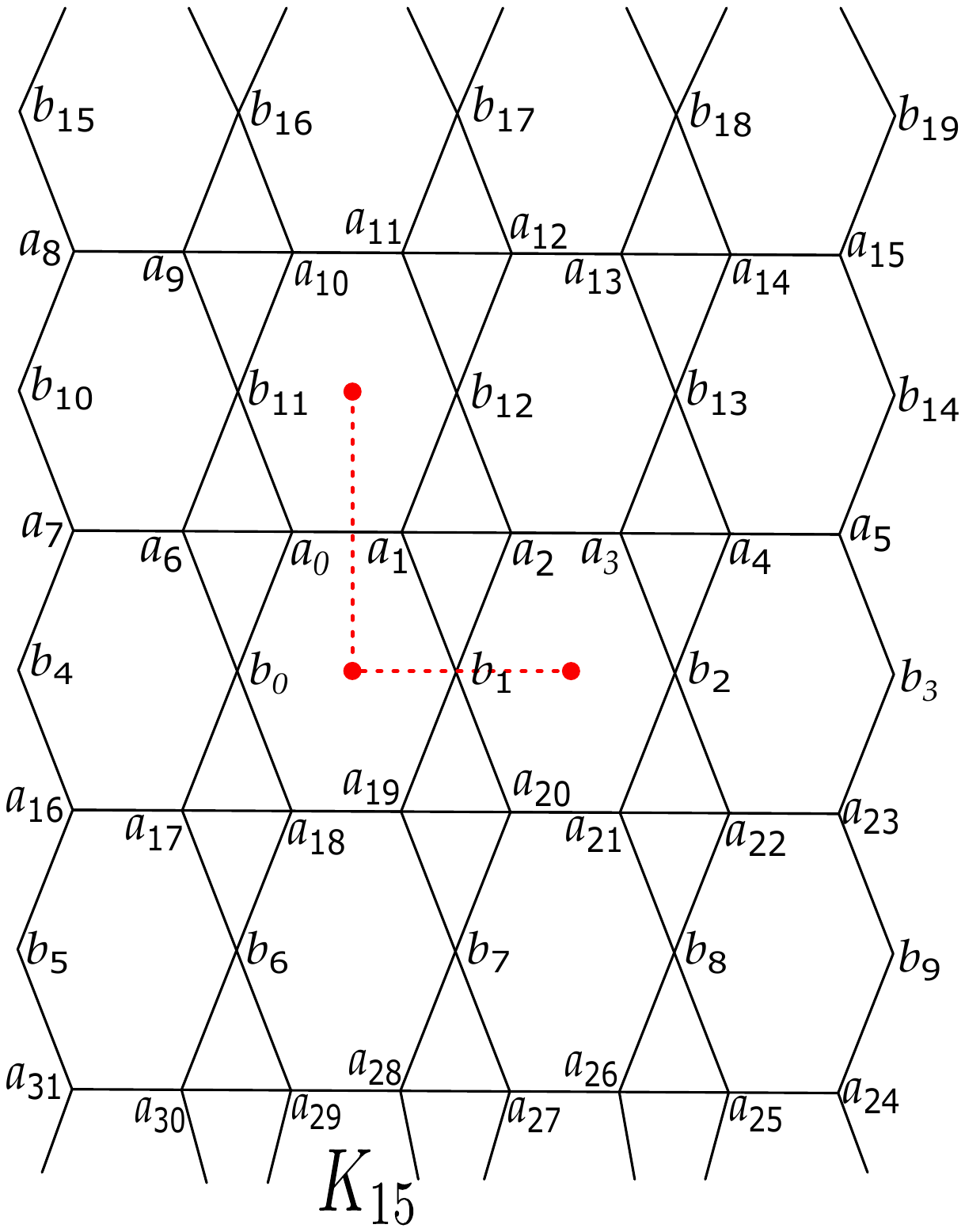}\hspace{5mm}
    \includegraphics[height=6cm, width= 6cm]{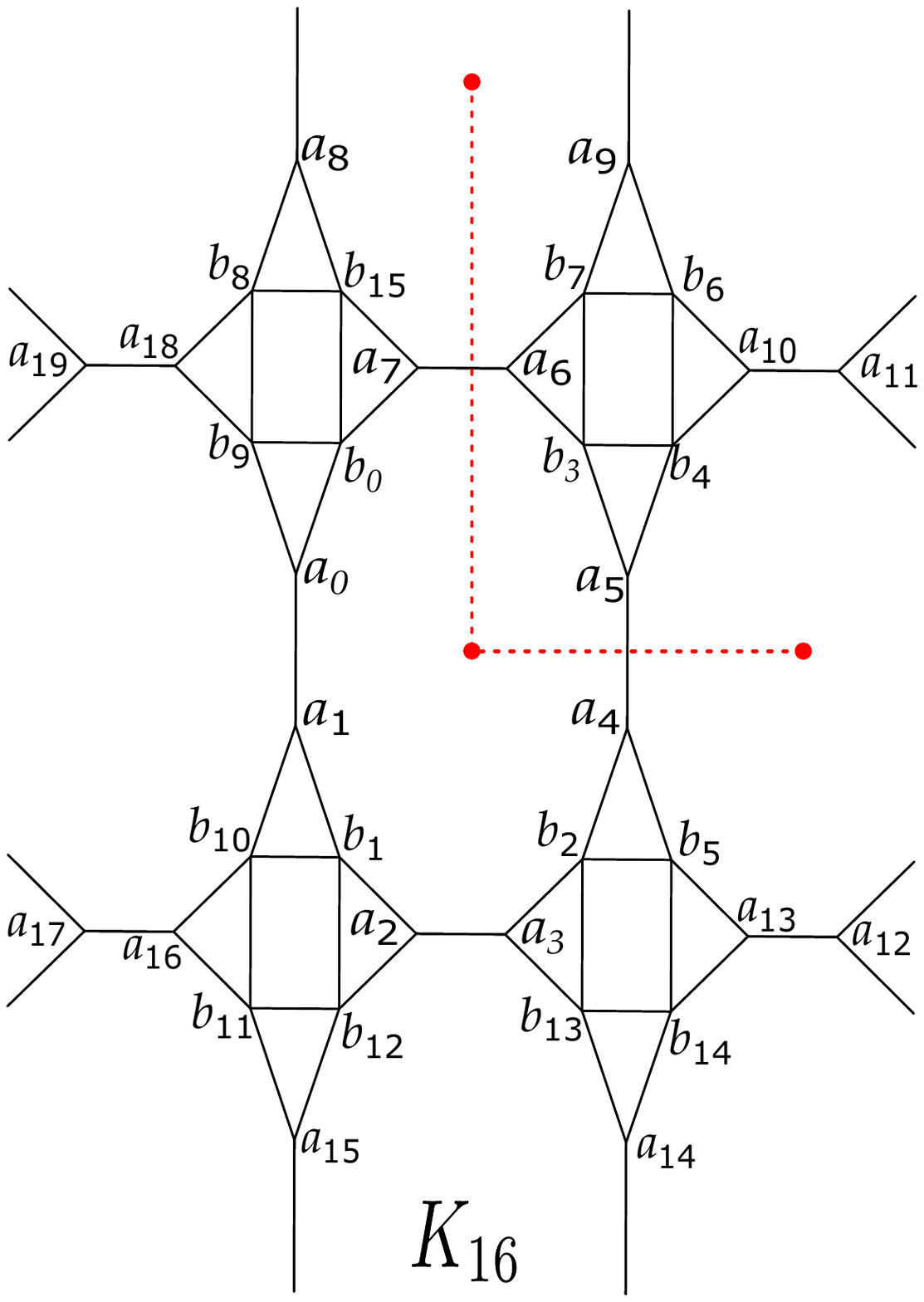}
    
     \vspace{10mm}
     
    \includegraphics[height=6cm, width= 6cm]{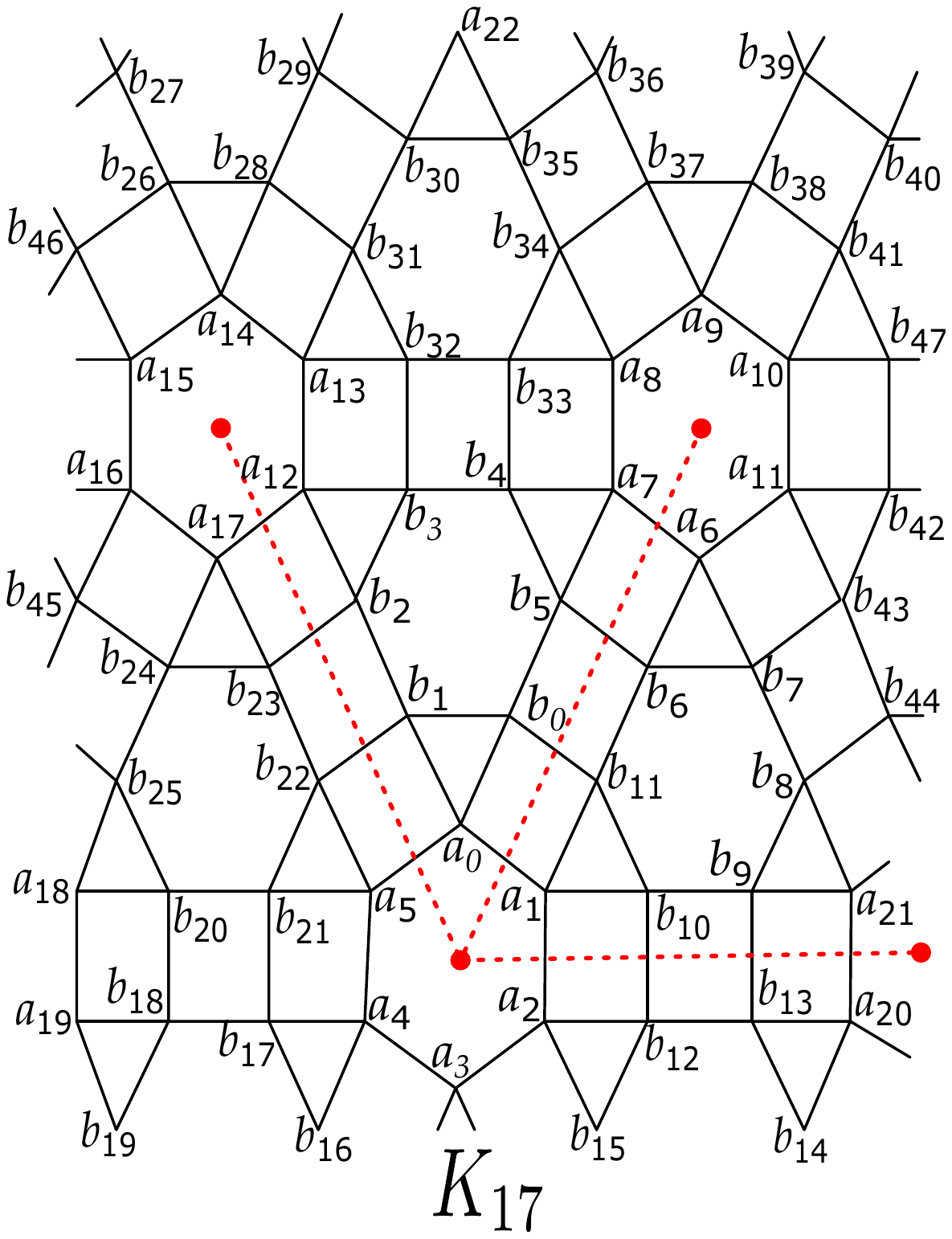}\hspace{5mm}
    \includegraphics[height=6cm, width= 6cm]{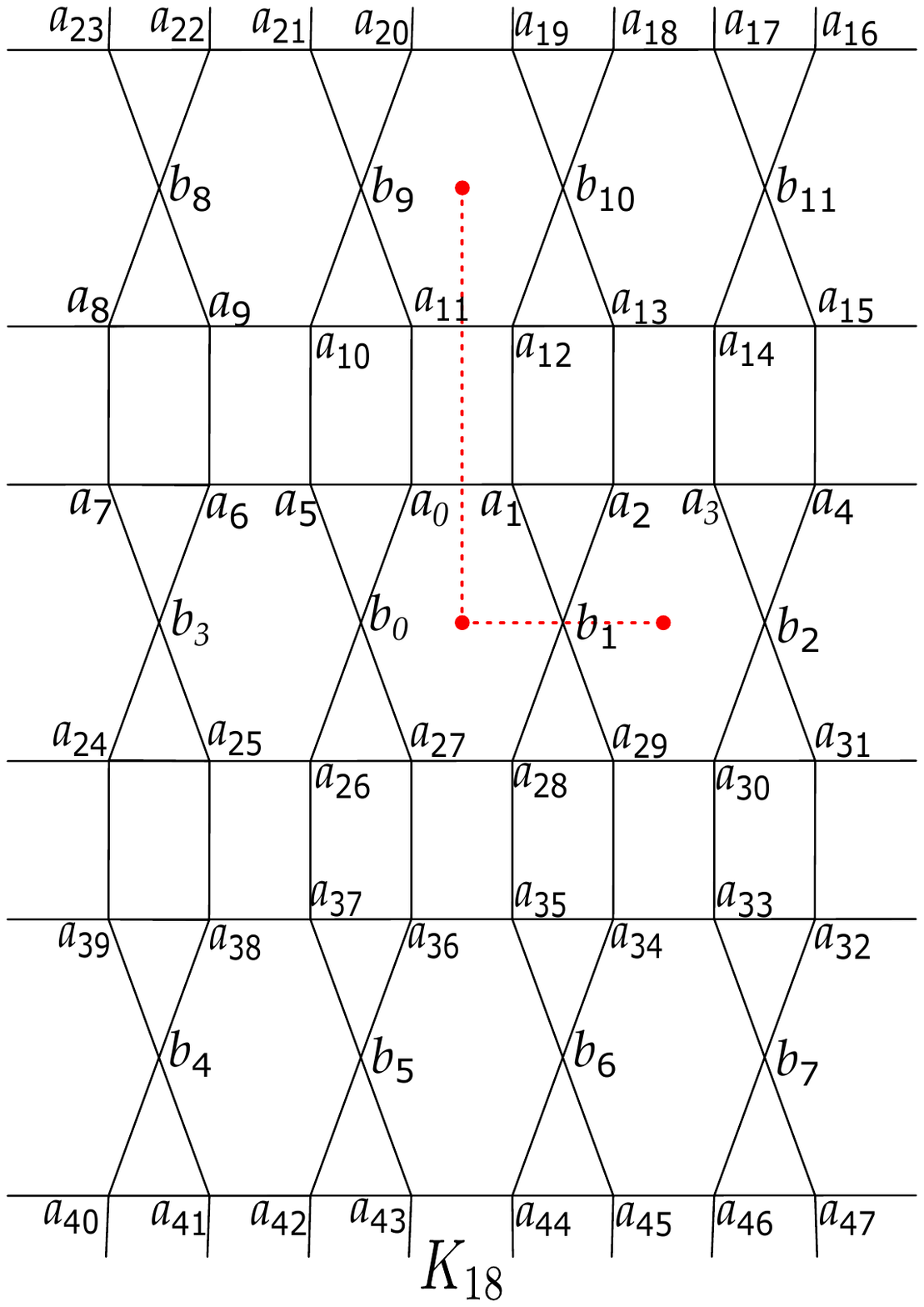}
     \vspace{10mm}
     
     \end{figure}
     
     \begin{figure}[H]
    \centering
    \includegraphics[height=6cm, width= 6cm]{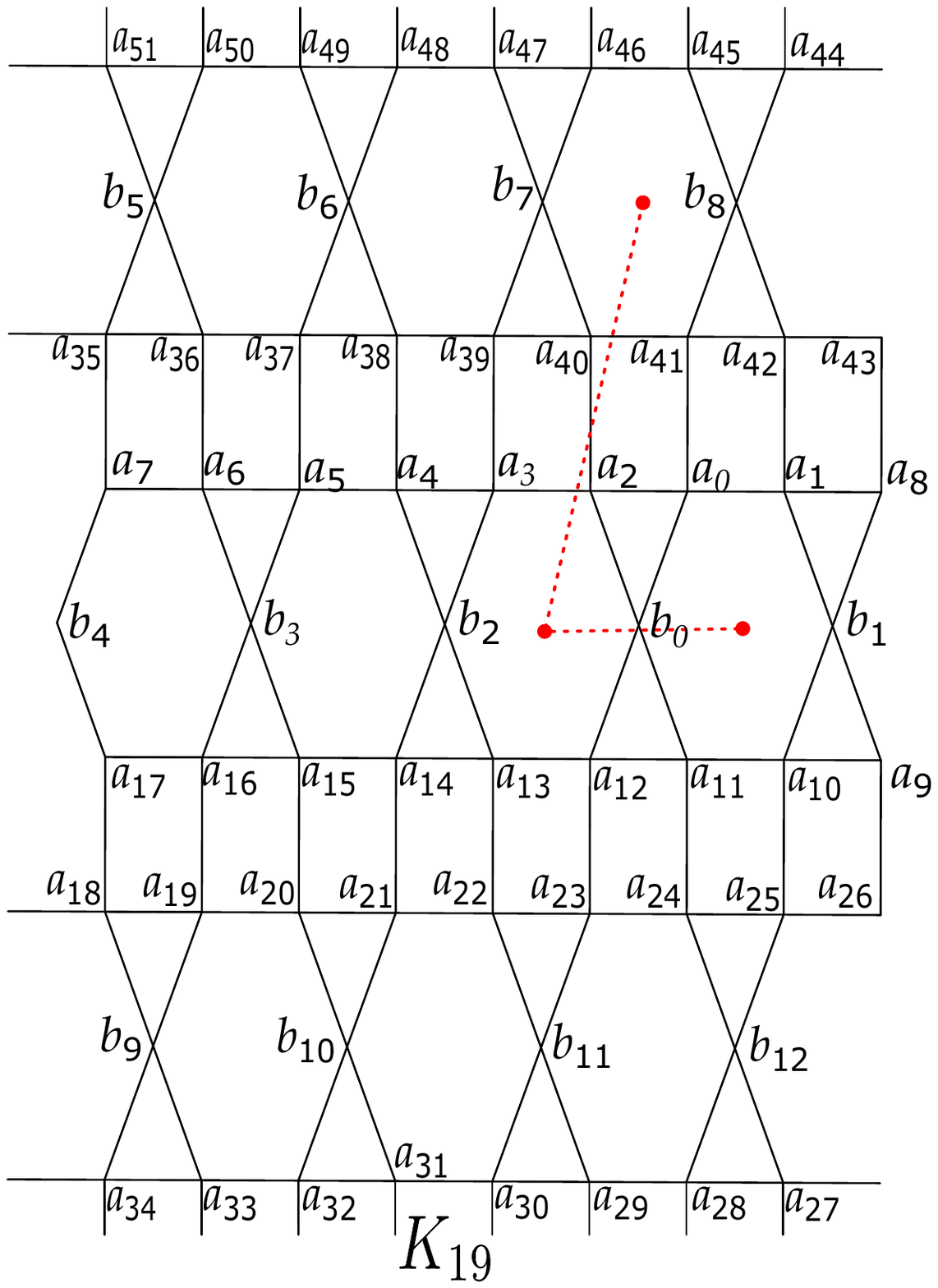}\hspace{5mm}
    \includegraphics[height=6cm, width= 6cm]{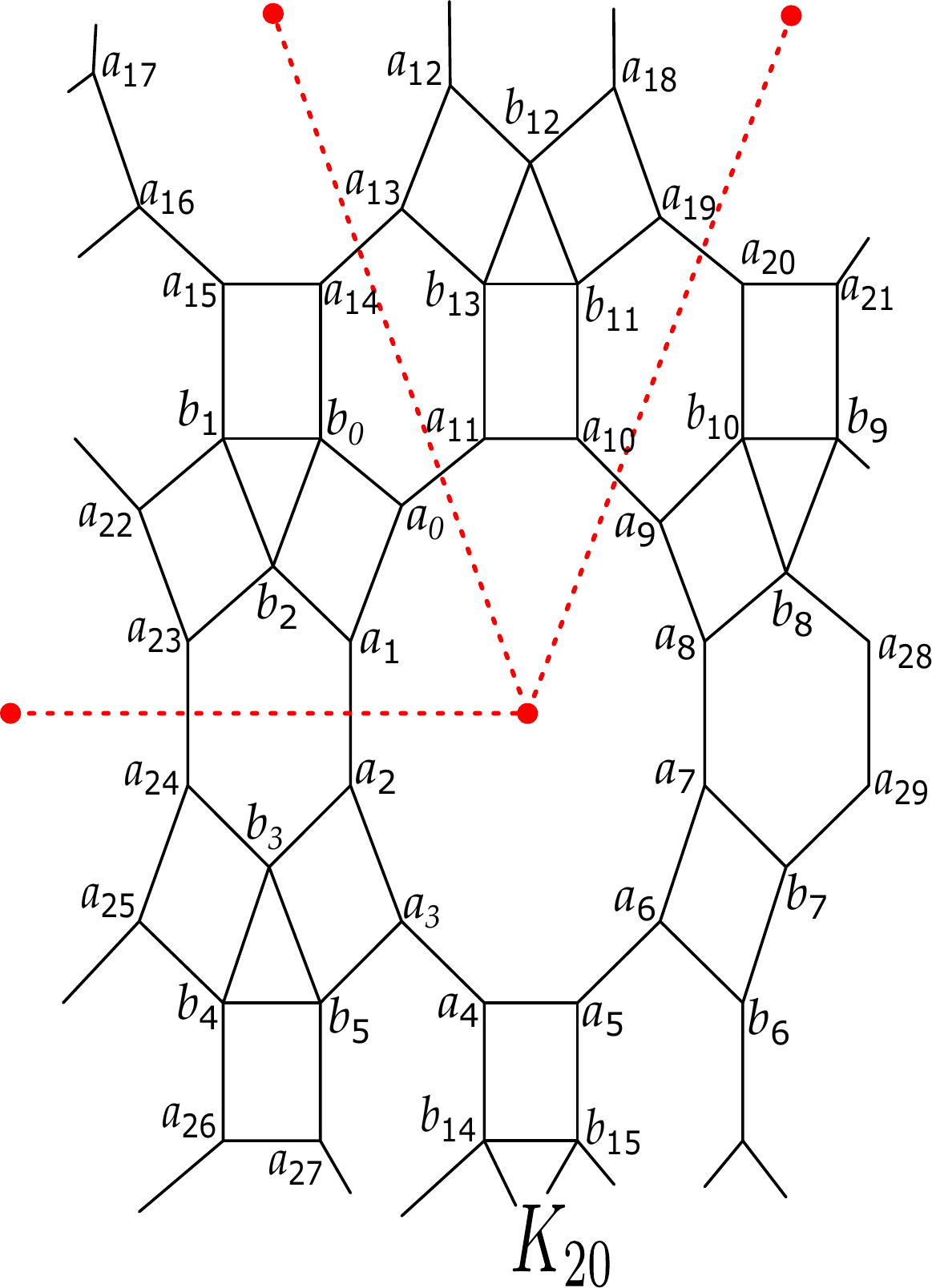}
\end{figure}

\section{$3$-uniform tilings of the plane}\label{3uniform}

\begin{figure}
    \centering
    \includegraphics[height=7cm, width= 6cm]{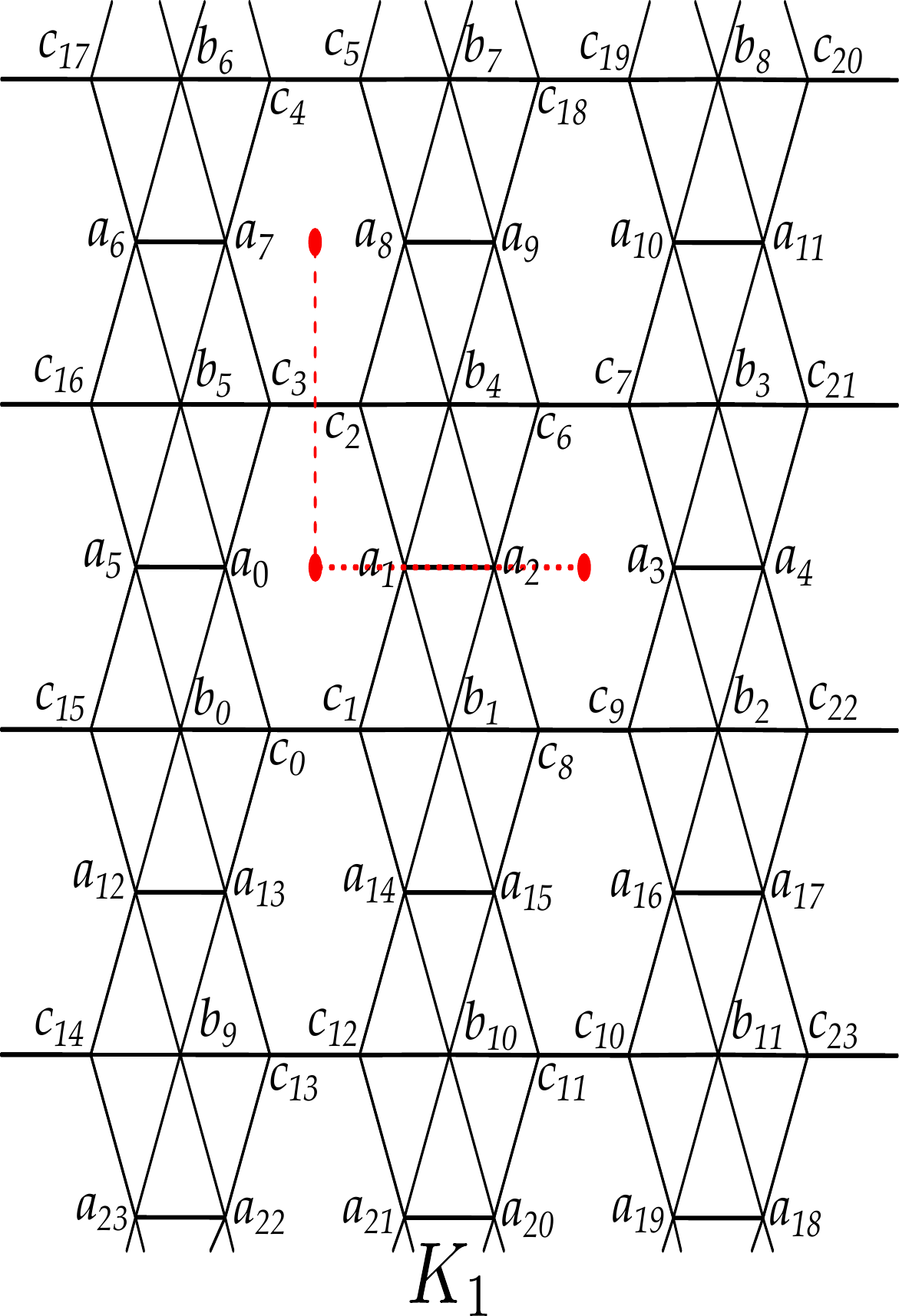}\hspace{5mm}
    \includegraphics[height=7cm, width= 6cm]{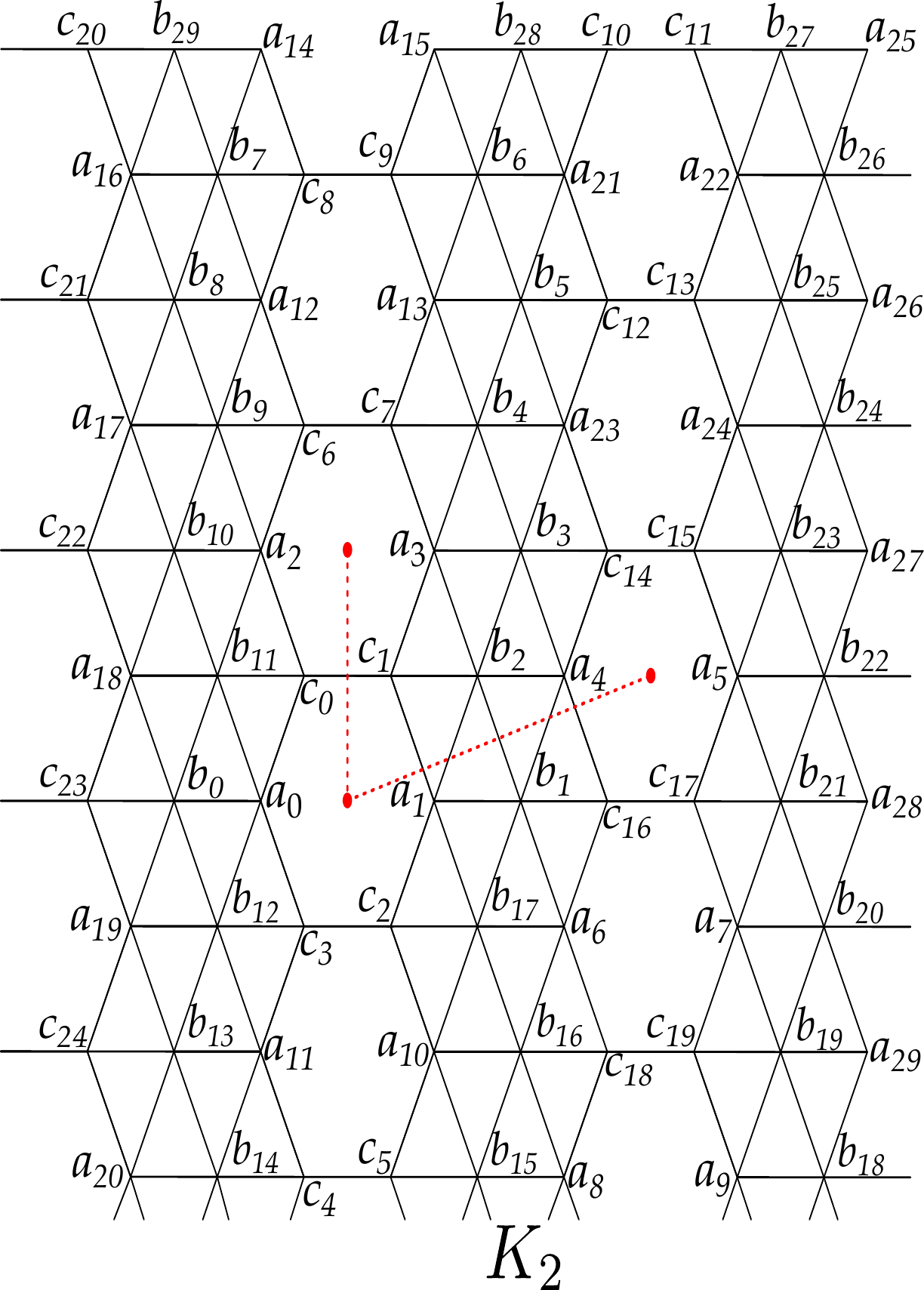}
    \end{figure}
    \begin{figure}
    \centering
    \includegraphics[height=6cm, width= 6cm]{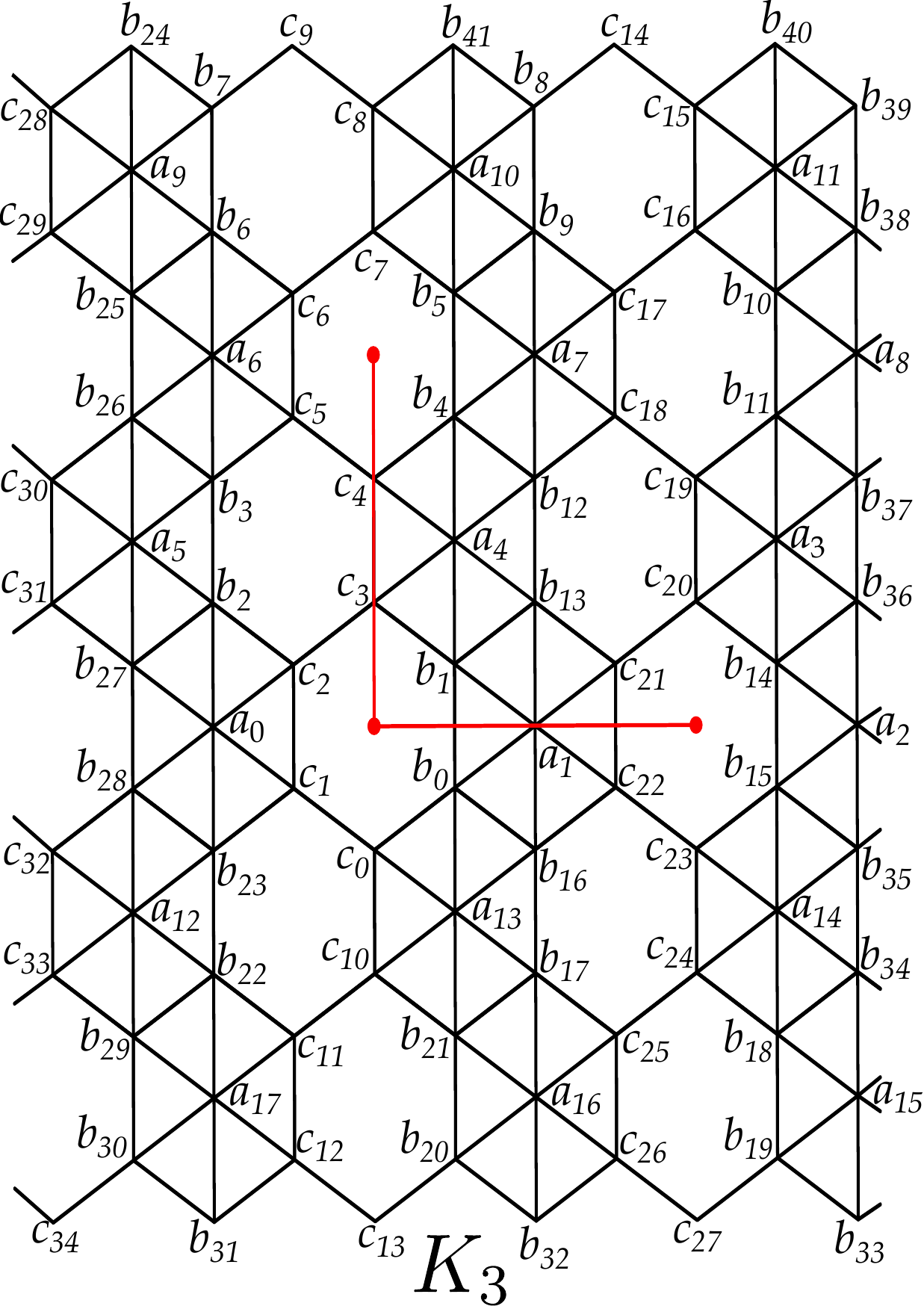}\hspace{5mm}
    \includegraphics[height=6cm, width= 6cm]{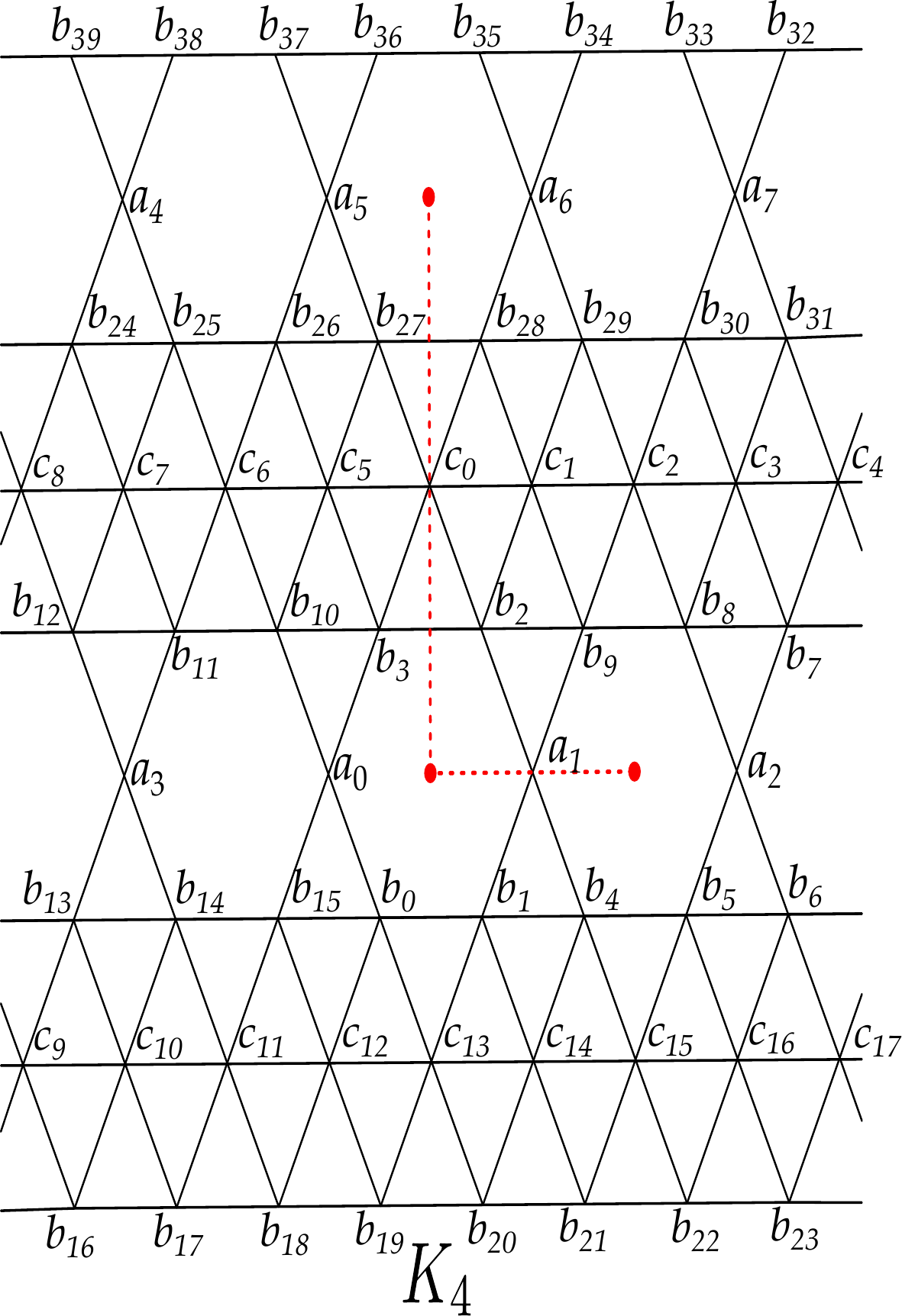}
     \end{figure}
     \begin{figure}
     \centering
    \includegraphics[height=6cm, width= 6cm]{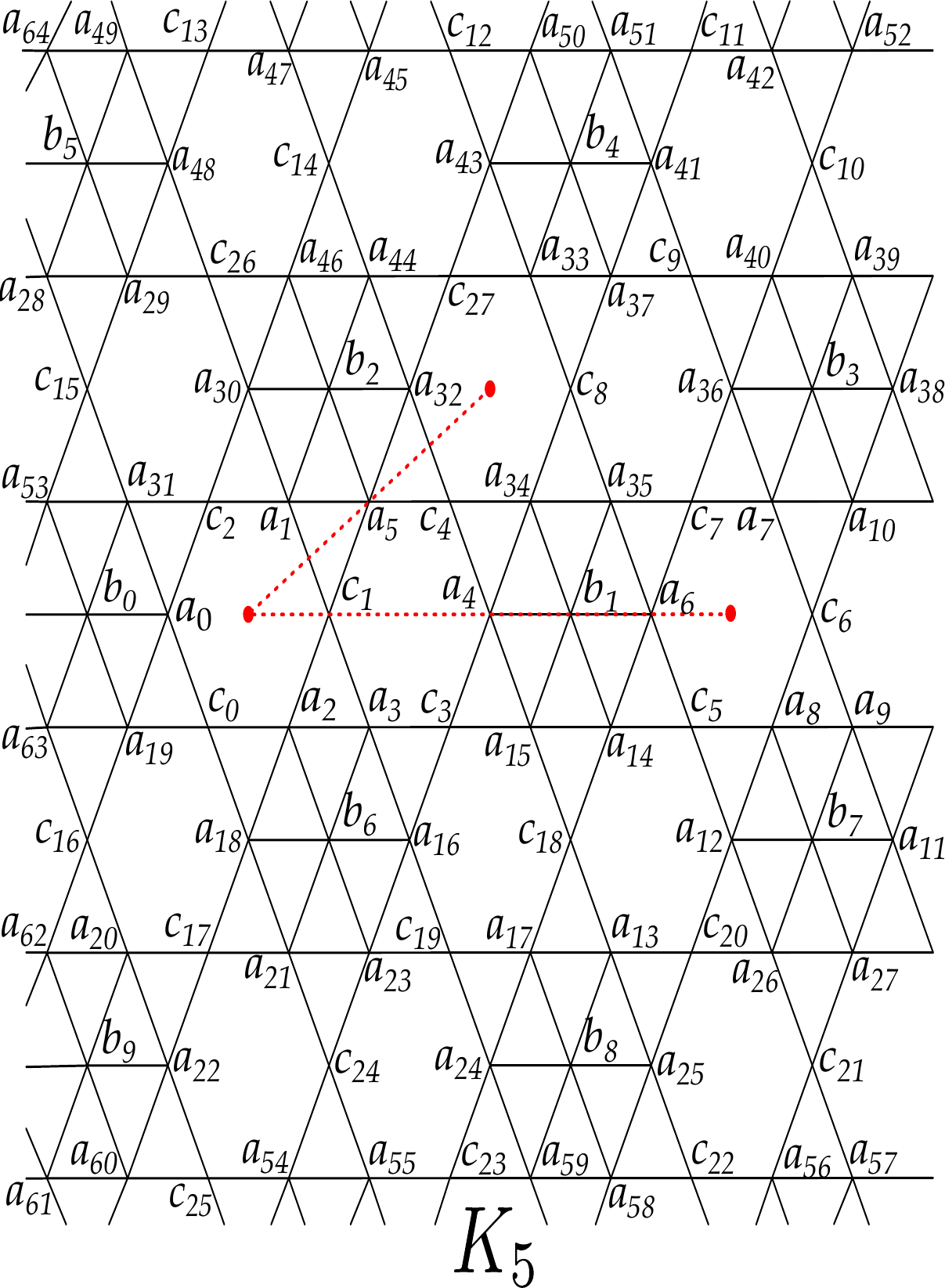}\hspace{5mm}
    \includegraphics[height=6cm, width= 6cm]{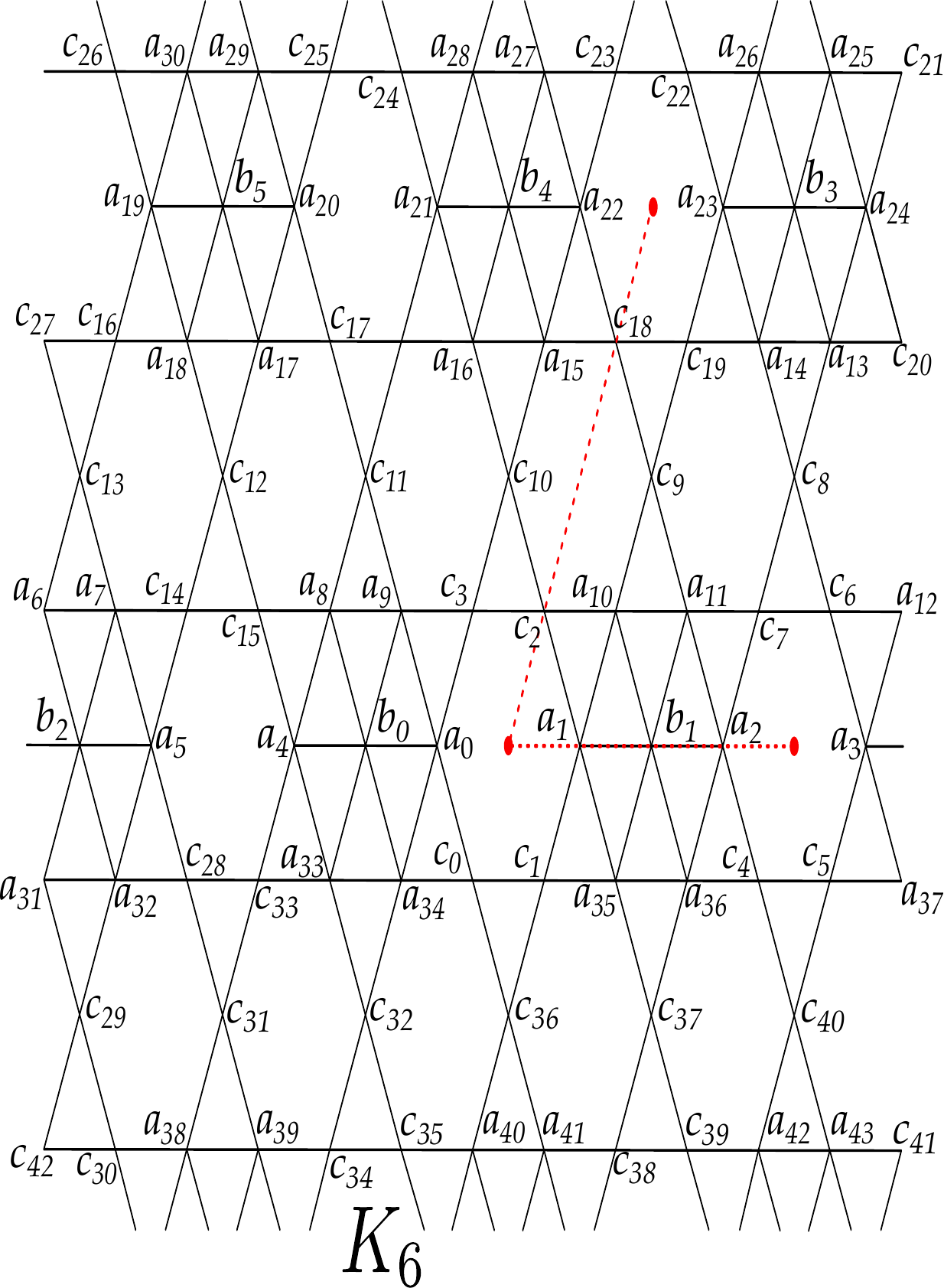}
    \end{figure}
     \begin{figure}
     \centering
    \includegraphics[height=6cm, width= 6cm]{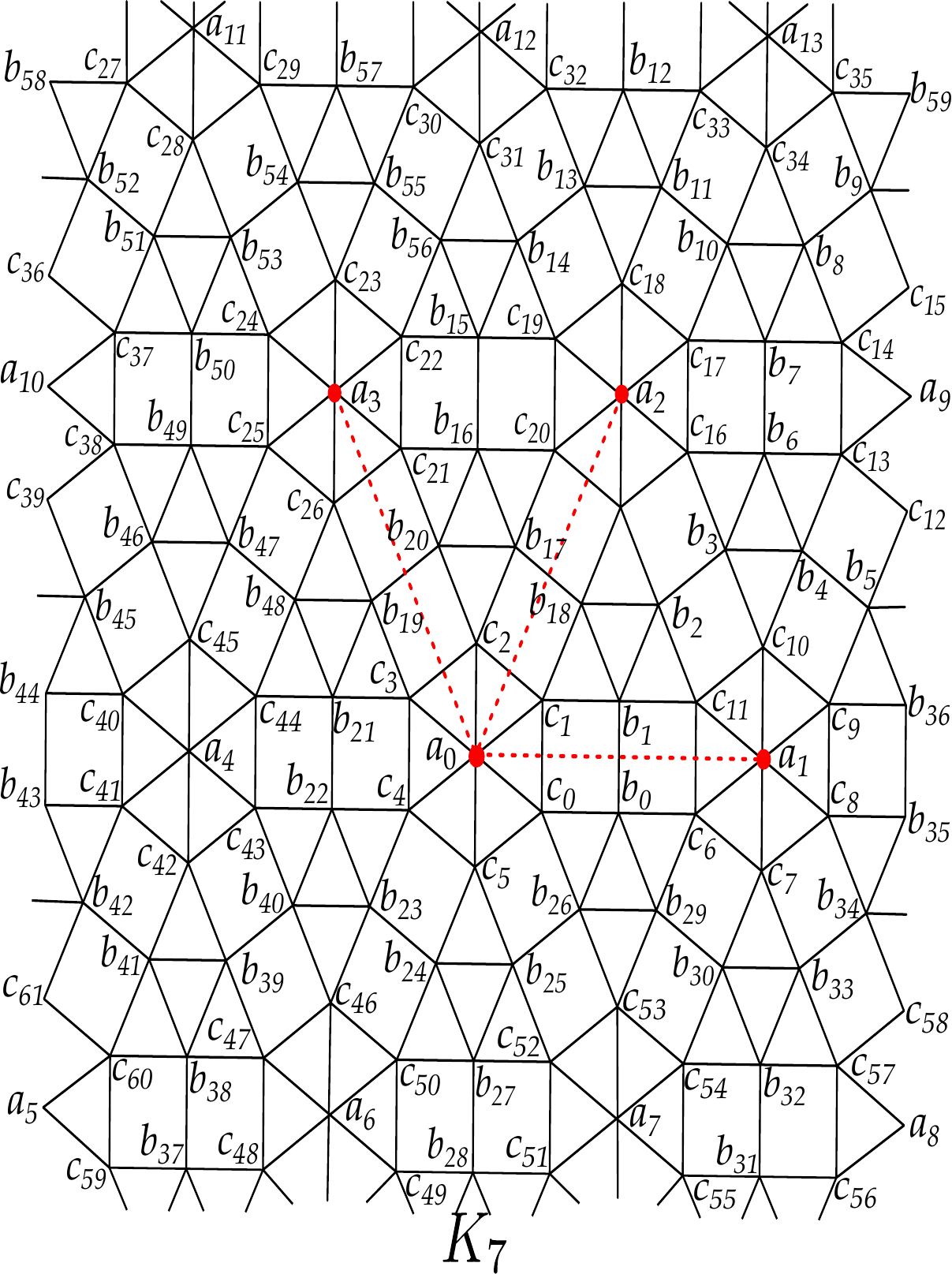}\hspace{5mm}
    \includegraphics[height=6cm, width= 6cm]{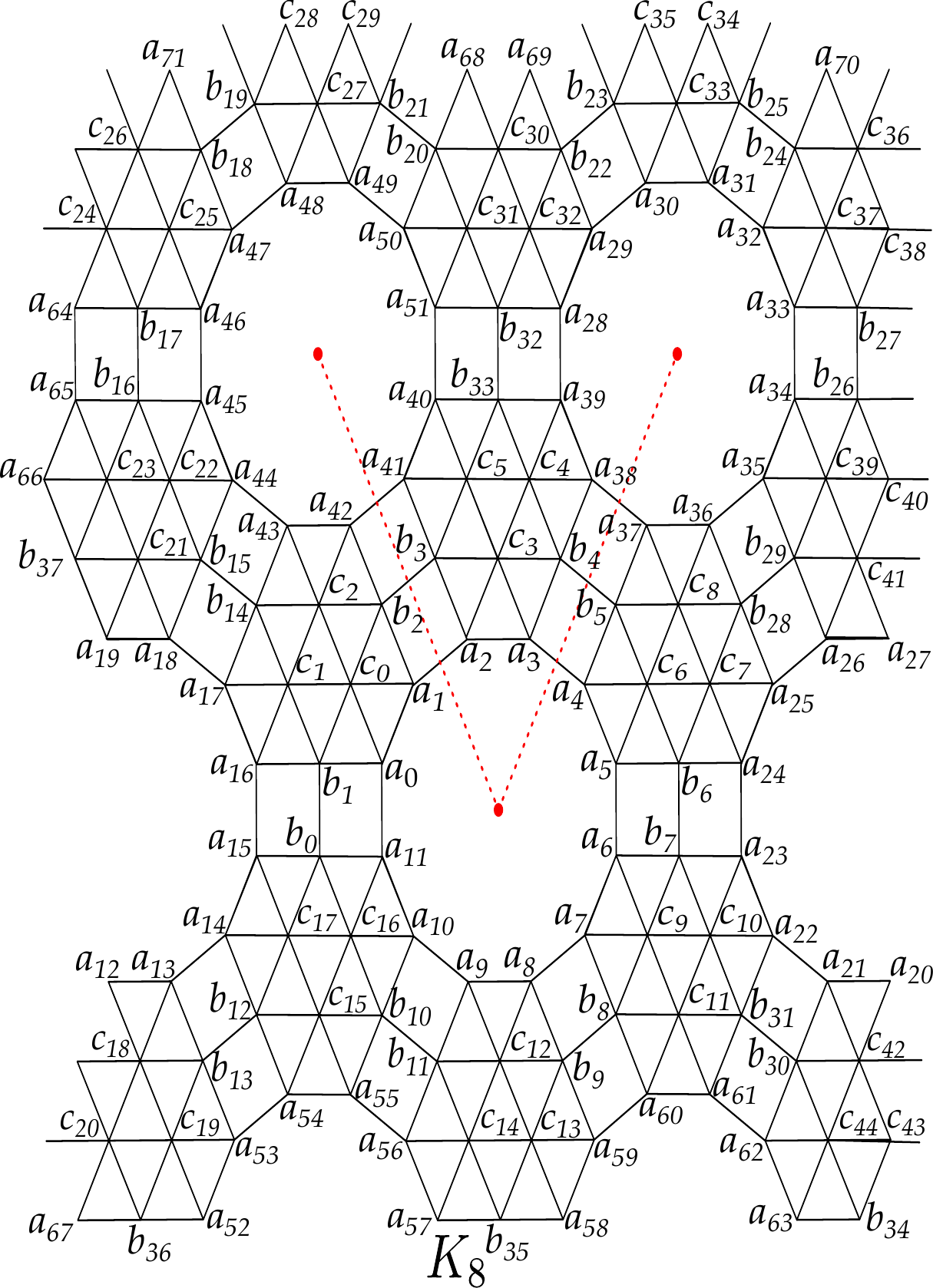}
    \end{figure}
     \begin{figure}
     \centering
    \includegraphics[height=6cm, width= 6cm]{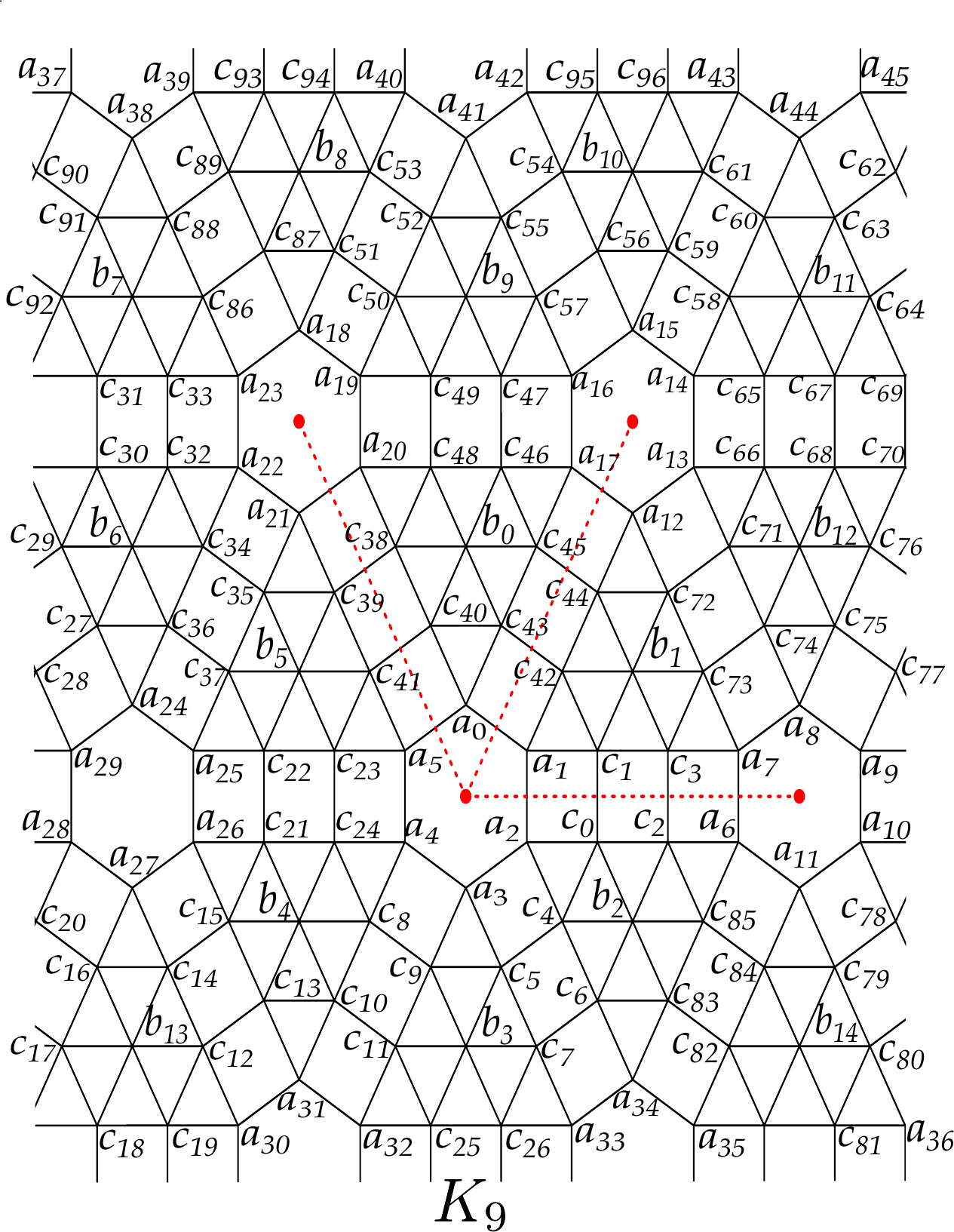}\hspace{5mm}
    \includegraphics[height=6cm, width= 6cm]{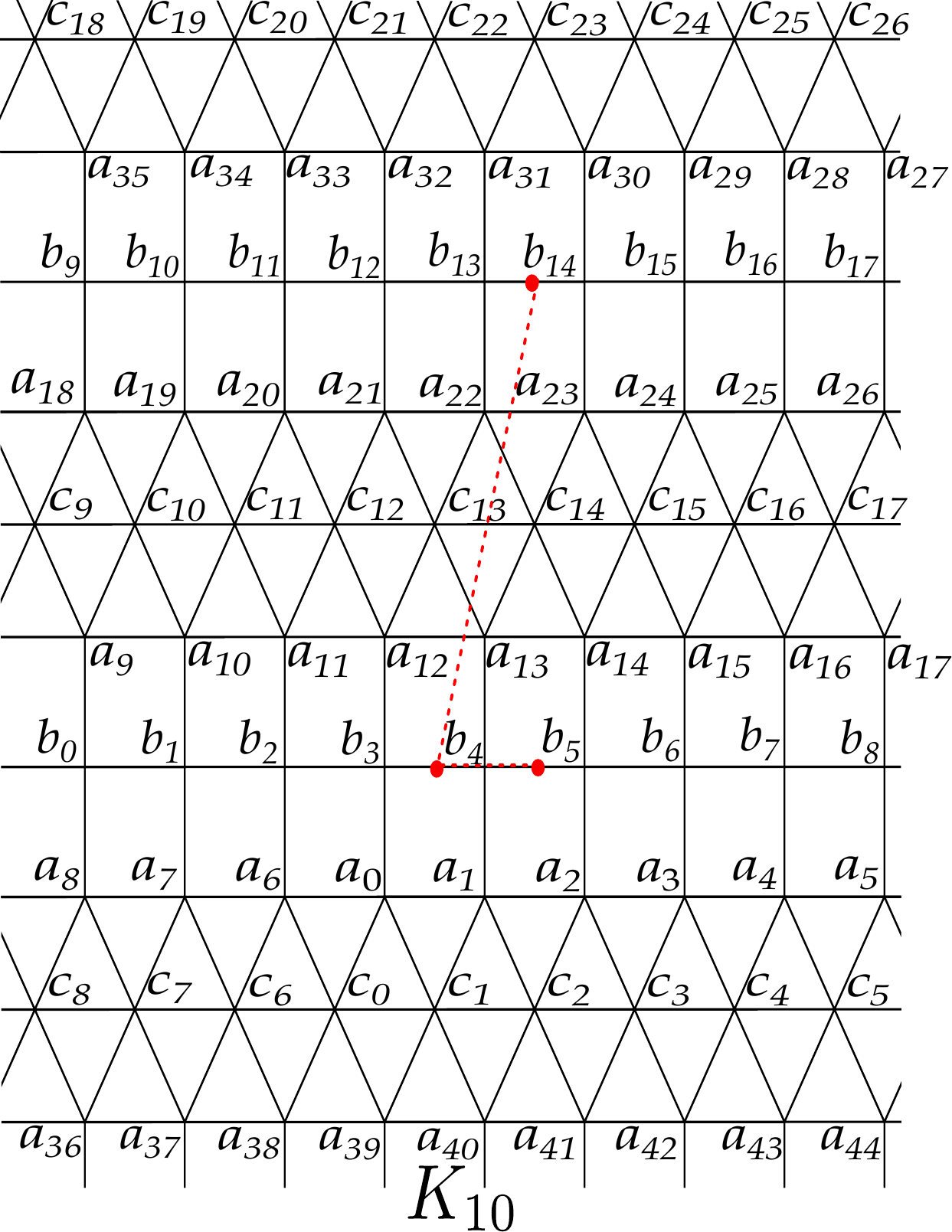}
    \end{figure}
    \begin{figure}
    \centering
    \includegraphics[height=6cm, width= 6cm]{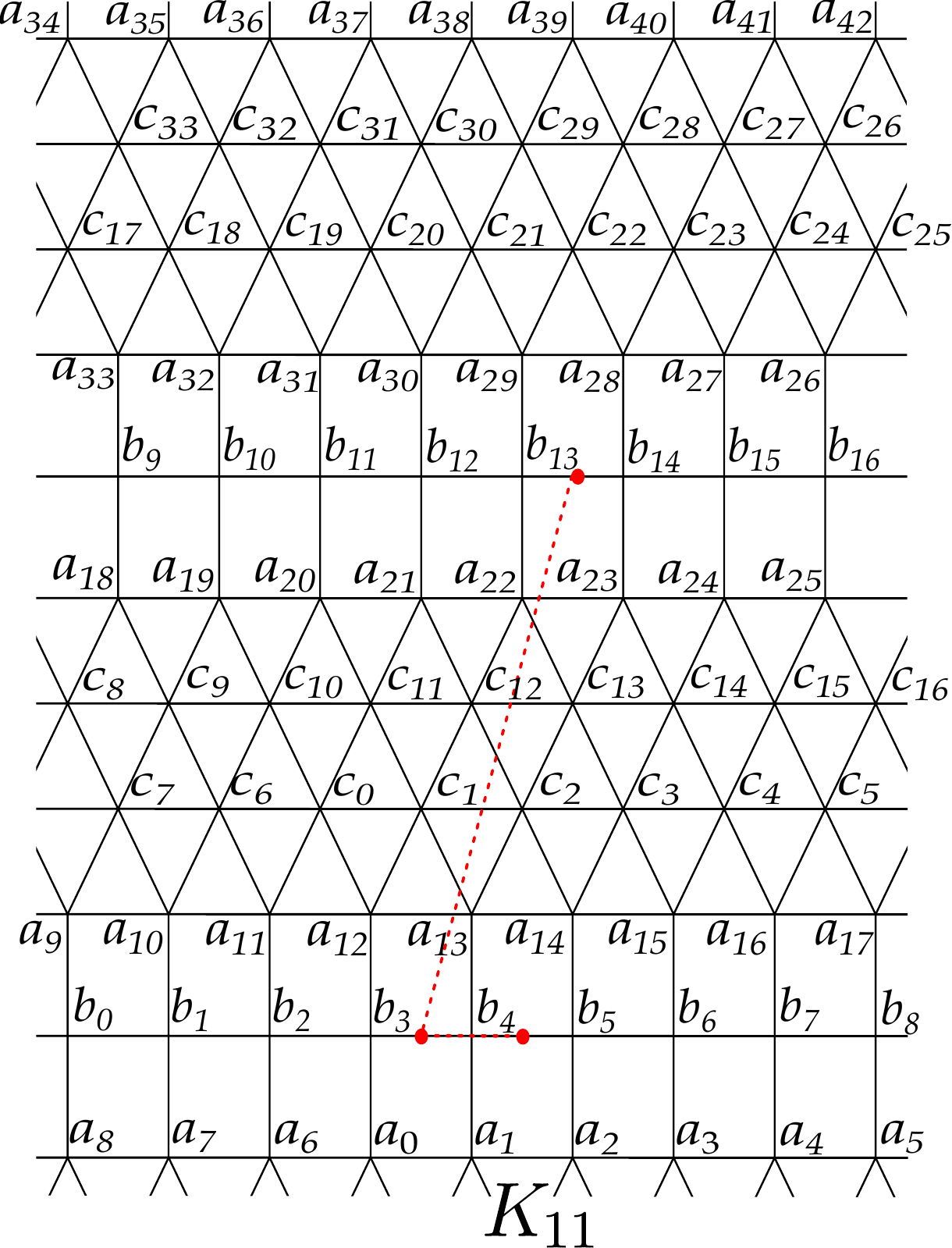}\hspace{5mm}
    \includegraphics[height=6cm, width= 6cm]{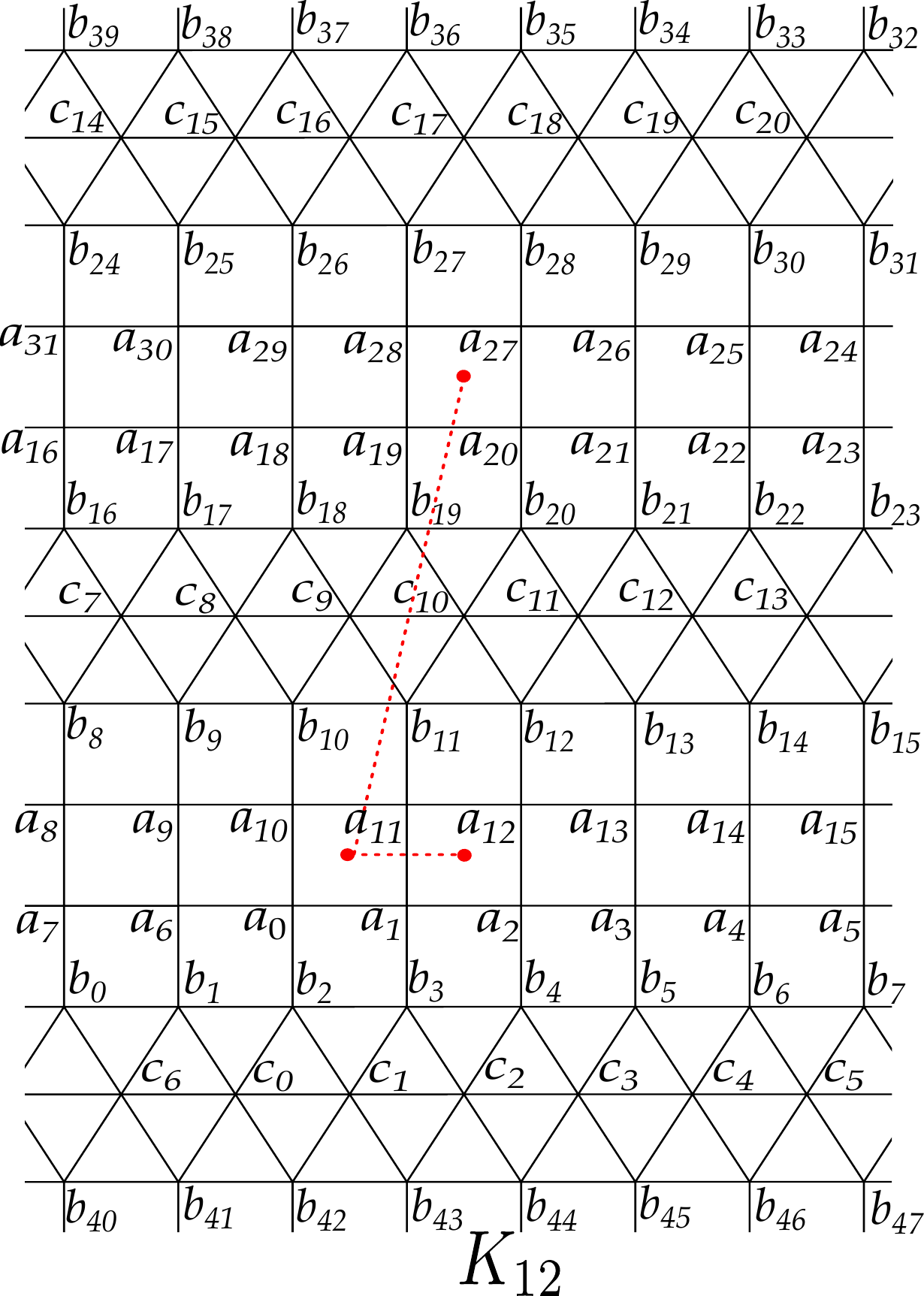}
    \end{figure}
     \begin{figure}
     \centering
    \includegraphics[height=6cm, width= 6cm]{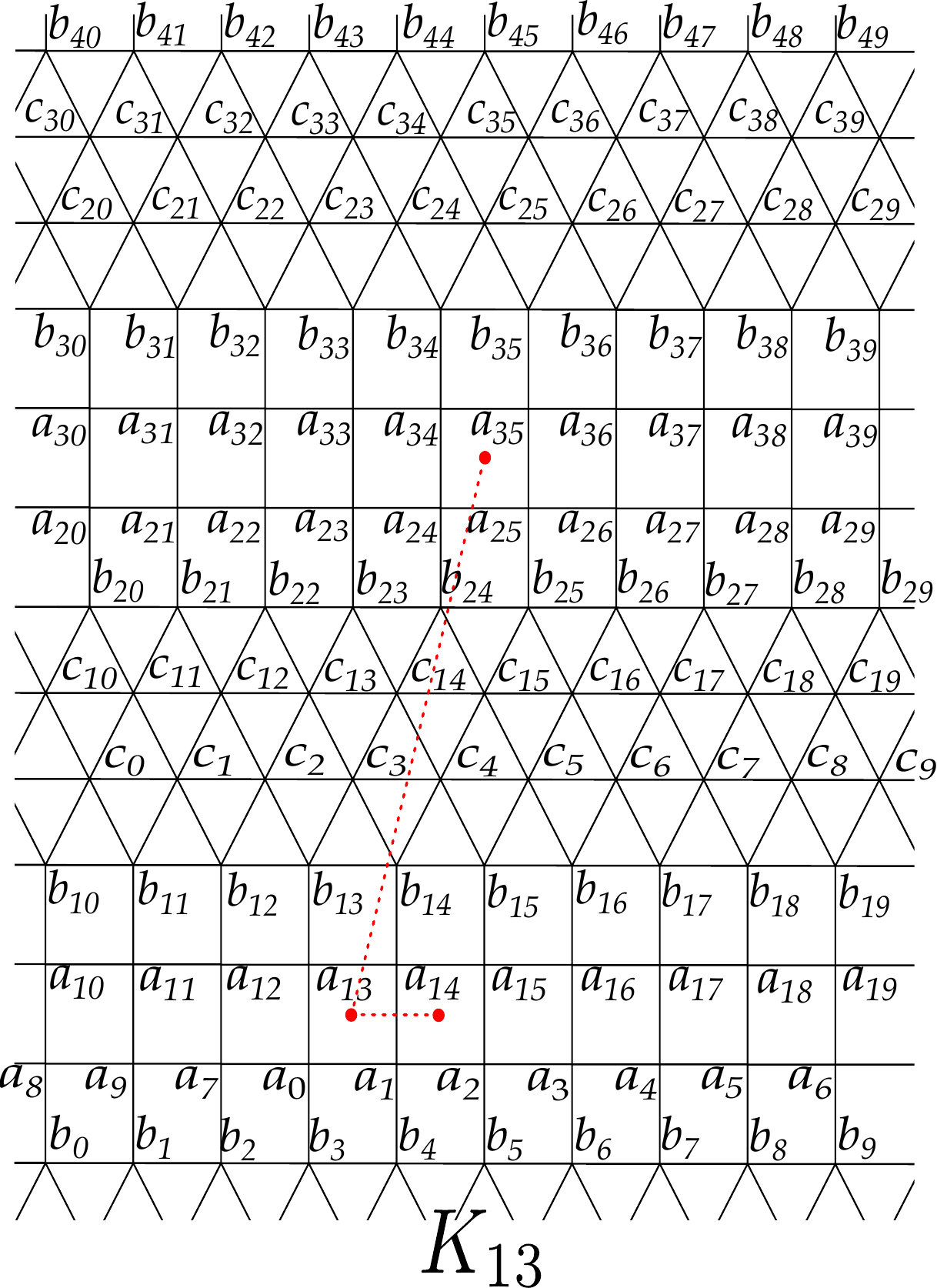}\hspace{5mm}
    \includegraphics[height=6cm, width= 6cm]{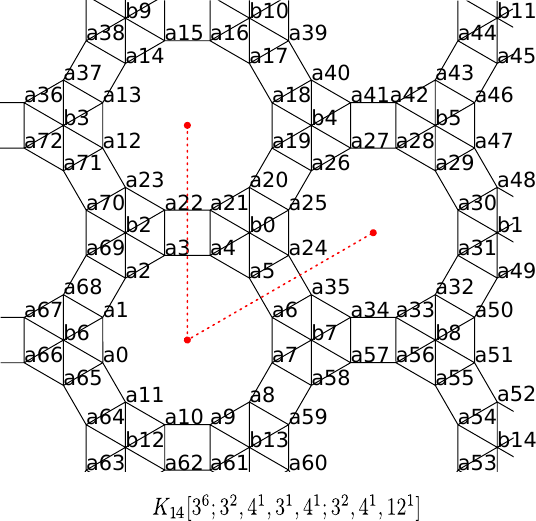}
    \end{figure}
     \begin{figure}
     \centering
    \includegraphics[height=6cm, width= 6cm]{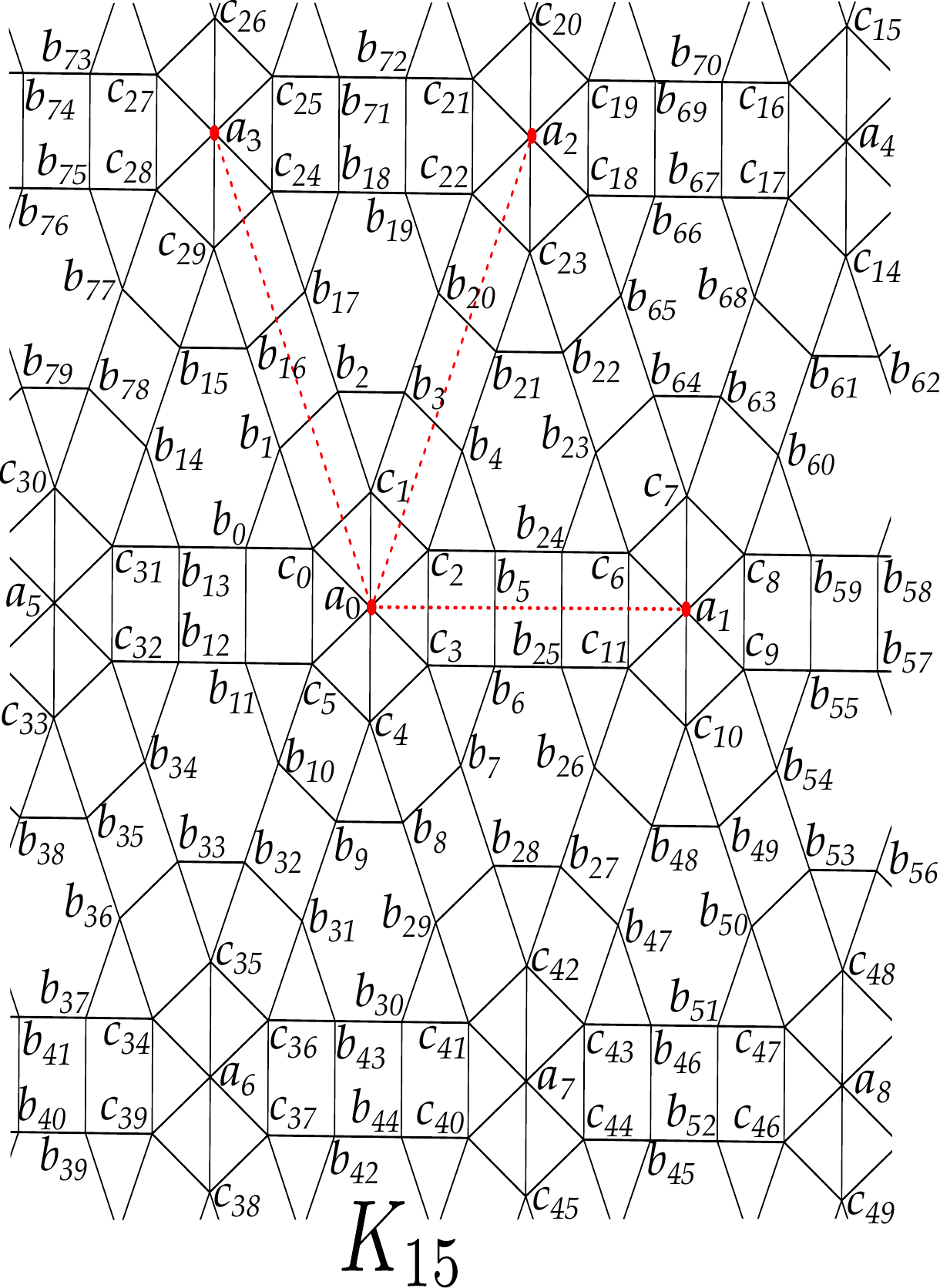}\hspace{5mm}
    \includegraphics[height=6cm, width= 6cm]{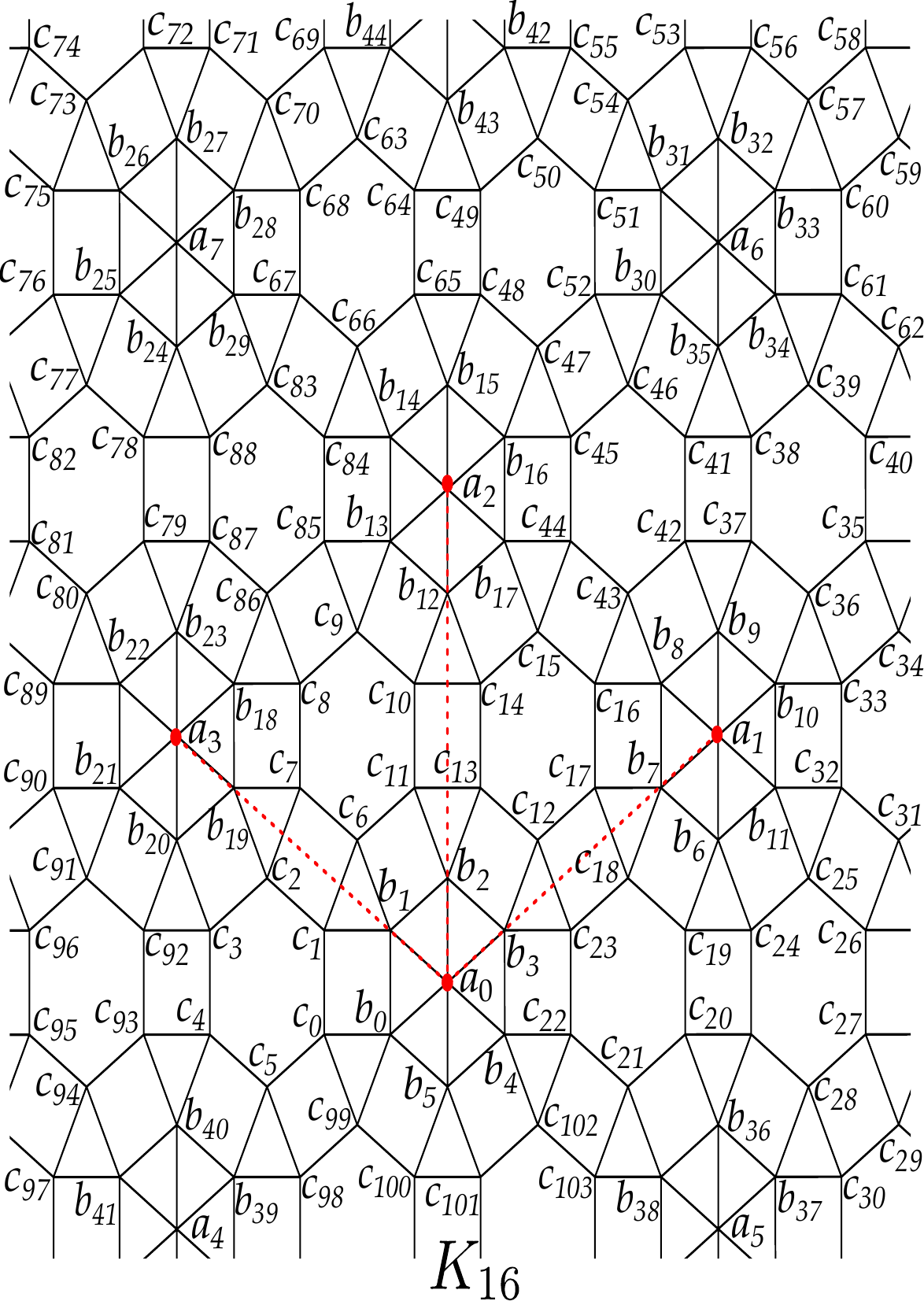}
\end{figure}
\begin{figure}
\centering
\includegraphics[height=6cm, width= 6cm]{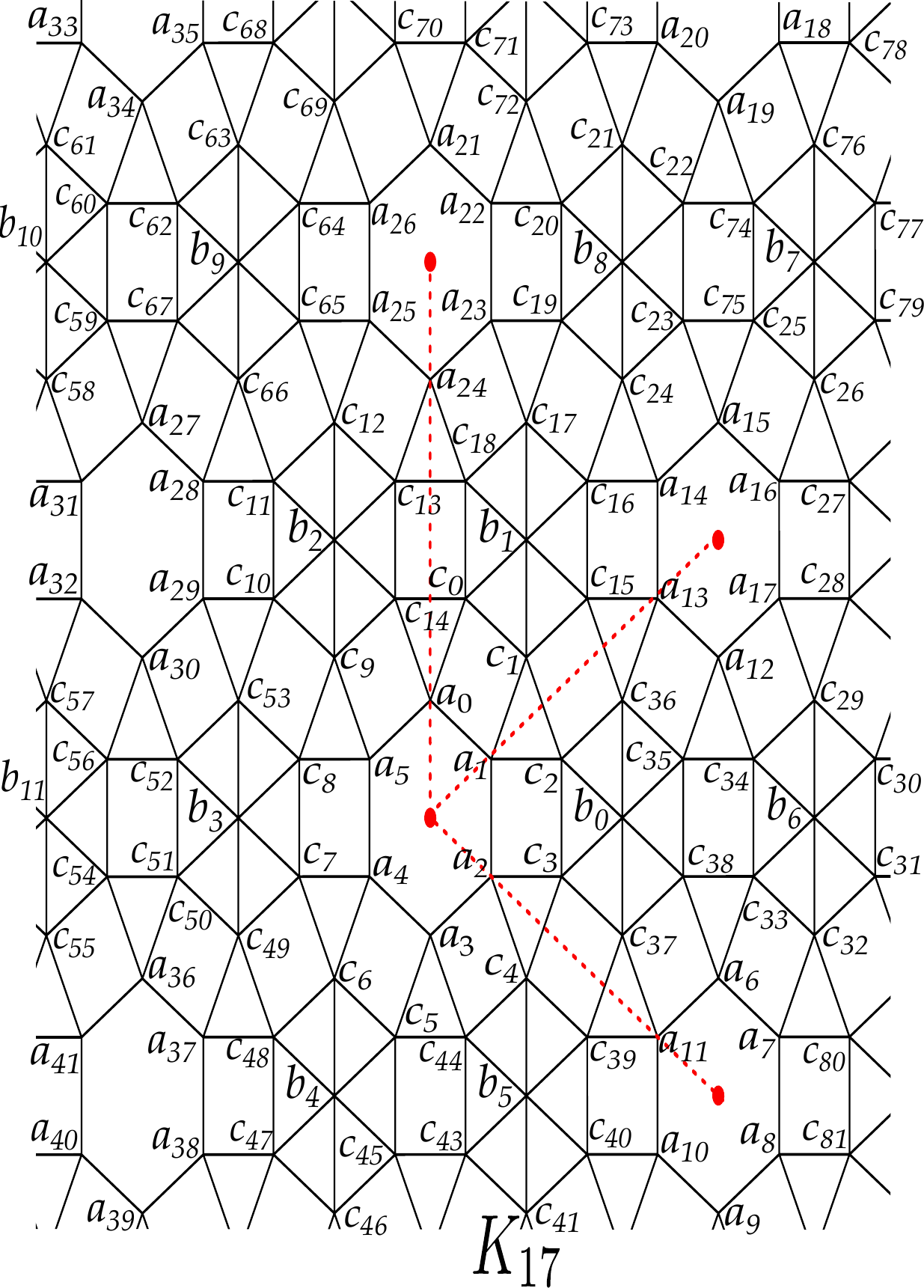}\hspace{5mm}
\includegraphics[height=6cm, width= 6cm]{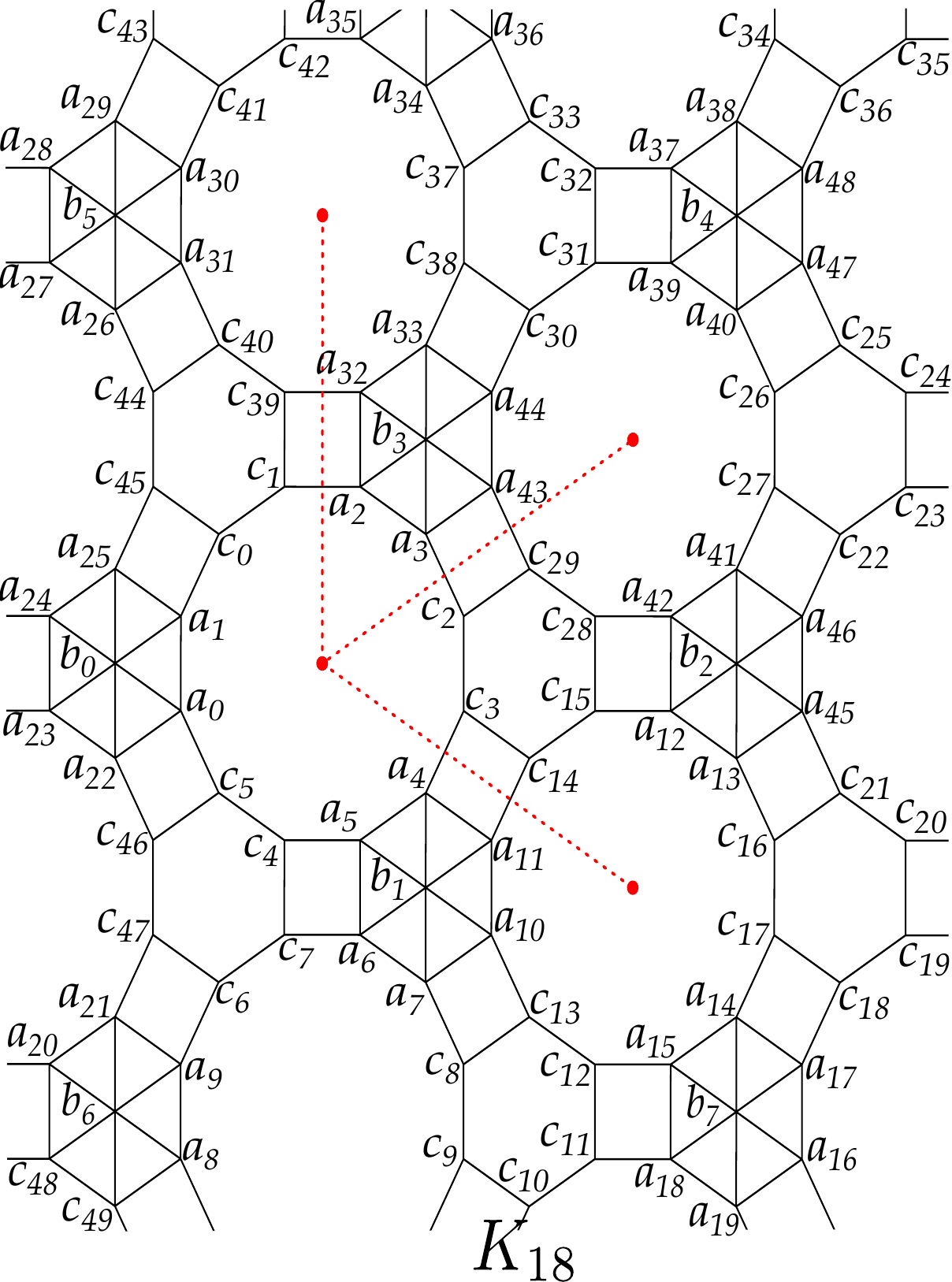}
\end{figure}
     \begin{figure}
     \centering
\includegraphics[height=6cm, width= 6cm]{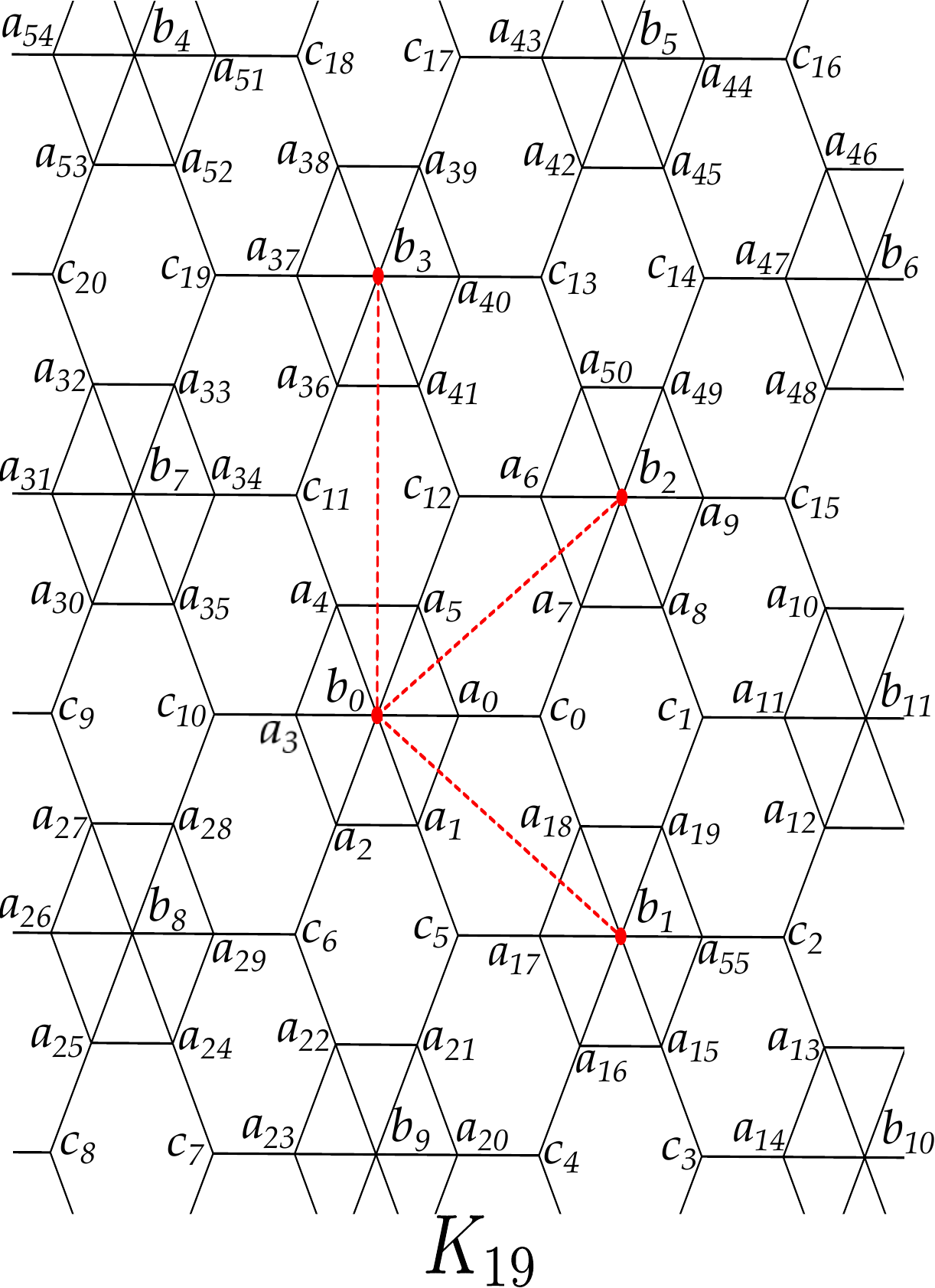}\hspace{5mm}
\includegraphics[height=6cm, width= 6cm]{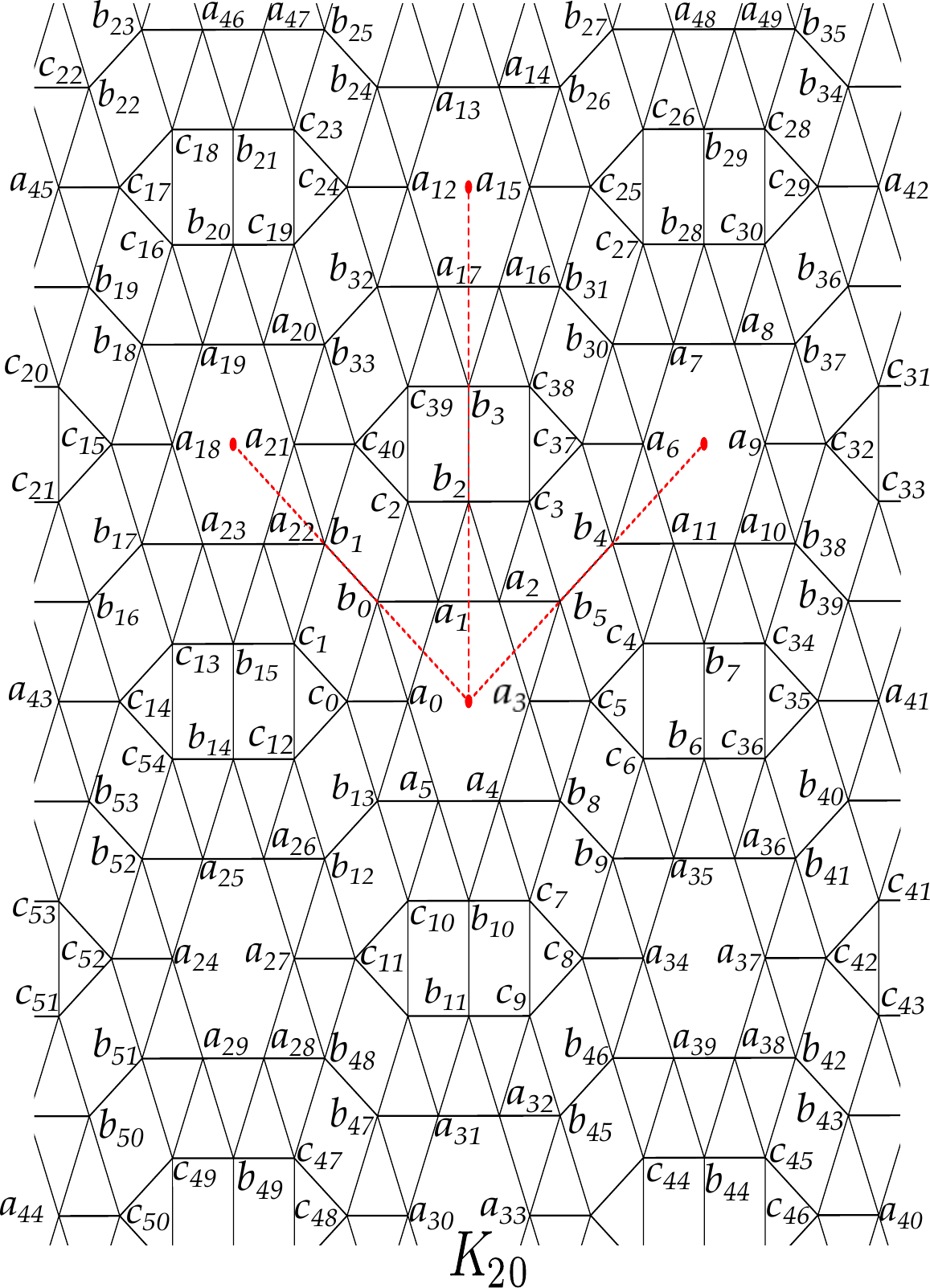}
\end{figure}
     \begin{figure}
     \centering
\includegraphics[height=6cm, width= 6cm]{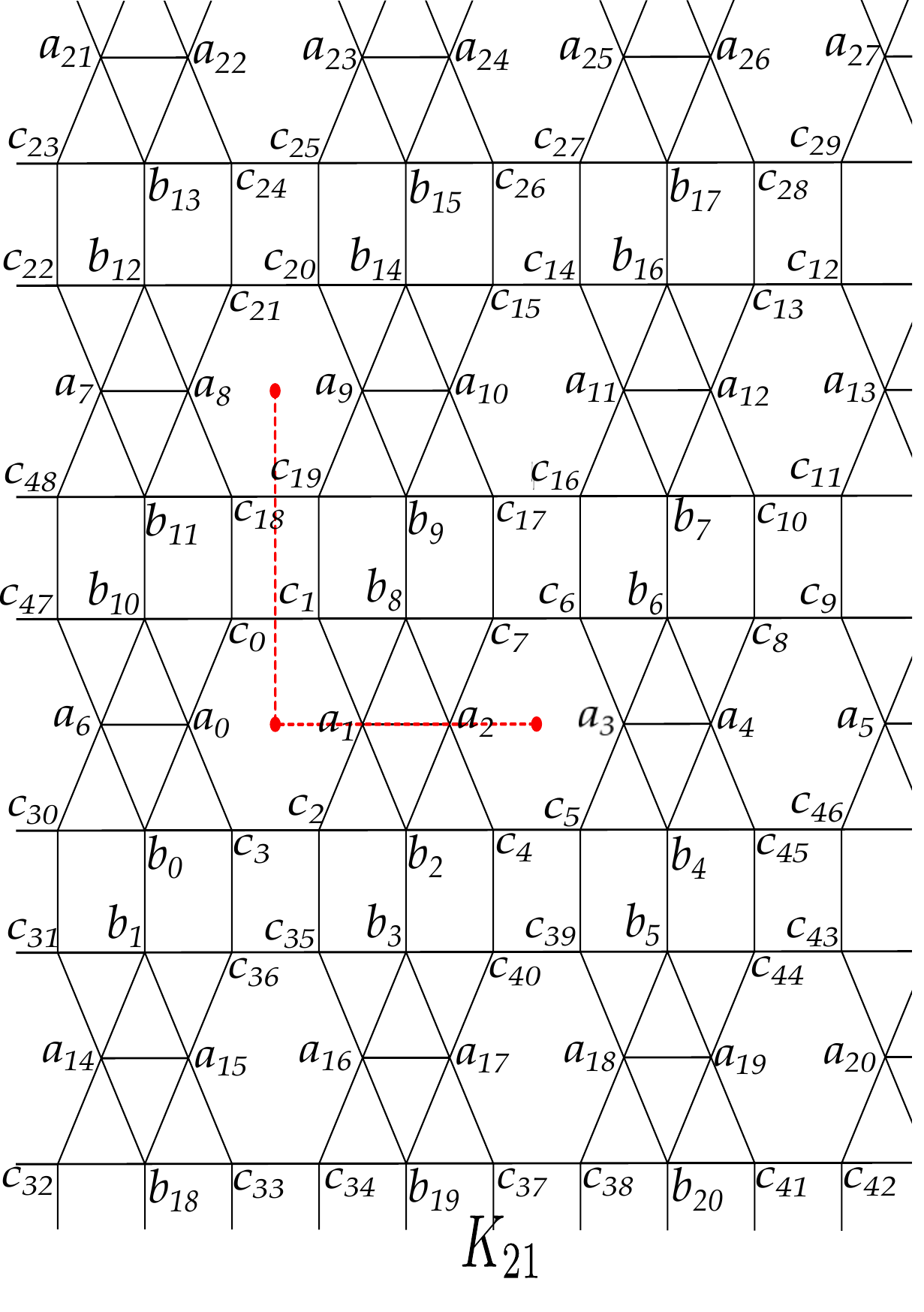}\hspace{5mm}
\includegraphics[height=6cm, width= 6cm]{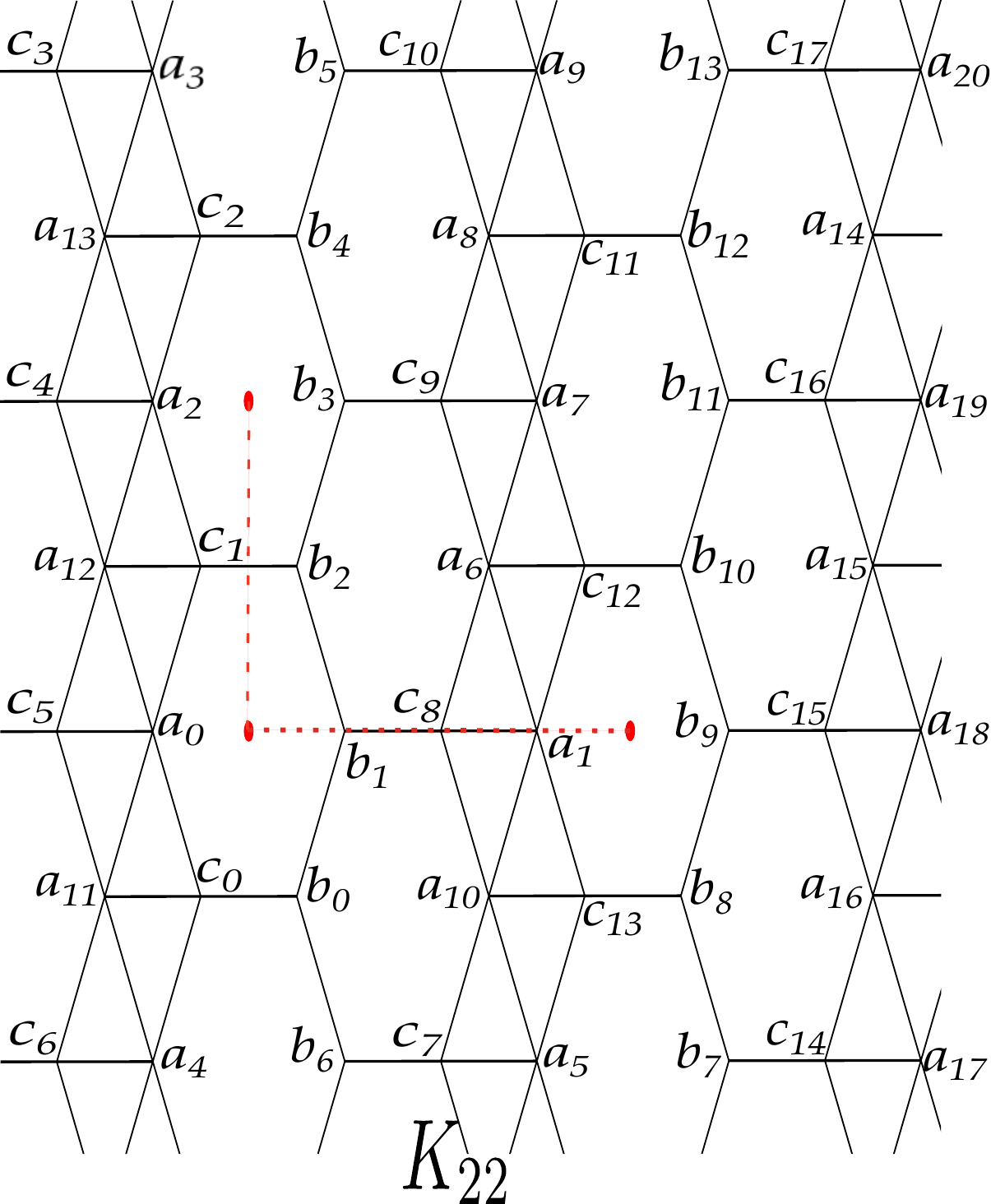}
\end{figure}
     \begin{figure}
     \centering
\includegraphics[height=6cm, width= 6cm]{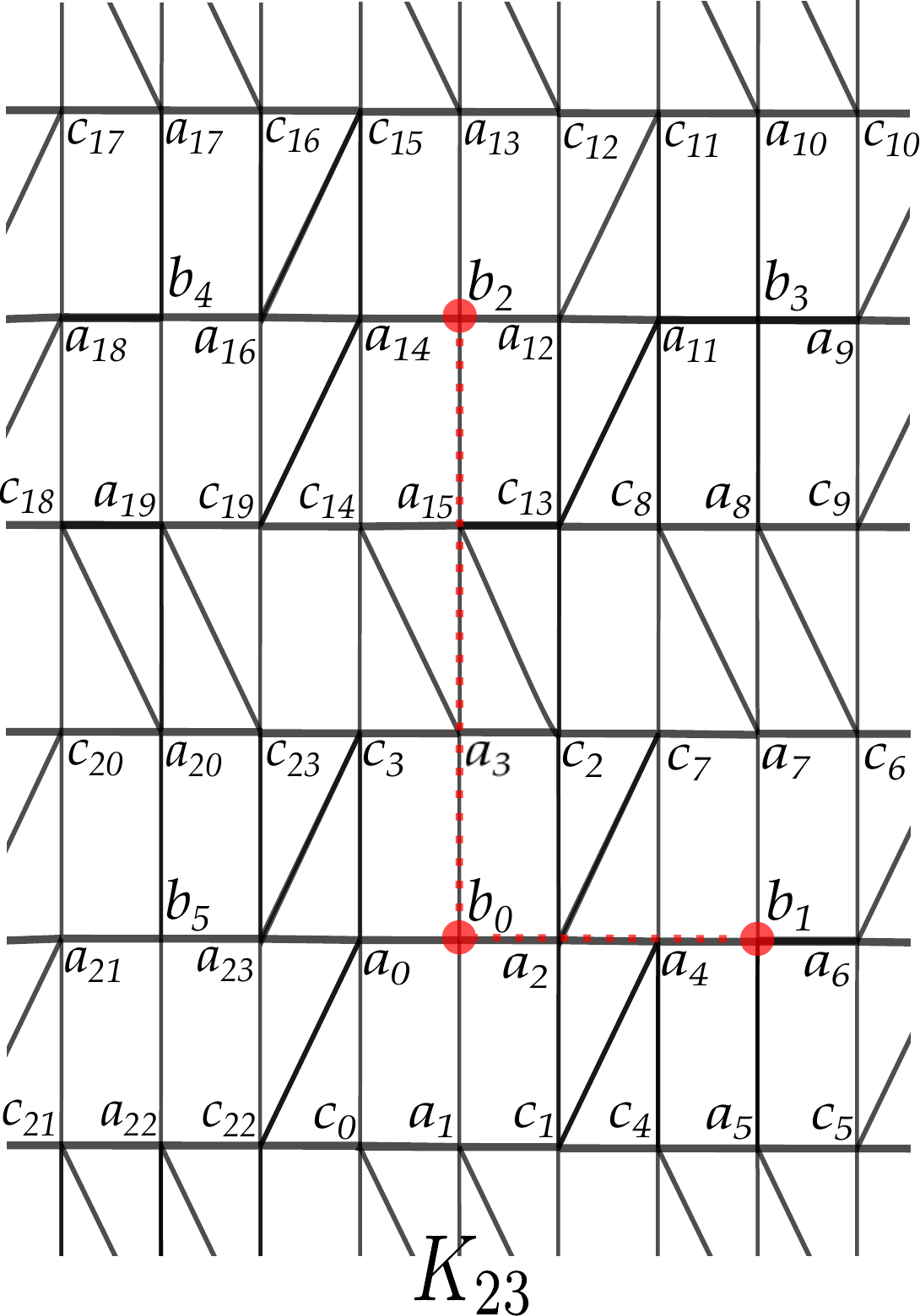}\hspace{5mm}
\includegraphics[height=6cm, width= 6cm]{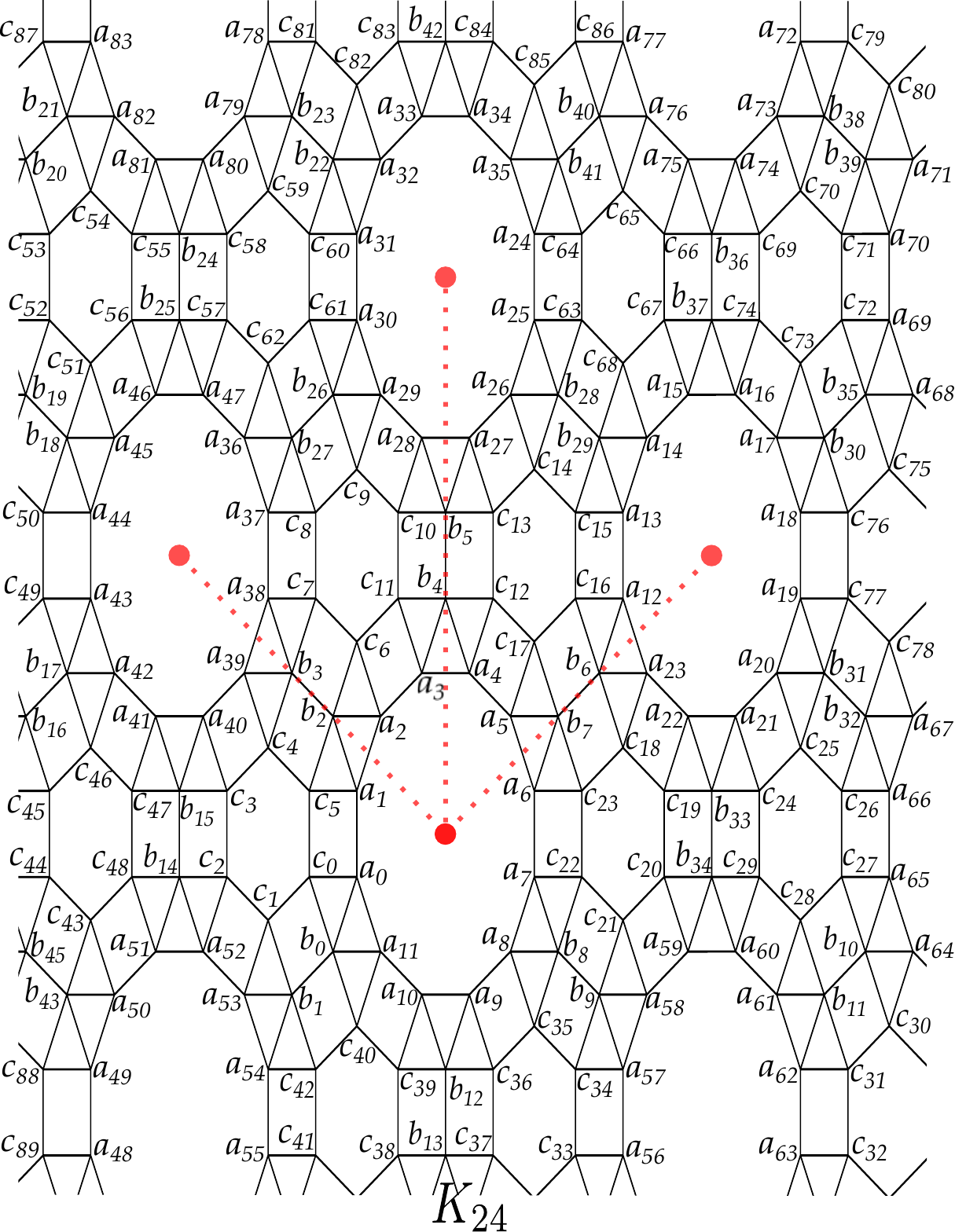}
\end{figure}
     \begin{figure}
     \centering
\includegraphics[height=6cm, width= 6cm]{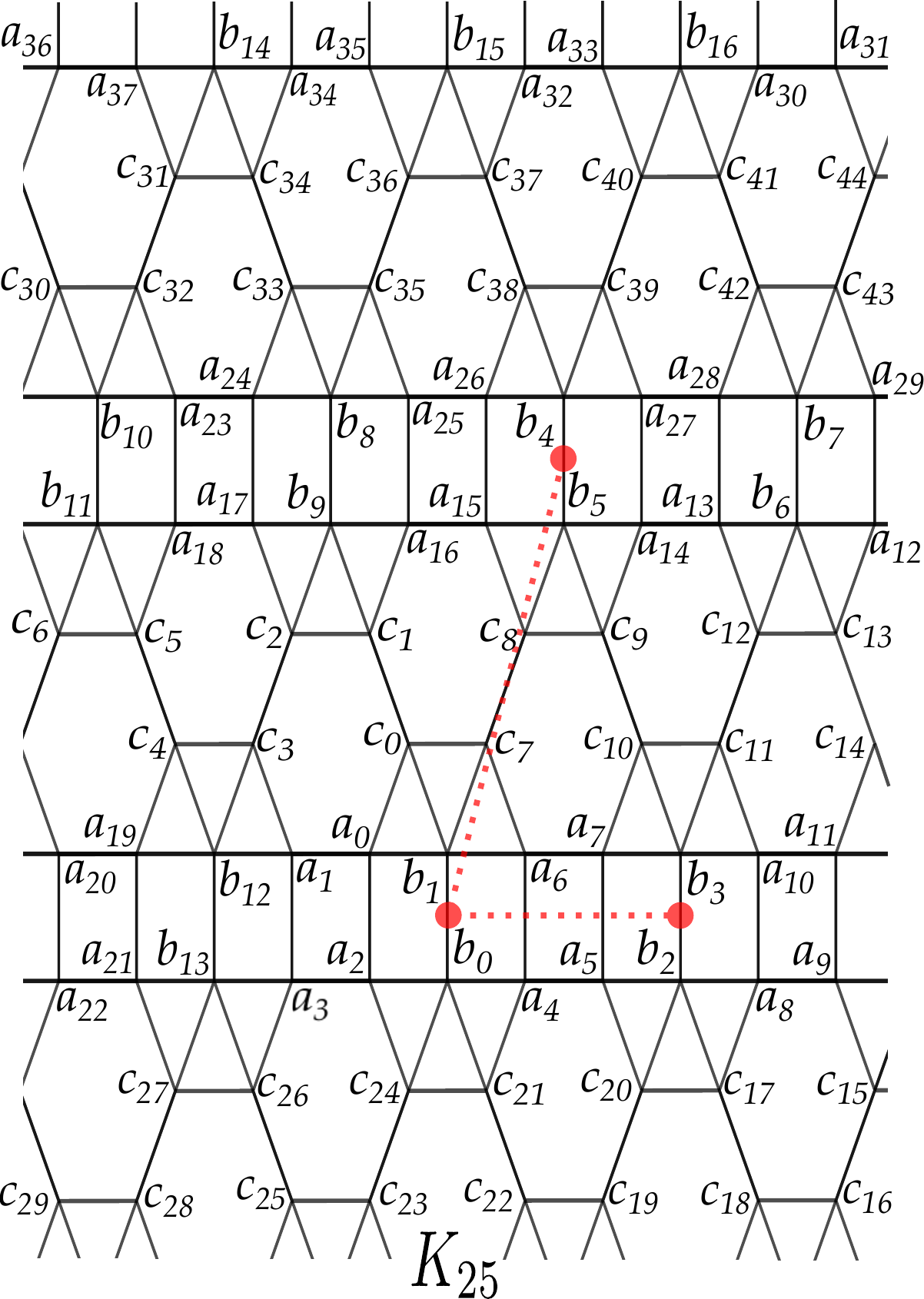}\hspace{5mm}
\includegraphics[height=6cm, width= 6cm]{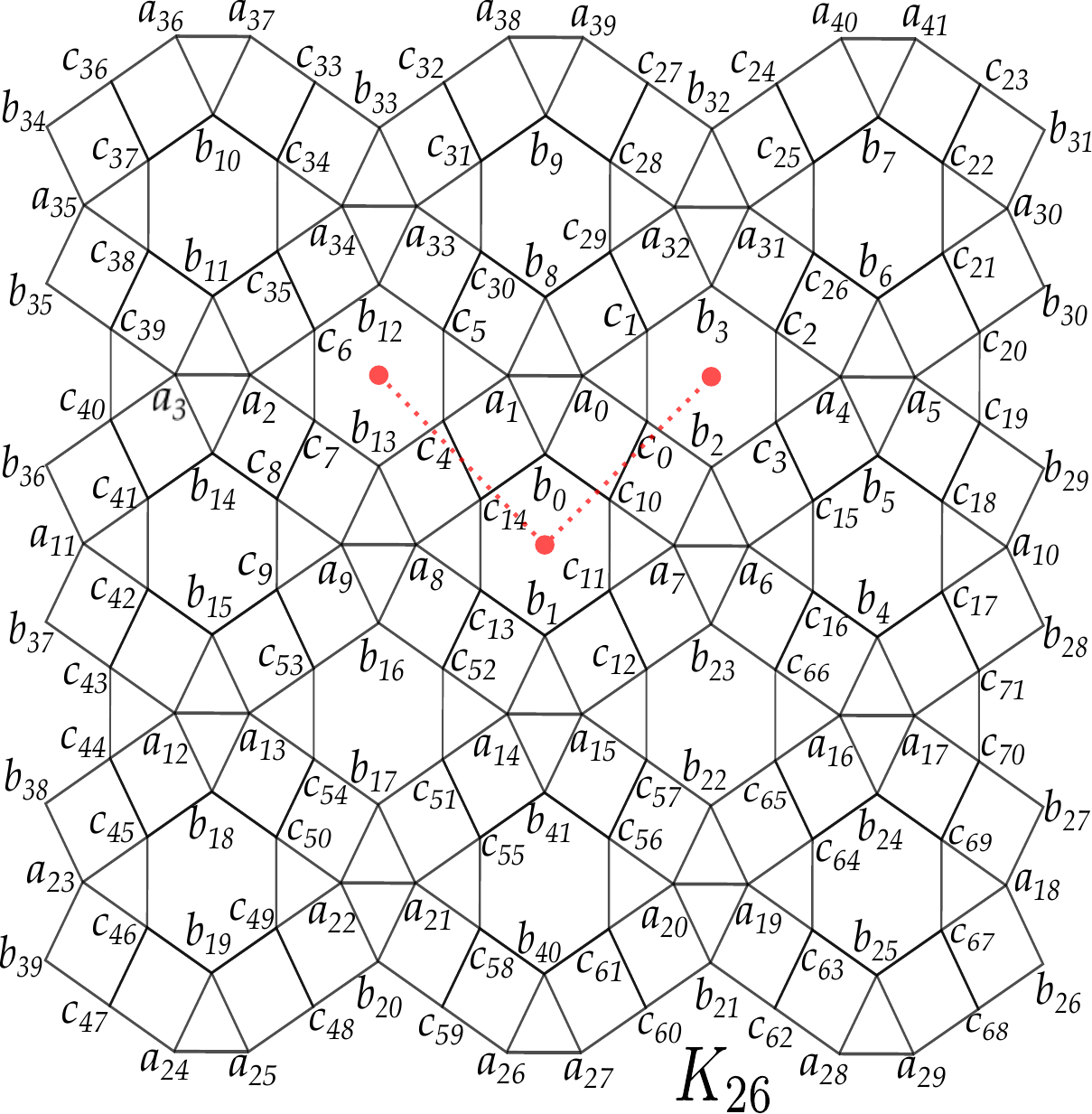}
\end{figure}
     \begin{figure}
     \centering
\includegraphics[height=6cm, width= 6cm]{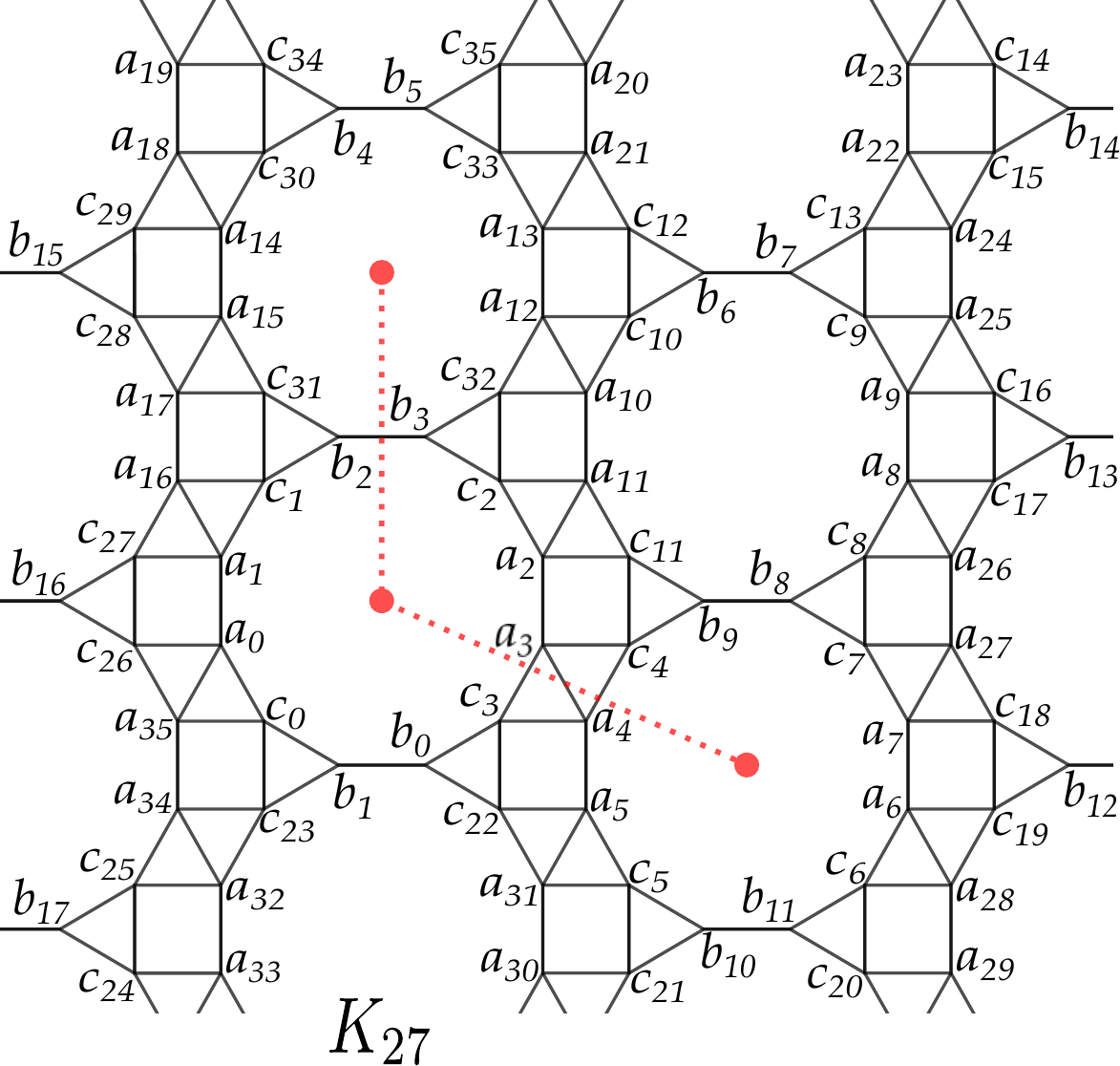}\hspace{5mm}
\includegraphics[height=6cm, width= 6cm]{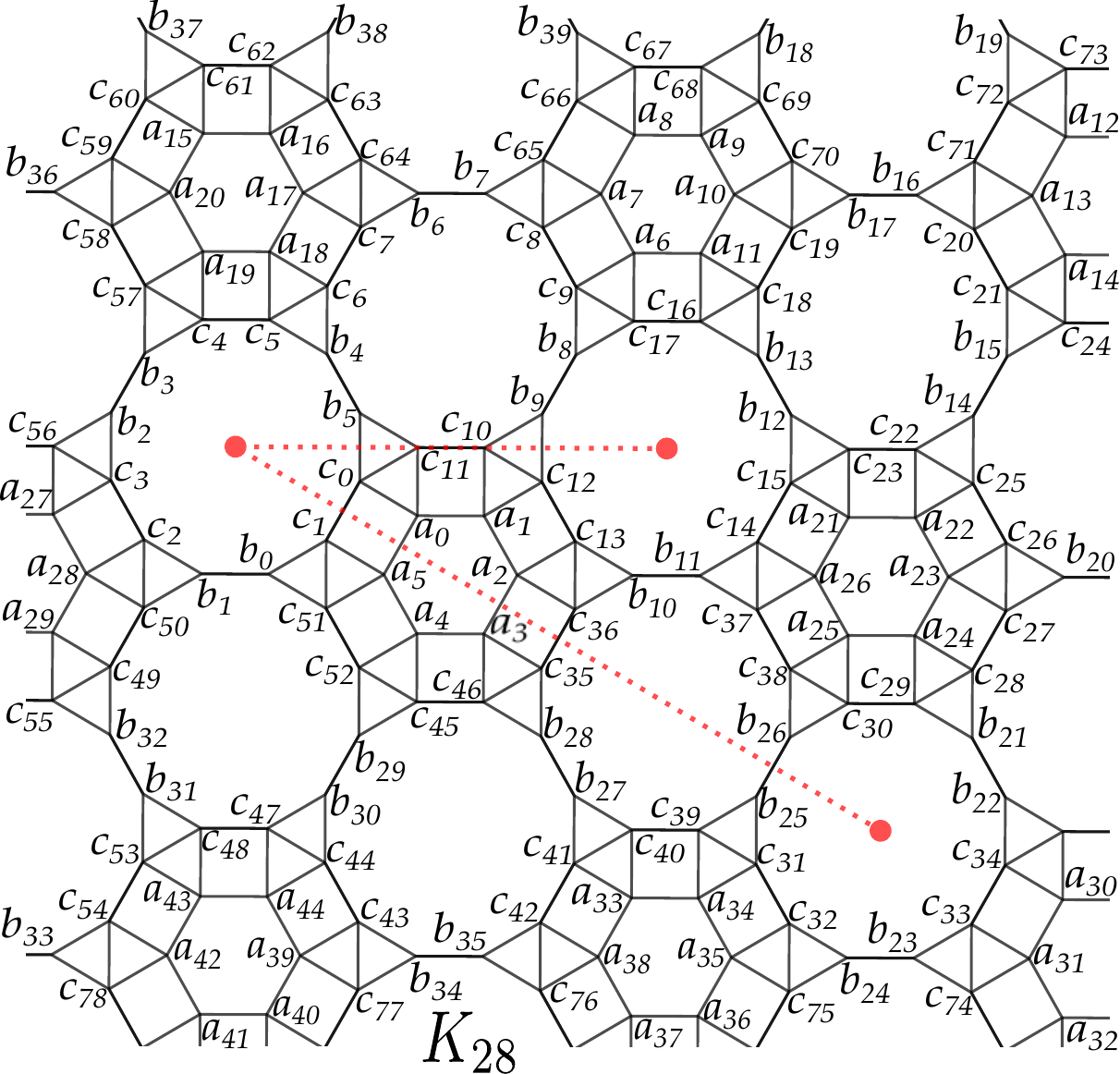}
\end{figure}
     \begin{figure}
     \centering
\includegraphics[height=6cm, width= 6cm]{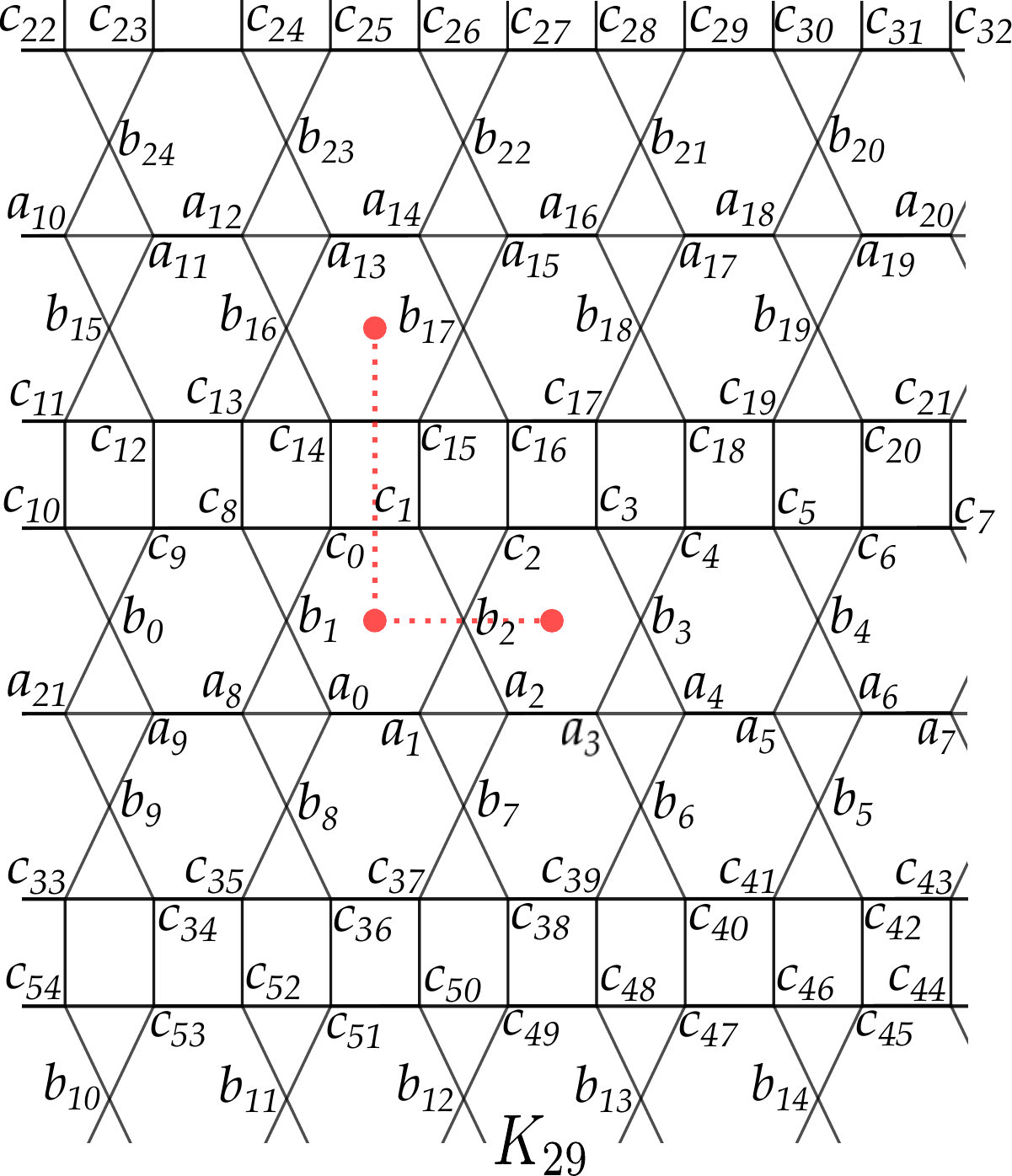}\hspace{5mm}
\includegraphics[height=6cm, width= 6cm]{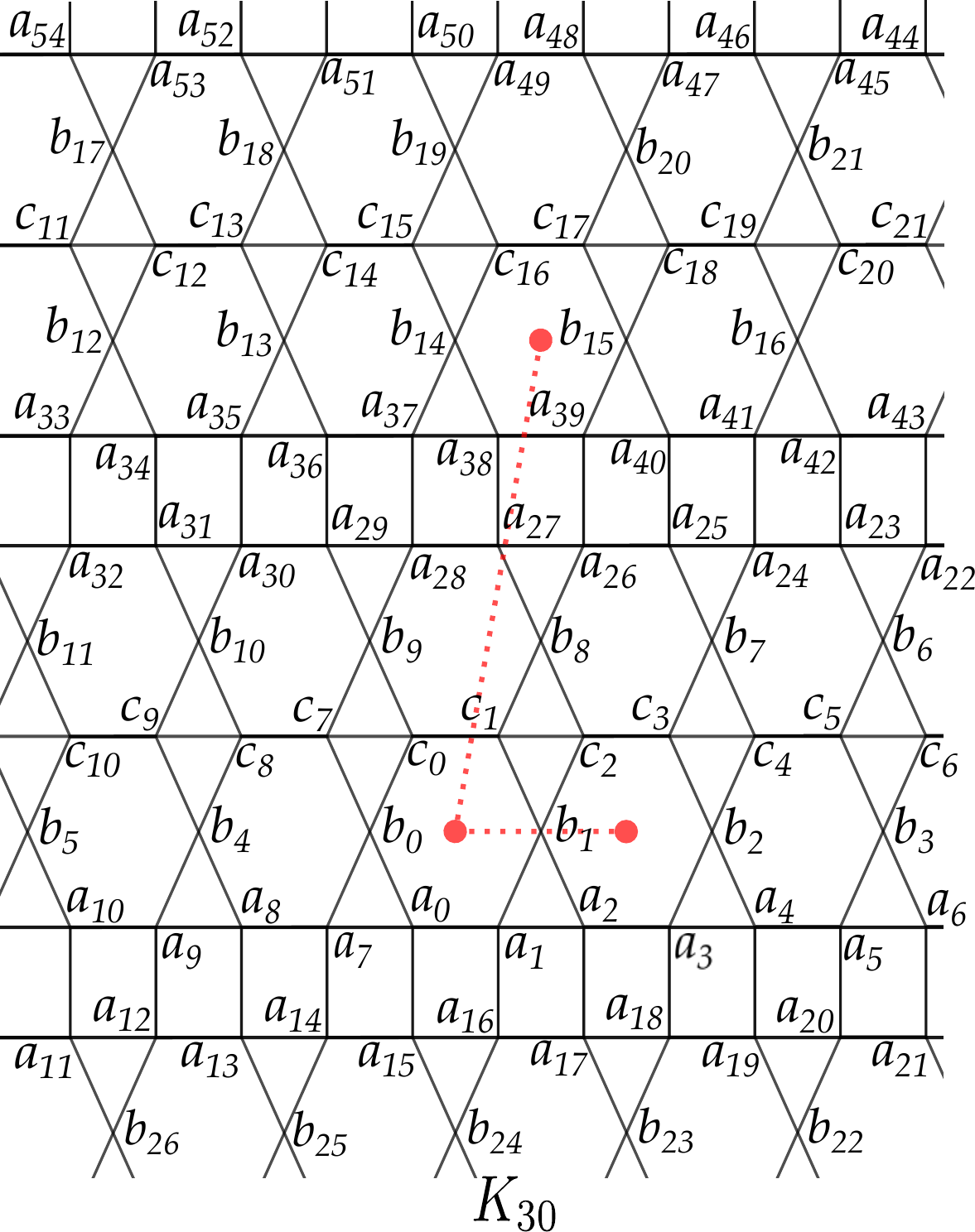}
\end{figure}
     \begin{figure}
     \centering
\includegraphics[height=6cm, width= 6cm]{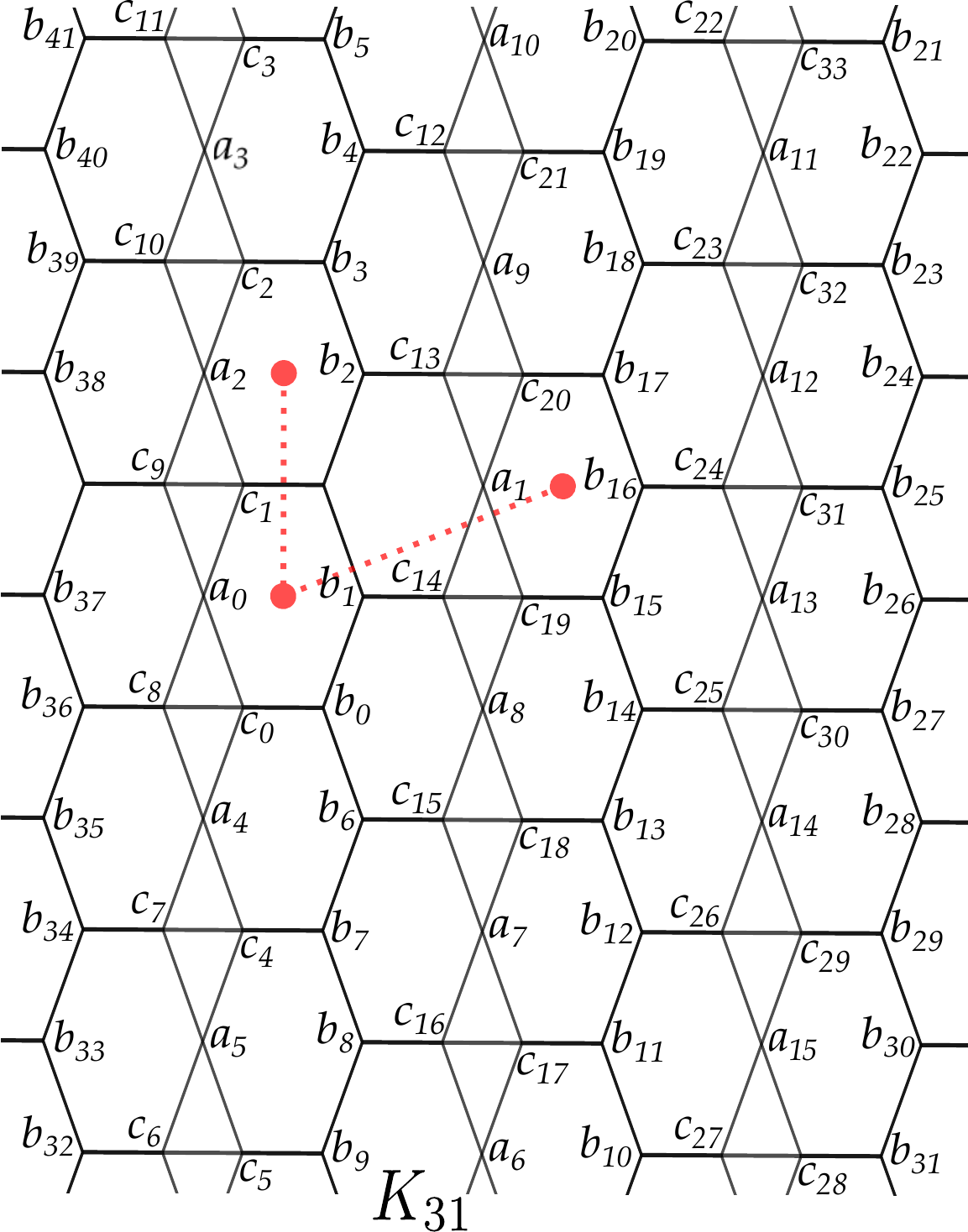}\hspace{5mm}
\includegraphics[height=6cm, width= 6cm]{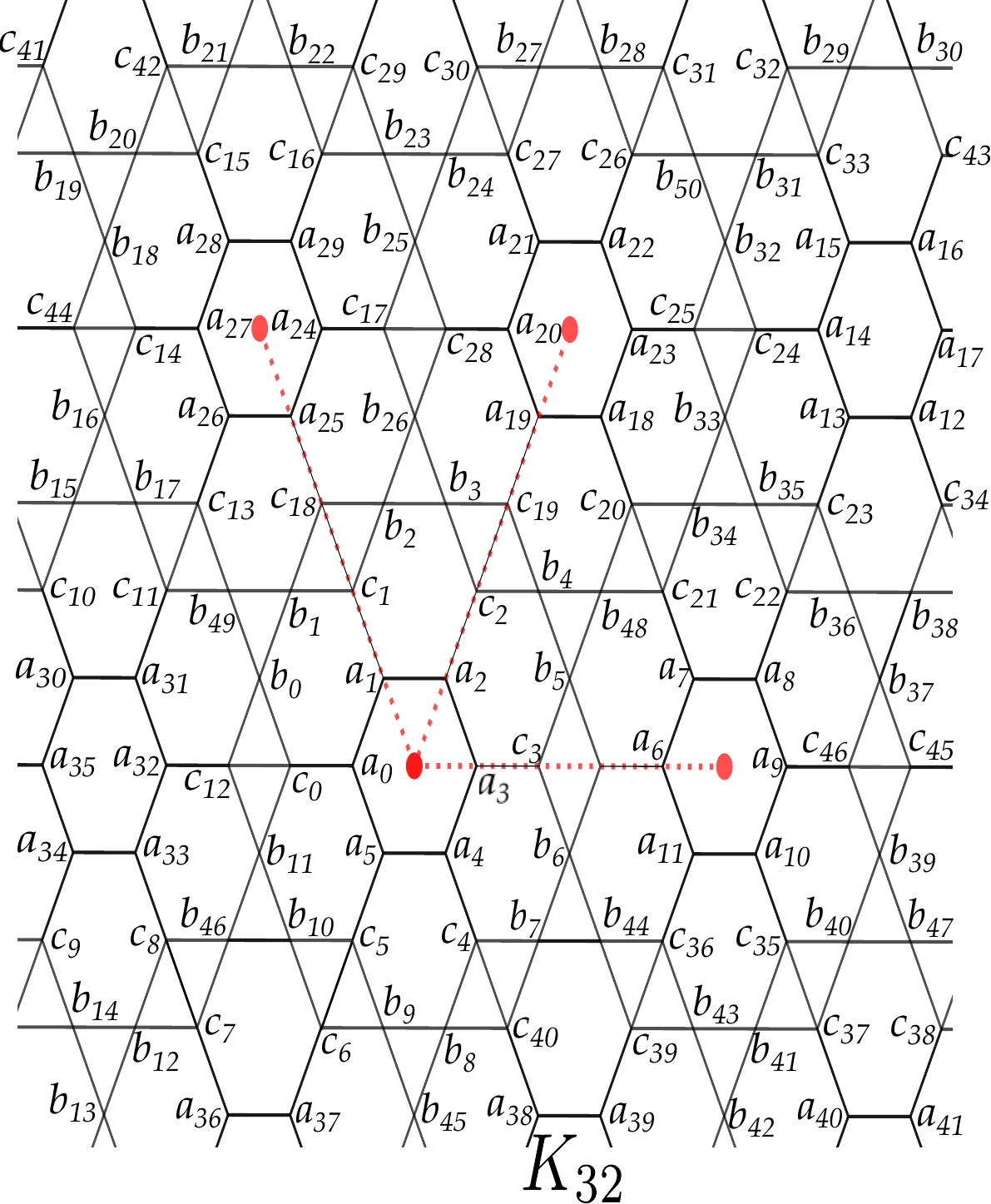}
\end{figure}
     \begin{figure}
     \centering
\includegraphics[height=6cm, width= 6cm]{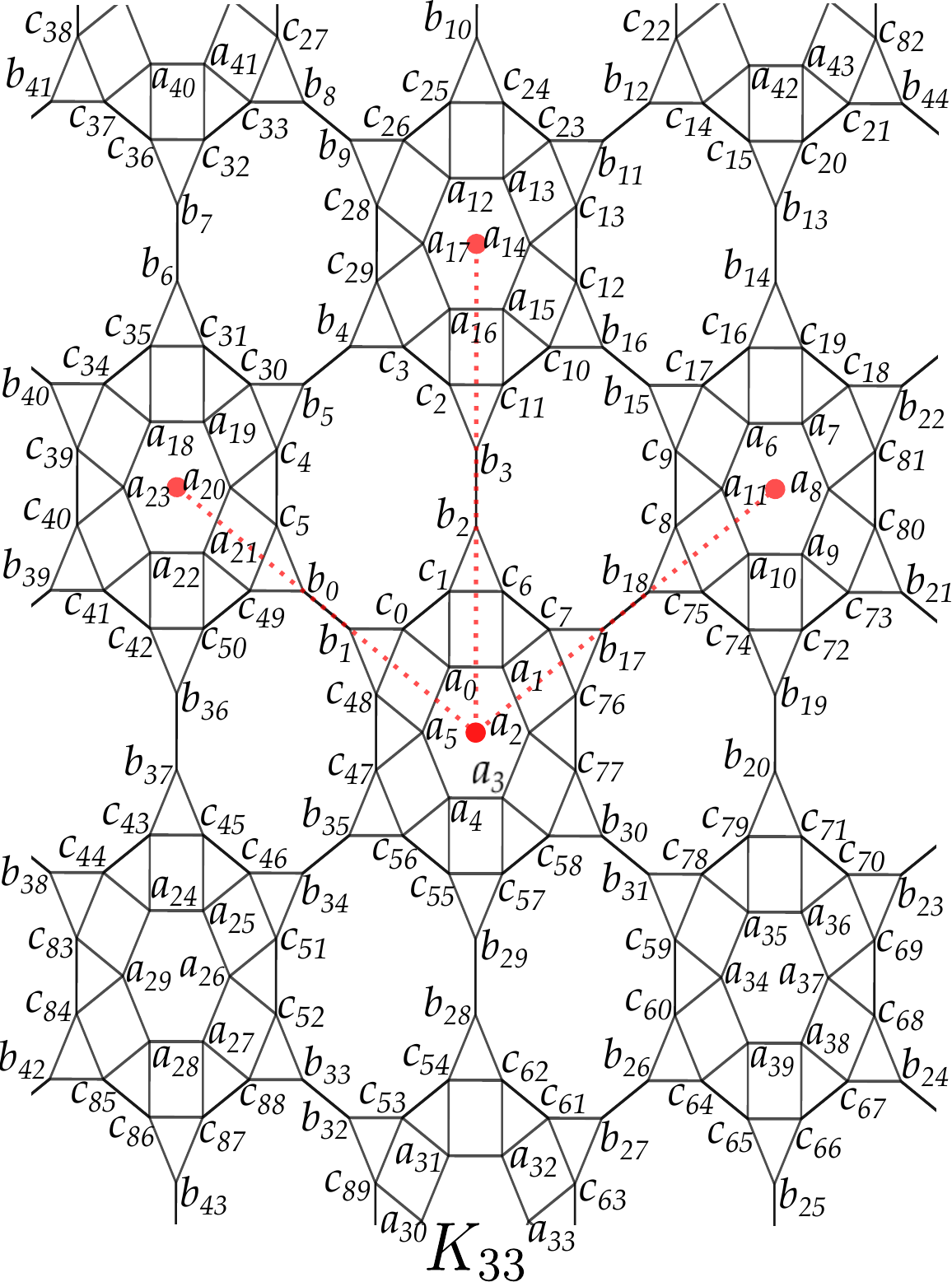}\hspace{5mm}
\includegraphics[height=6cm, width= 6cm]{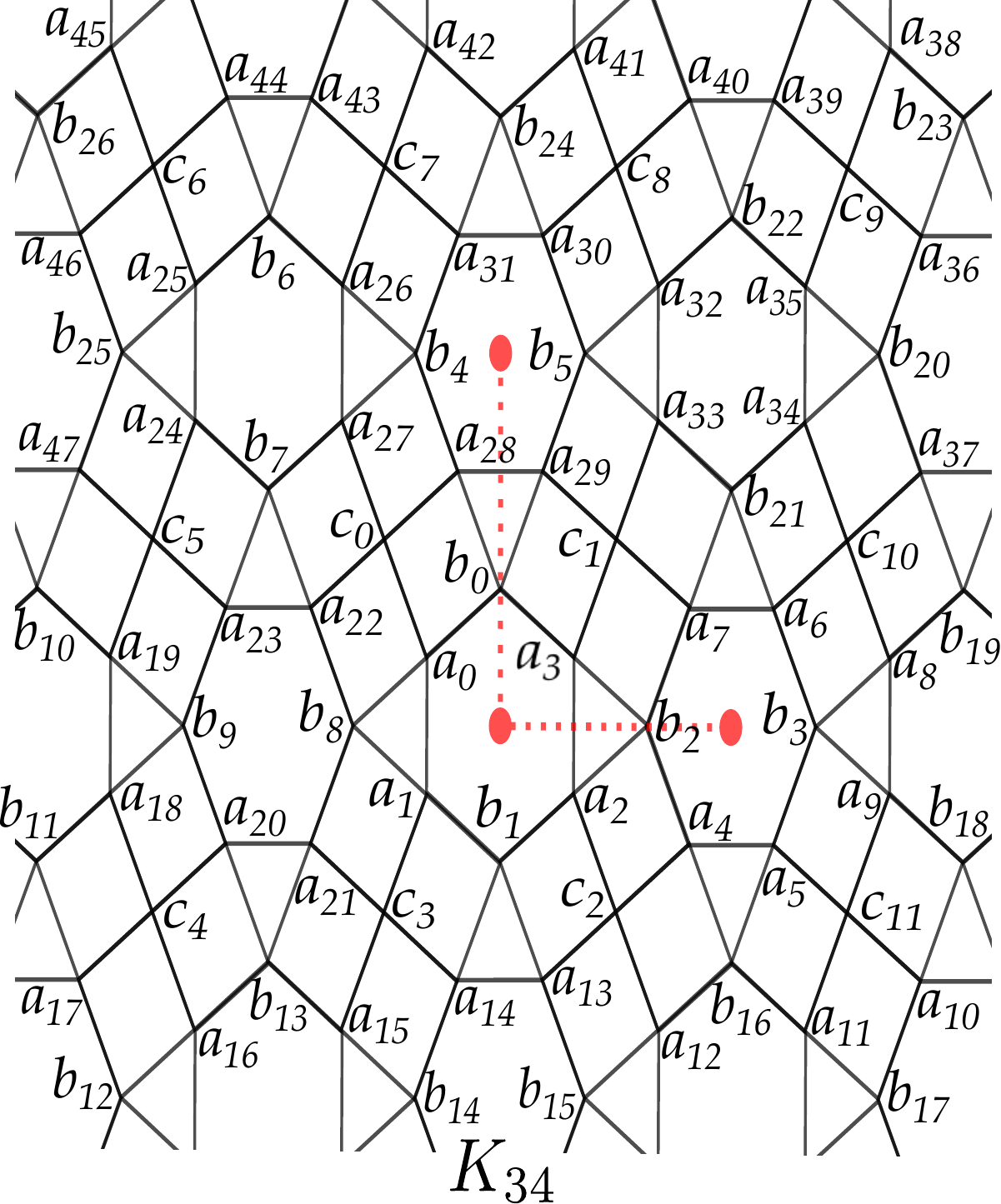}
\end{figure}
     \begin{figure}
     \centering
\includegraphics[height=6cm, width= 6cm]{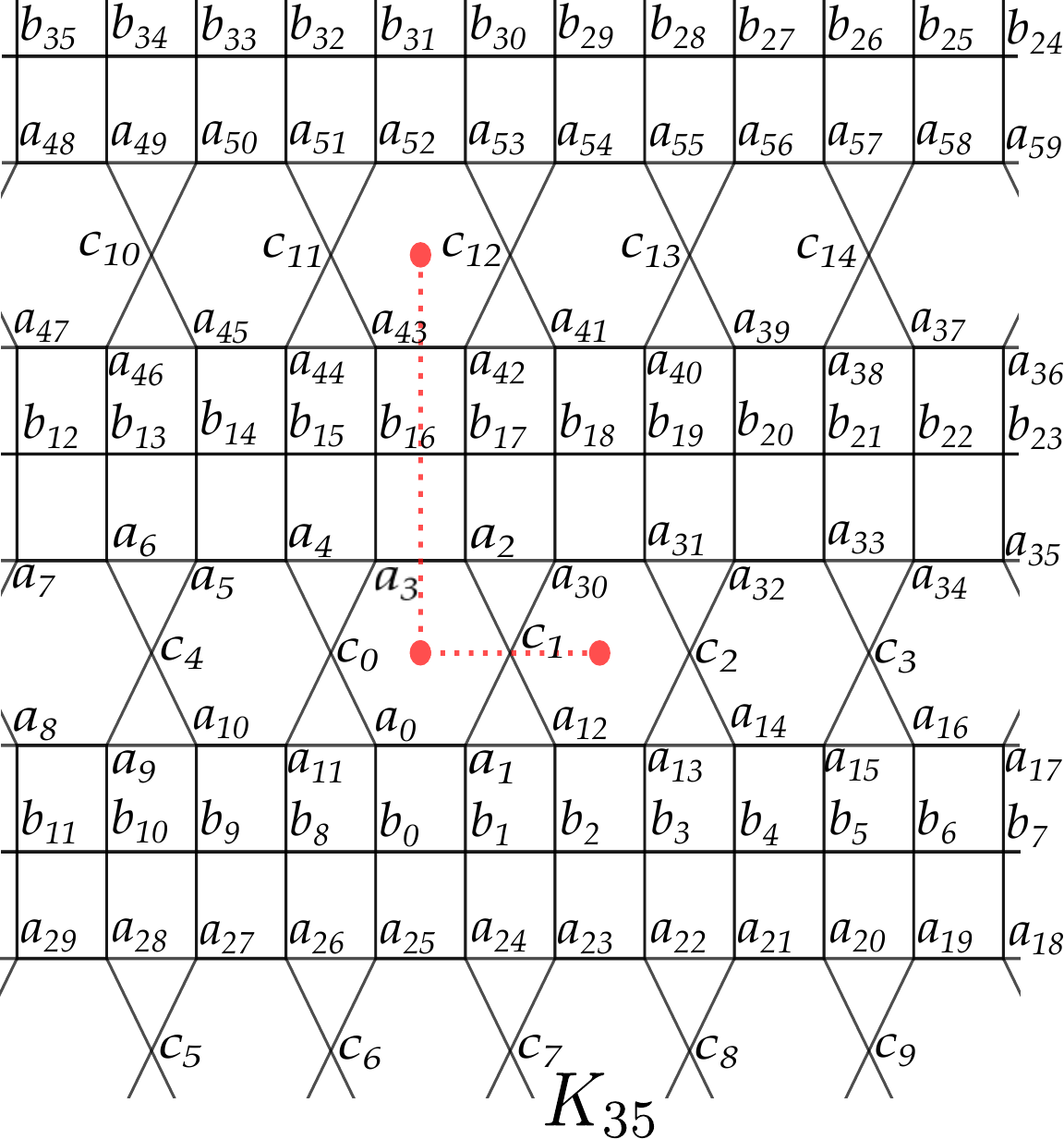}\hspace{5mm}
\includegraphics[height=6cm, width= 6cm]{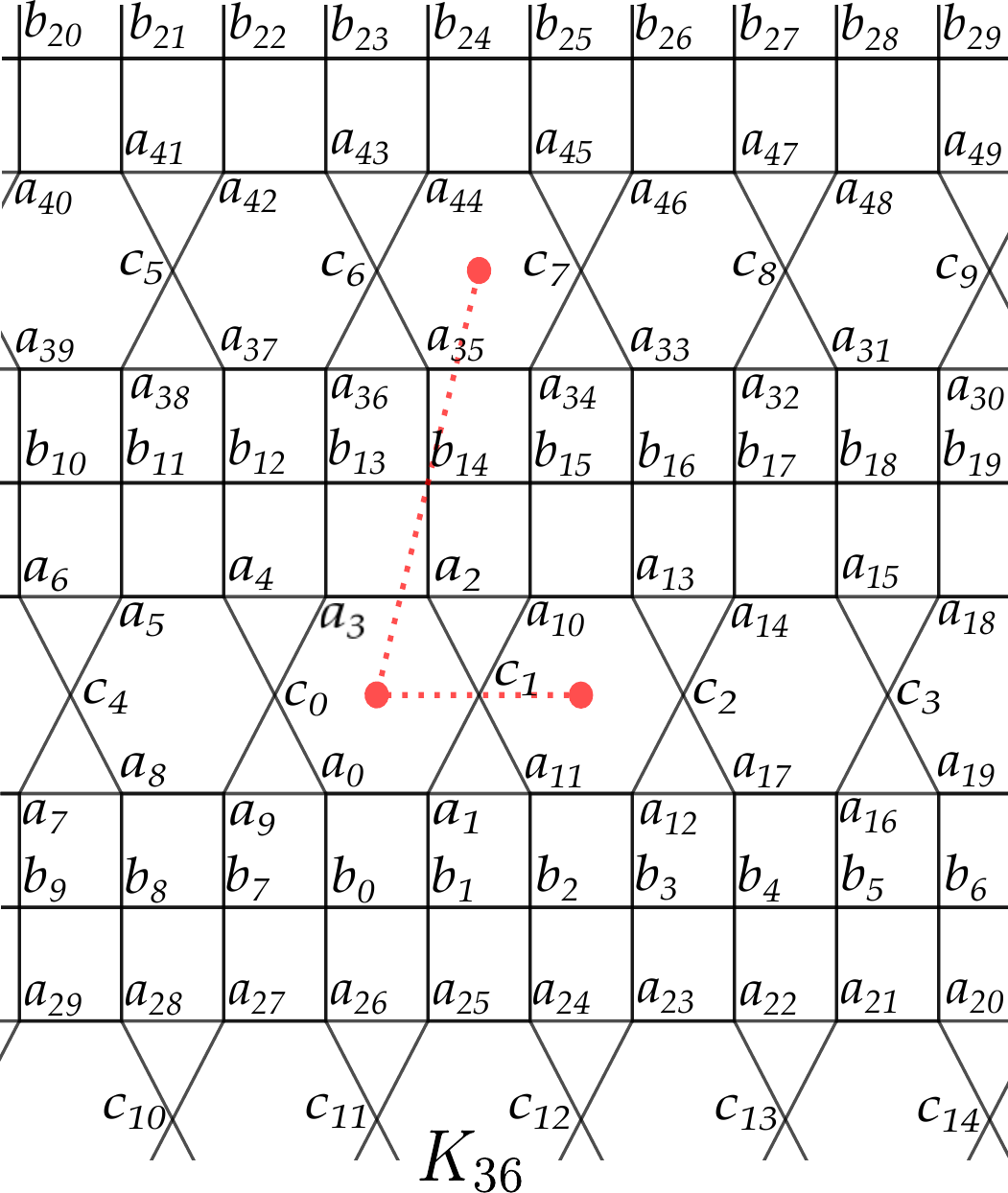}
\end{figure}
     \begin{figure}
     \centering
\includegraphics[height=6cm, width= 6cm]{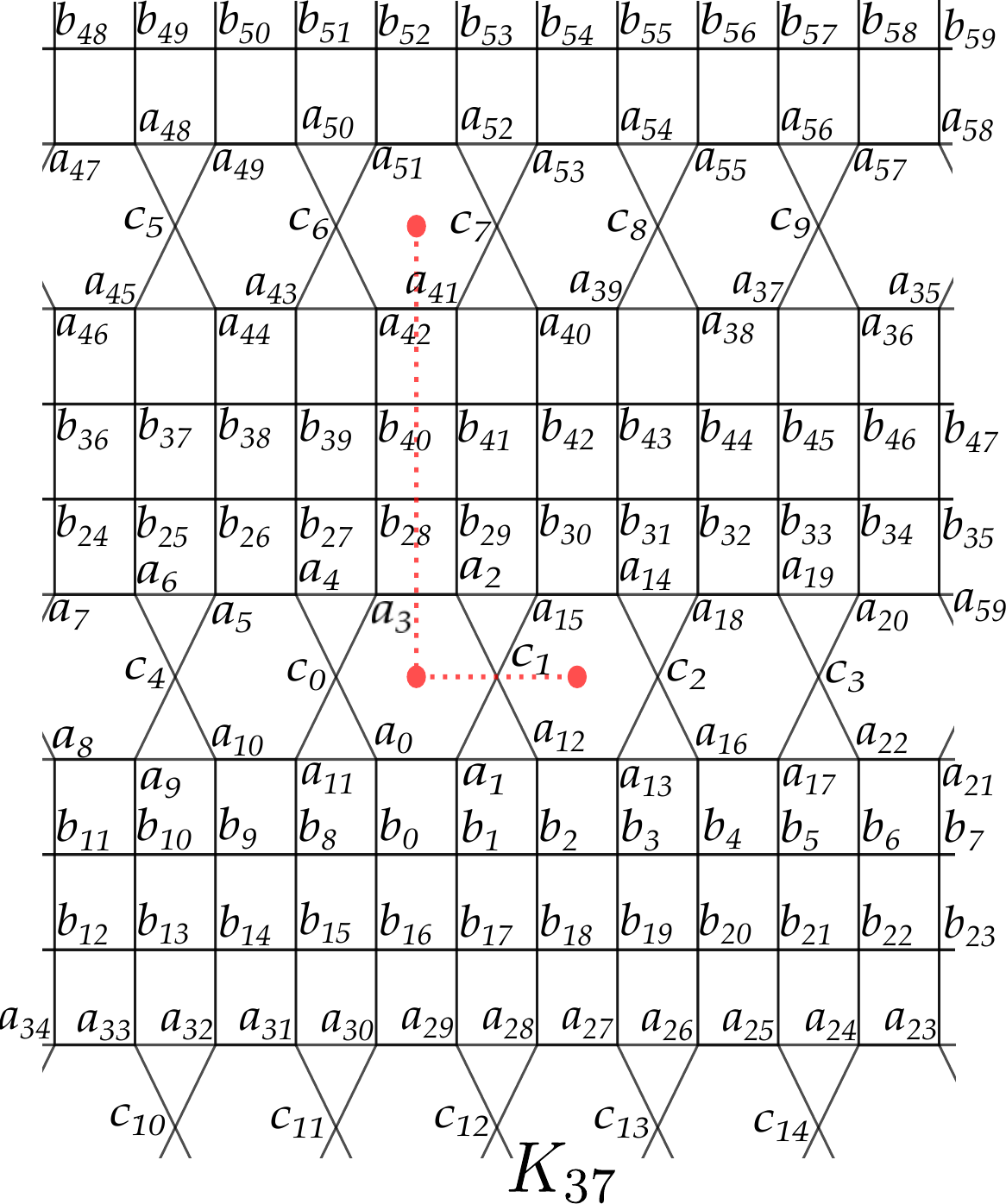}\hspace{5mm}
\includegraphics[height=6cm, width= 6cm]{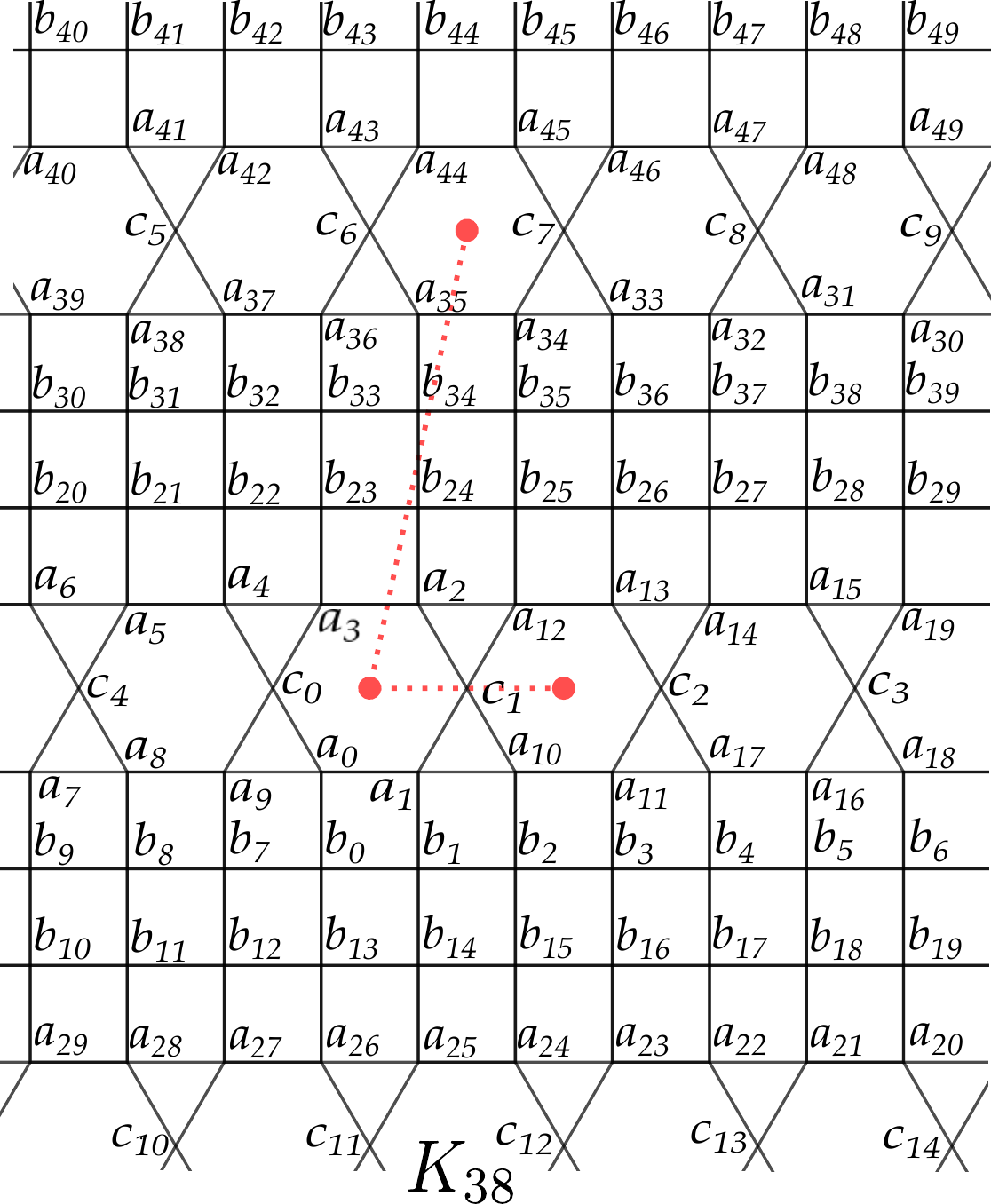}
\end{figure}     \begin{figure}
     \centering
\includegraphics[height=6cm, width= 6cm]{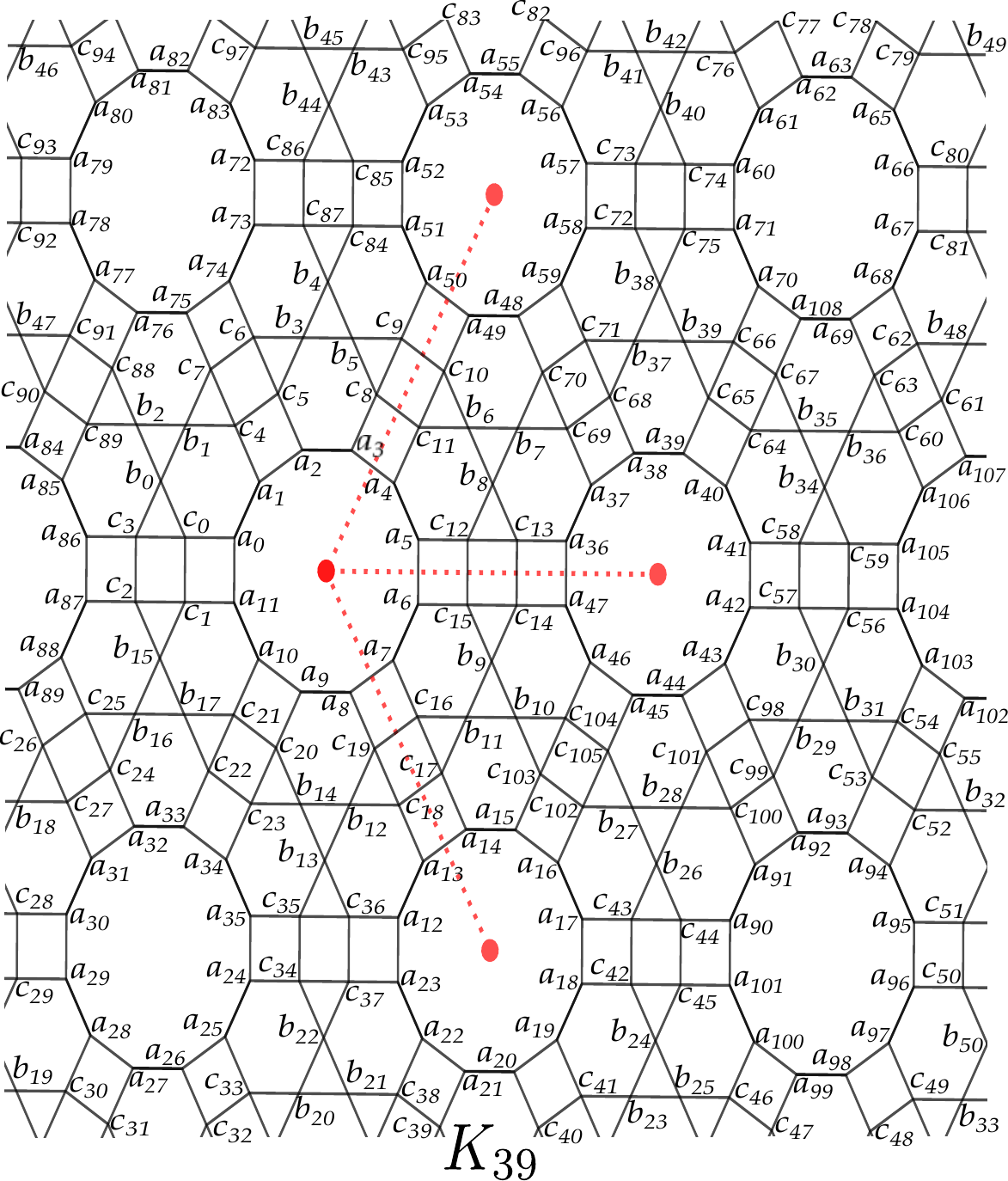}\hspace{5mm}
\includegraphics[height=6cm, width= 6cm]{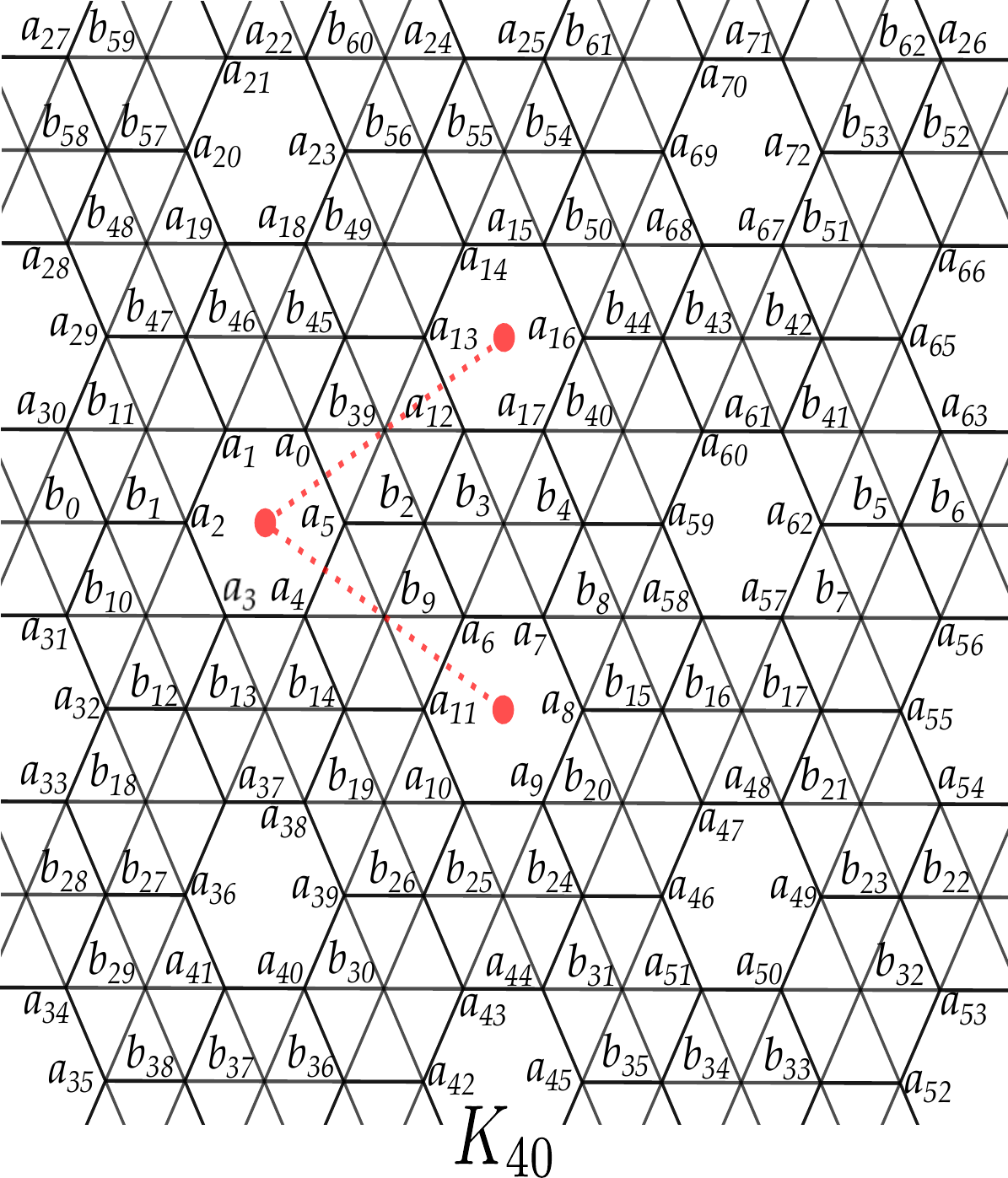}
\end{figure}     \begin{figure}
     \centering
\includegraphics[height=6cm, width= 6cm]{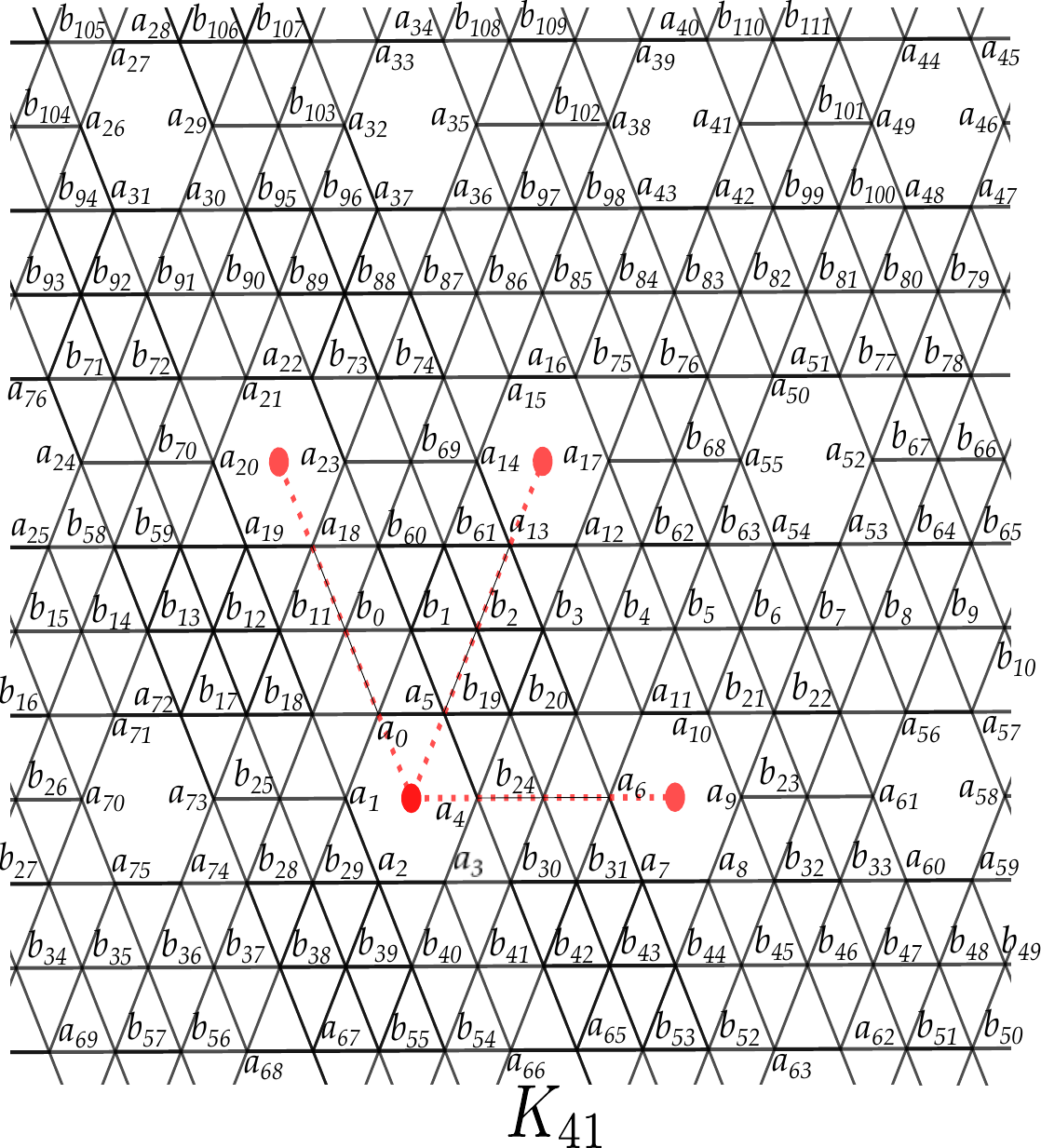}\hspace{5mm}
\includegraphics[height=6cm, width= 6cm]{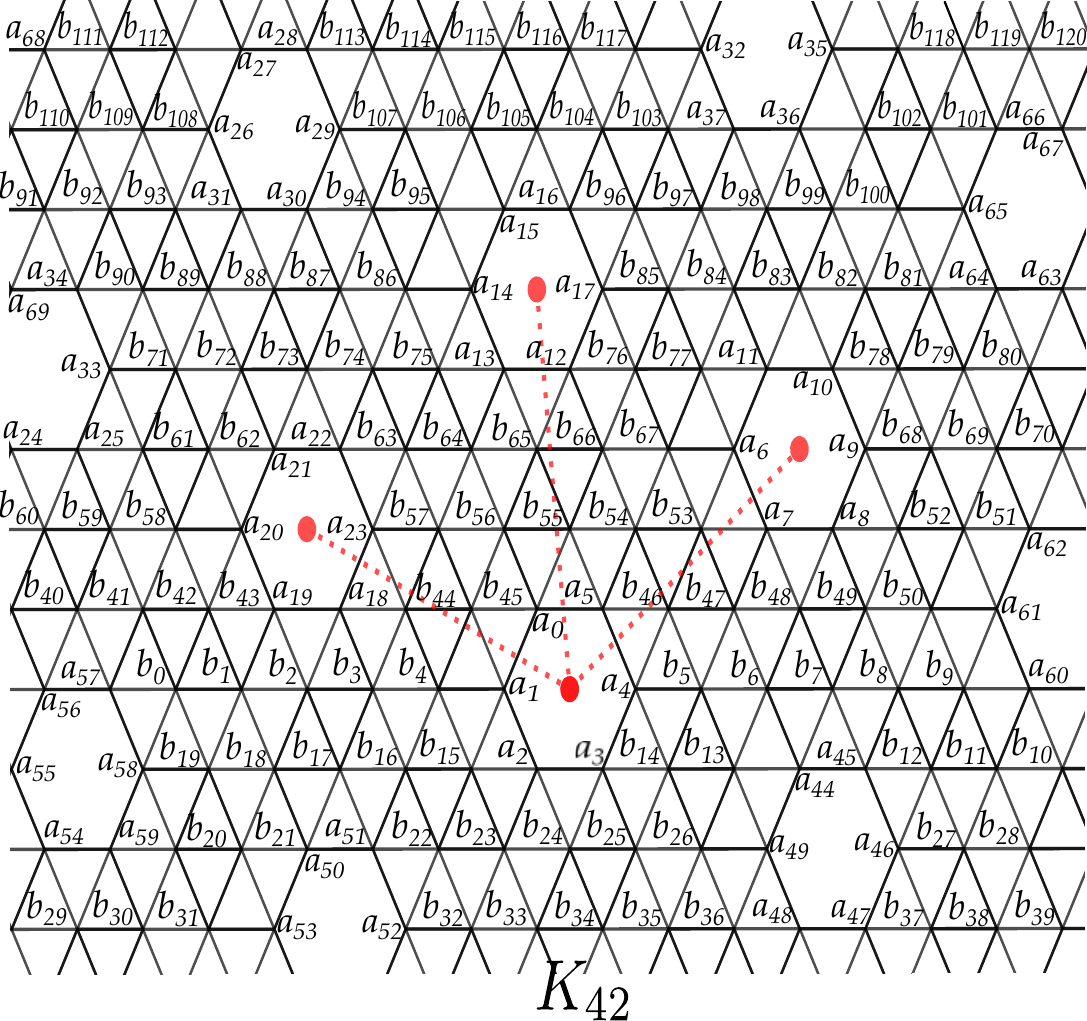}
\end{figure}     \begin{figure}
     \centering
\includegraphics[height=6cm, width= 6cm]{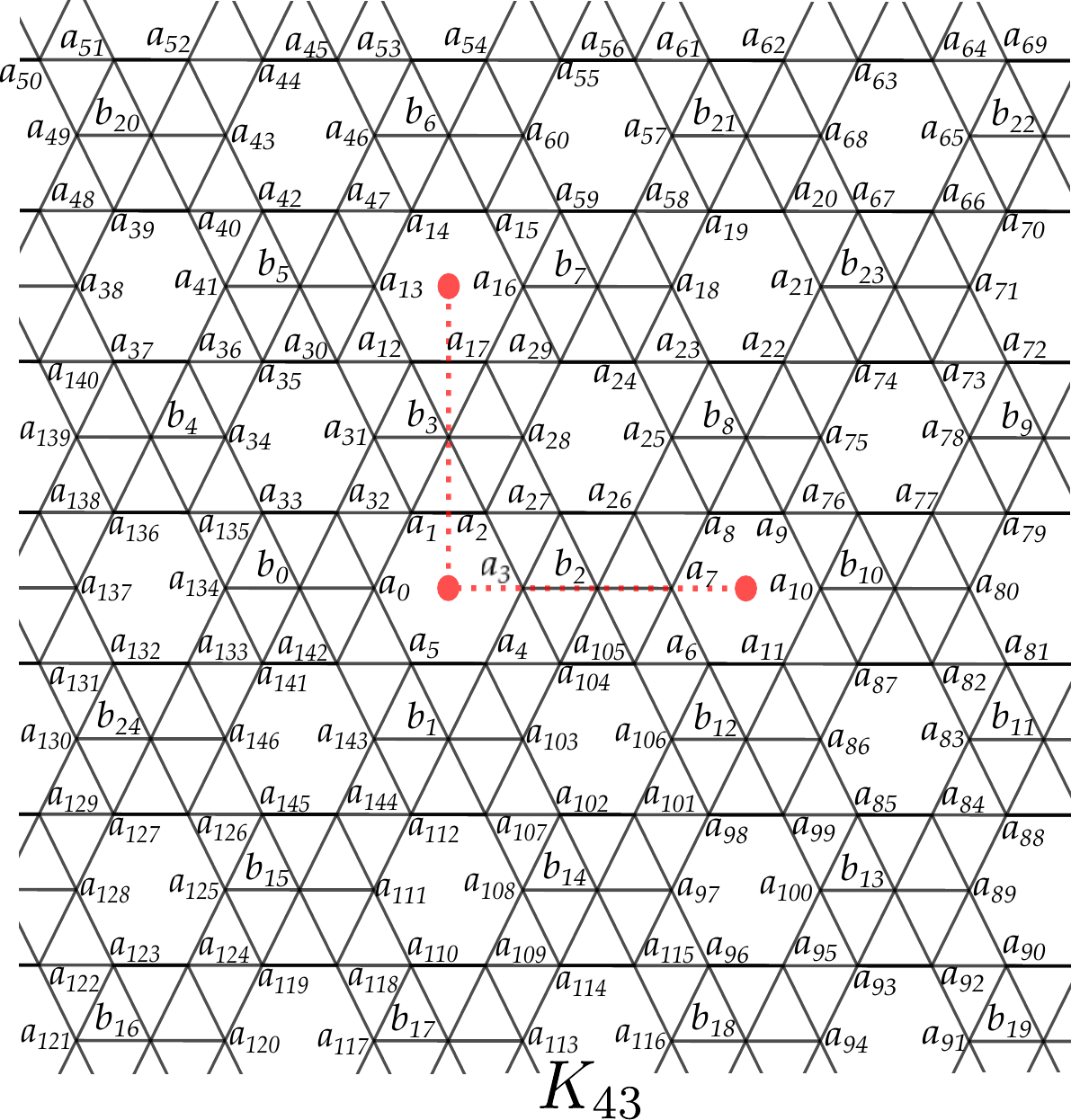}\hspace{5mm}
\includegraphics[height=6cm, width= 6cm]{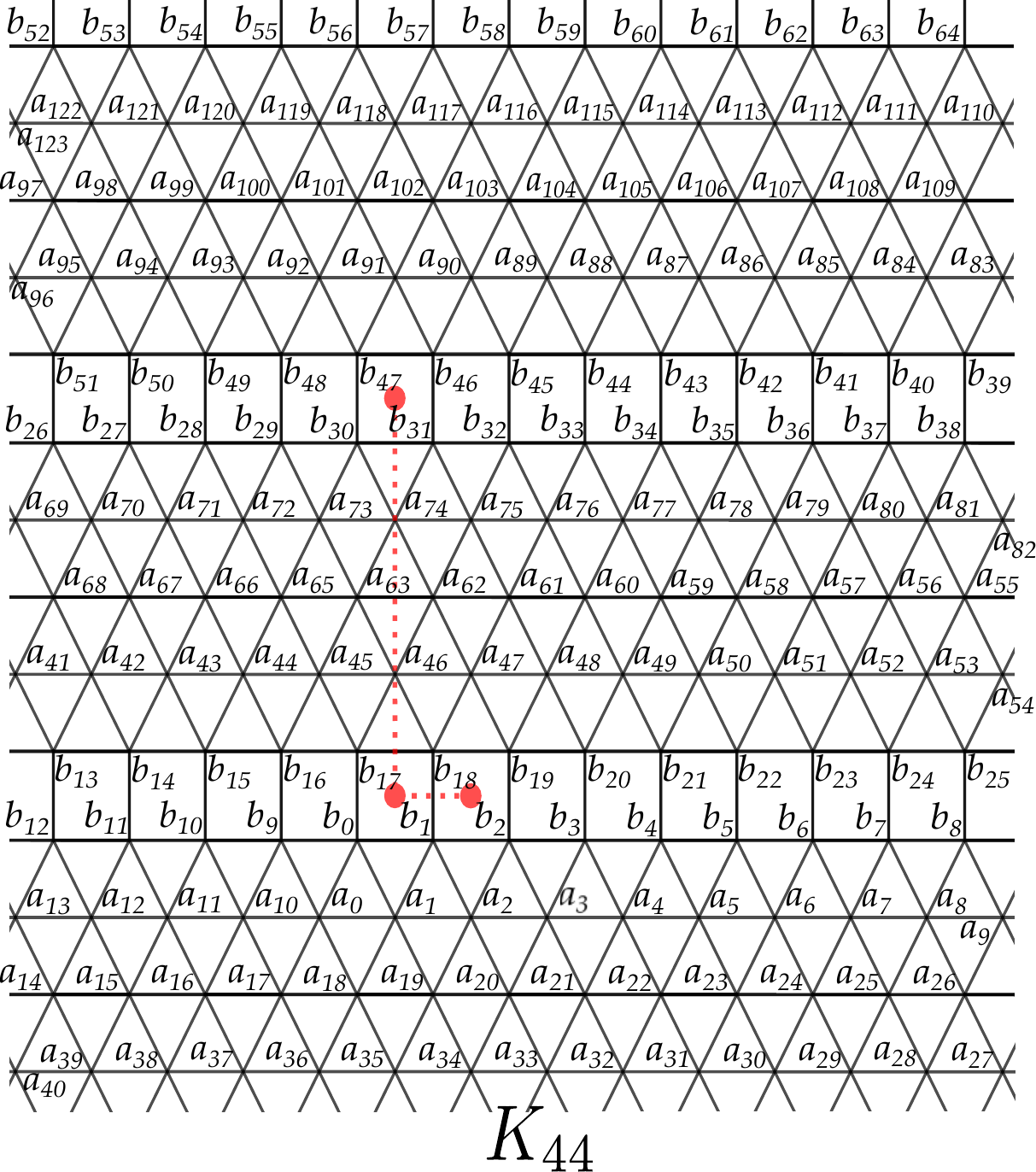}
\end{figure}     \begin{figure}
     \centering
\includegraphics[height=6cm, width= 6cm]{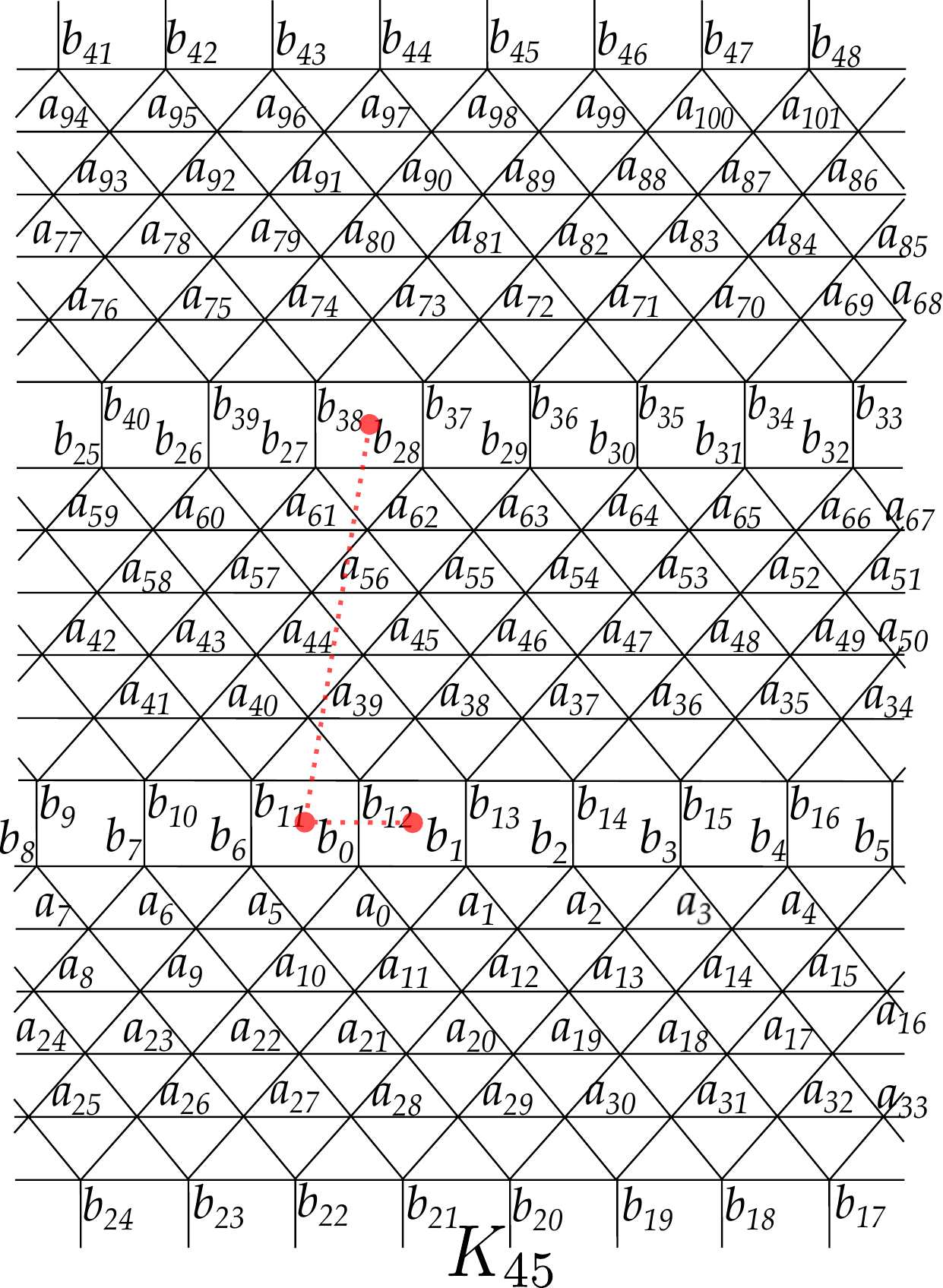}\hspace{5mm}
\includegraphics[height=6cm, width= 6cm]{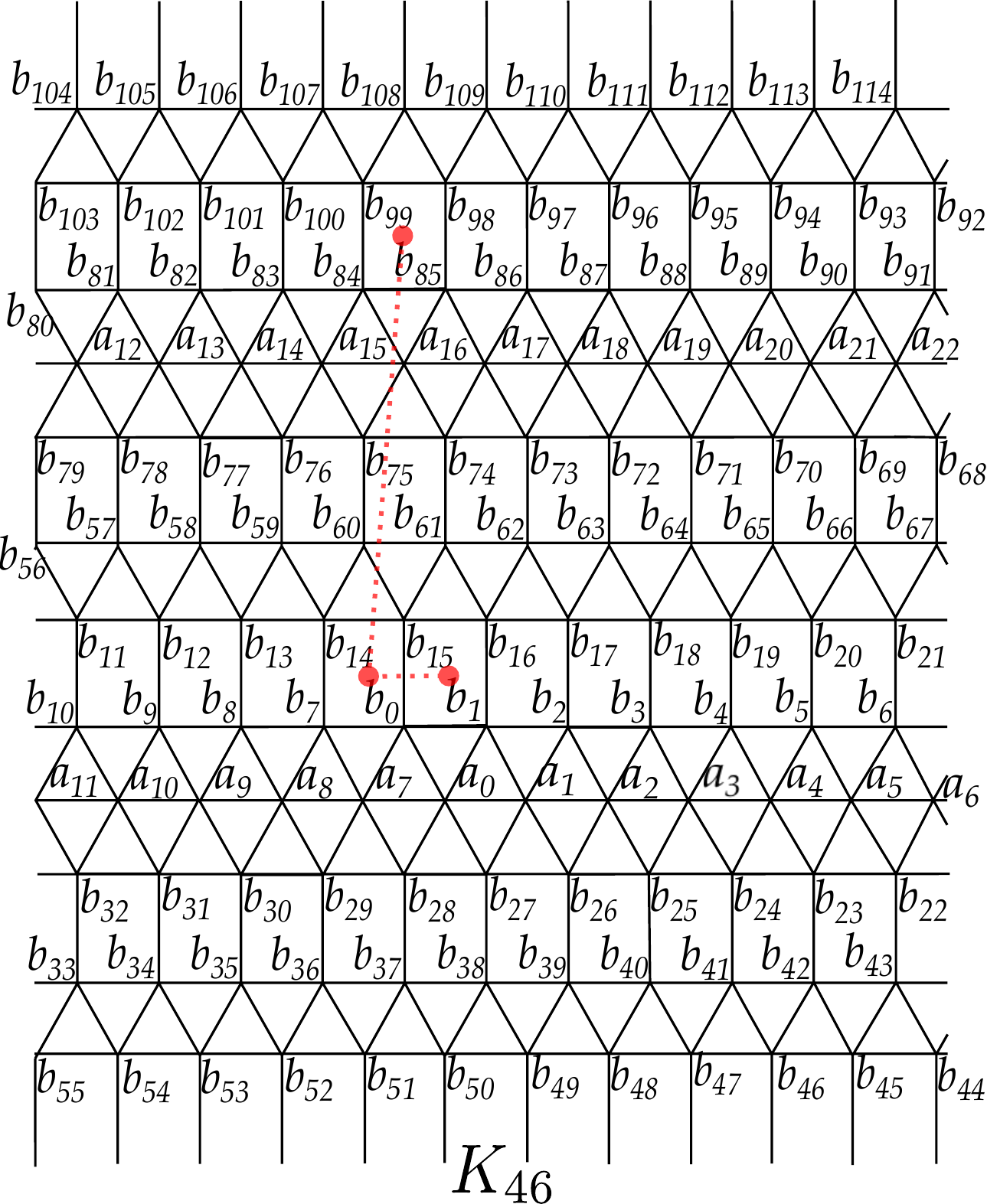}
\end{figure}
     \begin{figure}
     \centering
\includegraphics[height=6cm, width= 6cm]{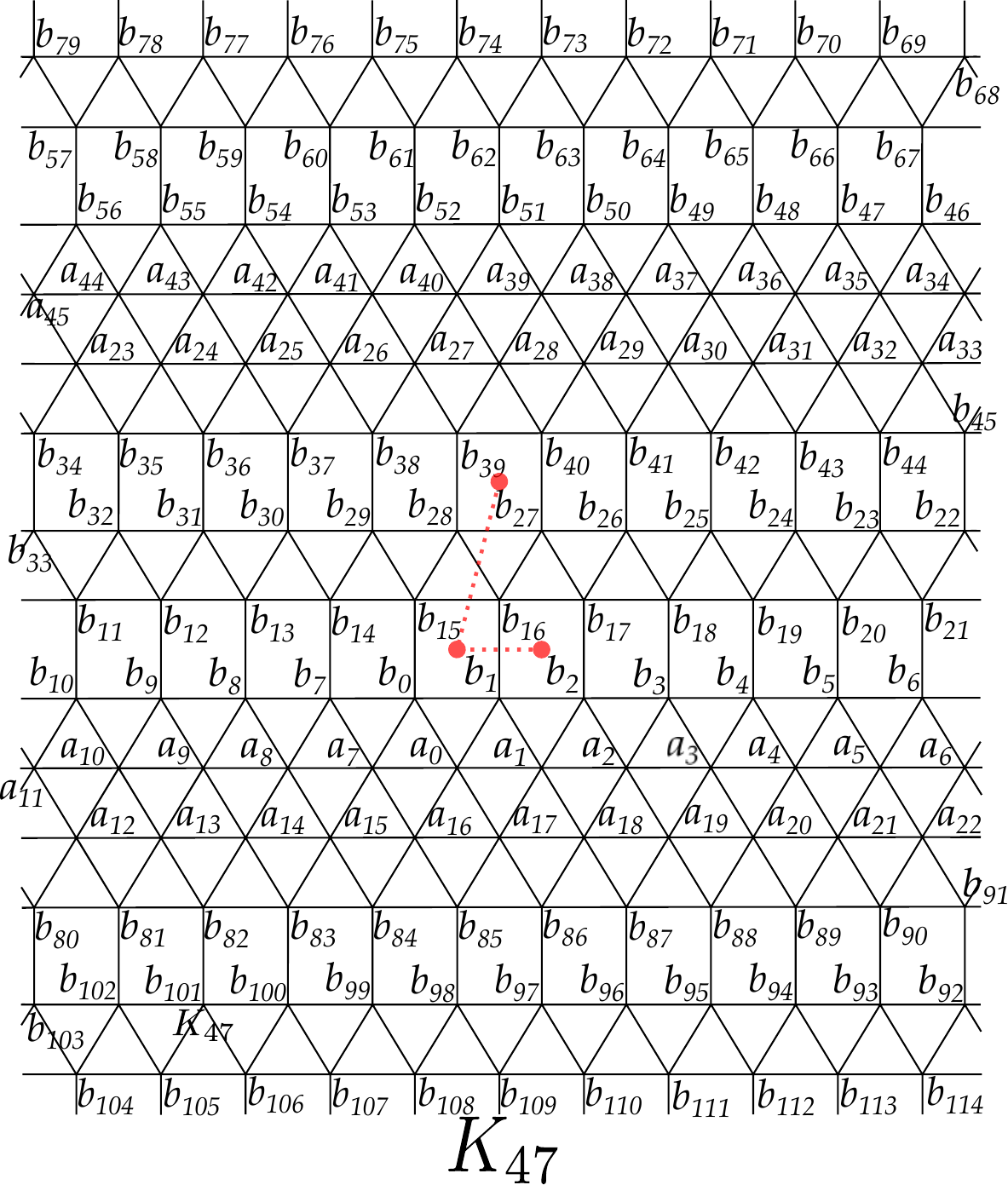}\hspace{5mm}
\includegraphics[height=6cm, width= 6cm]{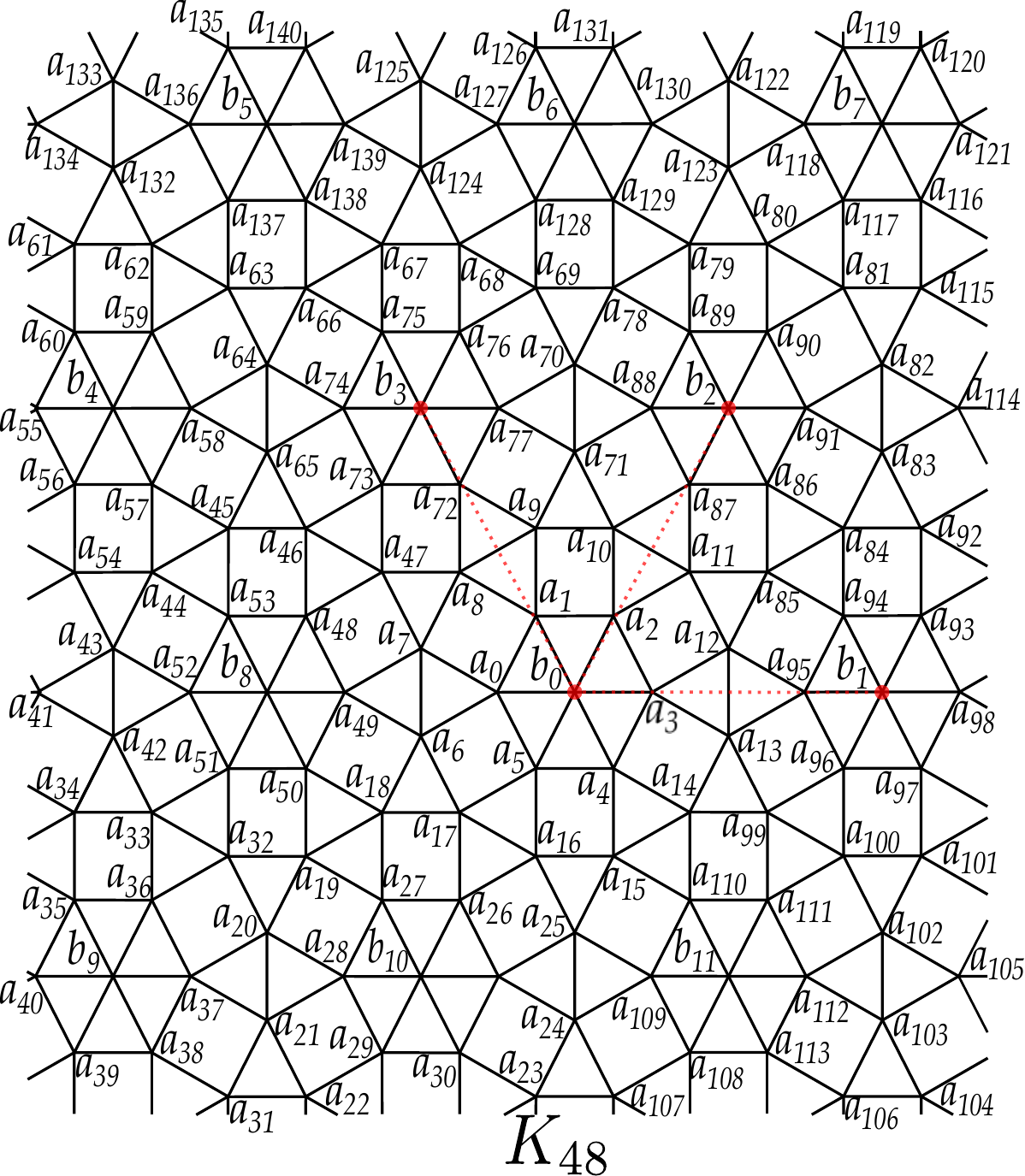}
\end{figure}     \begin{figure}
     \centering
\includegraphics[height=6cm, width= 6cm]{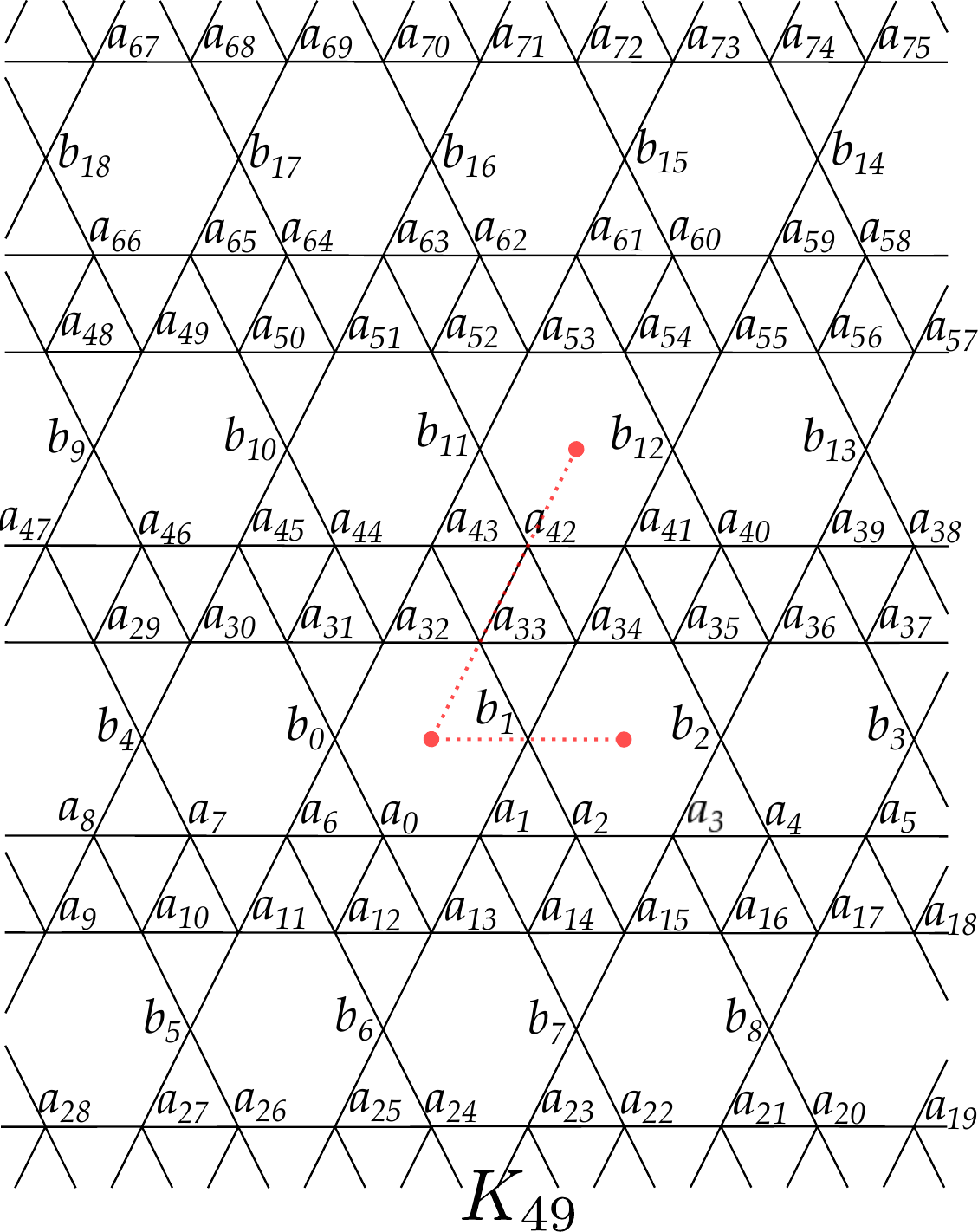}\hspace{5mm}
\includegraphics[height=6cm, width= 6cm]{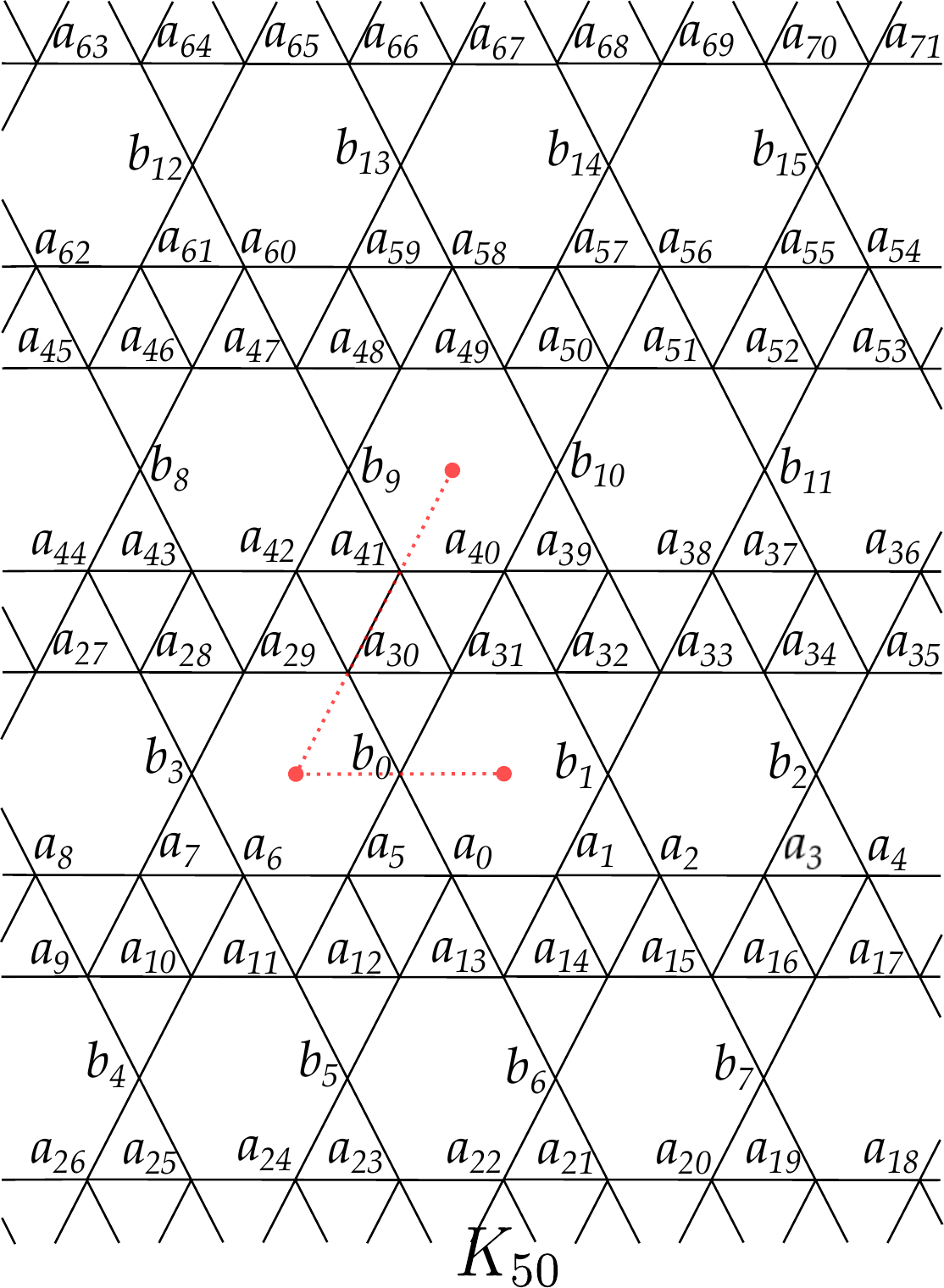}
\end{figure}     \begin{figure}
     \centering
\includegraphics[height=6cm, width= 6cm]{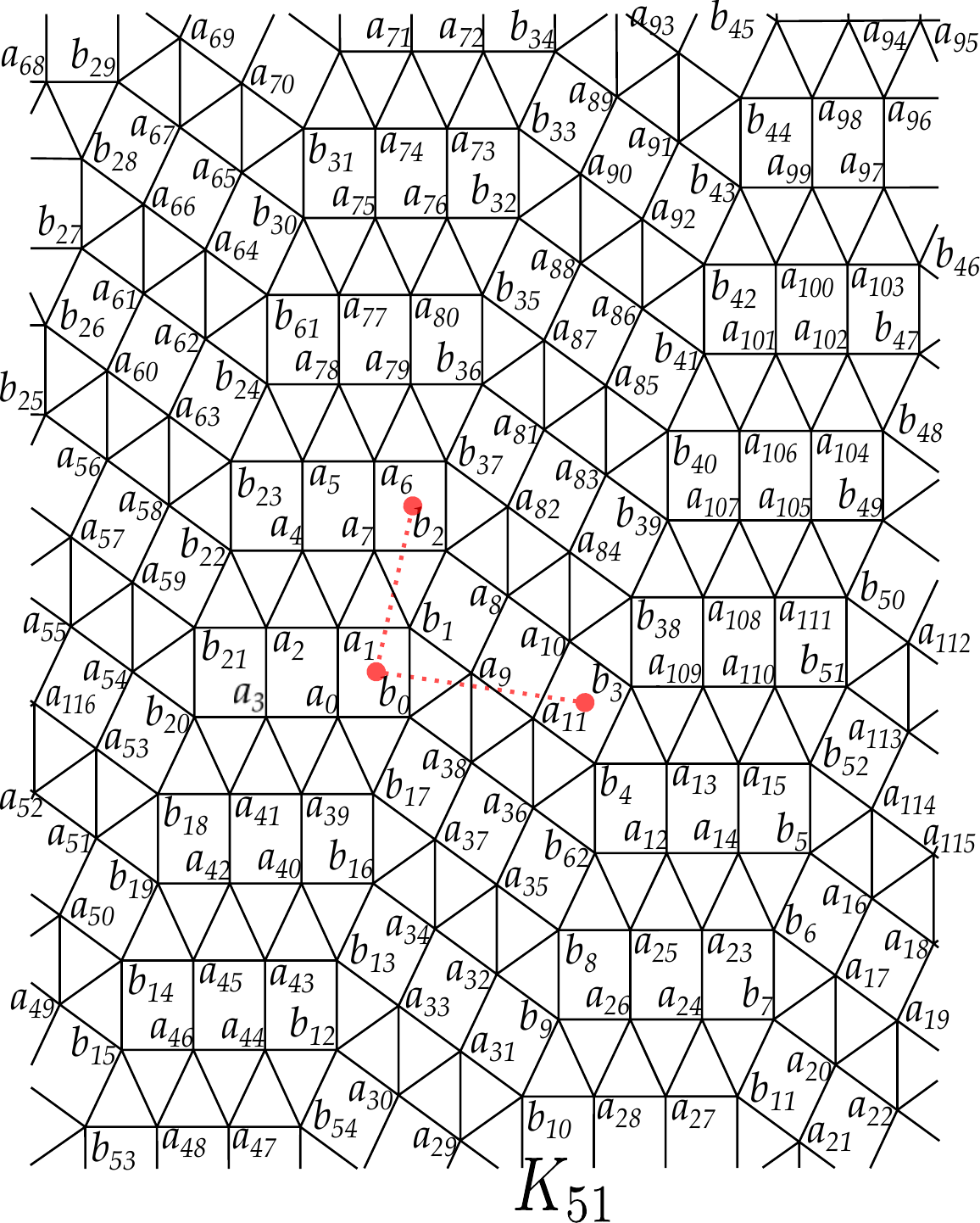}\hspace{5mm}
\includegraphics[height=6cm, width= 6cm]{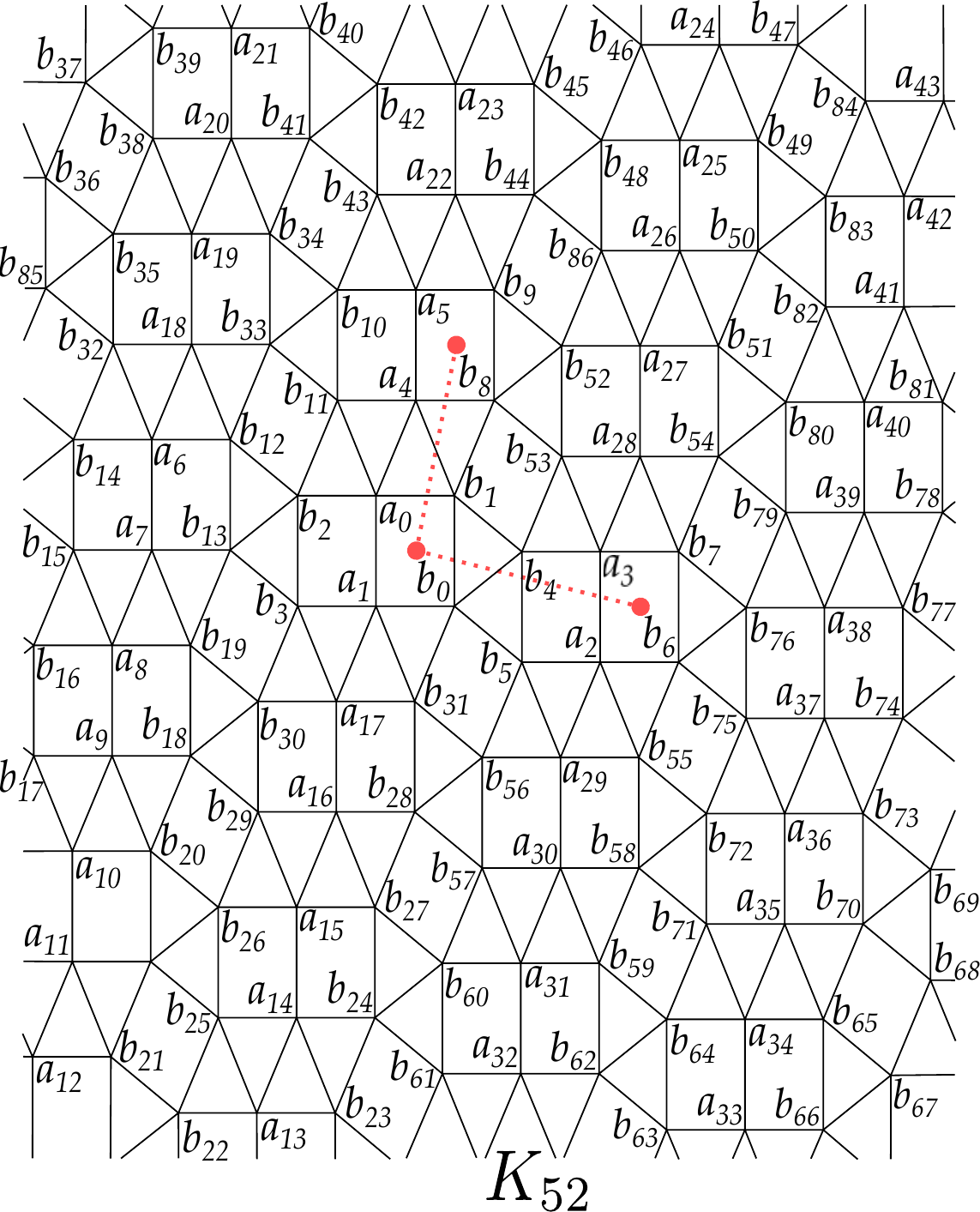}
\end{figure}     \begin{figure}
     \centering
\includegraphics[height=6cm, width= 6cm]{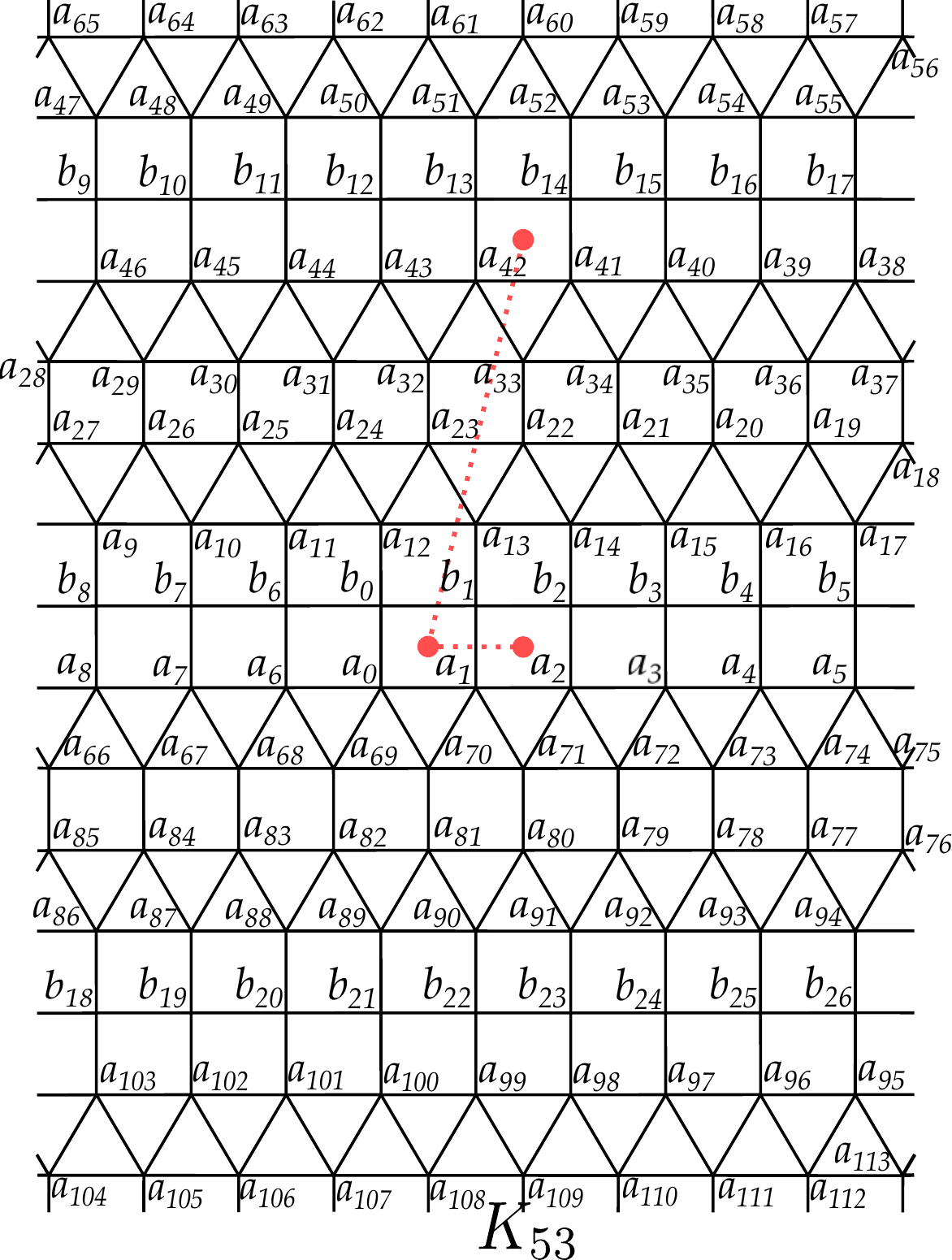}\hspace{5mm}
\includegraphics[height=6cm, width= 6cm]{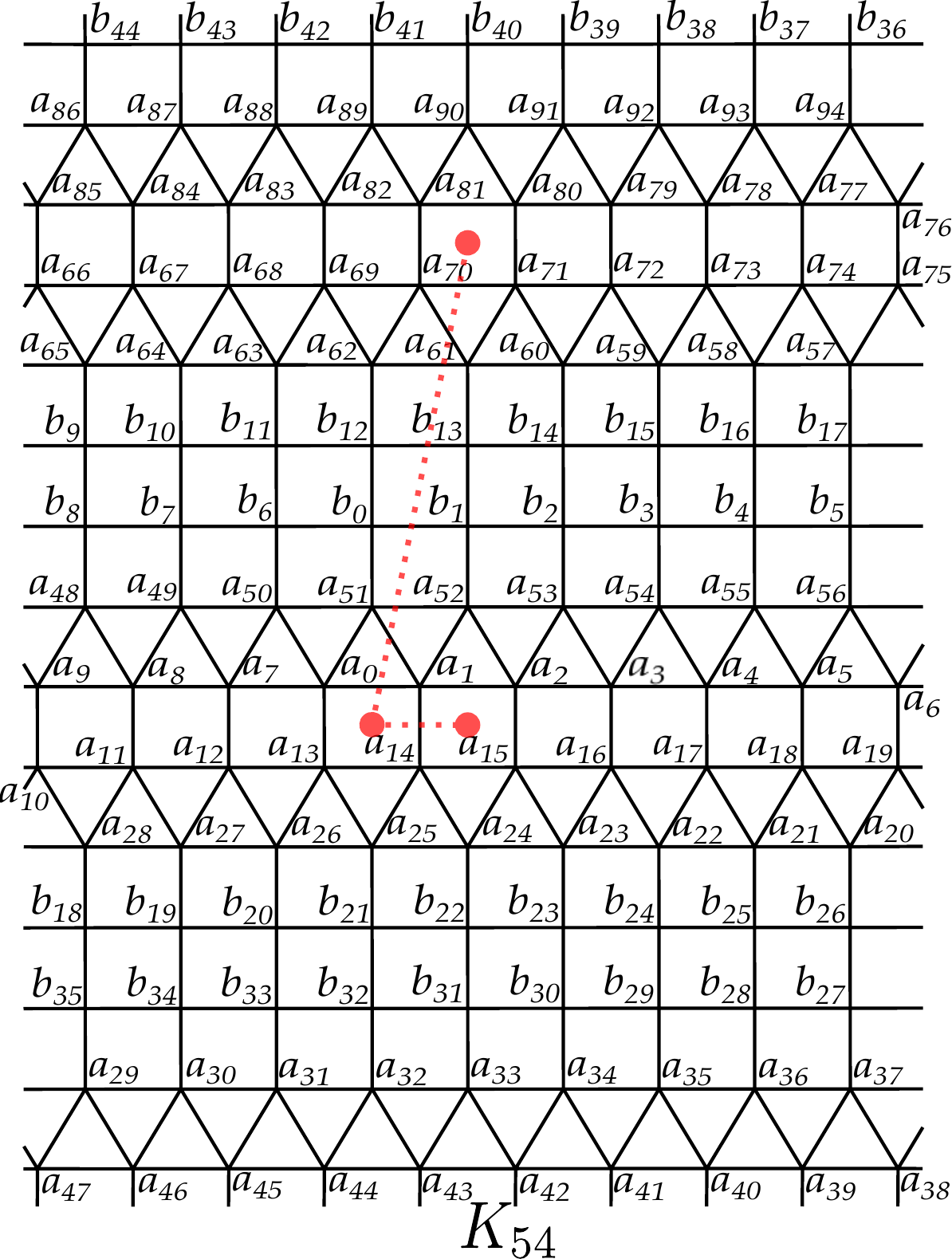}
\end{figure}     \begin{figure}
     \centering
\includegraphics[height=6cm, width= 6cm]{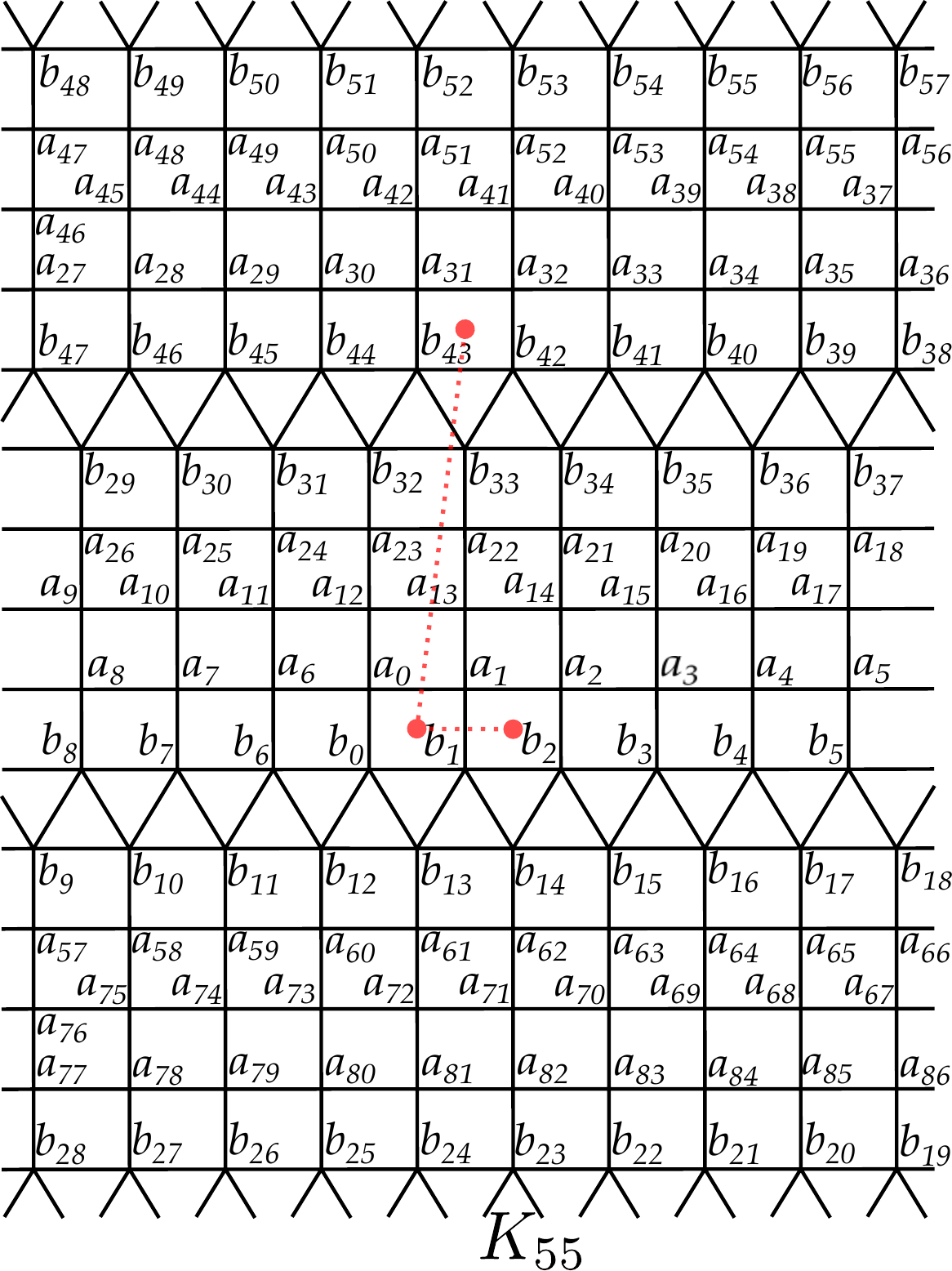}\hspace{5mm}
\includegraphics[height=6cm, width= 6cm]{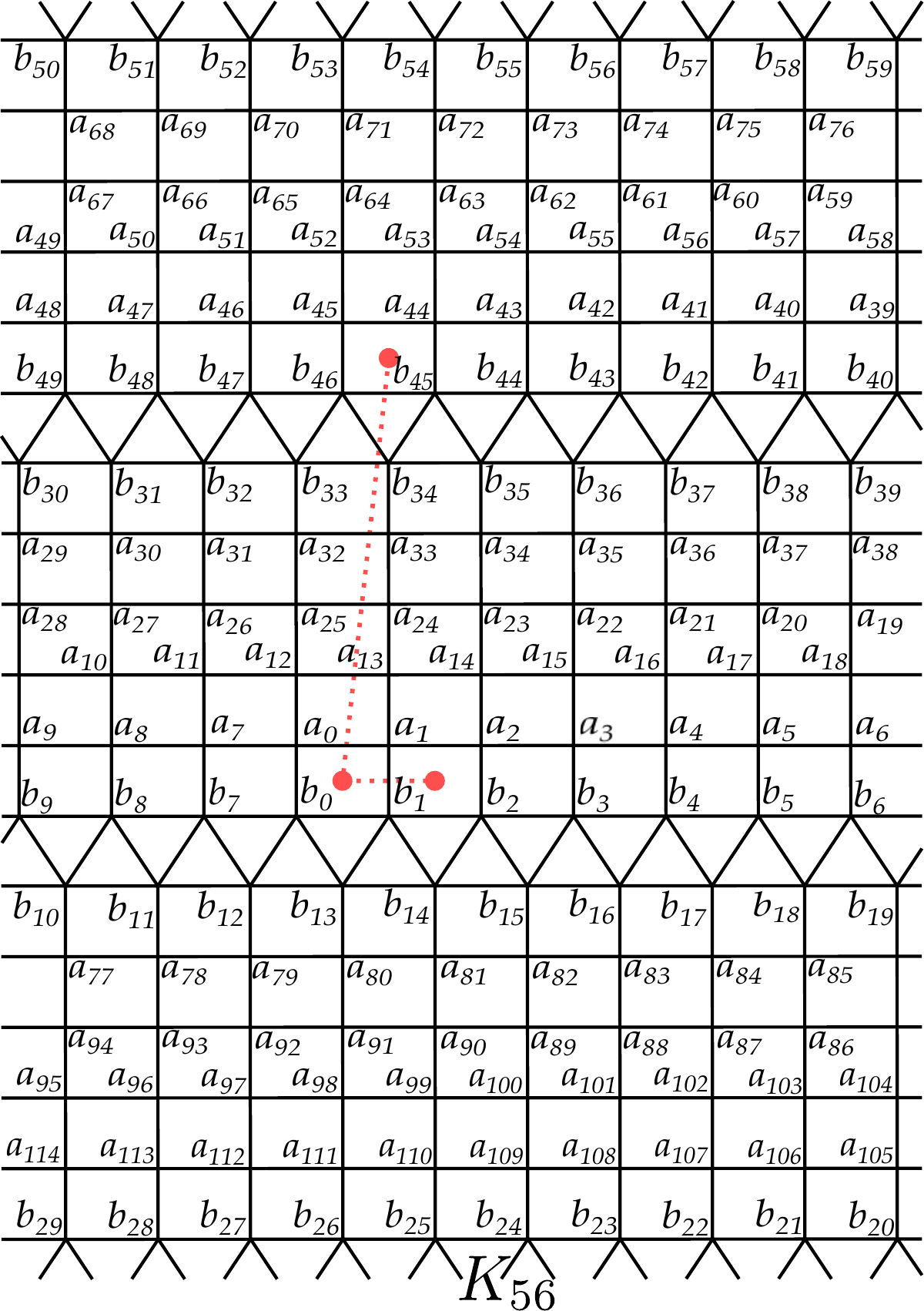}
\end{figure}     \begin{figure}
     \centering
\includegraphics[height=6cm, width= 6cm]{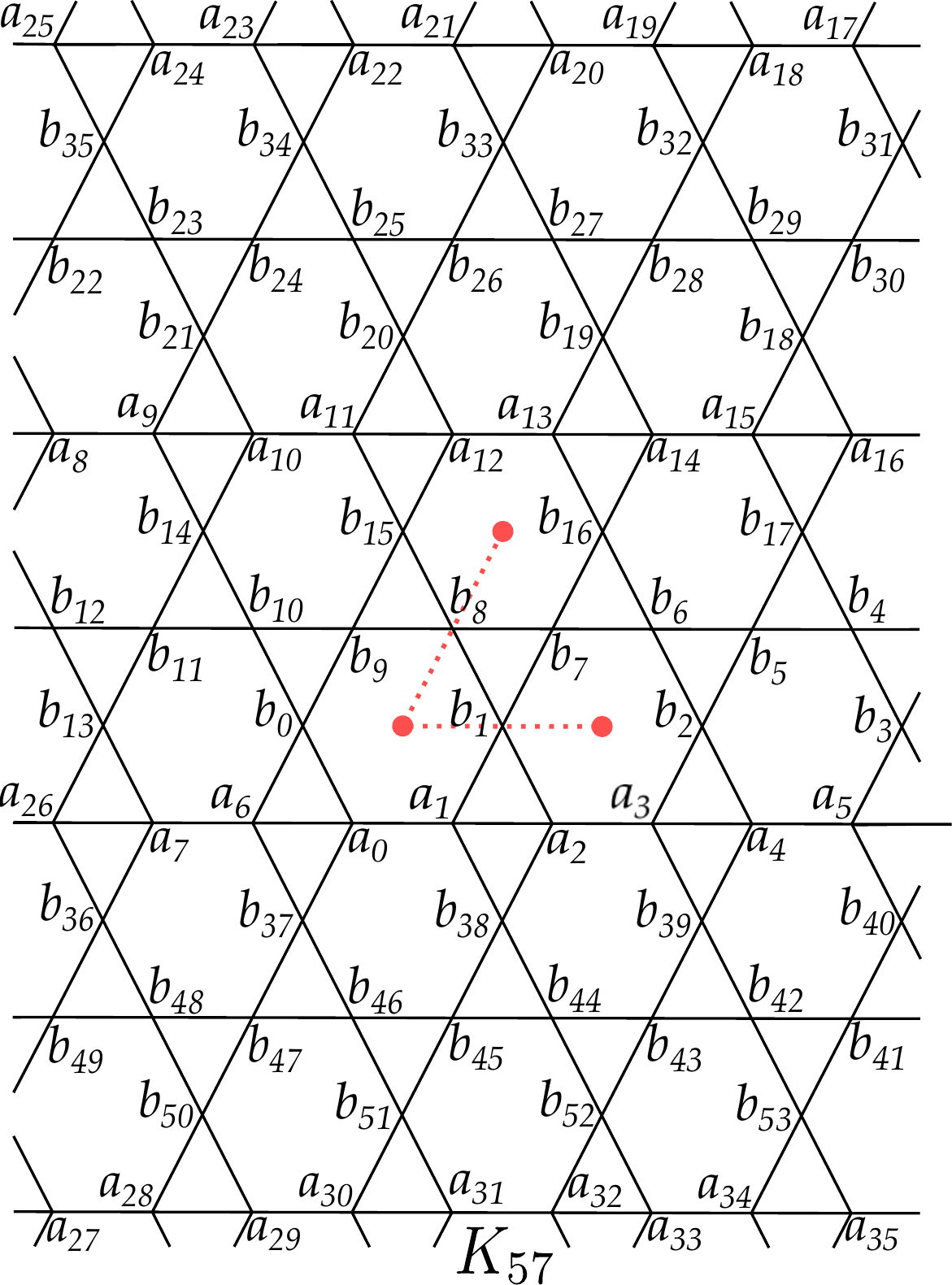}\hspace{5mm}
\includegraphics[height=6cm, width= 6cm]{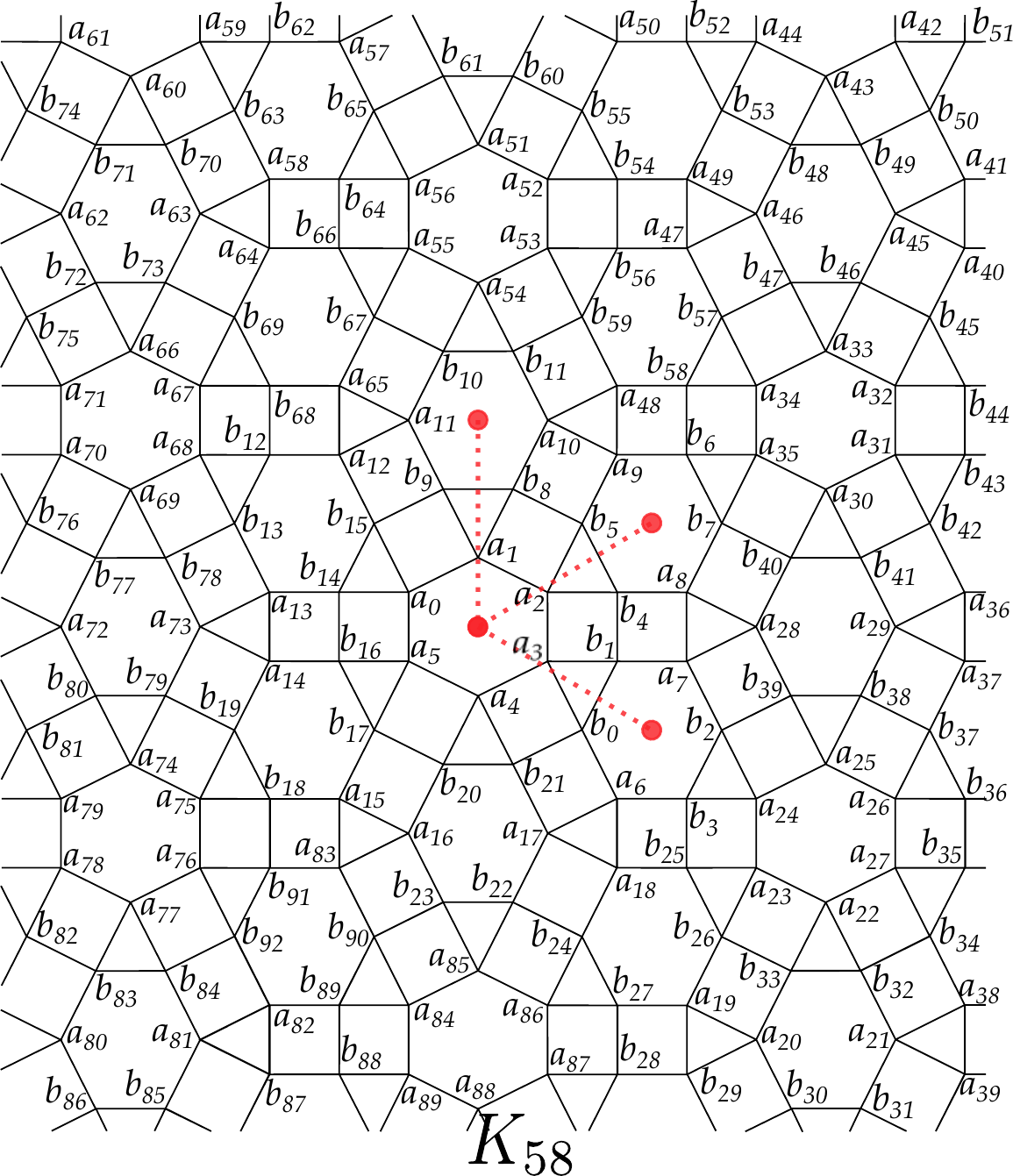}
\end{figure}     \begin{figure}
     \centering
\includegraphics[height=6cm, width= 6cm]{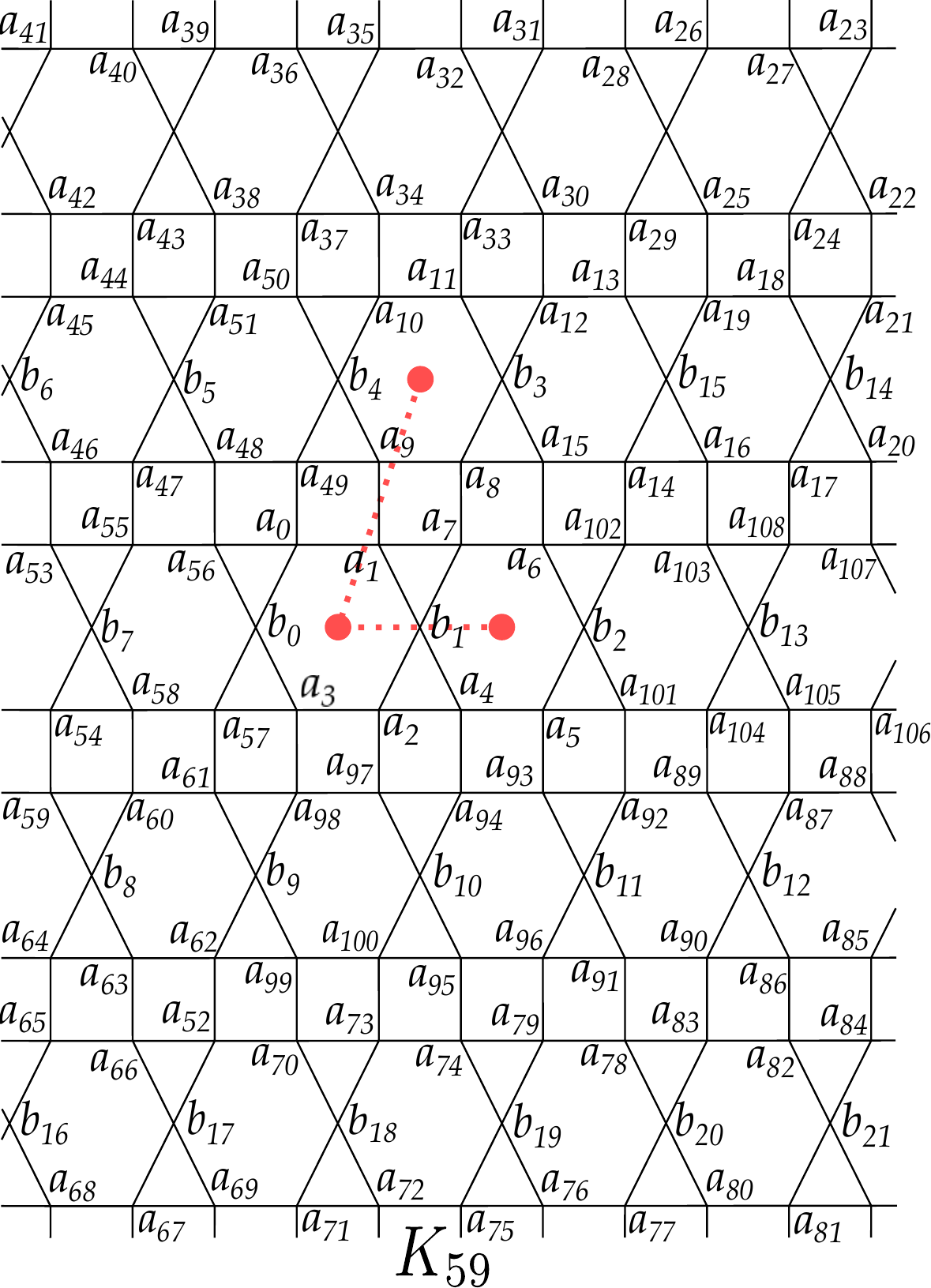}\hspace{5mm}
\includegraphics[height=6cm, width= 6cm]{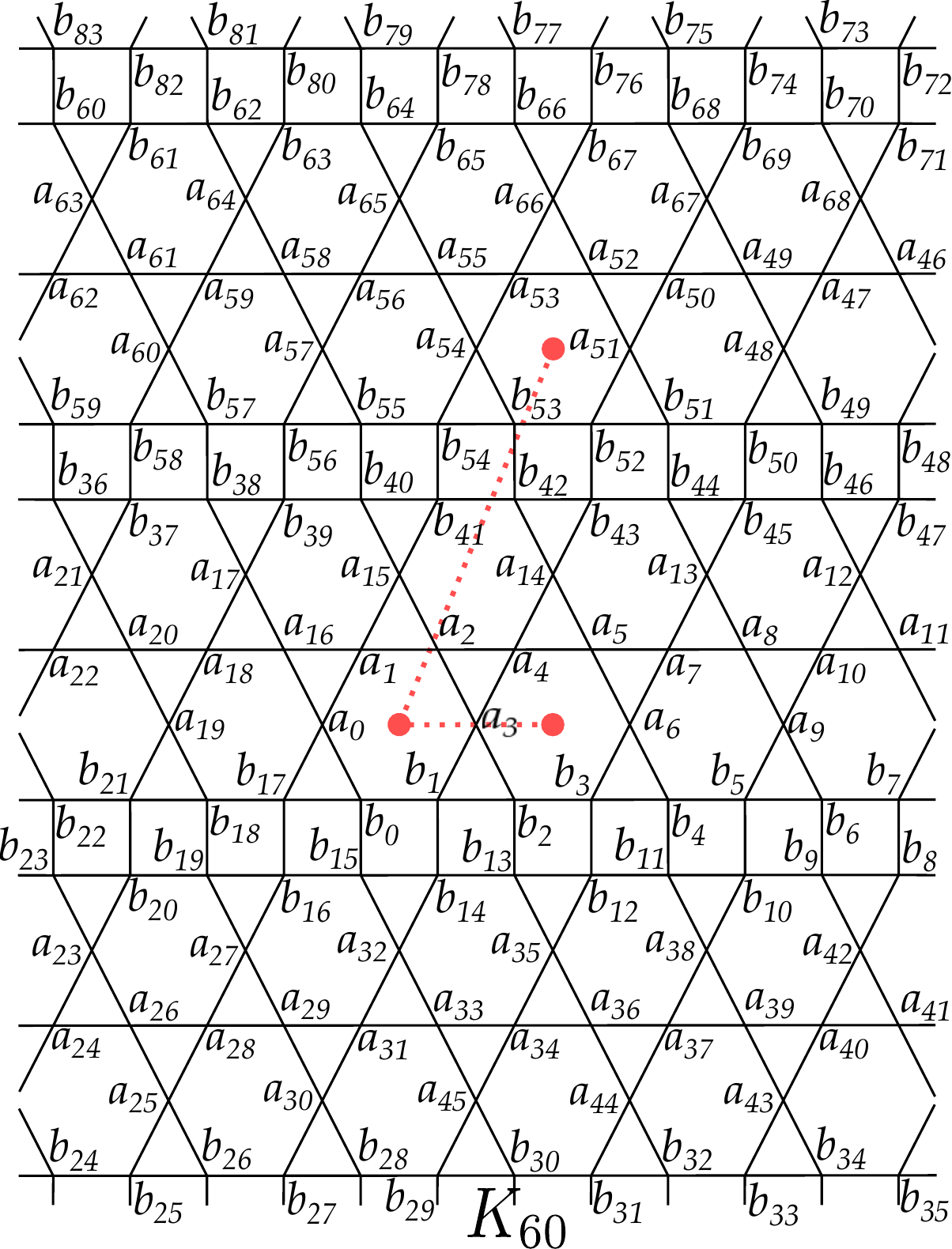}
\end{figure}     \begin{figure}
     \centering
\includegraphics[height=6cm, width= 6cm]{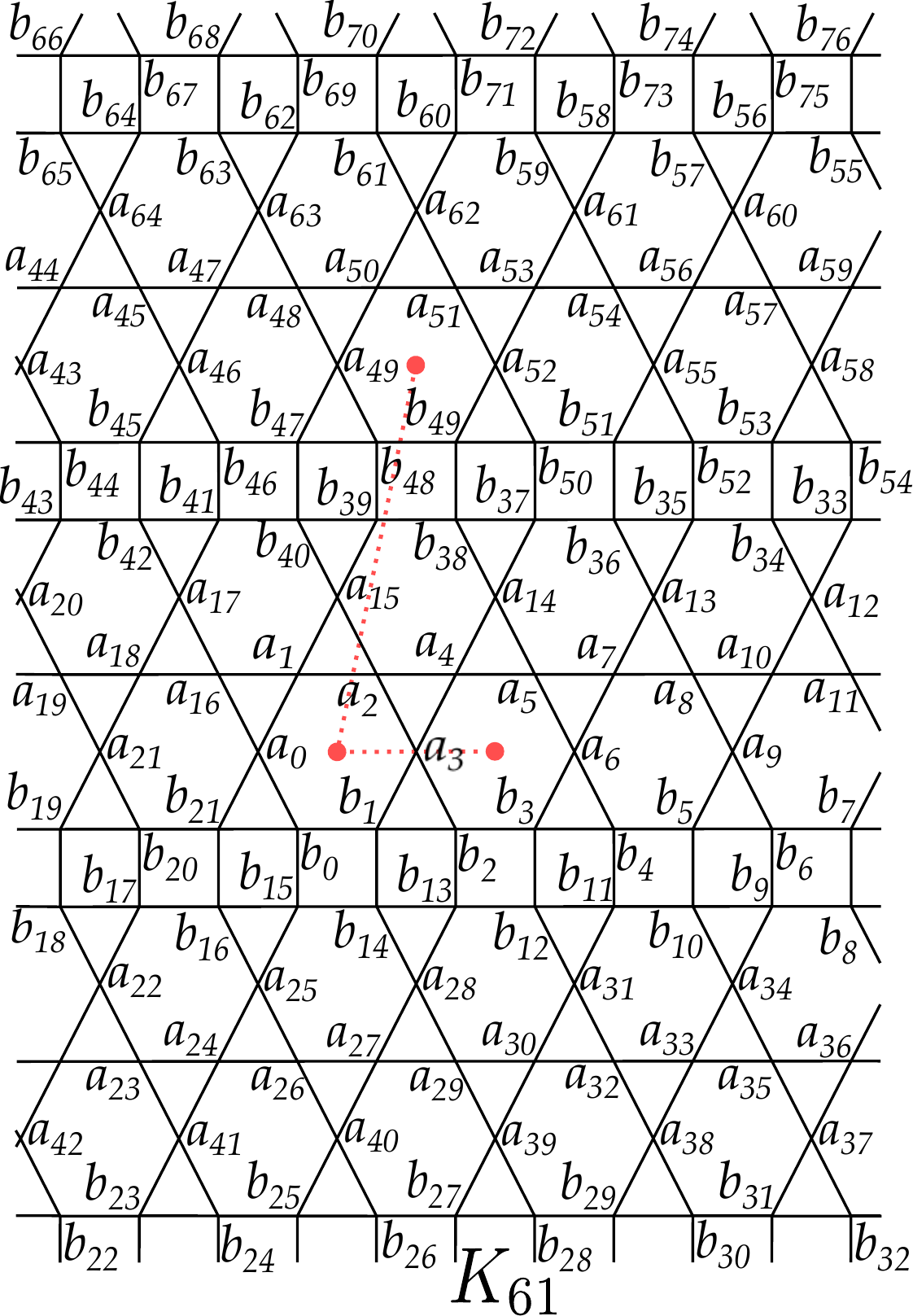}
\end{figure}

\newpage

\section{Proof of Theorem \ref{theo2}}

\begin{proposition}\label{prop2}
The maps $K_{i}$ ($1 \le i \le 61$) (in Section \ref{3uniform}) are unique up to isomorphism. 
\end{proposition} 

\begin{proof}[Proof of Theorem \ref{theo2}]
Let $X_1$ be a semiequivelar map on the torus that is the quotient of the plane's $3$-uniform lattice $K_1$. Let the vertices of $X_1$ form $m_1$ ${\rm Aut}(X_1)$-orbits. Let $K_1$ be as in Section \ref{3uniform}.
Let $V_{1} = V(K_1)$ be the vertex set of $K_1$. Let $H_{1}$ be the group of all the translations of $K_1$. So, $H_1 \leq {\rm Aut}(K_1)$.

Since $X_1$ is a semiequivelar map on the torus that is the quotient of the plane's $3$-uniform lattice $K_1$ (as by Proposition \ref{prop2} $K_1$ is unique), so, we can assume, there is a polyhedral covering map $\eta_{1} : K_1 \to X_1$ where $X_1 = K_1/\Gamma_{1}$  for some fixed element (vertex, edge or face) free subgroup $\Gamma_{1} \le {\rm Aut}(K_1)$. Hence $\Gamma_{1}$
consists of translations and glide reflections. Since $X_1 =
K_1/\Gamma_{1}$ is orientable, $\Gamma_{1}$ does not contain any glide reflection. Thus $\Gamma_{1} \leq H_{1}$.

 We take the middle point of the line segment joining vertices $a_{0}$ and $a_1$ as the origin $(0,0)$ of $K_1$. Let $\alpha_1 := a_{2} - a_{0}$ and $\beta_1 := a_{7}- a_{0}$ in $K_1$. Then $$H_1 := \langle x\mapsto x+\alpha_1, x\mapsto x+\beta_1\rangle.$$ Under the action of $H_1$, vertices of $K_1$ form five orbits.
Consider the subgroup $G_1$ of ${\rm Aut}(K_1)$ generated by $H_1$ and the map (the half rotation) $x\mapsto -x$. So,
\begin{align*}
  G_1 & =\{ \alpha : x\mapsto \varepsilon x + m\alpha_1 + n\beta_1  \, : \, \varepsilon=\pm 1, m, n \in \ZZ\} \cong H_1\rtimes \mathbb{Z}_2.
\end{align*}
Clearly, under the action of $G_1$, vertices of $K_1$ form three orbits. The orbits are 
\begin{align*}
& O_1 :=\langle a_{0} \rangle, O_2 :=\langle b_{0} \rangle, O_3 :=\langle c_{0} \rangle.
\end{align*}

\noindent {\bf Claim 1.} If $S \leq H_1$ then $S \unlhd G_1$.

\smallskip

Let $g \in G_1$ and $s\in S$. Then $g(x) = \varepsilon x+ma+nb+rc$ and $s(x) = x + pa+ qb+\ell c$ for some $m, n, r, p, q, \ell \in \mathbb{Z}$ and $\varepsilon\in\{1, -1\}$.
Therefore, 
\begin{align*}
(g\circ s\circ g^{-1})(x) & = (g\circ s)(\varepsilon(x-ma-nb-rc))\\
                          & = g(\varepsilon(x-ma-nb-rc)+pa+qb+\ell c)\\                          
                          & =x-ma-nb-rc+\varepsilon(pa+qb+\ell c)+ma+nb+rc\\
                          & = x+\varepsilon(pa+qb+\ell c)\\
                          & = s^{\varepsilon}(x).
\end{align*}
 Thus, $g\circ j\circ g^{-1} = s^{\varepsilon}\in S$. This completes the claim.

\smallskip

By Claim 1, $\Gamma_1$ is a normal subgroup of $G_1$. Therefore, $G_1/\Gamma_1$ acts on $X_1= K_1/\Gamma_1$.
Since 
\begin{align*}
& O_1 :=\langle a_{0} \rangle, O_2 :=\langle b_{0} \rangle, O_3 :=\langle c_{0} \rangle
\end{align*}
 are the $G_1$-orbits, it follows that $\eta_1(O_j)$ for $j=1, 2, \dots, 6$ are the $(G_1/\Gamma_1)$-orbits. Since the vertex set of $K_1$ is $\sqcup_{j=1}^{3}\eta_j(O_j)$ and $G_1/\Gamma_1 \leq {\rm Aut}(X_1)$, it follows that the number of ${\rm Aut}(X_1)$-orbits of vertices is $3$. 
 
 \medskip

Let $X_{2} = K_{2}/\Gamma_{2}$ be a semiequivelar map on the torus for some fixed element (vertex, edge or face) free subgroup $\Gamma_{2} \le {\rm Aut}(K_{2})$. Let the vertices of $X_{2}$ form $m_{2}$ ${\rm Aut}(X_{2})$-orbits.
We take the middle point of the line segment joining vertices $a_0$ and $a_1$ as the origin $(0,0)$ of $K_2$  (see Section \ref{3uniform}). Let  $\alpha_2 := a_{4} - a_0$ and $\beta_2 := a_{2} - a_{0} \in \mathbb{R}^2$. Similarly as above, define $H_2 := \langle x\mapsto x+\alpha_2, x\mapsto x+\beta_2\rangle$ and 
\begin{align*}
  G_2 & =\{ \alpha : x\mapsto \varepsilon x + m_2\alpha_2 + n_2\beta_2   \, : \, \varepsilon=\pm 1, m_2, n_2\in \ZZ\} \cong H_2\rtimes \mathbb{Z}_2.
\end{align*}
By the same arguments as above and in Claim 1, $\Gamma_2 \unlhd G_2$ and the number of $G_2/\Gamma_2$-orbits of vertices of $X_2$ is three. Therefore, $G_2/\Gamma_2$ acts on $X_2 = K_2/\Gamma_2$.
Since  $O_1 = \langle a_0 \rangle,$ $O_2 = \langle b_0 \rangle,$  $O_3 = \langle c_0 \rangle$
 are the $G_9$-orbits, it follows that $O_1 = \langle a_0 \rangle,$ $O_2 = \langle b_0 \rangle,$  $O_3 = \langle c_0 \rangle$ are the $(G_2/\Gamma_2)$-orbits.  
Since the vertex set of $X_2$ is $\sqcup_{j=1}^3 \eta_2(O_j)$ and $G_2/\Gamma_2 \leq {\rm Aut}(X_2)$, it follows that the number of ${\rm Aut}(X_2)$-orbits of vertices is $3$. 

\medskip

Let $X_{3} = K_{3}/\Gamma_{3}$ be a semiequivelar map on the torus for some fixed element (vertex, edge or face) free subgroup $\Gamma_{3} \le {\rm Aut}(K_{3})$. Let the vertices of $X_{3}$ form $m_{3}$ ${\rm Aut}(X_{3})$-orbits.
We take the middle point of the line segment joining vertices $b_1$ and $c_1$ as the origin $(0,0)$ of $K_3$  (see Section \ref{3uniform}). Let  $\alpha_3 := c_{22} - c_1$ and $\beta_3 := c_{5} - c_{1} \in \mathbb{R}^2$. Similarly as above, define $H_3 := \langle x\mapsto x+\alpha_3, x\mapsto x+\beta_3, x\mapsto x+\gamma_3\rangle$ and 
\begin{align*}
  G_3 & =\{ \alpha : x\mapsto \varepsilon x + m_3\alpha_3 + n_3\beta_3    \, : \, \varepsilon=\pm 1, m_3, n_3 \in \ZZ\} \cong H_3\rtimes \mathbb{Z}_2.
\end{align*}
By the same arguments as above and in Claim 1, $\Gamma_3 \unlhd G_3$ and the number of $G_3/\Gamma_3$-orbits of vertices of $X_3$ is eight. Therefore, $G_3/\Gamma_3$ acts on $X_3 = K_3/\Gamma_3$.
Since  $O_1 = \langle a_0 \rangle,$ $O_2 = \langle a_5 \rangle,$  $O_3 = \langle b_0 \rangle$, $O_4 = \langle b_1 \rangle,$ $O_5 = \langle b_2 \rangle,$ $O_6 = \langle c_0 \rangle,$ $O_7 = \langle c_1 \rangle,$ $O_8 = \langle c_2 \rangle$
 are the $G_9$-orbits, it follows that $O_1 = \langle a_0 \rangle,$ $O_2 = \langle a_5 \rangle,$  $O_3 = \langle b_0 \rangle$, $O_4 = \langle b_1 \rangle,$ $O_5 = \langle b_2 \rangle,$ $O_6 = \langle c_0 \rangle,$ $O_7 = \langle c_1 \rangle,$ $O_8 = \langle c_2 \rangle$ are the $(G_3/\Gamma_3)$-orbits.  
Since the vertex set of $X_3$ is $\sqcup_{j=1}^8 \eta_3(O_j)$ and $G_3/\Gamma_3 \leq {\rm Aut}(X_3)$, it follows that the number of ${\rm Aut}(X_3)$-orbits of vertices is $\le 8$. 
 
\medskip

Let $X_{4} = K_{4}/\Gamma_{4}$ be a semiequivelar map on the torus for some fixed element (vertex, edge or face) free subgroup $\Gamma_{4} \le {\rm Aut}(K_{4})$. Let the vertices of $X_{4}$ form $m_{4}$ ${\rm Aut}(X_{4})$-orbits.
We take the middle point of the line segment joining vertices $b_3$ and $b_4$ as the origin $(0,0)$ of $K_4$  (see Section \ref{3uniform}). Let  $\alpha_4 := b_{4} - b_3$ and $\beta_4 := b_{13} - b_{3} \in \mathbb{R}^2$. Similarly as above, define $H_4 := \langle x\mapsto x+\alpha_4, x\mapsto x+\beta_4\rangle$ and 
\begin{align*}
  G_4 & =\{ \alpha : x\mapsto \varepsilon x + m_4\alpha_4 + n_4\beta_4    \, : \, \varepsilon=\pm 1, m_4, n_4 \in \ZZ\} \cong H_4\rtimes \mathbb{Z}_2.
\end{align*}
By the same arguments as above and in Claim 1, $\Gamma_4 \unlhd G_4$ and the number of $G_4/\Gamma_4$-orbits of vertices of $X_4$ is three. Therefore, $G_4/\Gamma_4$ acts on $X_4 = K_4/\Gamma_4$.
Since  $O_1 = \langle a_0 \rangle,$ $O_2 = \langle b_3 \rangle,$  $O_3 = \langle c_0 \rangle$
 are the $G_4$-orbits, it follows that $O_1 = \langle a_0 \rangle,$ $O_2 = \langle b_3 \rangle,$  $O_3 = \langle c_0 \rangle$ are the $(G_4/\Gamma_4)$-orbits.  
Since the vertex set of $X_4$ is $\sqcup_{j=1}^3 \eta_4(O_j)$ and $G_4/\Gamma_4 \leq {\rm Aut}(X_4)$, it follows that the number of ${\rm Aut}(X_4)$-orbits of vertices is $3$.

\medskip

Let $X_{5} = K_{5}/\Gamma_{5}$ be a semiequivelar map on the torus for some fixed element (vertex, edge or face) free subgroup $\Gamma_{5} \le {\rm Aut}(K_{5})$. Let the vertices of $X_{5}$ form $m_{5}$ ${\rm Aut}(X_{5})$-orbits.
We take the middle point of the line segment joining vertices $a_0$ and $c_1$ as the origin $(0,0)$ of $K_5$  (see Section \ref{3uniform}). Let  $\alpha_5 := a_{6} - a_0$ and $\beta_5 := a_{32} - a_{0} \in \mathbb{R}^2$. Similarly as above, define $H_5 := \langle x\mapsto x+\alpha_5, x\mapsto x+\beta_5\rangle$ and 
\begin{align*}
  G_5 & =\{ \alpha : x\mapsto \varepsilon x + m_5\alpha_5 + n_5\beta_5    \, : \, \varepsilon=\pm 1, m_5, n_5 \in \ZZ\} \cong H_5\rtimes \mathbb{Z}_2.
\end{align*}
By the same arguments as above and in Claim 1, $\Gamma_5 \unlhd G_5$ and the number of $G_5/\Gamma_5$-orbits of vertices of $X_5$ is twelve. Therefore, $G_5/\Gamma_5$ acts on $X_5 = K_5/\Gamma_5$.
Since  $O_1 = \langle a_0 \rangle,$ $O_2 = \langle a_1 \rangle,$  $O_3 = \langle a_2 \rangle$ $O_4 = \langle a_3 \rangle,$$O_5 = \langle a_4 \rangle,$$O_6 = \langle a_5 \rangle,$$O_7 = \langle b_0 \rangle,$$O_8 = \langle c_0 \rangle,$$O_9 = \langle c_1 \rangle,$$O_{10} = \langle c_2 \rangle,$$O_{11} = \langle c_3 \rangle,$$O_{12} = \langle c_4 \rangle$
 are the $G_5$-orbits, it follows that $O_1 = \langle a_0 \rangle,$ $O_2 = \langle a_1 \rangle,$  $O_3 = \langle a_2 \rangle$ $O_4 = \langle a_3 \rangle,$$O_5 = \langle a_4 \rangle,$$O_6 = \langle a_5 \rangle,$$O_7 = \langle b_0 \rangle,$$O_8 = \langle c_0 \rangle,$$O_9 = \langle c_1 \rangle,$$O_{10} = \langle c_2 \rangle,$$O_{11} = \langle c_3 \rangle,$$O_{12} = \langle c_4 \rangle$ are the $(G_5/\Gamma_5)$-orbits.  
Since the vertex set of $X_5$ is $\sqcup_{j=1}^{12} \eta_5(O_j)$ and $G_5/\Gamma_5 \leq {\rm Aut}(X_5)$, it follows that the number of ${\rm Aut}(X_5)$-orbits of vertices is $\le12$. 
 
\medskip

Let $X_{6} = K_{6}/\Gamma_{6}$ be a semiequivelar map on the torus for some fixed element (vertex, edge or face) free subgroup $\Gamma_{6} \le {\rm Aut}(K_{6})$. Let the vertices of $X_{6}$ form $m_{6}$ ${\rm Aut}(X_{6})$-orbits.
We take the middle point of the line segment joining vertices $a_0$ and $a_1$ as the origin $(0,0)$ of $K_6$  (see Section \ref{3uniform}). Let  $\alpha_6 := a_{2} - a_0$ and $\beta_6 := a_{22} - a_{0} \in \mathbb{R}^2$. Similarly as above, define $H_6 := \langle x\mapsto x+\alpha_6, x\mapsto x+\beta_6\rangle$ and 
\begin{align*}
  G_6 & =\{ \alpha : x\mapsto \varepsilon x + m_6\alpha_6 + n_6\beta_6    \, : \, \varepsilon=\pm 1, m_6, n_6 \in \ZZ\} \cong H_6\rtimes \mathbb{Z}_2.
\end{align*}
By the same arguments as above and in Claim 1, $\Gamma_6 \unlhd G_6$ and the number of $G_6/\Gamma_6$-orbits of vertices of $X_6$ is seven. Therefore, $G_6/\Gamma_6$ acts on $X_6 = K_6/\Gamma_6$.
Since  $O_1 = \langle a_0 \rangle,$ $O_2 = \langle a_9 \rangle,$  $O_3 = \langle a_8 \rangle$ $O_4 = \langle b_0 \rangle,$$O_5 = \langle c_0 \rangle,$ $O_6 = \langle c_{1} \rangle,$$O_7 = \langle c_{10} \rangle$
 are the $G_6$-orbits, it follows that $O_1 = \langle a_0 \rangle,$ $O_2 = \langle a_9 \rangle,$  $O_3 = \langle a_8 \rangle$ $O_4 = \langle b_0 \rangle,$$O_5 = \langle c_0 \rangle,$ $O_6 = \langle c_{1} \rangle,$$O_7 = \langle c_{10} \rangle$ are the $(G_6/\Gamma_6)$-orbits.  
Since the vertex set of $X_6$ is $\sqcup_{j=1}^{7} \eta_6(O_j)$ and $G_6/\Gamma_6 \leq {\rm Aut}(X_6)$, it follows that the number of ${\rm Aut}(X_6)$-orbits of vertices is $\le7$. 

\medskip

Let $X_{7} = K_{7}/\Gamma_{7}$ be a semiequivelar map on the torus for some fixed element (vertex, edge or face) free subgroup $\Gamma_{7} \le {\rm Aut}(K_{7})$. Let the vertices of $X_{7}$ form $m_{7}$ ${\rm Aut}(X_{7})$-orbits.
We take the point $a_0$ as the origin $(0,0)$ of $K_7$  (see Section \ref{3uniform}). Let  $\alpha_7 := a_{1} - a_0$, $\beta_7 := a_{2} - a_{0}$ and $\gamma_7 := a_{3} - a_{0}\in \mathbb{R}^2$. Similarly as above, define $H_7 := \langle x\mapsto x+\alpha_7, x\mapsto x+\beta_7, x\mapsto x+\gamma_7\rangle$ and 
\begin{align*}
 G_7 & =\{ \alpha : x\mapsto \varepsilon x + m_7\alpha_7 + n_7\beta_7 +r_7\gamma_7   \, : \, \varepsilon=\pm 1, m_7, n_7, r_7 \in \ZZ\} \cong H_7\rtimes \mathbb{Z}_2
\end{align*}
By the same arguments as above and in Claim 1, $\Gamma_7 \unlhd G_7$ and the number of $G_7/\Gamma_7$-orbits of vertices of $X_7$ is seven. Therefore, $G_7/\Gamma_7$ acts on $X_7 = K_7/\Gamma_7$.
Since  $O_1 = \langle a_0 \rangle,$ $O_2 = \langle b_0 \rangle,$  $O_3 = \langle b_{19} \rangle$ $O_4 = \langle b_{17} \rangle,$$O_5 = \langle c_0 \rangle,$ $O_6 = \langle c_{1} \rangle,$ $O_7 = \langle c_{2} \rangle$
 are the $G_7$-orbits, it follows that $O_1 = \langle a_0 \rangle,$ $O_2 = \langle b_0 \rangle,$  $O_3 = \langle b_{19} \rangle$ $O_4 = \langle b_{17} \rangle,$$O_5 = \langle c_0 \rangle,$ $O_6 = \langle c_{1} \rangle,$ $O_7 = \langle c_{2} \rangle$ are the $(G_7/\Gamma_7)$-orbits.  
Since the vertex set of $X_7$ is $\sqcup_{j=1}^{7} \eta_7(O_j)$ and $G_7/\Gamma_7 \leq {\rm Aut}(X_7)$, it follows that the number of ${\rm Aut}(X_7)$-orbits of vertices is $\le7$.  
 
\medskip

Let $X_{8} = K_{8}/\Gamma_{8}$ be a semiequivelar map on the torus for some fixed element (vertex, edge or face) free subgroup $\Gamma_{8} \le {\rm Aut}(K_{8})$. Let the vertices of $X_{8}$ form $m_{8}$ ${\rm Aut}(X_{8})$-orbits.
We take the middle point of the line segment joining vertices $a_0$ and $a_6$ as the origin $(0,0)$ of $K_8$  (see Section \ref{3uniform}). Let  $\alpha_8 := a_{28} - a_0$ and $\beta_8 := a_{46} - a_{0} \in \mathbb{R}^2$. Similarly as above, define $H_8 := \langle x\mapsto x+\alpha_8, x\mapsto x+\beta_8\rangle$ and 
\begin{align*}
  G_8 & =\{ \alpha : x\mapsto \varepsilon x + m_8\alpha_8 + n_8\beta_8    \, : \, \varepsilon=\pm 1, m_8, n_8 \in \ZZ\} \cong H_8\rtimes \mathbb{Z}_2.
\end{align*}
By the same arguments as above and in Claim 1, $\Gamma_8 \unlhd G_8$ and the number of $G_8/\Gamma_8$-orbits of vertices of $X_8$ is ten. Therefore, $G_8/\Gamma_8$ acts on $X_8 = K_8/\Gamma_8$.
Since  $O_1 = \langle a_0 \rangle,$ $O_2 = \langle a_1 \rangle,$  $O_3 = \langle a_2 \rangle$ $O_4 = \langle a_3 \rangle,$$O_5 = \langle a_4 \rangle,$ $O_6 = \langle a_{5} \rangle,$ $O_7 = \langle b_{0} \rangle$, $O_8 = \langle c_{0} \rangle$, $O_{9} = \langle c_{1} \rangle$, $O_{10} = \langle c_{2} \rangle$ 
 are the $G_8$-orbits, it follows that $O_1 = \langle a_0 \rangle,$ $O_2 = \langle a_1 \rangle,$  $O_3 = \langle a_2 \rangle$ $O_4 = \langle a_3 \rangle,$$O_5 = \langle a_4 \rangle,$ $O_6 = \langle a_{5} \rangle,$ $O_7 = \langle b_{0} \rangle$, $O_8 = \langle c_{0} \rangle$, $O_{9} = \langle c_{1} \rangle$, $O_{10} = \langle c_{2} \rangle$ are the $(G_8/\Gamma_8)$-orbits.  
Since the vertex set of $X_8$ is $\sqcup_{j=1}^{10} \eta_8(O_j)$ and $G_8/\Gamma_8 \leq {\rm Aut}(X_8)$, it follows that the number of ${\rm Aut}(X_8)$-orbits of vertices is $\le10$. 
 
\medskip

Let $X_{9} = K_{9}/\Gamma_{9}$ be a semiequivelar map on the torus for some fixed element (vertex, edge or face) free subgroup $\Gamma_{9} \le {\rm Aut}(K_{9})$. Let the vertices of $X_{9}$ form $m_{9}$ ${\rm Aut}(X_{9})$-orbits.
We take the middle point of the line segment joining vertices $a_0$ and $a_3$ as the origin $(0,0)$ of $K_9$  (see Section \ref{3uniform}). Let  $\alpha_9 := a_{8} - a_0$, $\beta_9 := a_{15} - a_{0}$ and $\gamma_7 := a_{18} - a_{0}\in \mathbb{R}^2$. Similarly as above, define $H_9 := \langle x\mapsto x+\alpha_9, x\mapsto x+\beta_9, x\mapsto x+\gamma_9\rangle$ and 
\begin{align*}
 G_9 & =\{ \alpha : x\mapsto \varepsilon x + m_9\alpha_9 + n_9\beta_9 +r_9\gamma_9   \, : \, \varepsilon=\pm 1, m_9, n_9, r_9 \in \ZZ\} \cong H_9\rtimes \mathbb{Z}_2
\end{align*}
By the same arguments as above and in Claim 1, $\Gamma_9 \unlhd G_9$ and the number of $G_9/\Gamma_9$-orbits of vertices of $X_9$ is ten. Therefore, $G_9/\Gamma_9$ acts on $X_9 = K_9/\Gamma_9$.
Since  $O_1 = \langle a_0 \rangle,$ $O_2 = \langle a_1 \rangle,$  $O_3 = \langle a_2 \rangle$ $O_4 = \langle b_0 \rangle,$$O_5 = \langle c_0 \rangle,$ $O_6 = \langle c_{1} \rangle,$ $O_7 = \langle c_{4} \rangle$, $O_8 = \langle c_{5} \rangle$, $O_{9} = \langle c_{8} \rangle$, $O_{10} = \langle c_{9} \rangle$ 
 are the $G_9$-orbits, it follows that $O_1 = \langle a_0 \rangle,$ $O_2 = \langle a_1 \rangle,$  $O_3 = \langle a_2 \rangle$ $O_4 = \langle b_0 \rangle,$$O_5 = \langle c_0 \rangle,$ $O_6 = \langle c_{1} \rangle,$ $O_7 = \langle c_{4} \rangle$, $O_8 = \langle c_{5} \rangle$, $O_{9} = \langle c_{8} \rangle$, $O_{10} = \langle c_{9} \rangle$ are the $(G_9/\Gamma_9)$-orbits.  
Since the vertex set of $X_9$ is $\sqcup_{j=1}^{10} \eta_9(O_j)$ and $G_9/\Gamma_9 \leq {\rm Aut}(X_9)$, it follows that the number of ${\rm Aut}(X_9)$-orbits of vertices is $\le10$. 

 \medskip

Let $X_{10} = K_{10}/\Gamma_{10}$ be a semiequivelar map on the torus for some fixed element (vertex, edge or face) free subgroup $\Gamma_{10} \le {\rm Aut}(K_{10})$. Let the vertices of $X_{10}$ form $m_{10}$ ${\rm Aut}(X_{10})$-orbits.
We take the middle point of the line segment joining vertices $a_0$ and $a_{13}$ as the origin $(0,0)$ of $K_{10}$  (see Section \ref{3uniform}). Let  $\alpha_{10} := a_{1} - a_0$ and $\beta_{10} := a_{22} - a_{0} \in \mathbb{R}^2$. Similarly as above, define $H_{10} := \langle x\mapsto x+\alpha_{10}, x\mapsto x+\beta_{10}\rangle$ and 
\begin{align*}
  G_{10} & =\{ \alpha : x\mapsto \varepsilon x + m_{10}\alpha_{10} + n_{10}\beta_{10}   \, : \, \varepsilon=\pm 1, m_{10}, n_{10}\in \ZZ\} \cong H_{10}\rtimes \mathbb{Z}_2.
\end{align*}
By the same arguments as above and in Claim 1, $\Gamma_{10} \unlhd G_{10}$ and the number of $G_{10}/\Gamma_{10}$-orbits of vertices of $X_{10}$ is three. Therefore, $G_{10}/\Gamma_{10}$ acts on $X_{10} = K_{10}/\Gamma_{10}$.
Since  $O_1 = \langle a_0 \rangle,$ $O_2 = \langle b_3 \rangle,$  $O_3 = \langle c_0 \rangle$
 are the $G_{10}$-orbits, it follows that $O_1 = \langle a_0 \rangle,$ $O_2 = \langle b_3 \rangle,$  $O_3 = \langle c_0 \rangle$ are the $(G_{10}/\Gamma_{10})$-orbits.  
Since the vertex set of $X_{10}$ is $\sqcup_{j=1}^3 \eta_{10}(O_j)$ and $G_{10}/\Gamma_{10} \leq {\rm Aut}(X_{10})$, it follows that the number of ${\rm Aut}(X_{10})$-orbits of vertices is $3$. 

 \medskip

Let $X_{11} = K_{11}/\Gamma_{11}$ be a semiequivelar map on the torus for some fixed element (vertex, edge or face) free subgroup $\Gamma_{11} \le {\rm Aut}(K_{11})$. Let the vertices of $X_{11}$ form $m_{11}$ ${\rm Aut}(X_{11})$-orbits.
We take the middle point of the line segment joining vertices $a_0$ and $a_{13}$ as the origin $(0,0)$ of $K_{11}$  (see Section \ref{3uniform}). Let  $\alpha_{11} := a_{1} - a_0$ and $\beta_{11} := a_{22} - a_{0} \in \mathbb{R}^2$. Similarly as above, define $H_{11} := \langle x\mapsto x+\alpha_{11}, x\mapsto x+\beta_{11}\rangle$ and 
\begin{align*}
  G_{11} & =\{ \alpha : x\mapsto \varepsilon x + m_{11}\alpha_{11} + n_{11}\beta_{11}   \, : \, \varepsilon=\pm 1, m_{11}, n_{11}\in \ZZ\} \cong H_{11}\rtimes \mathbb{Z}_2.
\end{align*}
By the same arguments as above and in Claim 1, $\Gamma_{11} \unlhd G_{11}$ and the number of $G_{11}/\Gamma_{11}$-orbits of vertices of $X_{11}$ is three. Therefore, $G_{11}/\Gamma_{11}$ acts on $X_{11} = K_{11}/\Gamma_{11}$.
Since  $O_1 = \langle a_0 \rangle,$ $O_2 = \langle b_3 \rangle,$  $O_3 = \langle c_0 \rangle$
 are the $G_{11}$-orbits, it follows that $O_1 = \langle a_0 \rangle,$ $O_2 = \langle b_3 \rangle,$  $O_3 = \langle c_0 \rangle$ are the $(G_{11}/\Gamma_{11})$-orbits.  
Since the vertex set of $X_{11}$ is $\sqcup_{j=1}^3 \eta_{11}(O_j)$ and $G_{11}/\Gamma_{11} \leq {\rm Aut}(X_{11})$, it follows that the number of ${\rm Aut}(X_{11})$-orbits of vertices is $3$. 

 \medskip

Let $X_{12} = K_{12}/\Gamma_{12}$ be a semiequivelar map on the torus for some fixed element (vertex, edge or face) free subgroup $\Gamma_{12} \le {\rm Aut}(K_{12})$. Let the vertices of $X_{12}$ form $m_{12}$ ${\rm Aut}(X_{12})$-orbits.
We take the middle point of the line segment joining vertices $a_0$ and $a_{11}$ as the origin $(0,0)$ of $K_{12}$  (see Section \ref{3uniform}). Let  $\alpha_{12} := a_{1} - a_0$ and $\beta_{12} := a_{19} - a_{0} \in \mathbb{R}^2$. Similarly as above, define $H_{12} := \langle x\mapsto x+\alpha_{12}, x\mapsto x+\beta_{12}\rangle$ and 
\begin{align*}
  G_{12} & =\{ \alpha : x\mapsto \varepsilon x + m_{12}\alpha_{12} + n_{12}\beta_{12}   \, : \, \varepsilon=\pm 1, m_{12}, n_{12}\in \ZZ\} \cong H_{12}\rtimes \mathbb{Z}_2.
\end{align*}
By the same arguments as above and in Claim 1, $\Gamma_{12} \unlhd G_{12}$ and the number of $G_{12}/\Gamma_{12}$-orbits of vertices of $X_{12}$ is three. Therefore, $G_{12}/\Gamma_{12}$ acts on $X_{12} = K_{12}/\Gamma_{12}$.
Since  $O_1 = \langle a_0 \rangle,$ $O_2 = \langle b_2 \rangle,$  $O_3 = \langle c_0 \rangle$
 are the $G_{12}$-orbits, it follows that $O_1 = \langle a_0 \rangle,$ $O_2 = \langle b_2 \rangle,$  $O_3 = \langle c_0 \rangle$ are the $(G_{12}/\Gamma_{12})$-orbits.  
Since the vertex set of $X_{12}$ is $\sqcup_{j=1}^3 \eta_{12}(O_j)$ and $G_{12}/\Gamma_{12} \leq {\rm Aut}(X_{12})$, it follows that the number of ${\rm Aut}(X_{12})$-orbits of vertices is $3$. 

 \medskip

Let $X_{13} = K_{13}/\Gamma_{13}$ be a semiequivelar map on the torus for some fixed element (vertex, edge or face) free subgroup $\Gamma_{13} \le {\rm Aut}(K_{13})$. Let the vertices of $X_{13}$ form $m_{13}$ ${\rm Aut}(X_{13})$-orbits.
We take the middle point of the line segment joining vertices $a_0$ and $a_{14}$ as the origin $(0,0)$ of $K_{13}$  (see Section \ref{3uniform}). Let  $\alpha_{13} := a_{1} - a_0$ and $\beta_{13} := a_{24} - a_{0} \in \mathbb{R}^2$. Similarly as above, define $H_{13} := \langle x\mapsto x+\alpha_{13}, x\mapsto x+\beta_{13}\rangle$ and 
\begin{align*}
  G_{13} & =\{ \alpha : x\mapsto \varepsilon x + m_{13}\alpha_{13} + n_{13}\beta_{13}   \, : \, \varepsilon=\pm 1, m_{13}, n_{13}\in \ZZ\} \cong H_{13}\rtimes \mathbb{Z}_2.
\end{align*}
By the same arguments as above and in Claim 1, $\Gamma_{13} \unlhd G_{13}$ and the number of $G_{13}/\Gamma_{13}$-orbits of vertices of $X_{13}$ is three. Therefore, $G_{13}/\Gamma_{13}$ acts on $X_{13} = K_{13}/\Gamma_{13}$.
Since  $O_1 = \langle a_0 \rangle,$ $O_2 = \langle b_3 \rangle,$  $O_3 = \langle c_2 \rangle$
 are the $G_{13}$-orbits, it follows that $O_1 = \langle a_0 \rangle,$ $O_2 = \langle b_3 \rangle,$  $O_3 = \langle c_2 \rangle$ are the $(G_{13}/\Gamma_{13})$-orbits.  
Since the vertex set of $X_{13}$ is $\sqcup_{j=1}^3 \eta_{13}(O_j)$ and $G_{13}/\Gamma_{13} \leq {\rm Aut}(X_{13})$, it follows that the number of ${\rm Aut}(X_{13})$-orbits of vertices is $3$. 

\medskip

Let $X_{15} = K_{15}/\Gamma_{15}$ be a semiequivelar map on the torus for some fixed element (vertex, edge or face) free subgroup $\Gamma_{15} \le {\rm Aut}(K_{15})$. Let the vertices of $X_{15}$ form $m_{15}$ ${\rm Aut}(X_{15})$-orbits.
We take the middle point of the line segment joining vertices $c_0$ and $c_{3}$ as the origin $(0,0)$ of $K_{15}$  (see Section \ref{3uniform}). Let  $\alpha_{15} := a_{1} - a_0$, $\beta_{15} := a_{2} - a_{0}$ and $\gamma_{15} := a_{3} - a_{0}\in \mathbb{R}^2$. Similarly as above, define $H_{15} := \langle x\mapsto x+\alpha_{15}, x\mapsto x+\beta_{15}, x\mapsto x+\gamma_{15}\rangle$ and 
\begin{align*}
 G_{15} & =\{ \alpha : x\mapsto \varepsilon x + m_{15}\alpha_{15} + n_{15}\beta_{15} +r_{15}\gamma_{15}   \, : \, \varepsilon=\pm 1, m_{15}, n_{15}, r_{15} \in \ZZ\} \cong H_{15}\rtimes \mathbb{Z}_2
\end{align*}
By the same arguments as above and in Claim 1, $\Gamma_{15} \unlhd G_{15}$ and the number of $G_{15}/\Gamma_{15}$-orbits of vertices of $X_7$ is ten. Therefore, $G_{15}/\Gamma_{15}$ acts on $X_{15} = K_{15}/\Gamma_{15}$.
Since  $O_1 = \langle a_0 \rangle,$ $O_2 = \langle b_0 \rangle,$  $O_3 = \langle b_1 \rangle$ $O_4 = \langle b_{16} \rangle,$$O_5 = \langle b_{15} \rangle,$ $O_6 = \langle b_{14} \rangle,$ $O_7 = \langle b_{13} \rangle$, $O_8 = \langle c_{0} \rangle$, $O_{9} = \langle c_{1} \rangle$, $O_{10} = \langle c_{2} \rangle$
 are the $G_{15}$-orbits, it follows that $O_1 = \langle a_0 \rangle,$ $O_2 = \langle b_0 \rangle,$  $O_3 = \langle b_1 \rangle$ $O_4 = \langle b_{16} \rangle,$$O_5 = \langle b_{15} \rangle,$ $O_6 = \langle b_{14} \rangle,$ $O_7 = \langle b_{13} \rangle$, $O_8 = \langle c_{0} \rangle$, $O_{9} = \langle c_{1} \rangle$, $O_{10} = \langle c_{2} \rangle$ are the $(G_{15}/\Gamma_{15})$-orbits.  
Since the vertex set of $X_{15}$ is $\sqcup_{j=1}^{10} \eta_{15}(O_j)$ and $G_{15}/\Gamma_{15} \leq {\rm Aut}(X_{15})$, it follows that the number of ${\rm Aut}(X_{15})$-orbits of vertices is $\le10$.  

\medskip

Let $X_{16} = K_{16}/\Gamma_{16}$ be a semiequivelar map on the torus for some fixed element (vertex, edge or face) free subgroup $\Gamma_{16} \le {\rm Aut}(K_{16})$. Let the vertices of $X_{16}$ form $m_{16}$ ${\rm Aut}(X_{16})$-orbits.
We take the middle point of the line segment joining vertices $b_0$ and $b_{3}$ as the origin $(0,0)$ of $K_{16}$  (see Section \ref{3uniform}). Let  $\alpha_{16} := a_{1} - a_0$, $\beta_{16} := a_{2} - a_{0}$ and $\gamma_{16} := a_{3} - a_{0}\in \mathbb{R}^2$. Similarly as above, define $H_{16} := \langle x\mapsto x+\alpha_{16}, x\mapsto x+\beta_{16}, x\mapsto x+\gamma_{16}\rangle$ and 
\begin{align*}
 G_{16} & =\{ \alpha : x\mapsto \varepsilon x + m_{16}\alpha_{16} + n_{16}\beta_{16} +r_{16}\gamma_{16}   \, : \, \varepsilon=\pm 1, m_{16}, n_{16}, r_{16} \in \ZZ\} \cong H_{16}\rtimes \mathbb{Z}_2
\end{align*}
By the same arguments as above and in Claim 1, $\Gamma_{16} \unlhd G_{16}$ and the number of $G_{16}/\Gamma_{16}$-orbits of vertices of $X_{16}$ is eleven. Therefore, $G_{16}/\Gamma_{16}$ acts on $X_{16} = K_{16}/\Gamma_{16}$.
Since  $O_1 = \langle a_0 \rangle,$ $O_2 = \langle b_0 \rangle,$  $O_3 = \langle b_1 \rangle$ $O_4 = \langle b_{2} \rangle,$$O_5 = \langle b_{3} \rangle,$ $O_6 = \langle c_{0} \rangle,$ $O_7 = \langle c_{1} \rangle$, $O_8 = \langle c_{2} \rangle$, $O_{9} = \langle c_{3} \rangle$, $O_{10} = \langle c_{4} \rangle$, $O_{11} = \langle c_{5} \rangle$
 are the $G_{16}$-orbits, it follows that $O_1 = \langle a_0 \rangle,$ $O_2 = \langle b_0 \rangle,$  $O_3 = \langle b_1 \rangle$ $O_4 = \langle b_{2} \rangle,$$O_5 = \langle b_{3} \rangle,$ $O_6 = \langle c_{0} \rangle,$ $O_7 = \langle c_{1} \rangle$, $O_8 = \langle c_{2} \rangle$, $O_{9} = \langle c_{3} \rangle$, $O_{10} = \langle c_{4} \rangle$, $O_{11} = \langle c_{5} \rangle$ are the $(G_{16}/\Gamma_{16})$-orbits.  
Since the vertex set of $X_{16}$ is $\sqcup_{j=1}^{11} \eta_{16}(O_j)$ and $G_{16}/\Gamma_{16} \leq {\rm Aut}(X_{16})$, it follows that the number of ${\rm Aut}(X_{16})$-orbits of vertices is $\le11$.  

 \medskip

Let $X_{17} = K_{17}/\Gamma_{17}$ be a semiequivelar map on the torus for some fixed element (vertex, edge or face) free subgroup $\Gamma_{17} \le {\rm Aut}(K_{17})$. Let the vertices of $X_{17}$ form $m_{17}$ ${\rm Aut}(X_{17})$-orbits.
We take the middle point of the line segment joining vertices $a_0$ and $a_{3}$ as the origin $(0,0)$ of $K_{17}$  (see Section \ref{3uniform}). Let  $\alpha_{17} := a_{6} - a_0$, $\beta_{17} := a_{15} - a_{0}$ and $\gamma_{17} := a_{21} - a_{0}\in \mathbb{R}^2$. Similarly as above, define $H_{17} := \langle x\mapsto x+\alpha_{17}, x\mapsto x+\beta_{17}, x\mapsto x+\gamma_{17}\rangle$ and 
\begin{align*}
 G_{17} & =\{ \alpha : x\mapsto \varepsilon x + m_{17}\alpha_{17} + n_{17}\beta_{17} +r_{17}\gamma_{17}   \, : \, \varepsilon=\pm 1, m_{17}, n_{17}, r_{17} \in \ZZ\} \cong H_{17}\rtimes \mathbb{Z}_2
\end{align*}
By the same arguments as above and in Claim 1, $\Gamma_{17} \unlhd G_{17}$ and the number of $G_{17}/\Gamma_{17}$-orbits of vertices of $X_{17}$ is ten. Therefore, $G_{17}/\Gamma_{17}$ acts on $X_{17} = K_{17}/\Gamma_{17}$.
Since  $O_1 = \langle a_0 \rangle,$ $O_2 = \langle a_1 \rangle,$  $O_3 = \langle a_2 \rangle$ $O_4 = \langle b_{0} \rangle,$$O_5 = \langle c_{0} \rangle,$ $O_6 = \langle c_{1} \rangle,$ $O_7 = \langle c_{15} \rangle$, $O_8 = \langle c_{16} \rangle$, $O_{9} = \langle c_{17} \rangle$, $O_{10} = \langle c_{18} \rangle$
 are the $G_{17}$-orbits, it follows that $O_1 = \langle a_0 \rangle,$ $O_2 = \langle a_1 \rangle,$  $O_3 = \langle a_2 \rangle$ $O_4 = \langle b_{0} \rangle,$$O_5 = \langle c_{0} \rangle,$ $O_6 = \langle c_{1} \rangle,$ $O_7 = \langle c_{15} \rangle$, $O_8 = \langle c_{16} \rangle$, $O_{9} = \langle c_{17} \rangle$, $O_{10} = \langle c_{18} \rangle$ are the $(G_{17}/\Gamma_{17})$-orbits.  
Since the vertex set of $X_{17}$ is $\sqcup_{j=1}^{10} \eta_{17}(O_j)$ and $G_{17}/\Gamma_{17} \leq {\rm Aut}(X_{17})$, it follows that the number of ${\rm Aut}(X_{17})$-orbits of vertices is $\le10$.  
 
  \medskip

Let $X_{18} = K_{18}/\Gamma_{18}$ be a semiequivelar map on the torus for some fixed element (vertex, edge or face) free subgroup $\Gamma_{18} \le {\rm Aut}(K_{18})$. Let the vertices of $X_{18}$ form $m_{18}$ ${\rm Aut}(X_{18})$-orbits.
We take the middle point of the line segment joining vertices $a_0$ and $c_{2}$ as the origin $(0,0)$ of $K_{18}$  (see Section \ref{3uniform}). Let  $\alpha_{18} := a_{11} - a_0$, $\beta_{18} := a_{18} - a_{0}$ and $\gamma_{18} := a_{31} - a_{0}\in \mathbb{R}^2$. Similarly as above, define $H_{18} := \langle x\mapsto x+\alpha_{18}, x\mapsto x+\beta_{18}, x\mapsto x+\gamma_{18}\rangle$ and 
\begin{align*}
 G_{18} & =\{ \alpha : x\mapsto \varepsilon x + m_{18}\alpha_{18} + n_{18}\beta_{18} +r_{18}\gamma_{18}   \, : \, \varepsilon=\pm 1, m_{18}, n_{18}, r_{18} \in \ZZ\} \cong H_{18}\rtimes \mathbb{Z}_2
\end{align*}
By the same arguments as above and in Claim 1, $\Gamma_{18} \unlhd G_{18}$ and the number of $G_{18}/\Gamma_{18}$-orbits of vertices of $X_{18}$ is thirteen. Therefore, $G_{18}/\Gamma_{18}$ acts on $X_{18} = K_{18}/\Gamma_{18}$.
Since  $O_1 = \langle a_0 \rangle,$ $O_2 = \langle a_1 \rangle,$  $O_3 = \langle a_2 \rangle$ $O_4 = \langle a_3 \rangle,$$O_5 = \langle a_{4} \rangle,$ $O_6 = \langle a_{5} \rangle,$ $O_7 = \langle b_{0} \rangle$, $O_8 = \langle c_{0} \rangle$, $O_{9} = \langle c_{1} \rangle$, $O_{10} = \langle c_{2} \rangle$, $O_{11} = \langle c_{3} \rangle$, $O_{12} = \langle c_{4} \rangle$, $O_{13} = \langle c_{5} \rangle$
 are the $G_{18}$-orbits, it follows that $O_1 = \langle a_0 \rangle,$ $O_2 = \langle a_1 \rangle,$  $O_3 = \langle a_2 \rangle$ $O_4 = \langle a_3 \rangle,$$O_5 = \langle a_{4} \rangle,$ $O_6 = \langle a_{5} \rangle,$ $O_7 = \langle b_{0} \rangle$, $O_8 = \langle c_{0} \rangle$, $O_{9} = \langle c_{1} \rangle$, $O_{10} = \langle c_{2} \rangle$, $O_{11} = \langle c_{3} \rangle$, $O_{12} = \langle c_{4} \rangle$, $O_{13} = \langle c_{5} \rangle$ are the $(G_{18}/\Gamma_{18})$-orbits.  
Since the vertex set of $X_{18}$ is $\sqcup_{j=1}^{13} \eta_{18}(O_j)$ and $G_{18}/\Gamma_{18} \leq {\rm Aut}(X_{18})$, it follows that the number of ${\rm Aut}(X_{18})$-orbits of vertices is $\le13$.
 
  \medskip

Let $X_{19} = K_{19}/\Gamma_{19}$ be a semiequivelar map on the torus for some fixed element (vertex, edge or face) free subgroup $\Gamma_{19} \le {\rm Aut}(K_{19})$. Let the vertices of $X_{18}$ form $m_{19}$ ${\rm Aut}(X_{19})$-orbits.
We take the point $b_0$ as the origin $(0,0)$ of $K_{19}$  (see Section \ref{3uniform}). Let  $\alpha_{19} := b_{1} - b_0$, $\beta_{19} := b_{2} - b_{0}$ and $\gamma_{19} := b_{3} - b_{0}\in \mathbb{R}^2$. Similarly as above, define $H_{19} := \langle x\mapsto x+\alpha_{19}, x\mapsto x+\beta_{19}, x\mapsto x+\gamma_{19}\rangle$ and 
\begin{align*}
 G_{19} & =\{ \alpha : x\mapsto \varepsilon x + m_{19}\alpha_{19} + n_{19}\beta_{19} +r_{19}\gamma_{19}   \, : \, \varepsilon=\pm 1, m_{19}, n_{19}, r_{19} \in \ZZ\} \cong H_{19}\rtimes \mathbb{Z}_2
\end{align*}
By the same arguments as above and in Claim 1, $\Gamma_{19} \unlhd G_{19}$ and the number of $G_{19}/\Gamma_{19}$-orbits of vertices of $X_{19}$ is five. Therefore, $G_{19}/\Gamma_{19}$ acts on $X_{19} = K_{19}/\Gamma_{19}$.
Since  $O_1 = \langle a_0 \rangle,$ $O_2 = \langle a_1 \rangle,$  $O_3 = \langle a_5 \rangle$ $O_4 = \langle b_{0} \rangle,$$O_5 = \langle c_{0} \rangle$
 are the $G_{19}$-orbits, it follows that $O_1 = \langle a_0 \rangle,$ $O_2 = \langle a_1 \rangle,$  $O_3 = \langle a_5 \rangle$ $O_4 = \langle b_{0} \rangle,$$O_5 = \langle c_{0} \rangle$ are the $(G_{19}/\Gamma_{19})$-orbits.  
Since the vertex set of $X_{19}$ is $\sqcup_{j=1}^{5} \eta_{19}(O_j)$ and $G_{19}/\Gamma_{19} \leq {\rm Aut}(X_{19})$, it follows that the number of ${\rm Aut}(X_{19})$-orbits of vertices is $\le5$.

 \medskip

Let $X_{20} = K_{20}/\Gamma_{20}$ be a semiequivelar map on the torus for some fixed element (vertex, edge or face) free subgroup $\Gamma_{20} \le {\rm Aut}(K_{20})$. Let the vertices of $X_{20}$ form $m_{20}$ ${\rm Aut}(X_{20})$-orbits.
We take the middle point of the line segment joining vertices $a_0$ and $a_{3}$ as the origin $(0,0)$ of $K_{20}$  (see Section \ref{3uniform}). Let  $\alpha_{20} := a_{6} - a_0$, $\beta_{20} := a_{12} - a_{0}$ and $\gamma_{20} := a_{18} - a_{0}\in \mathbb{R}^2$. Similarly as above, define $H_{20} := \langle x\mapsto x+\alpha_{20}, x\mapsto x+\beta_{20}, x\mapsto x+\gamma_{20}\rangle$ and 
\begin{align*}
 G_{20} & =\{ \alpha : x\mapsto \varepsilon x + m_{20}\alpha_{20} + n_{20}\beta_{20} +r_{20}\gamma_{20}   \, : \, \varepsilon=\pm 1, m_{20}, n_{20}, r_{20} \in \ZZ\} \cong H_{20}\rtimes \mathbb{Z}_2
\end{align*}
By the same arguments as above and in Claim 1, $\Gamma_{20} \unlhd G_{20}$ and the number of $G_{20}/\Gamma_{20}$-orbits of vertices of $X_{20}$ is nine. Therefore, $G_{20}/\Gamma_{20}$ acts on $X_{20} = K_{20}/\Gamma_{20}$.
Since  $O_1 = \langle a_0 \rangle,$ $O_2 = \langle a_1 \rangle,$  $O_3 = \langle a_2 \rangle$ $O_4 = \langle b_{4} \rangle,$$O_5 = \langle b_{6} \rangle,$ $O_6 = \langle b_{8} \rangle,$ $O_7 = \langle c_{3} \rangle$, $O_8 = \langle c_{5} \rangle$, $O_{9} = \langle c_{6} \rangle$
 are the $G_{20}$-orbits, it follows that $O_1 = \langle a_0 \rangle,$ $O_2 = \langle a_1 \rangle,$  $O_3 = \langle a_2 \rangle$ $O_4 = \langle b_{4} \rangle,$$O_5 = \langle b_{6} \rangle,$ $O_6 = \langle b_{8} \rangle,$ $O_7 = \langle c_{3} \rangle$, $O_8 = \langle c_{5} \rangle$, $O_{9} = \langle c_{6} \rangle$ are the $(G_{20}/\Gamma_{20})$-orbits.  
Since the vertex set of $X_{20}$ is $\sqcup_{j=1}^{9} \eta_{20}(O_j)$ and $G_{20}/\Gamma_{20} \leq {\rm Aut}(X_{20})$, it follows that the number of ${\rm Aut}(X_{20})$-orbits of vertices is $\le9$.  

\medskip

Let $X_{21} = K_{21}/\Gamma_{21}$ be a semiequivelar map on the torus for some fixed element (vertex, edge or face) free subgroup $\Gamma_{21} \le {\rm Aut}(K_{21})$. Let the vertices of $X_{21}$ form $m_{21}$ ${\rm Aut}(X_{21})$-orbits.
We take the middle point of the line segment joining vertices $a_0$ and $a_1$ as the origin $(0,0)$ of $K_{21}$  (see Section \ref{3uniform}). Let  $\alpha_{21} := a_{2} - a_0$ and $\beta_{21} := a_{8} - a_{0} \in \mathbb{R}^2$. Similarly as above, define $H_{21} := \langle x\mapsto x+\alpha_{21}, x\mapsto x+\beta_{21}\rangle$ and 
\begin{align*}
  G_{21} & =\{ \alpha : x\mapsto \varepsilon x + m_{21}\alpha_{21} + n_{21}\beta_{21}    \, : \, \varepsilon=\pm 1, m_{21}, n_{21} \in \ZZ\} \cong H_{21}\rtimes \mathbb{Z}_2.
\end{align*}
By the same arguments as above and in Claim 1, $\Gamma_{21} \unlhd G_{21}$ and the number of $G_{21}/\Gamma_{21}$-orbits of vertices of $X_{21}$ is four. Therefore, $G_{21}/\Gamma_{21}$ acts on $X_{21} = K_{21}/\Gamma_{21}$.
Since  $O_1 = \langle a_0 \rangle,$ $O_2 = \langle c_0 \rangle,$  $O_3 = \langle c_1 \rangle$ $O_4 = \langle b_0 \rangle$
 are the $G_{21}$-orbits, it follows that $O_1 = \langle a_0 \rangle,$ $O_2 = \langle c_0 \rangle,$  $O_3 = \langle c_1 \rangle$ $O_4 = \langle b_0 \rangle$ are the $(G_{21}/\Gamma_{21})$-orbits.  
Since the vertex set of $X_{21}$ is $\sqcup_{j=1}^{4} \eta_{21}(O_j)$ and $G_{21}/\Gamma_{21} \leq {\rm Aut}(X_{21})$, it follows that the number of ${\rm Aut}(X_{21})$-orbits of vertices is $\le4$. 

\medskip

Let $X_{22} = K_{22}/\Gamma_{22}$ be a semiequivelar map on the torus for some fixed element (vertex, edge or face) free subgroup $\Gamma_{22}\le {\rm Aut}(K_{22})$. Let the vertices of $X_{22}$ form $m_{22}$ ${\rm Aut}(X_{22})$-orbits.
We take the middle point of the line segment joining vertices $a_0$ and $b_1$ as the origin $(0,0)$ of $K_{22}$  (see Section \ref{3uniform}).Let  $\alpha_{22} := a_{1} - a_0$ and $\beta_{22} := a_{2} - a_{0} \in \mathbb{R}^2$. Similarly as above, define $H_{22} := \langle x\mapsto x+\alpha_{22}, x\mapsto x+\beta_{22}\rangle$ and 
\begin{align*}
  G_{22} & =\{ \alpha : x\mapsto \varepsilon x + m_{22}\alpha_{22} + n_{22}\beta_{22}    \, : \, \varepsilon=\pm 1, m_{22}, n_{22} \in \ZZ\} \cong H_{22}\rtimes \mathbb{Z}_2.
\end{align*}
By the same arguments as above and in Claim 1, $\Gamma_{22} \unlhd G_{22}$ and the number of $G_{22}/\Gamma_{22}$-orbits of vertices of $X_{22}$ is six. Therefore, $G_{22}/\Gamma_{22}$ acts on $X_{22} = K_{22}/\Gamma_{22}$.
Since  $O_1 = \langle a_0 \rangle,$ $O_2 = \langle c_0 \rangle,$  $O_3 = \langle {b_0} \rangle$ $O_4 = \langle b_{1} \rangle,$$O_5 = \langle c_8 \rangle,$ $O_6 = \langle a_{10} \rangle$
 are the $G_{22}$-orbits, it follows that $O_1 = \langle a_0 \rangle,$ $O_2 = \langle c_0 \rangle,$  $O_3 = \langle {b_0} \rangle$ $O_4 = \langle b_{1} \rangle,$$O_5 = \langle c_8 \rangle,$ $O_6 = \langle a_{10} \rangle$ are the $(G_{22}/\Gamma_{22})$-orbits.  
Since the vertex set of $X_{22}$ is $\sqcup_{j=1}^{6} \eta_{22}(O_j)$ and $G_{22}/\Gamma_{22} \leq {\rm Aut}(X_{22})$, it follows that the number of ${\rm Aut}(X_{22})$-orbits of vertices is $\le6$.  
 
\medskip

Let $X_{23} = K_{23}/\Gamma_{23}$ be a semiequivelar map on the torus for some fixed element (vertex, edge or face) free subgroup $\Gamma_{23}\le {\rm Aut}(K_{23})$. Let the vertices of $X_{23}$ form $m_{23}$ ${\rm Aut}(X_{23})$-orbits.
We take the point $b_0$ as the origin $(0,0)$ of $K_{23}$  (see Section \ref{3uniform}).Let  $\alpha_{23} := b_{1} - b_0$ and $\beta_{23} := b_{2} - b_{0} \in \mathbb{R}^2$. Similarly as above, define $H_{23} := \langle x\mapsto x+\alpha_{23}, x\mapsto x+\beta_{23}\rangle$ and 
\begin{align*}
  G_{23} & =\{ \alpha : x\mapsto \varepsilon x + m_{23}\alpha_{23} + n_{23}\beta_{23}    \, : \, \varepsilon=\pm 1, m_{23}, n_{23} \in \ZZ\} \cong H_{23}\rtimes \mathbb{Z}_2.
\end{align*}
By the same arguments as above and in Claim 1, $\Gamma_{23} \unlhd G_{23}$ and the number of $G_{23}/\Gamma_{23}$-orbits of vertices of $X_{23}$ is five. Therefore, $G_{23}/\Gamma_{23}$ acts on $X_{23} = K_{23}/\Gamma_{23}$.
Since  $O_1 = \langle a_0 \rangle,$ $O_2 = \langle a_1 \rangle,$  $O_3 = \langle {b_0} \rangle$ $O_4 = \langle c_{0} \rangle,$$O_5 = \langle c_1 \rangle$
 are the $G_{23}$-orbits, it follows that $O_1 = \langle a_0 \rangle,$ $O_2 = \langle a_1 \rangle,$  $O_3 = \langle {b_0} \rangle$ $O_4 = \langle c_{0} \rangle,$$O_5 = \langle c_1 \rangle$ are the $(G_{23}/\Gamma_{23})$-orbits.  
Since the vertex set of $X_{23}$ is $\sqcup_{j=1}^{5} \eta_{23}(O_j)$ and $G_{23}/\Gamma_{23} \leq {\rm Aut}(X_{23})$, it follows that the number of ${\rm Aut}(X_{23})$-orbits of vertices is $\le5$.  

 \medskip

Let $X_{24} = K_{24}/\Gamma_{24}$ be a semiequivelar map on the torus for some fixed element (vertex, edge or face) free subgroup $\Gamma_{24} \le {\rm Aut}(K_{24})$. Let the vertices of $X_{24}$ form $m_{24}$ ${\rm Aut}(X_{24})$-orbits.
We take the middle point of the line segment joining vertices $a_0$ and $a_{6}$ as the origin $(0,0)$ of $K_{24}$  (see Section \ref{3uniform}). Let  $\alpha_{24} := a_{12} - a_0$, $\beta_{24} := a_{30} - a_{0}$ and $\gamma_{24} := a_{43} - a_{0}\in \mathbb{R}^2$. Similarly as above, define $H_{24} := \langle x\mapsto x+\alpha_{24}, x\mapsto x+\beta_{24}, x\mapsto x+\gamma_{24}\rangle$ and 
\begin{align*}
 G_{24} & =\{ \alpha : x\mapsto \varepsilon x + m_{24}\alpha_{24} + n_{24}\beta_{24} +r_{24}\gamma_{24}   \, : \, \varepsilon=\pm 1, m_{24}, n_{24}, r_{24} \in \ZZ\} \cong H_{24}\rtimes \mathbb{Z}_2
\end{align*}
By the same arguments as above and in Claim 1, $\Gamma_{24} \unlhd G_{24}$ and the number of $G_{24}/\Gamma_{24}$-orbits of vertices of $X_{24}$ is fifteen. Therefore, $G_{24}/\Gamma_{24}$ acts on $X_{24} = K_{24}/\Gamma_{24}$.
Since  $O_1 = \langle a_0 \rangle,$ $O_2 = \langle a_1 \rangle,$  $O_3 = \langle a_2 \rangle$ $O_4 = \langle a_3 \rangle,$$O_5 = \langle a_{4} \rangle,$ $O_6 = \langle a_{5} \rangle,$ $O_7 = \langle b_{0} \rangle$, $O_8 = \langle b_{2} \rangle$, $O_{9} = \langle b_{4} \rangle$, $O_{10} = \langle c_{0} \rangle$, $O_{11} = \langle c_{1} \rangle$, $O_{12} = \langle c_{2} \rangle$, $O_{13} = \langle c_{3} \rangle$, $O_{14} = \langle c_{4} \rangle$, $O_{15} = \langle c_{5} \rangle$
 are the $G_{24}$-orbits, it follows that $O_1 = \langle a_0 \rangle,$ $O_2 = \langle a_1 \rangle,$  $O_3 = \langle a_2 \rangle$ $O_4 = \langle a_3 \rangle,$$O_5 = \langle a_{4} \rangle,$ $O_6 = \langle a_{5} \rangle,$ $O_7 = \langle b_{0} \rangle$, $O_8 = \langle b_{2} \rangle$, $O_{9} = \langle b_{4} \rangle$, $O_{10} = \langle c_{0} \rangle$, $O_{11} = \langle c_{1} \rangle$, $O_{12} = \langle c_{2} \rangle$, $O_{13} = \langle c_{3} \rangle$, $O_{14} = \langle c_{4} \rangle$, $O_{15} = \langle c_{5} \rangle$ are the $(G_{24}/\Gamma_{24})$-orbits.  
Since the vertex set of $X_{24}$ is $\sqcup_{j=1}^{15} \eta_{24}(O_j)$ and $G_{24}/\Gamma_{24} \leq {\rm Aut}(X_{24})$, it follows that the number of ${\rm Aut}(X_{24})$-orbits of vertices is $\le15$.
 
\medskip

Let $X_{25} = K_{25}/\Gamma_{25}$ be a semiequivelar map on the torus for some fixed element (vertex, edge or face) free subgroup $\Gamma_{25}\le {\rm Aut}(K_{25})$. Let the vertices of $X_{25}$ form $m_{25}$ ${\rm Aut}(X_{25})$-orbits.
We take the middle point of the line segment joining vertices $b_0$ and $b_1$ as the origin $(0,0)$ of $K_{25}$  (see Section \ref{3uniform}).Let  $\alpha_{25} := b_{2} - b_0$ and $\beta_{25} := b_{5} - b_{0} \in \mathbb{R}^2$. Similarly as above, define $H_{25} := \langle x\mapsto x+\alpha_{25}, x\mapsto x+\beta_{25}\rangle$ and 
\begin{align*}
  G_{25} & =\{ \alpha : x\mapsto \varepsilon x + m_{25}\alpha_{25} + n_{25}\beta_{25}    \, : \, \varepsilon=\pm 1, m_{25}, n_{25} \in \ZZ\} \cong H_{25}\rtimes \mathbb{Z}_2.
\end{align*}
By the same arguments as above and in Claim 1, $\Gamma_{25} \unlhd G_{25}$ and the number of $G_{25}/\Gamma_{25}$-orbits of vertices of $X_{25}$ is five. Therefore, $G_{25}/\Gamma_{25}$ acts on $X_{25} = K_{25}/\Gamma_{25}$.
Since  $O_1 = \langle a_0 \rangle,$ $O_2 = \langle a_1 \rangle,$  $O_3 = \langle {b_0} \rangle$ $O_4 = \langle c_{0} \rangle,$$O_5 = \langle c_2 \rangle$
 are the $G_{25}$-orbits, it follows that $O_1 = \langle a_0 \rangle,$ $O_2 = \langle a_1 \rangle,$  $O_3 = \langle {b_0} \rangle$ $O_4 = \langle c_{0} \rangle,$$O_5 = \langle c_2 \rangle$ are the $(G_{25}/\Gamma_{25})$-orbits.  
Since the vertex set of $X_{25}$ is $\sqcup_{j=1}^{5} \eta_{25}(O_j)$ and $G_{25}/\Gamma_{25} \leq {\rm Aut}(X_{25})$, it follows that the number of ${\rm Aut}(X_{25})$-orbits of vertices is $\le5$.  
 
\medskip

Let $X_{26} = K_{26}/\Gamma_{26}$ be a semiequivelar map on the torus for some fixed element (vertex, edge or face) free subgroup $\Gamma_{26}\le {\rm Aut}(K_{26})$. Let the vertices of $X_{26}$ form $m_{26}$ ${\rm Aut}(X_{26})$-orbits.
We take the middle point of the line segment joining vertices $b_0$ and $b_1$ as the origin $(0,0)$ of $K_{26}$  (see Section \ref{3uniform}).Let  $\alpha_{26} := b_{3} - b_0$ and $\beta_{26} := b_{12} - b_{0} \in \mathbb{R}^2$. Similarly as above, define $H_{26} := \langle x\mapsto x+\alpha_{26}, x\mapsto x+\beta_{26}\rangle$ and 
\begin{align*}
  G_{26} & =\{ \alpha : x\mapsto \varepsilon x + m_{26}\alpha_{26} + n_{26}\beta_{26}    \, : \, \varepsilon=\pm 1, m_{26}, n_{26} \in \ZZ\} \cong H_{26}\rtimes \mathbb{Z}_2.
\end{align*}
By the same arguments as above and in Claim 1, $\Gamma_{26} \unlhd G_{26}$ and the number of $G_{26}/\Gamma_{26}$-orbits of vertices of $X_{26}$ is four. Therefore, $G_{26}/\Gamma_{26}$ acts on $X_{26} = K_{26}/\Gamma_{26}$.
Since  $O_1 = \langle a_0 \rangle,$ $O_2 = \langle b_0 \rangle,$  $O_3 = \langle c_{10} \rangle$ $O_4 = \langle c_{11} \rangle$
 are the $G_{26}$-orbits, it follows that $O_1 = \langle a_0 \rangle,$ $O_2 = \langle b_0 \rangle,$  $O_3 = \langle c_{10} \rangle$ $O_4 = \langle c_{11} \rangle$ are the $(G_{26}/\Gamma_{26})$-orbits.  
Since the vertex set of $X_{26}$ is $\sqcup_{j=1}^{4} \eta_{26}(O_j)$ and $G_{26}/\Gamma_{26} \leq {\rm Aut}(X_{26})$, it follows that the number of ${\rm Aut}(X_{26})$-orbits of vertices is $\le4$.  
  
\medskip

Let $X_{27} = K_{27}/\Gamma_{27}$ be a semiequivelar map on the torus for some fixed element (vertex, edge or face) free subgroup $\Gamma_{27}\le {\rm Aut}(K_{27})$. Let the vertices of $X_{27}$ form $m_{27}$ ${\rm Aut}(X_{27})$-orbits.
We take the middle point of the line segment joining vertices $a_0$ and $a_2$ as the origin $(0,0)$ of $K_{27}$  (see Section \ref{3uniform}).Let  $\alpha_{27} := a_{5} - a_0$ and $\beta_{27} := a_{15} - a_{0} \in \mathbb{R}^2$. Similarly as above, define $H_{27} := \langle x\mapsto x+\alpha_{27}, x\mapsto x+\beta_{27}\rangle$ and 
\begin{align*}
  G_{27} & =\{ \alpha : x\mapsto \varepsilon x + m_{27}\alpha_{27} + n_{27}\beta_{27}    \, : \, \varepsilon=\pm 1, m_{27}, n_{27} \in \ZZ\} \cong H_{27}\rtimes \mathbb{Z}_2.
\end{align*}
By the same arguments as above and in Claim 1, $\Gamma_{27} \unlhd G_{27}$ and the number of $G_{27}/\Gamma_{27}$-orbits of vertices of $X_{27}$ is five. Therefore, $G_{27}/\Gamma_{27}$ acts on $X_{27} = K_{27}/\Gamma_{27}$.
Since  $O_1 = \langle a_0 \rangle,$ $O_2 = \langle a_1 \rangle,$  $O_3 = \langle {b_0} \rangle$ $O_4 = \langle c_{0} \rangle,$$O_5 = \langle c_1 \rangle$
 are the $G_{27}$-orbits, it follows that $O_1 = \langle a_0 \rangle,$ $O_2 = \langle a_1 \rangle,$  $O_3 = \langle {b_0} \rangle$ $O_4 = \langle c_{0} \rangle,$$O_5 = \langle c_1 \rangle$ are the $(G_{27}/\Gamma_{27})$-orbits.  
Since the vertex set of $X_{27}$ is $\sqcup_{j=1}^{5} \eta_{27}(O_j)$ and $G_{27}/\Gamma_{27} \leq {\rm Aut}(X_{27})$, it follows that the number of ${\rm Aut}(X_{27})$-orbits of vertices is $\le5$.  
 
 \medskip

Let $X_{29} = K_{29}/\Gamma_{29}$ be a semiequivelar map on the torus for some fixed element (vertex, edge or face) free subgroup $\Gamma_{29}\le {\rm Aut}(K_{29})$. Let the vertices of $X_{29}$ form $m_{29}$ ${\rm Aut}(X_{29})$-orbits.
We take the middle point of the line segment joining vertices $a_0$ and $c_1$ as the origin $(0,0)$ of $K_{29}$  (see Section \ref{3uniform}).Let  $\alpha_{29} := a_{2} - a_0$ and $\beta_{29} := a_{13} - a_{0} \in \mathbb{R}^2$. Similarly as above, define $H_{29} := \langle x\mapsto x+\alpha_{29}, x\mapsto x+\beta_{29}\rangle$ and 
\begin{align*}
  G_{29} & =\{ \alpha : x\mapsto \varepsilon x + m_{29}\alpha_{29} + n_{29}\beta_{29}    \, : \, \varepsilon=\pm 1, m_{29}, n_{29} \in \ZZ\} \cong H_{29}\rtimes \mathbb{Z}_2.
\end{align*}
By the same arguments as above and in Claim 1, $\Gamma_{29} \unlhd G_{29}$ and the number of $G_{29}/\Gamma_{29}$-orbits of vertices of $X_{29}$ is nine. Therefore, $G_{29}/\Gamma_{29}$ acts on $X_{29} = K_{29}/\Gamma_{29}$.
Since $O_1 = \langle a_0 \rangle,$ $O_2 = \langle a_1 \rangle,$  $O_3 = \langle b_1 \rangle$ $O_4 = \langle b_{7} \rangle,$$O_5 = \langle b_{8} \rangle,$ $O_6 = \langle c_{0} \rangle,$ $O_7 = \langle c_{1} \rangle$, $O_8 = \langle c_{36} \rangle$, $O_{9} = \langle c_{37} \rangle$
 are the $G_{29}$-orbits, it follows that $O_1 = \langle a_0 \rangle,$ $O_2 = \langle a_1 \rangle,$  $O_3 = \langle b_1 \rangle$ $O_4 = \langle b_{7} \rangle,$$O_5 = \langle b_{8} \rangle,$ $O_6 = \langle c_{0} \rangle,$ $O_7 = \langle c_{1} \rangle$, $O_8 = \langle c_{36} \rangle$, $O_{9} = \langle c_{37} \rangle$ are the $(G_{29}/\Gamma_{29})$-orbits.  
Since the vertex set of $X_{29}$ is $\sqcup_{j=1}^{9} \eta_{29}(O_j)$ and $G_{29}/\Gamma_{29} \leq {\rm Aut}(X_{29})$, it follows that the number of ${\rm Aut}(X_{29})$-orbits of vertices is $\le9$.  
 
 \medskip

Let $X_{30} = K_{30}/\Gamma_{30}$ be a semiequivelar map on the torus for some fixed element (vertex, edge or face) free subgroup $\Gamma_{30}\le {\rm Aut}(K_{30})$. Let the vertices of $X_{30}$ form $m_{30}$ ${\rm Aut}(X_{30})$-orbits.
We take the middle point of the line segment joining vertices $a_0$ and $c_1$ as the origin $(0,0)$ of $K_{30}$  (see Section \ref{3uniform}).Let  $\alpha_{30} := a_{2} - a_0$ and $\beta_{30} := a_{38} - a_{0} \in \mathbb{R}^2$. Similarly as above, define $H_{30} := \langle x\mapsto x+\alpha_{30}, x\mapsto x+\beta_{30}\rangle$ and 
\begin{align*}
  G_{30} & =\{ \alpha : x\mapsto \varepsilon x + m_{30}\alpha_{30} + n_{30}\beta_{30}    \, : \, \varepsilon=\pm 1, m_{30}, n_{30} \in \ZZ\} \cong H_{30}\rtimes \mathbb{Z}_2.
\end{align*}
By the same arguments as above and in Claim 1, $\Gamma_{30} \unlhd G_{30}$ and the number of $G_{30}/\Gamma_{30}$-orbits of vertices of $X_{30}$ is eight. Therefore, $G_{30}/\Gamma_{30}$ acts on $X_{30} = K_{30}/\Gamma_{30}$.
Since $O_1 = \langle a_0 \rangle,$ $O_2 = \langle a_1 \rangle,$  $O_3 = \langle a_{27} \rangle$ $O_4 = \langle a_{28} \rangle,$$O_5 = \langle b_{0} \rangle,$ $O_6 = \langle b_{9} \rangle,$ $O_7 = \langle c_{0} \rangle$, $O_8 = \langle c_{1} \rangle$
 are the $G_{30}$-orbits, it follows that $O_1 = \langle a_0 \rangle,$ $O_2 = \langle a_1 \rangle,$  $O_3 = \langle a_{27} \rangle$ $O_4 = \langle a_{28} \rangle,$$O_5 = \langle b_{0} \rangle,$ $O_6 = \langle b_{9} \rangle,$ $O_7 = \langle c_{0} \rangle$, $O_8 = \langle c_{1} \rangle$ are the $(G_{30}/\Gamma_{30})$-orbits.  
Since the vertex set of $X_{30}$ is $\sqcup_{j=1}^{8} \eta_{30}(O_j)$ and $G_{30}/\Gamma_{30} \leq {\rm Aut}(X_{30})$, it follows that the number of ${\rm Aut}(X_{30})$-orbits of vertices is $\le8$.  
 
 \medskip

Let $X_{31} = K_{31}/\Gamma_{31}$ be a semiequivelar map on the torus for some fixed element (vertex, edge or face) free subgroup $\Gamma_{31}\le {\rm Aut}(K_{31})$. Let the vertices of $X_{31}$ form $m_{31}$ ${\rm Aut}(X_{31})$-orbits.
We take the middle point of the line segment joining vertices $a_0$ and $b_1$ as the origin $(0,0)$ of $K_{31}$  (see Section \ref{3uniform}).Let  $\alpha_{31} := a_{1} - a_0$ and $\beta_{31} := a_{2} - a_{0} \in \mathbb{R}^2$. Similarly as above, define $H_{31} := \langle x\mapsto x+\alpha_{31}, x\mapsto x+\beta_{31}\rangle$ and 
\begin{align*}
  G_{31} & =\{ \alpha : x\mapsto \varepsilon x + m_{31}\alpha_{31} + n_{31}\beta_{31}    \, : \, \varepsilon=\pm 1, m_{31}, n_{31} \in \ZZ\} \cong H_{31}\rtimes \mathbb{Z}_2.
\end{align*}
By the same arguments as above and in Claim 1, $\Gamma_{31} \unlhd G_{31}$ and the number of $G_{31}/\Gamma_{31}$-orbits of vertices of $X_{31}$ is six. Therefore, $G_{31}/\Gamma_{31}$ acts on $X_{31} = K_{31}/\Gamma_{31}$.
Since $O_1 = \langle a_0 \rangle,$ $O_2 = \langle c_0 \rangle,$  $O_3 = \langle b_{0} \rangle$ $O_4 = \langle b_{1} \rangle,$$O_5 = \langle c_{14} \rangle,$ $O_6 = \langle a_{1} \rangle$
 are the $G_{31}$-orbits, it follows that $O_1 = \langle a_0 \rangle,$ $O_2 = \langle c_0 \rangle,$  $O_3 = \langle b_{0} \rangle$ $O_4 = \langle b_{1} \rangle,$$O_5 = \langle c_{14} \rangle,$ $O_6 = \langle a_{1} \rangle$ are the $(G_{31}/\Gamma_{31})$-orbits.  
Since the vertex set of $X_{31}$ is $\sqcup_{j=1}^{6} \eta_{31}(O_j)$ and $G_{31}/\Gamma_{31} \leq {\rm Aut}(X_{31})$, it follows that the number of ${\rm Aut}(X_{31})$-orbits of vertices is $\le6$.  
 
  \medskip

Let $X_{32} = K_{32}/\Gamma_{32}$ be a semiequivelar map on the torus for some fixed element (vertex, edge or face) free subgroup $\Gamma_{32} \le {\rm Aut}(K_{32})$. Let the vertices of $X_{32}$ form $m_{32}$ ${\rm Aut}(X_{32})$-orbits.
We take the middle point of the line segment joining vertices $a_0$ and $a_{3}$ as the origin $(0,0)$ of $K_{32}$  (see Section \ref{3uniform}). Let  $\alpha_{32} := a_{6} - a_0$, $\beta_{32} := a_{20} - a_{0}$ and $\gamma_{32} := a_{27} - a_{0}\in \mathbb{R}^2$. Similarly as above, define $H_{32} := \langle x\mapsto x+\alpha_{32}, x\mapsto x+\beta_{32}, x\mapsto x+\gamma_{32}\rangle$ and 
\begin{align*}
 G_{32} & =\{ \alpha : x\mapsto \varepsilon x + m_{32}\alpha_{32} + n_{32}\beta_{32} +r_{32}\gamma_{32}   \, : \, \varepsilon=\pm 1, m_{32}, n_{32}, r_{32} \in \ZZ\} \cong H_{32}\rtimes \mathbb{Z}_2
\end{align*}
By the same arguments as above and in Claim 1, $\Gamma_{32} \unlhd G_{32}$ and the number of $G_{32}/\Gamma_{32}$-orbits of vertices of $X_{32}$ is twelve. Therefore, $G_{32}/\Gamma_{32}$ acts on $X_{32} = K_{32}/\Gamma_{32}$.
Since  $O_1 = \langle a_0 \rangle,$ $O_2 = \langle a_1 \rangle,$  $O_3 = \langle a_2 \rangle$ $O_4 = \langle c_0 \rangle,$$O_5 = \langle b_{0} \rangle,$ $O_6 = \langle b_{1} \rangle,$ $O_7 = \langle c_{1} \rangle$, $O_8 = \langle b_{2} \rangle$, $O_{9} = \langle b_{3} \rangle$, $O_{10} = \langle c_{2} \rangle$, $O_{11} = \langle b_{4} \rangle$, $O_{12} = \langle b_{5} \rangle$
 are the $G_{32}$-orbits, it follows that $O_1 = \langle a_0 \rangle,$ $O_2 = \langle a_1 \rangle,$  $O_3 = \langle a_2 \rangle$ $O_4 = \langle c_0 \rangle,$$O_5 = \langle b_{0} \rangle,$ $O_6 = \langle b_{1} \rangle,$ $O_7 = \langle c_{1} \rangle$, $O_8 = \langle b_{2} \rangle$, $O_{9} = \langle b_{3} \rangle$, $O_{10} = \langle c_{2} \rangle$, $O_{11} = \langle b_{4} \rangle$, $O_{12} = \langle b_{5} \rangle$ are the $(G_{32}/\Gamma_{32})$-orbits.  
Since the vertex set of $X_{32}$ is $\sqcup_{j=1}^{12} \eta_{32}(O_j)$ and $G_{32}/\Gamma_{32} \leq {\rm Aut}(X_{32})$, it follows that the number of ${\rm Aut}(X_{32})$-orbits of vertices is $\le12$.
 
\medskip

Let $X_{33} = K_{33}/\Gamma_{33}$ be a semiequivelar map on the torus for some fixed element (vertex, edge or face) free subgroup $\Gamma_{33} \le {\rm Aut}(K_{33})$. Let the vertices of $X_{33}$ form $m_{33}$ ${\rm Aut}(X_{33})$-orbits.
We take the middle point of the line segment joining vertices $a_0$ and $a_{3}$ as the origin $(0,0)$ of $K_{33}$  (see Section \ref{3uniform}). Let  $\alpha_{33} := a_{6} - a_0$, $\beta_{33} := a_{12} - a_{0}$ and $\gamma_{33} := a_{18} - a_{0}\in \mathbb{R}^2$. Similarly as above, define $H_{33} := \langle x\mapsto x+\alpha_{33}, x\mapsto x+\beta_{33}, x\mapsto x+\gamma_{33}\rangle$ and 
\begin{align*}
 G_{33} & =\{ \alpha : x\mapsto \varepsilon x + m_{33}\alpha_{33} + n_{33}\beta_{33} +r_{33}\gamma_{33}   \, : \, \varepsilon=\pm 1, m_{33}, n_{33}, r_{33} \in \ZZ\} \cong H_{32}\rtimes \mathbb{Z}_2
\end{align*}
By the same arguments as above and in Claim 1, $\Gamma_{33} \unlhd G_{33}$ and the number of $G_{33}/\Gamma_{33}$-orbits of vertices of $X_{33}$ is twelve. Therefore, $G_{32}/\Gamma_{33}$ acts on $X_{33} = K_{33}/\Gamma_{33}$.
Since $O_1 = \langle a_0 \rangle,$ $O_2 = \langle a_1 \rangle,$  $O_3 = \langle a_2 \rangle$ $O_4 = \langle b_0 \rangle,$$O_5 = \langle c_{0} \rangle,$ $O_6 = \langle c_{1} \rangle,$ $O_7 = \langle b_{2} \rangle$, $O_8 = \langle c_{2} \rangle$, $O_{9} = \langle c_{3} \rangle$, $O_{10} = \langle b_{4} \rangle$, $O_{11} = \langle c_{4} \rangle$, $O_{12} = \langle c_{5} \rangle$
 are the $G_{33}$-orbits, it follows that $O_1 = \langle a_0 \rangle,$ $O_2 = \langle a_1 \rangle,$  $O_3 = \langle a_2 \rangle$ $O_4 = \langle b_0 \rangle,$$O_5 = \langle c_{0} \rangle,$ $O_6 = \langle c_{1} \rangle,$ $O_7 = \langle b_{2} \rangle$, $O_8 = \langle c_{2} \rangle$, $O_{9} = \langle c_{3} \rangle$, $O_{10} = \langle b_{4} \rangle$, $O_{11} = \langle c_{4} \rangle$, $O_{12} = \langle c_{5} \rangle$ are the $(G_{33}/\Gamma_{33})$-orbits.  
Since the vertex set of $X_{33}$ is $\sqcup_{j=1}^{12} \eta_{33}(O_j)$ and $G_{33}/\Gamma_{33} \leq {\rm Aut}(X_{33})$, it follows that the number of ${\rm Aut}(X_{33})$-orbits of vertices is $\le12$.  
\medskip

Let $X_{35} = K_{35}/\Gamma_{35}$ be a semiequivelar map on the torus for some fixed element (vertex, edge or face) free subgroup $\Gamma_{35}\le {\rm Aut}(K_{35})$. Let the vertices of $X_{35}$ form $m_{35}$ ${\rm Aut}(X_{35})$-orbits.
We take the middle point of the line segment joining vertices $a_0$ and $a_2$ as the origin $(0,0)$ of $K_{35}$  (see Section \ref{3uniform}).Let  $\alpha_{35} := a_{12} - a_0$ and $\beta_{35} := a_{43} - a_{0} \in \mathbb{R}^2$. Similarly as above, define $H_{35} := \langle x\mapsto x+\alpha_{35}, x\mapsto x+\beta_{35}\rangle$ and 
\begin{align*}
  G_{35} & =\{ \alpha : x\mapsto \varepsilon x + m_{35}\alpha_{35} + n_{35}\beta_{35}    \, : \, \varepsilon=\pm 1, m_{35}, n_{35} \in \ZZ\} \cong H_{35}\rtimes \mathbb{Z}_2.
\end{align*}
By the same arguments as above and in Claim 1, $\Gamma_{35} \unlhd G_{35}$ and the number of $G_{35}/\Gamma_{35}$-orbits of vertices of $X_{35}$ is four. Therefore, $G_{35}/\Gamma_{35}$ acts on $X_{35} = K_{35}/\Gamma_{35}$.
Since $O_1 = \langle a_0 \rangle,$ $O_2 = \langle a_1 \rangle,$  $O_3 = \langle c_{1} \rangle$ $O_4 = \langle b_{0} \rangle$
 are the $G_{35}$-orbits, it follows that $O_1 = \langle a_0 \rangle,$ $O_2 = \langle a_1 \rangle,$  $O_3 = \langle c_{1} \rangle$ $O_4 = \langle b_{0} \rangle$ are the $(G_{35}/\Gamma_{35})$-orbits.  
Since the vertex set of $X_{35}$ is $\sqcup_{j=1}^{4} \eta_{35}(O_j)$ and $G_{35}/\Gamma_{35} \leq {\rm Aut}(X_{35})$, it follows that the number of ${\rm Aut}(X_{35})$-orbits of vertices is $\le4$.  
 
 \medskip

Let $X_{36} = K_{36}/\Gamma_{36}$ be a semiequivelar map on the torus for some fixed element (vertex, edge or face) free subgroup $\Gamma_{36}\le {\rm Aut}(K_{36})$. Let the vertices of $X_{36}$ form $m_{36}$ ${\rm Aut}(X_{36})$-orbits.
We take the middle point of the line segment joining vertices $a_0$ and $a_2$ as the origin $(0,0)$ of $K_{36}$  (see Section \ref{3uniform}).Let  $\alpha_{36} := a_{11} - a_0$ and $\beta_{36} := a_{35} - a_{0} \in \mathbb{R}^2$. Similarly as above, define $H_{36} := \langle x\mapsto x+\alpha_{36}, x\mapsto x+\beta_{36}\rangle$ and 
\begin{align*}
  G_{36} & =\{ \alpha : x\mapsto \varepsilon x + m_{36}\alpha_{36} + n_{36}\beta_{36}    \, : \, \varepsilon=\pm 1, m_{36}, n_{36} \in \ZZ\} \cong H_{36}\rtimes \mathbb{Z}_2.
\end{align*}
By the same arguments as above and in Claim 1, $\Gamma_{36} \unlhd G_{36}$ and the number of $G_{36}/\Gamma_{36}$-orbits of vertices of $X_{36}$ is five. Therefore, $G_{36}/\Gamma_{36}$ acts on $X_{36} = K_{36}/\Gamma_{36}$.
Since $O_1 = \langle a_0 \rangle,$ $O_2 = \langle a_1 \rangle,$  $O_3 = \langle c_{1} \rangle$ $O_4 = \langle b_{0} \rangle,$$O_5 = \langle b_{1} \rangle$
 are the $G_{36}$-orbits, it follows that $O_1 = \langle a_0 \rangle,$ $O_2 = \langle a_1 \rangle,$  $O_3 = \langle c_{1} \rangle$ $O_4 = \langle b_{0} \rangle,$$O_5 = \langle b_{1} \rangle$ are the $(G_{36}/\Gamma_{36})$-orbits.  
Since the vertex set of $X_{36}$ is $\sqcup_{j=1}^{5} \eta_{36}(O_j)$ and $G_{36}/\Gamma_{36} \leq {\rm Aut}(X_{36})$, it follows that the number of ${\rm Aut}(X_{36})$-orbits of vertices is $\le5$.  
 
 \medskip

Let $X_{37} = K_{37}/\Gamma_{37}$ be a semiequivelar map on the torus for some fixed element (vertex, edge or face) free subgroup $\Gamma_{37}\le {\rm Aut}(K_{37})$. Let the vertices of $X_{37}$ form $m_{37}$ ${\rm Aut}(X_{37})$-orbits.
We take the middle point of the line segment joining vertices $a_0$ and $a_2$ as the origin $(0,0)$ of $K_{37}$  (see Section \ref{3uniform}).Let  $\alpha_{37} := a_{12} - a_0$ and $\beta_{37} := a_{42} - a_{0} \in \mathbb{R}^2$. Similarly as above, define $H_{37} := \langle x\mapsto x+\alpha_{37}, x\mapsto x+\beta_{37}\rangle$ and 
\begin{align*}
  G_{37} & =\{ \alpha : x\mapsto \varepsilon x + m_{37}\alpha_{37} + n_{37}\beta_{37}    \, : \, \varepsilon=\pm 1, m_{37}, n_{37} \in \ZZ\} \cong H_{37}\rtimes \mathbb{Z}_2.
\end{align*}
By the same arguments as above and in Claim 1, $\Gamma_{37} \unlhd G_{37}$ and the number of $G_{37}/\Gamma_{37}$-orbits of vertices of $X_{37}$ is five. Therefore, $G_{37}/\Gamma_{37}$ acts on $X_{37} = K_{37}/\Gamma_{37}$.
Since $O_1 = \langle a_0 \rangle,$ $O_2 = \langle a_1 \rangle,$  $O_3 = \langle c_{1} \rangle$, $O_4 = \langle b_{0} \rangle,$$O_5 = \langle b_{1} \rangle$
 are the $G_{37}$-orbits, it follows that $O_1 = \langle a_0 \rangle,$ $O_2 = \langle a_1 \rangle,$  $O_3 = \langle c_{1} \rangle$ $O_4 = \langle b_{0} \rangle,$$O_5 = \langle b_{1} \rangle$ are the $(G_{37}/\Gamma_{37})$-orbits.  
Since the vertex set of $X_{37}$ is $\sqcup_{j=1}^{5} \eta_{37}(O_j)$ and $G_{37}/\Gamma_{37} \leq {\rm Aut}(X_{37})$, it follows that the number of ${\rm Aut}(X_{37})$-orbits of vertices is $\le5$.  
 
 \medskip

Let $X_{38} = K_{38}/\Gamma_{38}$ be a semiequivelar map on the torus for some fixed element (vertex, edge or face) free subgroup $\Gamma_{38}\le {\rm Aut}(K_{38})$. Let the vertices of $X_{38}$ form $m_{38}$ ${\rm Aut}(X_{38})$-orbits.
We take the middle point of the line segment joining vertices $a_0$ and $a_2$ as the origin $(0,0)$ of $K_{38}$  (see Section \ref{3uniform}).Let  $\alpha_{38} := a_{10} - a_0$ and $\beta_{38} := a_{35} - a_{0} \in \mathbb{R}^2$. Similarly as above, define $H_{38} := \langle x\mapsto x+\alpha_{38}, x\mapsto x+\beta_{38}\rangle$ and 
\begin{align*}
  G_{38} & =\{ \alpha : x\mapsto \varepsilon x + m_{38}\alpha_{38} + n_{38}\beta_{38}    \, : \, \varepsilon=\pm 1, m_{38}, n_{38} \in \ZZ\} \cong H_{38}\rtimes \mathbb{Z}_2.
\end{align*}
By the same arguments as above and in Claim 1, $\Gamma_{38} \unlhd G_{38}$ and the number of $G_{38}/\Gamma_{38}$-orbits of vertices of $X_{38}$ is five. Therefore, $G_{38}/\Gamma_{38}$ acts on $X_{38} = K_{38}/\Gamma_{38}$.
Since $O_1 = \langle a_0 \rangle,$ $O_2 = \langle b_0 \rangle,$  $O_3 = \langle b_{23} \rangle$, $O_4 = \langle a_{1} \rangle,$ $O_5 = \langle c_{1} \rangle$
 are the $G_{38}$-orbits, it follows that $O_1 = \langle a_0 \rangle,$ $O_2 = \langle b_0 \rangle,$  $O_3 = \langle b_{23} \rangle$ $O_4 = \langle a_{1} \rangle,$ $O_5 = \langle c_{1} \rangle$ are the $(G_{38}/\Gamma_{38})$-orbits.  
Since the vertex set of $X_{38}$ is $\sqcup_{j=1}^{5} \eta_{38}(O_j)$ and $G_{38}/\Gamma_{38} \leq {\rm Aut}(X_{38})$, it follows that the number of ${\rm Aut}(X_{38})$-orbits of vertices is $\le5$.  
 
 \medskip

Let $X_{39} = K_{39}/\Gamma_{39}$ be a semiequivelar map on the torus for some fixed element (vertex, edge or face) free subgroup $\Gamma_{39} \le {\rm Aut}(K_{39})$. Let the vertices of $X_{39}$ form $m_{39}$ ${\rm Aut}(X_{39})$-orbits.
We take the middle point of the line segment joining vertices $a_0$ and $a_{6}$ as the origin $(0,0)$ of $K_{39}$  (see Section \ref{3uniform}). Let  $\alpha_{39} := a_{12} - a_0$, $\beta_{39} := a_{36} - a_{0}$ and $\gamma_{39} := a_{52} - a_{0}\in \mathbb{R}^2$. Similarly as above, define $H_{39} := \langle x\mapsto x+\alpha_{39}, x\mapsto x+\beta_{39}, x\mapsto x+\gamma_{39}\rangle$ and 
\begin{align*}
 G_{39} & =\{ \alpha : x\mapsto \varepsilon x + m_{39}\alpha_{39} + n_{39}\beta_{39} +r_{39}\gamma_{39}   \, : \, \varepsilon=\pm 1, m_{39}, n_{39}, r_{39} \in \ZZ\} \cong H_{39}\rtimes \mathbb{Z}_2
\end{align*}
By the same arguments as above and in Claim 1, $\Gamma_{39} \unlhd G_{39}$ and the number of $G_{39}/\Gamma_{39}$-orbits of vertices of $X_{39}$ is fifteen. Therefore, $G_{39}/\Gamma_{39}$ acts on $X_{39} = K_{39}/\Gamma_{39}$.
Since  $O_1 = \langle a_0 \rangle,$ $O_2 = \langle a_1 \rangle,$  $O_3 = \langle a_2 \rangle$, $O_4 = \langle a_3 \rangle,$ $O_5 = \langle a_{4} \rangle,$ $O_6 = \langle a_{5} \rangle,$ $O_7 = \langle c_{0} \rangle$, $O_8 = \langle b_{0} \rangle$, $O_{9} = \langle b_{1} \rangle$, $O_{10} = \langle c_{4} \rangle$, $O_{11} = \langle c_{5} \rangle$, $O_{12} = \langle b_{5} \rangle$, $O_{13} = \langle c_{8} \rangle$, $O_{14} = \langle c_{11} \rangle$, $O_{15} = \langle c_{12} \rangle$
 are the $G_{32}$-orbits, it follows that $O_1 = \langle a_0 \rangle,$ $O_2 = \langle a_1 \rangle,$  $O_3 = \langle a_2 \rangle$, $O_4 = \langle a_3 \rangle,$ $O_5 = \langle a_{4} \rangle,$ $O_6 = \langle a_{5} \rangle,$ $O_7 = \langle c_{0} \rangle$, $O_8 = \langle b_{0} \rangle$, $O_{9} = \langle b_{1} \rangle$, $O_{10} = \langle c_{4} \rangle$, $O_{11} = \langle c_{5} \rangle$, $O_{12} = \langle b_{5} \rangle$, $O_{13} = \langle c_{8} \rangle$, $O_{14} = \langle c_{11} \rangle$, $O_{15} = \langle c_{12} \rangle$ are the $(G_{39}/\Gamma_{39})$-orbits.  
Since the vertex set of $X_{39}$ is $\sqcup_{j=1}^{15} \eta_{39}(O_j)$ and $G_{39}/\Gamma_{39} \leq {\rm Aut}(X_{39})$, it follows that the number of ${\rm Aut}(X_{39})$-orbits of vertices is $\le15$.

 \medskip

Let $X_{40} = K_{40}/\Gamma_{40}$ be a semiequivelar map on the torus for some fixed element (vertex, edge or face) free subgroup $\Gamma_{40}\le {\rm Aut}(K_{40})$. Let the vertices of $X_{40}$ form $m_{40}$ ${\rm Aut}(X_{40})$-orbits.
We take the middle point of the line segment joining vertices $a_0$ and $a_3$ as the origin $(0,0)$ of $K_{40}$  (see Section \ref{3uniform}).Let  $\alpha_{40} := a_{7} - a_0$ and $\beta_{40} := a_{15} - a_{0} \in \mathbb{R}^2$. Similarly as above, define $H_{40} := \langle x\mapsto x+\alpha_{40}, x\mapsto x+\beta_{40}\rangle$ and 
\begin{align*}
  G_{40} & =\{ \alpha : x\mapsto \varepsilon x + m_{40}\alpha_{40} + n_{40}\beta_{40}    \, : \, \varepsilon=\pm 1, m_{40}, n_{40} \in \ZZ\} \cong H_{40}\rtimes \mathbb{Z}_2.
\end{align*}
By the same arguments as above and in Claim 1, $\Gamma_{40} \unlhd G_{40}$ and the number of $G_{40}/\Gamma_{40}$-orbits of vertices of $X_{40}$ is seven. Therefore, $G_{40}/\Gamma_{40}$ acts on $X_{40} = K_{40}/\Gamma_{40}$.
Since $O_1 = \langle a_0 \rangle,$ $O_2 = \langle a_1 \rangle,$  $O_3 = \langle a_{2} \rangle$ $O_4 = \langle b_{13} \rangle,$$O_5 = \langle b_{12} \rangle,$ $O_6 = \langle b_{10} \rangle,$ $O_7 = \langle  b_{11} \rangle$
 are the $G_{40}$-orbits, it follows that $O_1 = \langle a_0 \rangle,$ $O_2 = \langle a_1 \rangle,$  $O_3 = \langle a_{2} \rangle$ $O_4 = \langle b_{13} \rangle,$$O_5 = \langle b_{12} \rangle,$ $O_6 = \langle b_{10} \rangle,$ $O_7 = \langle  b_{11} \rangle$ are the $(G_{40}/\Gamma_{40})$-orbits.  
Since the vertex set of $X_{40}$ is $\sqcup_{j=1}^{7} \eta_{40}(O_j)$ and $G_{40}/\Gamma_{40} \leq {\rm Aut}(X_{40})$, it follows that the number of ${\rm Aut}(X_{40})$-orbits of vertices is $\le7$.  
 
 \medskip

Let $X_{41} = K_{41}/\Gamma_{41}$ be a semiequivelar map on the torus for some fixed element (vertex, edge or face) free subgroup $\Gamma_{41} \le {\rm Aut}(K_{41})$. Let the vertices of $X_{41}$ form $m_{41}$ ${\rm Aut}(X_{41})$-orbits.
We take the middle point of the line segment joining vertices $a_0$ and $a_{6}$ as the origin $(0,0)$ of $K_{41}$  (see Section \ref{3uniform}). Let  $\alpha_{41} := a_{11} - a_0$, $\beta_{41} := a_{15} - a_{0}$ and $\gamma_{41} := a_{21} - a_{0}\in \mathbb{R}^2$. Similarly as above, define $H_{41} := \langle x\mapsto x+\alpha_{41}, x\mapsto x+\beta_{41}, x\mapsto x+\gamma_{41}\rangle$ and 
\begin{align*}
 G_{41} & =\{ \alpha : x\mapsto \varepsilon x + m_{41}\alpha_{41} + n_{41}\beta_{41} +r_{41}\gamma_{41}   \, : \, \varepsilon=\pm 1, m_{41}, n_{41}, r_{41} \in \ZZ\} \cong H_{41}\rtimes \mathbb{Z}_2
\end{align*}
By the same arguments as above and in Claim 1, $\Gamma_{41} \unlhd G_{41}$ and the number of $G_{41}/\Gamma_{41}$-orbits of vertices of $X_{41}$ is nine. Therefore, $G_{41}/\Gamma_{41}$ acts on $X_{41} = K_{41}/\Gamma_{41}$.
Since  $O_1 = \langle a_0 \rangle,$ $O_2 = \langle a_1 \rangle,$  $O_3 = \langle a_2 \rangle$, $O_4 = \langle b_{18} \rangle,$ $O_5 = \langle b_{0} \rangle,$ $O_6 = \langle b_{1} \rangle,$ $O_7 = \langle b_{2} \rangle$, $O_8 = \langle b_{19} \rangle$, $O_{9} = \langle b_{24} \rangle$
 are the $G_{41}$-orbits, it follows that $O_1 = \langle a_0 \rangle,$ $O_2 = \langle a_1 \rangle,$  $O_3 = \langle a_2 \rangle$, $O_4 = \langle b_{18} \rangle,$ $O_5 = \langle b_{0} \rangle,$ $O_6 = \langle b_{1} \rangle,$ $O_7 = \langle b_{2} \rangle$, $O_8 = \langle b_{19} \rangle$, $O_{9} = \langle b_{24} \rangle$ are the $(G_{41}/\Gamma_{41})$-orbits.  
Since the vertex set of $X_{41}$ is $\sqcup_{j=1}^{9} \eta_{41}(O_j)$ and $G_{41}/\Gamma_{41} \leq {\rm Aut}(X_{41})$, it follows that the number of ${\rm Aut}(X_{41})$-orbits of vertices is $\le9$.
 
 \medskip

Let $X_{42} = K_{42}/\Gamma_{42}$ be a semiequivelar map on the torus for some fixed element (vertex, edge or face) free subgroup $\Gamma_{42} \le {\rm Aut}(K_{42})$. Let the vertices of $X_{42}$ form $m_{42}$ ${\rm Aut}(X_{42})$-orbits.
We take the middle point of the line segment joining vertices $a_0$ and $a_{3}$ as the origin $(0,0)$ of $K_{42}$  (see Section \ref{3uniform}). Let  $\alpha_{42} := a_{11} - a_0$, $\beta_{42} := a_{15} - a_{0}$ and $\gamma_{42} := a_{21} - a_{0}\in \mathbb{R}^2$. Similarly as above, define $H_{42} := \langle x\mapsto x+\alpha_{42}, x\mapsto x+\beta_{42}, x\mapsto x+\gamma_{42}\rangle$ and 
\begin{align*}
 G_{42} & =\{ \alpha : x\mapsto \varepsilon x + m_{42}\alpha_{42} + n_{42}\beta_{42} +r_{42}\gamma_{42}   \, : \, \varepsilon=\pm 1, m_{42}, n_{42}, r_{42} \in \ZZ\} \cong H_{42}\rtimes \mathbb{Z}_2
\end{align*}
By the same arguments as above and in Claim 1, $\Gamma_{42} \unlhd G_{42}$ and the number of $G_{42}/\Gamma_{42}$-orbits of vertices of $X_{42}$ is nine. Therefore, $G_{42}/\Gamma_{42}$ acts on $X_{42} = K_{42}/\Gamma_{42}$.
Since  $O_1 = \langle a_0 \rangle,$ $O_2 = \langle a_1 \rangle,$  $O_3 = \langle a_2 \rangle$, $O_4 = \langle b_{24} \rangle,$ $O_5 = \langle b_{25} \rangle,$ $O_6 = \langle b_{14} \rangle,$ $O_7 = \langle b_{5} \rangle$, $O_8 = \langle b_{47} \rangle$, $O_{9} = \langle b_{46} \rangle$
 are the $G_{42}$-orbits, it follows that $O_1 = \langle a_0 \rangle,$ $O_2 = \langle a_1 \rangle,$ $O_3 = \langle a_2 \rangle$, $O_4 = \langle b_{24} \rangle,$ $O_5 = \langle b_{25} \rangle,$ $O_6 = \langle b_{14} \rangle,$ $O_7 = \langle b_{5} \rangle$, $O_8 = \langle b_{47} \rangle$, $O_{9} = \langle b_{46} \rangle$ are the $(G_{41}/\Gamma_{42})$-orbits.  
Since the vertex set of $X_{42}$ is $\sqcup_{j=1}^{9} \eta_{42}(O_j)$ and $G_{42}/\Gamma_{42} \leq {\rm Aut}(X_{42})$, it follows that the number of ${\rm Aut}(X_{42})$-orbits of vertices is $\le9$.

 \medskip

Let $X_{43} = K_{43}/\Gamma_{43}$ be a semiequivelar map on the torus for some fixed element (vertex, edge or face) free subgroup $\Gamma_{43}\le {\rm Aut}(K_{43})$. Let the vertices of $X_{43}$ form $m_{43}$ ${\rm Aut}(X_{43})$-orbits.
We take the middle point of the line segment joining vertices $a_0$ and $a_3$ as the origin $(0,0)$ of $K_{43}$  (see Section \ref{3uniform}).Let  $\alpha_{43} := a_{7} - a_0$ and $\beta_{43} := a_{13} - a_{0} \in \mathbb{R}^2$. Similarly as above, define $H_{43} := \langle x\mapsto x+\alpha_{43}, x\mapsto x+\beta_{43}\rangle$ and 
\begin{align*}
  G_{43} & =\{ \alpha : x\mapsto \varepsilon x + m_{43}\alpha_{43} + n_{43}\beta_{43}    \, : \, \varepsilon=\pm 1, m_{43}, n_{43} \in \ZZ\} \cong H_{43}\rtimes \mathbb{Z}_2.
\end{align*}
By the same arguments as above and in Claim 1, $\Gamma_{43} \unlhd G_{43}$ and the number of $G_{43}/\Gamma_{43}$-orbits of vertices of $X_{43}$ is seven. Therefore, $G_{43}/\Gamma_{43}$ acts on $X_{43} = K_{43}/\Gamma_{43}$.
Since $O_1 = \langle a_0 \rangle,$ $O_2 = \langle a_1 \rangle,$  $O_3 = \langle a_{2} \rangle$ $O_4 = \langle a_{24} \rangle,$ $O_5 = \langle a_{25} \rangle,$ $O_6 = \langle a_{26} \rangle,$ $O_7 = \langle  b_{0} \rangle$
 are the $G_{43}$-orbits, it follows that $O_1 = \langle a_0 \rangle,$ $O_2 = \langle a_1 \rangle,$  $O_3 = \langle a_{2} \rangle$ $O_4 = \langle a_{24} \rangle,$ $O_5 = \langle a_{25} \rangle,$ $O_6 = \langle a_{26} \rangle,$ $O_7 = \langle  b_{0} \rangle$ are the $(G_{43}/\Gamma_{43})$-orbits.  
Since the vertex set of $X_{43}$ is $\sqcup_{j=1}^{7} \eta_{43}(O_j)$ and $G_{43}/\Gamma_{43} \leq {\rm Aut}(X_{43})$, it follows that the number of ${\rm Aut}(X_{43})$-orbits of vertices is $\le7$.  

 \medskip

Let $X_{44} = K_{44}/\Gamma_{44}$ be a semiequivelar map on the torus for some fixed element (vertex, edge or face) free subgroup $\Gamma_{44}\le {\rm Aut}(K_{44})$. Let the vertices of $X_{44}$ form $m_{44}$ ${\rm Aut}(X_{44})$-orbits.
We take the middle point of the line segment joining vertices $b_0$ and $b_{18}$ as the origin $(0,0)$ of $K_{44}$  (see Section \ref{3uniform}).Let  $\alpha_{44} := b_{1} - b_0$ and $\beta_{44} := b_{30} - b_{0} \in \mathbb{R}^2$. Similarly as above, define $H_{44} := \langle x\mapsto x+\alpha_{44}, x\mapsto x+\beta_{44}\rangle$ and 
\begin{align*}
  G_{44} & =\{ \alpha : x\mapsto \varepsilon x + m_{44}\alpha_{44} + n_{44}\beta_{44}    \, : \, \varepsilon=\pm 1, m_{44}, n_{44} \in \ZZ\} \cong H_{44}\rtimes \mathbb{Z}_2.
\end{align*}
By the same arguments as above and in Claim 1, $\Gamma_{44} \unlhd G_{44}$ and the number of $G_{44}/\Gamma_{44}$-orbits of vertices of $X_{44}$ is three. Therefore, $G_{44}/\Gamma_{44}$ acts on $X_{44} = K_{44}/\Gamma_{44}$.
Since $O_1 = \langle a_0 \rangle,$ $O_2 = \langle a_{18} \rangle,$  $O_3 = \langle b_{0} \rangle$
 are the $G_{44}$-orbits, it follows that $O_1 = \langle a_0 \rangle,$ $O_2 = \langle a_{18} \rangle,$  $O_3 = \langle b_{0} \rangle$ are the $(G_{44}/\Gamma_{44})$-orbits.  
Since the vertex set of $X_{44}$ is $\sqcup_{j=1}^{3} \eta_{44}(O_j)$ and $G_{44}/\Gamma_{44} \leq {\rm Aut}(X_{44})$, it follows that the number of ${\rm Aut}(X_{44})$-orbits of vertices is $3$.  

 \medskip

Let $X_{45} = K_{45}/\Gamma_{45}$ be a semiequivelar map on the torus for some fixed element (vertex, edge or face) free subgroup $\Gamma_{45}\le {\rm Aut}(K_{45})$. Let the vertices of $X_{45}$ form $m_{45}$ ${\rm Aut}(X_{45})$-orbits.
We take the middle point of the line segment joining vertices $b_0$ and $b_{11}$ as the origin $(0,0)$ of $K_{45}$  (see Section \ref{3uniform}).Let  $\alpha_{45} := b_{1} - b_0$ and $\beta_{45} := b_{28} - b_{0} \in \mathbb{R}^2$. Similarly as above, define $H_{45} := \langle x\mapsto x+\alpha_{45}, x\mapsto x+\beta_{45}\rangle$ and 
\begin{align*}
  G_{45} & =\{ \alpha : x\mapsto \varepsilon x + m_{45}\alpha_{45} + n_{45}\beta_{45}    \, : \, \varepsilon=\pm 1, m_{45}, n_{45} \in \ZZ\} \cong H_{45}\rtimes \mathbb{Z}_2.
\end{align*}
By the same arguments as above and in Claim 1, $\Gamma_{45} \unlhd G_{45}$ and the number of $G_{45}/\Gamma_{45}$-orbits of vertices of $X_{45}$ is four. Therefore, $G_{45}/\Gamma_{45}$ acts on $X_{45} = K_{45}/\Gamma_{45}$.
Since $O_1 = \langle a_0 \rangle,$ $O_2 = \langle a_{11} \rangle,$  $O_3 = \langle a_{21} \rangle$, $O_4 = \langle b_{0} \rangle$
 are the $G_{45}$-orbits, it follows that  $O_1 = \langle a_0 \rangle,$ $O_2 = \langle a_{11} \rangle,$  $O_3 = \langle a_{21} \rangle$, $O_4 = \langle b_{0} \rangle$ are the $(G_{45}/\Gamma_{45})$-orbits.  
Since the vertex set of $X_{45}$ is $\sqcup_{j=1}^{4} \eta_{45}(O_j)$ and $G_{45}/\Gamma_{45} \leq {\rm Aut}(X_{45})$, it follows that the number of ${\rm Aut}(X_{45})$-orbits of vertices is $\le4$.  

 \medskip

Let $X_{46} = K_{46}/\Gamma_{46}$ be a semiequivelar map on the torus for some fixed element (vertex, edge or face) free subgroup $\Gamma_{46}\le {\rm Aut}(K_{46})$. Let the vertices of $X_{46}$ form $m_{46}$ ${\rm Aut}(X_{46})$-orbits.
We take the middle point of the line segment joining vertices $b_0$ and $b_{14}$ as the origin $(0,0)$ of $K_{46}$  (see Section \ref{3uniform}).Let  $\alpha_{46} := b_{1} - b_0$ and $\beta_{46} := b_{85} - b_{0} \in \mathbb{R}^2$. Similarly as above, define $H_{46} := \langle x\mapsto x+\alpha_{46}, x\mapsto x+\beta_{46}\rangle$ and 
\begin{align*}
  G_{46} & =\{ \alpha : x\mapsto \varepsilon x + m_{46}\alpha_{46} + n_{46}\beta_{46}    \, : \, \varepsilon=\pm 1, m_{46}, n_{46} \in \ZZ\} \cong H_{46}\rtimes \mathbb{Z}_2.
\end{align*}
By the same arguments as above and in Claim 1, $\Gamma_{46} \unlhd G_{46}$ and the number of $G_{46}/\Gamma_{46}$-orbits of vertices of $X_{46}$ is four. Therefore, $G_{46}/\Gamma_{46}$ acts on $X_{46} = K_{46}/\Gamma_{46}$.
Since $O_1 = \langle a_0 \rangle,$ $O_2 = \langle b_{0} \rangle,$  $O_3 = \langle b_{60} \rangle$, $O_4 = \langle b_{75} \rangle$
 are the $G_{46}$-orbits, it follows that  $O_1 = \langle a_0 \rangle,$ $O_2 = \langle b_{0} \rangle,$  $O_3 = \langle b_{60} \rangle$, $O_4 = \langle b_{75} \rangle$ are the $(G_{46}/\Gamma_{46})$-orbits.  
Since the vertex set of $X_{46}$ is $\sqcup_{j=1}^{4} \eta_{46}(O_j)$ and $G_{46}/\Gamma_{46} \leq {\rm Aut}(X_{46})$, it follows that the number of ${\rm Aut}(X_{46})$-orbits of vertices is $\le4$.  
 
 \medskip

Let $X_{47} = K_{47}/\Gamma_{47}$ be a semiequivelar map on the torus for some fixed element (vertex, edge or face) free subgroup $\Gamma_{47}\le {\rm Aut}(K_{47})$. Let the vertices of $X_{47}$ form $m_{47}$ ${\rm Aut}(X_{47})$-orbits.
We take the middle point of the line segment joining vertices $b_0$ and $b_{16}$ as the origin $(0,0)$ of $K_{47}$  (see Section \ref{3uniform}).Let  $\alpha_{47} := b_{1} - b_0$ and $\beta_{47} := b_{51} - b_{0} \in \mathbb{R}^2$. Similarly as above, define $H_{47} := \langle x\mapsto x+\alpha_{47}, x\mapsto x+\beta_{47}\rangle$ and 
\begin{align*}
  G_{47} & =\{ \alpha : x\mapsto \varepsilon x + m_{47}\alpha_{47} + n_{47}\beta_{47}    \, : \, \varepsilon=\pm 1, m_{47}, n_{47} \in \ZZ\} \cong H_{47}\rtimes \mathbb{Z}_2.
\end{align*}
By the same arguments as above and in Claim 1, $\Gamma_{47} \unlhd G_{47}$ and the number of $G_{47}/\Gamma_{47}$-orbits of vertices of $X_{47}$ is five. Therefore, $G_{47}/\Gamma_{47}$ acts on $X_{47} = K_{47}/\Gamma_{47}$.
Since $O_1 = \langle b_{0} \rangle,$ $O_2 = \langle b_{28} \rangle,$  $O_3 = \langle b_{39} \rangle$, $O_4 = \langle a_{0} \rangle$, $O_5 = \langle a_{16} \rangle$
 are the $G_{47}$-orbits, it follows that  $O_1 = \langle b_{0} \rangle,$ $O_2 = \langle b_{28} \rangle,$  $O_3 = \langle b_{39} \rangle$, $O_4 = \langle a_{0} \rangle$, $O_5 = \langle a_{16} \rangle
 $ are the $(G_{47}/\Gamma_{47})$-orbits.  
Since the vertex set of $X_{47}$ is $\sqcup_{j=1}^{5} \eta_{47}(O_j)$ and $G_{47}/\Gamma_{47} \leq {\rm Aut}(X_{47})$, it follows that the number of ${\rm Aut}(X_{47})$-orbits of vertices is $\le5$.  
 
 \medskip

Let $X_{48} = K_{48}/\Gamma_{48}$ be a semiequivelar map on the torus for some fixed element (vertex, edge or face) free subgroup $\Gamma_{48} \le {\rm Aut}(K_{48})$. Let the vertices of $X_{48}$ form $m_{48}$ ${\rm Aut}(X_{48})$-orbits.
We take the point $b_0$ as the origin $(0,0)$ of $K_{48}$  (see Section \ref{3uniform}). Let  $\alpha_{48} := b_{1} - b_0$, $\beta_{48} := b_{2} - b_{0}$ and $\gamma_{48} := b_{3} - b_{0}\in \mathbb{R}^2$. Similarly as above, define $H_{48} := \langle x\mapsto x+\alpha_{48}, x\mapsto x+\beta_{48}, x\mapsto x+\gamma_{48}\rangle$ and 
\begin{align*}
 G_{48} & =\{ \alpha : x\mapsto \varepsilon x + m_{48}\alpha_{48} + n_{48}\beta_{48} +r_{48}\gamma_{48}   \, : \, \varepsilon=\pm 1, m_{48}, n_{48}, r_{48} \in \ZZ\} \cong H_{48}\rtimes \mathbb{Z}_2
\end{align*}
By the same arguments as above and in Claim 1, $\Gamma_{48} \unlhd G_{48}$ and the number of $G_{48}/\Gamma_{48}$-orbits of vertices of $X_{48}$ is seven. Therefore, $G_{48}/\Gamma_{48}$ acts on $X_{48} = K_{48}/\Gamma_{48}$.
Since  $O_1 = \langle a_0 \rangle,$ $O_2 = \langle a_1 \rangle,$  $O_3 = \langle a_2 \rangle$, $O_4 = \langle a_{6} \rangle,$ $O_5 = \langle a_8 \rangle,$ $O_6 = \langle a_{10} \rangle,$ $O_7 = \langle b_{0} \rangle$
 are the $G_{48}$-orbits, it follows that $O_1 = \langle a_0 \rangle,$ $O_2 = \langle a_1 \rangle,$  $O_3 = \langle a_2 \rangle$, $O_4 = \langle a_{6} \rangle,$ $O_5 = \langle a_8 \rangle,$ $O_6 = \langle a_{10} \rangle,$ $O_7 = \langle b_{0} \rangle$ are the $(G_{48}/\Gamma_{48})$-orbits.  
Since the vertex set of $X_{48}$ is $\sqcup_{j=1}^{7} \eta_{48}(O_j)$ and $G_{48}/\Gamma_{48} \leq {\rm Aut}(X_{48})$, it follows that the number of ${\rm Aut}(X_{48})$-orbits of vertices is $\le7$.

 \medskip

Let $X_{49} = K_{49}/\Gamma_{49}$ be a semiequivelar map on the torus for some fixed element (vertex, edge or face) free subgroup $\Gamma_{49}\le {\rm Aut}(K_{49})$. Let the vertices of $X_{49}$ form $m_{49}$ ${\rm Aut}(X_{49})$-orbits.
We take the middle point of the line segment joining vertices $b_0$ and $b_{1}$ as the origin $(0,0)$ of $K_{49}$  (see Section \ref{3uniform}).Let  $\alpha_{49} := b_{1} - b_0$ and $\beta_{49} := b_{11} - b_{0} \in \mathbb{R}^2$. Similarly as above, define $H_{49} := \langle x\mapsto x+\alpha_{49}, x\mapsto x+\beta_{49}\rangle$ and 
\begin{align*}
  G_{49} & =\{ \alpha : x\mapsto \varepsilon x + m_{49}\alpha_{49} + n_{49}\beta_{49}    \, : \, \varepsilon=\pm 1, m_{49}, n_{49} \in \ZZ\} \cong H_{49}\rtimes \mathbb{Z}_2.
\end{align*}
By the same arguments as above and in Claim 1, $\Gamma_{49} \unlhd G_{49}$ and the number of $G_{49}/\Gamma_{49}$-orbits of vertices of $X_{49}$ is three. Therefore, $G_{49}/\Gamma_{49}$ acts on $X_{49} = K_{49}/\Gamma_{49}$.
Since $O_1 = \langle a_0 \rangle,$ $O_2 = \langle a_{1} \rangle,$  $O_3 = \langle b_{0} \rangle$
 are the $G_{44}$-orbits, it follows that $O_1 = \langle a_0 \rangle,$ $O_2 = \langle a_{1} \rangle,$  $O_3 = \langle b_{0} \rangle$ are the $(G_{49}/\Gamma_{49})$-orbits.  
Since the vertex set of $X_{49}$ is $\sqcup_{j=1}^{3} \eta_{49}(O_j)$ and $G_{49}/\Gamma_{49} \leq {\rm Aut}(X_{49})$, it follows that the number of ${\rm Aut}(X_{49})$-orbits of vertices is $3$.  
    
 \medskip

Let $X_{50} = K_{50}/\Gamma_{50}$ be a semiequivelar map on the torus for some fixed element (vertex, edge or face) free subgroup $\Gamma_{50}\le {\rm Aut}(K_{50})$. Let the vertices of $X_{50}$ form $m_{50}$ ${\rm Aut}(X_{50})$-orbits.
We take the middle point of the line segment joining vertices $b_0$ and $b_{3}$ as the origin $(0,0)$ of $K_{50}$  (see Section \ref{3uniform}).Let  $\alpha_{50} := b_{1} - b_0$ and $\beta_{50} := b_{10} - b_{0} \in \mathbb{R}^2$. Similarly as above, define $H_{50} := \langle x\mapsto x+\alpha_{50}, x\mapsto x+\beta_{50}\rangle$ and 
\begin{align*}
  G_{50} & =\{ \alpha : x\mapsto \varepsilon x + m_{50}\alpha_{50} + n_{50}\beta_{50}    \, : \, \varepsilon=\pm 1, m_{50}, n_{50} \in \ZZ\} \cong H_{50}\rtimes \mathbb{Z}_2.
\end{align*}
By the same arguments as above and in Claim 1, $\Gamma_{50} \unlhd G_{50}$ and the number of $G_{50}/\Gamma_{50}$-orbits of vertices of $X_{50}$ is three. Therefore, $G_{50}/\Gamma_{50}$ acts on $X_{50} = K_{50}/\Gamma_{50}$.
Since $O_1 = \langle a_5 \rangle,$ $O_2 = \langle a_{6} \rangle,$  $O_3 = \langle b_{0} \rangle$
 are the $G_{50}$-orbits, it follows that $O_1 = \langle a_5 \rangle,$ $O_2 = \langle a_{6} \rangle,$  $O_3 = \langle b_{0} \rangle$ are the $(G_{50}/\Gamma_{50})$-orbits.  
Since the vertex set of $X_{50}$ is $\sqcup_{j=1}^{3} \eta_{50}(O_j)$ and $G_{50}/\Gamma_{50} \leq {\rm Aut}(X_{50})$, it follows that the number of ${\rm Aut}(X_{50})$-orbits of vertices is $3$.   
    
 \medskip

Let $X_{51} = K_{51}/\Gamma_{51}$ be a semiequivelar map on the torus for some fixed element (vertex, edge or face) free subgroup $\Gamma_{51}\le {\rm Aut}(K_{51})$. Let the vertices of $X_{51}$ form $m_{51}$ ${\rm Aut}(X_{51})$-orbits.
We take the middle point of the line segment joining vertices $a_0$ and $b_{1}$ as the origin $(0,0)$ of $K_{51}$  (see Section \ref{3uniform}).Let  $\alpha_{51} := a_{14} - a_0$ and $\beta_{51} := a_{7} - a_{0} \in \mathbb{R}^2$. Similarly as above, define $H_{51} := \langle x\mapsto x+\alpha_{51}, x\mapsto x+\beta_{51}\rangle$ and 
\begin{align*}
  G_{51} & =\{ \alpha : x\mapsto \varepsilon x + m_{51}\alpha_{51} + n_{51}\beta_{51}    \, : \, \varepsilon=\pm 1, m_{51}, n_{51} \in \ZZ\} \cong H_{51}\rtimes \mathbb{Z}_2.
\end{align*}
By the same arguments as above and in Claim 1, $\Gamma_{51} \unlhd G_{51}$ and the number of $G_{51}/\Gamma_{51}$-orbits of vertices of $X_{51}$ is three. Therefore, $G_{51}/\Gamma_{51}$ acts on $X_{51} = K_{51}/\Gamma_{51}$.
Since $O_1 = \langle a_1 \rangle,$ $O_2 = \langle a_2 \rangle,$  $O_3 = \langle b_{21} \rangle$, $O_4 = \langle b_{20} \rangle,$ $O_5 = \langle a_{3} \rangle,$ $O_6 = \langle a_{0} \rangle,$ $O_7 = \langle b_{0} \rangle$, $O_8 = \langle b_{1} \rangle$, $O_{9} = \langle b_{2} \rangle$, $O_{10} = \langle a_{8} \rangle$, $O_{11} = \langle a_{10} \rangle$, $O_{12} = \langle b_{3} \rangle$, $O_{13} = \langle b_{4} \rangle$, $O_{14} = \langle a_{11} \rangle$, $O_{15} = \langle a_{9} \rangle$ are the $G_{50}$-orbits, it follows that $O_1 = \langle a_1 \rangle,$ $O_2 = \langle a_2 \rangle,$  $O_3 = \langle b_{21} \rangle$, $O_4 = \langle b_{20} \rangle,$ $O_5 = \langle a_{3} \rangle,$ $O_6 = \langle a_{0} \rangle,$ $O_7 = \langle b_{0} \rangle$, $O_8 = \langle b_{1} \rangle$, $O_{9} = \langle b_{2} \rangle$, $O_{10} = \langle a_{8} \rangle$, $O_{11} = \langle a_{10} \rangle$, $O_{12} = \langle b_{3} \rangle$, $O_{13} = \langle b_{4} \rangle$, $O_{14} = \langle a_{11} \rangle$, $O_{15} = \langle a_{9} \rangle$ are the $(G_{51}/\Gamma_{51})$-orbits.  
Since the vertex set of $X_{51}$ is $\sqcup_{j=1}^{15} \eta_{51}(O_j)$ and $G_{51}/\Gamma_{51} \leq {\rm Aut}(X_{51})$, it follows that the number of ${\rm Aut}(X_{51})$-orbits of vertices is $\le15$.   
    
 \medskip

Let $X_{52} = K_{52}/\Gamma_{52}$ be a semiequivelar map on the torus for some fixed element (vertex, edge or face) free subgroup $\Gamma_{52}\le {\rm Aut}(K_{52})$. Let the vertices of $X_{52}$ form $m_{52}$ ${\rm Aut}(X_{52})$-orbits.
We take the middle point of the line segment joining vertices $a_0$ and $b_{0}$ as the origin $(0,0)$ of $K_{52}$  (see Section \ref{3uniform}).Let  $\alpha_{52} := a_{3} - a_0$ and $\beta_{52} := a_{5} - a_{0} \in \mathbb{R}^2$. Similarly as above, define $H_{52} := \langle x\mapsto x+\alpha_{52}, x\mapsto x+\beta_{52}\rangle$ and 
\begin{align*}
  G_{52} & =\{ \alpha : x\mapsto \varepsilon x + m_{52}\alpha_{52} + n_{52}\beta_{52}    \, : \, \varepsilon=\pm 1, m_{52}, n_{52} \in \ZZ\} \cong H_{52}\rtimes \mathbb{Z}_2.
\end{align*}
By the same arguments as above and in Claim 1, $\Gamma_{52} \unlhd G_{52}$ and the number of $G_{52}/\Gamma_{52}$-orbits of vertices of $X_{52}$ is six. Therefore, $G_{52}/\Gamma_{52}$ acts on $X_{52} = K_{52}/\Gamma_{52}$.
Since $O_1 = \langle a_0 \rangle,$ $O_2 = \langle b_1 \rangle,$  $O_3 = \langle b_{0} \rangle$, $O_4 = \langle b_{1} \rangle,$ $O_5 = \langle b_{3} \rangle,$ $O_6 = \langle b_{2} \rangle$
 are the $G_{52}$-orbits, it follows that $O_1 = \langle a_0 \rangle,$ $O_2 = \langle b_1 \rangle,$  $O_3 = \langle b_{0} \rangle$, $O_4 = \langle b_{1} \rangle,$ $O_5 = \langle b_{3} \rangle,$ $O_6 = \langle b_{2} \rangle$ are the $(G_{52}/\Gamma_{52})$-orbits.  
Since the vertex set of $X_{52}$ is $\sqcup_{j=1}^{6} \eta_{52}(O_j)$ and $G_{52}/\Gamma_{52} \leq {\rm Aut}(X_{52})$, it follows that the number of ${\rm Aut}(X_{52})$-orbits of vertices is $\le6$.   

\medskip

Let $X_{53} = K_{53}/\Gamma_{53}$ be a semiequivelar map on the torus for some fixed element (vertex, edge or face) free subgroup $\Gamma_{53}\le {\rm Aut}(K_{53})$. Let the vertices of $X_{53}$ form $m_{53}$ ${\rm Aut}(X_{53})$-orbits.
We take the middle point of the line segment joining vertices $a_0$ and $b_{1}$ as the origin $(0,0)$ of $K_{53}$  (see Section \ref{3uniform}).Let  $\alpha_{53} := a_{1} - a_0$ and $\beta_{53} := a_{42} - a_{0} \in \mathbb{R}^2$. Similarly as above, define $H_{53} := \langle x\mapsto x+\alpha_{53}, x\mapsto x+\beta_{53}\rangle$ and 
\begin{align*}
  G_{53} & =\{ \alpha : x\mapsto \varepsilon x + m_{53}\alpha_{53} + n_{53}\beta_{53}    \, : \, \varepsilon=\pm 1, m_{53}, n_{53} \in \ZZ\} \cong H_{53}\rtimes \mathbb{Z}_2.
\end{align*}
By the same arguments as above and in Claim 1, $\Gamma_{53} \unlhd G_{53}$ and the number of $G_{53}/\Gamma_{53}$-orbits of vertices of $X_{53}$ is five. Therefore, $G_{53}/\Gamma_{53}$ acts on $X_{53} = K_{53}/\Gamma_{53}$.
Since $O_1 = \langle a_0 \rangle,$ $O_2 = \langle b_0 \rangle,$  $O_3 = \langle a_{12} \rangle$, $O_4 = \langle a_{23} \rangle,$ $O_5 = \langle a_{32} \rangle$
 are the $G_{53}$-orbits, it follows that $O_1 = \langle a_0 \rangle,$ $O_2 = \langle b_0 \rangle,$  $O_3 = \langle a_{12} \rangle$, $O_4 = \langle a_{23} \rangle,$ $O_5 = \langle a_{32} \rangle$ are the $(G_{53}/\Gamma_{53})$-orbits.  
Since the vertex set of $X_{53}$ is $\sqcup_{j=1}^{5} \eta_{53}(O_j)$ and $G_{53}/\Gamma_{53} \leq {\rm Aut}(X_{53})$, it follows that the number of ${\rm Aut}(X_{53})$-orbits of vertices is $\le5$.   
 
\medskip

Let $X_{54} = K_{54}/\Gamma_{54}$ be a semiequivelar map on the torus for some fixed element (vertex, edge or face) free subgroup $\Gamma_{54}\le {\rm Aut}(K_{54})$. Let the vertices of $X_{54}$ form $m_{54}$ ${\rm Aut}(X_{54})$-orbits.
We take the middle point of the line segment joining vertices $a_0$ and $a_{14}$ as the origin $(0,0)$ of $K_{54}$  (see Section \ref{3uniform}).Let  $\alpha_{54} := a_{1} - a_0$ and $\beta_{54} := a_{81} - a_{0} \in \mathbb{R}^2$. Similarly as above, define $H_{54} := \langle x\mapsto x+\alpha_{54}, x\mapsto x+\beta_{54}\rangle$ and 
\begin{align*}
  G_{54} & =\{ \alpha : x\mapsto \varepsilon x + m_{54}\alpha_{54} + n_{54}\beta_{54}    \, : \, \varepsilon=\pm 1, m_{54}, n_{54} \in \ZZ\} \cong H_{54}\rtimes \mathbb{Z}_2.
\end{align*}
By the same arguments as above and in Claim 1, $\Gamma_{54} \unlhd G_{54}$ and the number of $G_{54}/\Gamma_{54}$-orbits of vertices of $X_{54}$ is three. Therefore, $G_{54}/\Gamma_{54}$ acts on $X_{54} = K_{54}/\Gamma_{54}$.
Since $O_1 = \langle a_0 \rangle,$ $O_2 = \langle a_{51} \rangle,$  $O_3 = \langle b_{0} \rangle$,
 are the $G_{54}$-orbits, it follows that $O_1 = \langle a_0 \rangle,$ $O_2 = \langle a_{51} \rangle,$  $O_3 = \langle b_{0} \rangle$nare the $(G_{54}/\Gamma_{54})$-orbits.  
Since the vertex set of $X_{54}$ is $\sqcup_{j=1}^{3} \eta_{54}(O_j)$ and $G_{54}/\Gamma_{54} \leq {\rm Aut}(X_{54})$, it follows that the number of ${\rm Aut}(X_{54})$-orbits of vertices is $3$.   
 
\medskip

Let $X_{55} = K_{55}/\Gamma_{55}$ be a semiequivelar map on the torus for some fixed element (vertex, edge or face) free subgroup $\Gamma_{55}\le {\rm Aut}(K_{55})$. Let the vertices of $X_{55}$ form $m_{55}$ ${\rm Aut}(X_{55})$-orbits.
We take the middle point of the line segment joining vertices $a_0$ and $b_{1}$ as the origin $(0,0)$ of $K_{55}$  (see Section \ref{3uniform}).Let  $\alpha_{55} := a_{1} - a_0$ and $\beta_{55} := a_{31} - a_{0} \in \mathbb{R}^2$. Similarly as above, define $H_{55} := \langle x\mapsto x+\alpha_{55}, x\mapsto x+\beta_{55}\rangle$ and 
\begin{align*}
  G_{55} & =\{ \alpha : x\mapsto \varepsilon x + m_{55}\alpha_{55} + n_{55}\beta_{55}    \, : \, \varepsilon=\pm 1, m_{55}, n_{55} \in \ZZ\} \cong H_{55}\rtimes \mathbb{Z}_2.
\end{align*}
By the same arguments as above and in Claim 1, $\Gamma_{55} \unlhd G_{55}$ and the number of $G_{55}/\Gamma_{55}$-orbits of vertices of $X_{55}$ is five. Therefore, $G_{55}/\Gamma_{55}$ acts on $X_{55} = K_{55}/\Gamma_{55}$.
Since $O_1 = \langle b_0 \rangle,$ $O_2 = \langle a_0 \rangle,$  $O_3 = \langle a_{12} \rangle$, $O_4 = \langle a_{23} \rangle,$ $O_5 = \langle a_{32} \rangle$
 are the $G_{55}$-orbits, it follows that $O_1 = \langle b_0 \rangle,$ $O_2 = \langle a_0 \rangle,$  $O_3 = \langle a_{12} \rangle$, $O_4 = \langle a_{23} \rangle,$ $O_5 = \langle a_{32} \rangle$ are the $(G_{55}/\Gamma_{55})$-orbits.  
Since the vertex set of $X_{55}$ is $\sqcup_{j=1}^{5} \eta_{55}(O_j)$ and $G_{55}/\Gamma_{55} \leq {\rm Aut}(X_{55})$, it follows that the number of ${\rm Aut}(X_{55})$-orbits of vertices is $\le5$.   
 
\medskip

Let $X_{56} = K_{56}/\Gamma_{56}$ be a semiequivelar map on the torus for some fixed element (vertex, edge or face) free subgroup $\Gamma_{56}\le {\rm Aut}(K_{56})$. Let the vertices of $X_{56}$ form $m_{56}$ ${\rm Aut}(X_{56})$-orbits.
We take the middle point of the line segment joining vertices $a_0$ and $b_{1}$ as the origin $(0,0)$ of $K_{56}$  (see Section \ref{3uniform}).Let  $\alpha_{56} := a_{1} - a_0$ and $\beta_{56} := a_{45} - a_{0} \in \mathbb{R}^2$. Similarly as above, define $H_{56} := \langle x\mapsto x+\alpha_{56}, x\mapsto x+\beta_{56}\rangle$ and 
\begin{align*}
  G_{56} & =\{ \alpha : x\mapsto \varepsilon x + m_{56}\alpha_{56} + n_{56}\beta_{56}    \, : \, \varepsilon=\pm 1, m_{56}, n_{56} \in \ZZ\} \cong H_{56}\rtimes \mathbb{Z}_2.
\end{align*}
By the same arguments as above and in Claim 1, $\Gamma_{56} \unlhd G_{56}$ and the number of $G_{56}/\Gamma_{56}$-orbits of vertices of $X_{56}$ is six. Therefore, $G_{56}/\Gamma_{56}$ acts on $X_{56} = K_{56}/\Gamma_{56}$.
Since $O_1 = \langle b_0 \rangle,$ $O_2 = \langle a_0 \rangle,$  $O_3 = \langle a_{12} \rangle$, $O_4 = \langle a_{25} \rangle,$ $O_5 = \langle a_{32} \rangle$, $O_6 = \langle a_{33} \rangle$
 are the $G_{56}$-orbits, it follows that $O_1 = \langle b_0 \rangle,$ $O_2 = \langle a_0 \rangle,$  $O_3 = \langle a_{12} \rangle$, $O_4 = \langle a_{25} \rangle,$ $O_5 = \langle a_{32} \rangle$, $O_6 = \langle a_{33} \rangle$ are the $(G_{56}/\Gamma_{56})$-orbits.  
Since the vertex set of $X_{56}$ is $\sqcup_{j=1}^{6} \eta_{56}(O_j)$ and $G_{56}/\Gamma_{56} \leq {\rm Aut}(X_{56})$, it follows that the number of ${\rm Aut}(X_{56})$-orbits of vertices is $\le6$.   
  
\medskip

Let $X_{57} = K_{57}/\Gamma_{57}$ be a semiequivelar map on the torus for some fixed element (vertex, edge or face) free subgroup $\Gamma_{57}\le {\rm Aut}(K_{57})$. Let the vertices of $X_{57}$ form $m_{57}$ ${\rm Aut}(X_{57})$-orbits.
We take the middle point of the line segment joining vertices $b_0$ and $b_{1}$ as the origin $(0,0)$ of $K_{57}$  (see Section \ref{3uniform}).Let  $\alpha_{57} := b_{1} - b_0$ and $\beta_{57} := b_{20} - b_{0} \in \mathbb{R}^2$. Similarly as above, define $H_{57} := \langle x\mapsto x+\alpha_{57}, x\mapsto x+\beta_{57}\rangle$ and 
\begin{align*}
  G_{57} & =\{ \alpha : x\mapsto \varepsilon x + m_{57}\alpha_{57} + n_{57}\beta_{57}    \, : \, \varepsilon=\pm 1, m_{57}, n_{57} \in \ZZ\} \cong H_{57}\rtimes \mathbb{Z}_2.
\end{align*}
By the same arguments as above and in Claim 1, $\Gamma_{57} \unlhd G_{57}$ and the number of $G_{57}/\Gamma_{57}$-orbits of vertices of $X_{57}$ is six. Therefore, $G_{57}/\Gamma_{57}$ acts on $X_{57} = K_{57}/\Gamma_{57}$.
Since $O_1 = \langle a_1 \rangle,$ $O_2 = \langle a_0 \rangle,$  $O_3 = \langle b_{0} \rangle$, $O_4 = \langle b_{9} \rangle,$ $O_5 = \langle b_{8} \rangle$, $O_6 = \langle b_{15} \rangle$
 are the $G_{57}$-orbits, it follows that $O_1 = \langle a_1 \rangle,$ $O_2 = \langle a_0 \rangle,$  $O_3 = \langle b_{0} \rangle$, $O_4 = \langle b_{9} \rangle,$ $O_5 = \langle b_{8} \rangle$, $O_6 = \langle b_{15} \rangle$ are the $(G_{57}/\Gamma_{57})$-orbits.  
Since the vertex set of $X_{56}$ is $\sqcup_{j=1}^{6} \eta_{57}(O_j)$ and $G_{57}/\Gamma_{57} \leq {\rm Aut}(X_{57})$, it follows that the number of ${\rm Aut}(X_{57})$-orbits of vertices is $\le6$.   

\medskip

Let $X_{58} = K_{58}/\Gamma_{58}$ be a semiequivelar map on the torus for some fixed element (vertex, edge or face) free subgroup $\Gamma_{58} \le {\rm Aut}(K_{58})$. Let the vertices of $X_{58}$ form $m_{58}$ ${\rm Aut}(X_{58})$-orbits.
We take the middle point of the line segment joining vertices $a_0$ and $a_{3}$ as the origin $(0,0)$ of $K_{58}$  (see Section \ref{3uniform}). Let  $\alpha_{58} := a_{24} - a_0$, $\beta_{58} := a_{34} - a_{0}$ and $\gamma_{58} := a_{56} - a_{0}\in \mathbb{R}^2$. Similarly as above, define $H_{58} := \langle x\mapsto x+\alpha_{58}, x\mapsto x+\beta_{58}, x\mapsto x+\gamma_{58}\rangle$ and 
\begin{align*}
 G_{58} & =\{ \alpha : x\mapsto \varepsilon x + m_{58}\alpha_{58} + n_{58}\beta_{58} +r_{58}\gamma_{58}   \, : \, \varepsilon=\pm 1, m_{58}, n_{58}, r_{58} \in \ZZ\} \cong H_{58}\rtimes \mathbb{Z}_2
\end{align*}
By the same arguments as above and in Claim 1, $\Gamma_{58} \unlhd G_{58}$ and the number of $G_{58}/\Gamma_{58}$-orbits of vertices of $X_{58}$ is twelve. Therefore, $G_{58}/\Gamma_{58}$ acts on $X_{58} = K_{58}/\Gamma_{58}$.
Since  $O_1 = \langle a_0 \rangle,$ $O_2 = \langle a_1 \rangle,$  $O_3 = \langle a_2 \rangle$, $O_4 = \langle b_0 \rangle,$ $O_5 = \langle a_{6} \rangle,$ $O_6 = \langle b_{3} \rangle,$ $O_7 = \langle b_{4} \rangle$, $O_8 = \langle a_{8} \rangle$, $O_{9} = \langle b_{7} \rangle$, $O_{10} = \langle b_{8} \rangle$, $O_{11} = \langle b_{9} \rangle$, $O_{12} = \langle a_{10} \rangle$
 are the $G_{58}$-orbits, it follows that$O_1 = \langle a_0 \rangle,$ $O_2 = \langle a_1 \rangle,$  $O_3 = \langle a_2 \rangle$, $O_4 = \langle b_0 \rangle,$ $O_5 = \langle a_{6} \rangle,$ $O_6 = \langle b_{3} \rangle,$ $O_7 = \langle b_{4} \rangle$, $O_8 = \langle a_{8} \rangle$, $O_{9} = \langle b_{7} \rangle$, $O_{10} = \langle b_{8} \rangle$, $O_{11} = \langle b_{9} \rangle$, $O_{12} = \langle a_{10} \rangle$ are the $(G_{58}/\Gamma_{58})$-orbits.  
Since the vertex set of $X_{58}$ is $\sqcup_{j=1}^{12} \eta_{58}(O_j)$ and $G_{58}/\Gamma_{58} \leq {\rm Aut}(X_{58})$, it follows that the number of ${\rm Aut}(X_{58})$-orbits of vertices is $\le12$.
  
\medskip

Let $X_{59} = K_{59}/\Gamma_{59}$ be a semiequivelar map on the torus for some fixed element (vertex, edge or face) free subgroup $\Gamma_{59}\le {\rm Aut}(K_{59})$. Let the vertices of $X_{59}$ form $m_{59}$ ${\rm Aut}(X_{59})$-orbits.
We take the middle point of the line segment joining vertices $b_0$ and $b_{1}$ as the origin $(0,0)$ of $K_{59}$  (see Section \ref{3uniform}).Let  $\alpha_{59} := b_{1} - b_0$ and $\beta_{59} := b_{18} - b_{0} \in \mathbb{R}^2$. Similarly as above, define $H_{59} := \langle x\mapsto x+\alpha_{59}, x\mapsto x+\beta_{59}\rangle$ and 
\begin{align*}
  G_{59} & =\{ \alpha : x\mapsto \varepsilon x + m_{59}\alpha_{59} + n_{59}\beta_{59}    \, : \, \varepsilon=\pm 1, m_{59}, n_{59} \in \ZZ\} \cong H_{59}\rtimes \mathbb{Z}_2.
\end{align*}
By the same arguments as above and in Claim 1, $\Gamma_{59} \unlhd G_{59}$ and the number of $G_{59}/\Gamma_{59}$-orbits of vertices of $X_{59}$ is six. Therefore, $G_{59}/\Gamma_{59}$ acts on $X_{59} = K_{59}/\Gamma_{59}$.
Since $O_1 = \langle a_0 \rangle,$ $O_2 = \langle a_1 \rangle,$  $O_3 = \langle b_{1} \rangle$, $O_4 = \langle a_{9} \rangle,$ $O_5 = \langle b_{4} \rangle$, $O_6 = \langle a_{11} \rangle$
 are the $G_{59}$-orbits, it follows that $O_1 = \langle a_0 \rangle,$ $O_2 = \langle a_1 \rangle,$  $O_3 = \langle b_{1} \rangle$, $O_4 = \langle a_{9} \rangle,$ $O_5 = \langle b_{4} \rangle$, $O_6 = \langle a_{11} \rangle$ are the $(G_{59}/\Gamma_{59})$-orbits.  
Since the vertex set of $X_{59}$ is $\sqcup_{j=1}^{6} \eta_{59}(O_j)$ and $G_{59}/\Gamma_{59} \leq {\rm Aut}(X_{59})$, it follows that the number of ${\rm Aut}(X_{59})$-orbits of vertices is $\le6$.   
  
\medskip

Let $X_{60} = K_{60}/\Gamma_{60}$ be a semiequivelar map on the torus for some fixed element (vertex, edge or face) free subgroup $\Gamma_{60}\le {\rm Aut}(K_{60})$. Let the vertices of $X_{60}$ form $m_{60}$ ${\rm Aut}(X_{60})$-orbits.
We take the middle point of the line segment joining vertices $a_0$ and $a_{3}$ as the origin $(0,0)$ of $K_{60}$  (see Section \ref{3uniform}).Let  $\alpha_{60} := a_{3} - a_0$ and $\beta_{60} := a_{54} - a_{0} \in \mathbb{R}^2$. Similarly as above, define $H_{60} := \langle x\mapsto x+\alpha_{60}, x\mapsto x+\beta_{60}\rangle$ and 
\begin{align*}
  G_{60} & =\{ \alpha : x\mapsto \varepsilon x + m_{60}\alpha_{60} + n_{60}\beta_{60}    \, : \, \varepsilon=\pm 1, m_{60}, n_{60} \in \ZZ\} \cong H_{60}\rtimes \mathbb{Z}_2.
\end{align*}
By the same arguments as above and in Claim 1, $\Gamma_{60} \unlhd G_{60}$ and the number of $G_{60}/\Gamma_{60}$-orbits of vertices of $X_{60}$ is eight. Therefore, $G_{60}/\Gamma_{60}$ acts on $X_{60} = K_{60}/\Gamma_{60}$.
Since $O_1 = \langle a_0 \rangle,$ $O_2 = \langle a_1 \rangle,$  $O_3 = \langle a_{2} \rangle$, $O_4 = \langle b_{1} \rangle,$ $O_5 = \langle b_{0} \rangle$, $O_6 = \langle b_{15} \rangle$, $O_6 = \langle b_{14} \rangle$, $O_6 = \langle a_{15} \rangle$
 are the $G_{60}$-orbits, it follows that $O_1 = \langle a_0 \rangle,$ $O_2 = \langle a_1 \rangle,$  $O_3 = \langle a_{2} \rangle$, $O_4 = \langle b_{1} \rangle,$ $O_5 = \langle b_{0} \rangle$, $O_6 = \langle b_{15} \rangle$, $O_6 = \langle b_{14} \rangle$, $O_6 = \langle a_{15} \rangle$ are the $(G_{60}/\Gamma_{60})$-orbits.  
Since the vertex set of $X_{60}$ is $\sqcup_{j=1}^{8} \eta_{60}(O_j)$ and $G_{60}/\Gamma_{60} \leq {\rm Aut}(X_{60})$, it follows that the number of ${\rm Aut}(X_{60})$-orbits of vertices is $\le8$.   

\medskip

Let $X_{61} = K_{61}/\Gamma_{61}$ be a semiequivelar map on the torus for some fixed element (vertex, edge or face) free subgroup $\Gamma_{61}\le {\rm Aut}(K_{61})$. Let the vertices of $X_{61}$ form $m_{61}$ ${\rm Aut}(X_{61})$-orbits.
We take the middle point of the line segment joining vertices $a_0$ and $a_{3}$ as the origin $(0,0)$ of $K_{61}$  (see Section \ref{3uniform}).Let  $\alpha_{61} := a_{3} - a_0$ and $\beta_{61} := a_{49} - a_{0} \in \mathbb{R}^2$. Similarly as above, define $H_{61} := \langle x\mapsto x+\alpha_{61}, x\mapsto x+\beta_{61}\rangle$ and 
\begin{align*}
  G_{61} & =\{ \alpha : x\mapsto \varepsilon x + m_{61}\alpha_{61} + n_{61}\beta_{61}    \, : \, \varepsilon=\pm 1, m_{61}, n_{61} \in \ZZ\} \cong H_{61}\rtimes \mathbb{Z}_2.
\end{align*}
By the same arguments as above and in Claim 1, $\Gamma_{61} \unlhd G_{61}$ and the number of $G_{61}/\Gamma_{61}$-orbits of vertices of $X_{61}$ is eight. Therefore, $G_{61}/\Gamma_{61}$ acts on $X_{61} = K_{61}/\Gamma_{61}$.
Since $O_1 = \langle a_0 \rangle,$ $O_2 = \langle a_1 \rangle,$  $O_3 = \langle a_{2} \rangle$, $O_4 = \langle a_{15} \rangle,$ $O_5 = \langle b_{0} \rangle$, $O_6 = \langle b_{1} \rangle$, $O_6 = \langle b_{14} \rangle$, $O_6 = \langle b_{15} \rangle$
 are the $G_{61}$-orbits, it follows that $O_1 = \langle a_0 \rangle,$ $O_2 = \langle a_1 \rangle,$  $O_3 = \langle a_{2} \rangle$, $O_4 = \langle a_{15} \rangle,$ $O_5 = \langle b_{0} \rangle$, $O_6 = \langle b_{1} \rangle$, $O_6 = \langle b_{14} \rangle$, $O_6 = \langle b_{15} \rangle$ are the $(G_{61}/\Gamma_{61})$-orbits.  
Since the vertex set of $X_{61}$ is $\sqcup_{j=1}^{8} \eta_{61}(O_j)$ and $G_{61}/\Gamma_{61} \leq {\rm Aut}(X_{61})$, it follows that the number of ${\rm Aut}(X_{61})$-orbits of vertices is $\le8$.   
\end{proof}

\section{Acknowledgements}

Marbarisha M. Kharkongor is supported by University Grants Commission (UGC) (1182/(CSIR-UGC NET JUNE 2019)) and Dipendu Maity is supported by NBHM, DAE (No. 02011/9/2021-NBHM(R.P.)/R$\&$D-II/9101).

\section{Conflicts of Interest Statement} 
On behalf of all authors, the corresponding author states that there is no conflict of interest.


{\small

}

\end{document}